\documentclass[a4paper,11pt]{amsart}

\usepackage{amsmath,amsxtra,amssymb,latexsym, amscd,amsthm,dsfont,subcaption,cases,longtable,booktabs,makecell}
\usepackage[mathscr]{eucal}
\usepackage{mathrsfs,bbm,enumerate,calc,yhmath,mathtools,faktor}
\usepackage{stackengine}
\usepackage{pict2e}
\usepackage{tikz}
\usepackage{graphicx}
\usepackage{color}
\usepackage[foot]{amsaddr}

\usepackage[a4paper]{geometry}
\usepackage{hyperref}
\hypersetup{colorlinks=true, breaklinks=true, urlcolor= blue, linkcolor= black, citecolor=teal,
  pdftitle={Friedman-Ramanujan functions II}, 
  pdfauthor={Nalini Anantharaman and Laura Monk}, 
  pdfsubject={}}

\usepackage[capitalize]{cleveref}
\crefname{equation}{equation}{equations}
\Crefname{equation}{Equation}{Equations}
\crefname{figure}{Figure}{Figures}
\Crefname{figure}{Figure}{Figures}

\numberwithin{equation}{section}
\renewcommand{\arraystretch}{1.2}

\newcommand{\para}[1]{ \medskip {\textbf {#1.}}}

\title[Friedman--Ramanujan functions in random hyperbolic
geometry II]{Friedman--Ramanujan functions \protect\\ in random hyperbolic geometry
  \protect\\ and application to spectral gaps II}

\author{Nalini Anantharaman\textsuperscript{1} and Laura Monk\textsuperscript{2}}

\address[1]{Coll\`ege de France, 11 place Marcelin Berthelot, 75005 Paris / IRMA, 7 rue Ren\'e Descartes, 67084 Strasbourg Cedex, France} 

\address[2]{School of Mathematics, University of Bristol, Bristol BS8 1UG, U.K.}

\email{nalini.anantharaman@college-de-france.fr}
\email{laura.monk@bristol.ac.uk}

\makeatletter
\@namedef{subjclassname@2020}{\textup{2020} Mathematics Subject Classification}
\makeatother



\date{\today}

\theoremstyle{plain}
\newtheorem{thm}{Theorem}[section]
\newtheorem{prp}[thm]{Proposition}
\newtheorem{cor}[thm]{Corollary}
\newtheorem{lem}[thm]{Lemma}

\theoremstyle{definition}
\newtheorem{defa}[thm]{Definition}

\newtheorem{rem}[thm]{Remark}

\newtheorem{exa}[thm]{Example}
\newtheorem{nota}[thm]{Notation}

\renewcommand{\d}{\, \mathrm{d}}

\makeatletter
\newcommand*{\ov}[1]{%
  $\m@th\overline{\mbox{#1}}$%
}
\newcommand*{\ovA}[1]{%
  $\m@th\overline{\mbox{#1}\raisebox{3mm}{}}$%
}
\newcommand*{\ovB}[1]{%
  $\m@th\overline{\mbox{#1\rule{0pt}{3mm}}}$%
}
\newcommand*{\ovC}[1]{%
  $\m@th\overline{\mbox{#1\strut}}$%
}
\newcommand*{\ovD}[1]{%
  $\m@th\overline{\mbox{#1\vphantom{\"A}}}$%
}
\newcommand*{\ovE}[1]{%
  $\m@th\overline{\raisebox{0pt}[1.2\height]{#1}}$%
}
\newcommand*{\ovF}[1]{%
  $\m@th\overline{\raisebox{0pt}[\dimexpr\height+0.3mm\relax]{#1}}$%
}
\newcommand*{\ovG}[1]{%
  $\m@th\overline{\raisebox{0pt}[\dimexpr\height+1mm\relax]{#1\vphantom{A}}}$%
}
\makeatother
\newcommand\runderset[2][\sim]{\mathrel{\ensurestackMath{%
  \stackengine{-.2pt}{\scriptscriptstyle#2}{#1}{O}{c}{F}{F}{S}}}}

\newcommand{\N}{\mathbb{N}}
\newcommand{\Z}{\mathbb{Z}}

\newcommand{\R}{\mathbb{R}}
\newcommand{\bD}{\mathbf{D}}

\newcommand{\C}{\mathbb{C}}

\newcommand{\D}{\mathcal{D}}

\DeclareMathOperator{\dist}{dist}
\DeclareMathOperator{\Cell}{Cell}
\DeclareMathOperator{\Junc}{Junc}
\DeclareMathOperator{\Br}{Br}

\DeclareMathOperator{\Id}{id}

\DeclareMathOperator{\hyp}{hyp}

\DeclareMathOperator{\id}{id}

\DeclareMathOperator\argch{argcosh}
\DeclareMathOperator\argcosh{argcosh}
\DeclareMathOperator\argsh{argsinh}

\DeclareSymbolFont{extraup}{U}{zavm}{m}{n}
\DeclareMathSymbol{\varheart}{\mathalpha}{extraup}{86}
\DeclareMathSymbol{\vardiamond}{\mathalpha}{extraup}{87}
\newcommand{\nwc}{\newcommand}
\nwc{\mf}{\mathbf} %Latex (as in \bf not tilted math letters)
\nwc{\blds}{\boldsymbol} %Latex 
\nwc{\ml}{\mathcal} %Latex

\nwc{\lam}{\lambda}
\nwc{\del}{\delta}
\nwc{\Del}{\Delta}
\nwc{\Lam}{\Lambda}
\nwc{\elll}{\ell}
\newcommand{\av}[2][\mathrm{all}]{\langle #2 \rangle_g^{{#1}}}
\newcommand{\avN}[2][\mathrm{all}]{\langle #2 \rangle_{g, Q}^{{#1}}}

\newcommand{\avTbtf}[2][\mathrm{all}]{\langle #2 \rangle_{g, \kappa, \tfL}^{{\type}}} 
 
\newcommand{\avb}[2][\mathrm{all}]{\left\langle #2 \right\rangle_g^{{#1}}}
\newcommand{\pavbtf}[2][]{\left\langle #2 \right\rangle_{g, \kappa, \tfL}^{{#1}}}
\newcommand{\avbtf}[2][]{\left\langle #2 \right\rangle_{g, \kappa, \tfL, \chic}^{{#1}}}

\newcommand{\tfL}{R}

\nwc{\IA}{\mathbb{A}} %algebraic
\nwc{\rA}{\mathrm{A}} 
\nwc{\Jn}{\mathrm{Jn}} 
\nwc{\IB}{\mathbb{B}} %ball
\nwc{\IC}{\mathbb{C}} %complex
\nwc{\ID}{\mathbb{D}} %Dedekind
\nwc{\IE}{\mathbb{E}} %Euklides
\nwc{\IF}{\mathbb{F}} %finite field
\nwc{\IG}{\mathbb{G}} %Gauss
\nwc{\IH}{\mathbb{H}} %Hilbert\N-subgroup
\nwc{\IN}{\mathbb{N}} %natural
\nwc{\IP}{\mathbb{P}} %prime
\nwc{\IQ}{\mathbb{Q}} %rational
\nwc{\IR}{\mathbb{R}} %real
\nwc{\IS}{\mathbb{S}} %sphere
\nwc{\IT}{\mathbb{T}} %torus
\nwc{\IZ}{\mathbb{Z}} %integers
\def\bbbone{{\mathchoice {1\mskip-4mu {\rm{l}}} {1\mskip-4mu {\rm{l}}}
{ 1\mskip-4.5mu {\rm{l}}} { 1\mskip-5mu {\rm{l}}}}}
\def\bbleft{{\mathchoice {[\mskip-3mu {[}} {[\mskip-3mu {[}}{[\mskip-4mu {[}}{[\mskip-5mu {[}}}}
\def\bbright{{\mathchoice {]\mskip-3mu {]}} {]\mskip-3mu {]}}{]\mskip-4mu {]}}{]\mskip-5mu {]}}}}
\nwc{\setK}{\bbleft 1,K \bbright}
\nwc{\setN}{\bbleft 1,\cN \bbright}
 \newcommand{\Lim}{\mathop{\longrightarrow}\limits}

\nwc{\va}{{\bf a}}
\nwc{\vb}{{\bf b}}
\nwc{\vc}{{\bf c}}
\nwc{\vd}{{\bf d}}
\nwc{\ve}{{\bf e}}
\nwc{\vf}{{\bf f}}
\nwc{\vg}{{\bf g}}
\nwc{\vh}{{\bf h}}
\nwc{\vI}{{\bf I}}
\nwc{\vj}{{\bf j}}
\nwc{\vk}{{\bf k}}
\nwc{\vl}{{\bf l}}
\nwc{\vm}{{\bf m}}
\nwc{\vM}{{\bf M}}
\nwc{\vn}{{\bf n}}
\nwc{\vo}{{\it o}}
\nwc{\vp}{{\bf p}}
\nwc{\vq}{{\bf q}}
\nwc{\vr}{{\bf r}}
\nwc{\vs}{{\bf s}}
\nwc{\rz}{{\mathrm z}}
\nwc{\vt}{{\bf t}}
\nwc{\vu}{{\bf u}}
\nwc{\vv}{{\bf v}}
\nwc{\vw}{{\bf w}}
\nwc{\vx}{{\bf x}}
\nwc{\vy}{{\bf y}}
\nwc{\vz}{{\bf z}}
\nwc{\bal}{\blds{\alpha}}
\nwc{\bep}{\blds{\epsilon}}
\nwc{\barbep}{\overline{\blds{\epsilon}}}
\nwc{\bnu}{\blds{\nu}}
\nwc{\bmu}{\blds{\mu}}
\nwc{\bet}{\blds{\eta}}

\nwc{\bk}{\blds{k}}
\nwc{\bm}{\blds{m}}
\nwc{\bM}{\blds{M}}
\nwc{\bp}{\blds{p}}
\nwc{\bq}{\blds{q}}
\nwc{\bn}{\blds{n}}
\nwc{\bv}{\blds{v}}
\nwc{\bw}{\blds{w}}
\nwc{\bx}{\blds{x}}
\nwc{\bxi}{\blds{\xi}}
\nwc{\by}{\blds{y}}
\nwc{\bz}{\blds{z}}

\nwc{\cA}{\ml{A}}
\nwc{\cB}{\ml{B}}
\nwc{\cC}{\ml{C}}
\nwc{\cD}{\ml{D}}
\nwc{\cE}{\ml{E}}
\nwc{\cF}{\ml{F}}
\nwc{\cG}{\ml{G}}
\nwc{\cH}{\ml{H}}
\nwc{\cI}{\ml{I}}
\nwc{\cJ}{\ml{J}}
\nwc{\cK}{\ml{K}}
\nwc{\cL}{\ml{L}}
\nwc{\cM}{\ml{M}}
\nwc{\cN}{\ml{N}}
\nwc{\cO}{\ml{O}}
\nwc{\cP}{\ml{P}}
\nwc{\cQ}{\ml{Q}}
\nwc{\cR}{\ml{R}}
\nwc{\cS}{\ml{S}}
\nwc{\cT}{\ml{T}}
\nwc{\cU}{\ml{U}}
\nwc{\cV}{\ml{V}}
\nwc{\cW}{\ml{W}}
\nwc{\cX}{\mlf{X}}
\nwc{\cY}{\ml{Y}}
\nwc{\cZ}{\ml{Z}}

\nwc{\rb}{\mathrm{b}}
\nwc{\rl}{\mathrm{l}}
\nwc{\rr}{\mathrm{r}}
\nwc{\ru}{\mathrm{u}}
\nwc{\rt}{\mathrm{t}}
\nwc{\rK}{\mathrm{K}}
\nwc{\rM}{\mathrm{M}}
\nwc{\rN}{\mathrm{N}}

\newcommand{\eqc}[1]{[ #1 ]_{\mathrm{loc}}}
\newcommand{\eq}{\, \raisebox{-1mm}{$\runderset{\mathrm{loc}}$} \,}

\newcommand{\MC}{\mathrm{MC}}

\nwc{\fA}{\mathfrak{a}}
\nwc{\fB}{\mathfrak{b}}
\nwc{\fC}{\mathfrak{c}}
\nwc{\fD}{\mathfrak{d}}
\nwc{\fE}{\mathfrak{e}}
\nwc{\fF}{\mathfrak{f}}
\nwc{\fG}{\mathfrak{g}}
\nwc{\fH}{\mathfrak{h}}
\nwc{\fI}{\mathfrak{i}}
\nwc{\fJ}{\mathfrak{j}}
\nwc{\fK}{\mathfrak{k}}
\nwc{\fL}{\mathfrak{l}}
\nwc{\fM}{\mathfrak{m}}
\nwc{\fN}{\mathfrak{n}}
\nwc{\fO}{\mathfrak{o}}
\nwc{\fP}{\mathfrak{p}}
\nwc{\fQ}{\mathfrak{q}}
\nwc{\fR}{\mathfrak{R}}
\nwc{\fS}{\mathfrak{s}}
\nwc{\fT}{\mathfrak{t}}
\nwc{\fU}{\mathfrak{u}}
\nwc{\fV}{\mathfrak{v}}
\nwc{\fW}{\mathfrak{w}}
\nwc{\fX}{\mathfrak{x}}
\nwc{\fY}{\mathfrak{y}}
\nwc{\fZ}{\mathfrak{z}}

\nwc{\tA}{\widetilde{A}}
\nwc{\tB}{\widetilde{B}}
\nwc{\tE}{E^{\vareps}}
\nwc{\tk}{\tilde k}
\nwc{\tN}{\tilde N}
\nwc{\bS}{\mathbf{S}}
\nwc{\bP}{\mathbf{P}}
\nwc{\bI}{\mathbf{I}}
\nwc{\fn}{\mathrm{fn}}
\nwc{\Term}{\mathrm{Term}}
\nwc{\TermBC}{\mathrm{Term}_{\mathrm{BC}}(\mathbf{\Gamma},\mathbf{d})}
\nwc{\n}{n_{\mathbf{S}}}
\nwc{\g}{g_{\mathbf{S}}}
\nwc{\nn}{n_{\mathbf{S}'}}
\nwc{\Gg}{g_{\mathbf{S}'}}
\nwc{\tP}{\widetilde{P}}
\nwc{\tQ}{\widetilde{Q}}
\nwc{\tR}{\widetilde{R}}
\nwc{\tV}{\widetilde{V}}
\nwc{\tW}{\widetilde{W}}
\nwc{\ty}{\tilde y}
\nwc{\teta}{\tilde \eta}
\nwc{\tdelta}{\tilde \delta}
\nwc{\tlambda}{\tilde \lambda}
\nwc{\ttheta}{\tilde \theta}
\nwc{\tvartheta}{\tilde \vartheta}
\nwc{\tPhi}{\widetilde \Phi}
\nwc{\tpsi}{\tilde \psi}
\nwc{\tmu}{\tilde \mu}

\nwc{\To}{\longrightarrow} %limits

\nwc{\ad}{\rm ad}
\nwc{\eps}{\epsilon}
\nwc{\ep}{\epsilon}
\nwc{\vareps}{\varepsilon}

\def\ep{\epsilon}

\def\Tr{{\rm Tr}}

\def\sq2{\sqrt{2}}

\def\t2{{\mathbb T}^2}
\def\s2{{\mathbb S}^2}

\def\N{\mathbb{N}}

\def\R{\mathbb{R}}

\def\Z{\mathbb{Z}}
\def\C{\mathbb{C}}
\def\O{\mathcal{O}}

\nwc{\lap}{\bigtriangleup}
\nwc{\rest}{\restriction}
\nwc{\Diff}{\operatorname{Diff}}
\nwc{\diam}{\operatorname{diam}}
\nwc{\Res}{\operatorname{Res}}
\nwc{\Spec}{\operatorname{Spec}}
\nwc{\Vol}{\operatorname{Vol}}
\nwc{\Op}{\operatorname{Op}}
\nwc{\supp}{\operatorname{supp}}
\nwc{\Span}{\operatorname{span}}

\nwc{\dia}{\varepsilon}
\nwc{\cut}{f}
\nwc{\qm}{u_\hbar}

\def\hto0{\xrightarrow{\hbar\to 0}}

\def\rto0{\xrightarrow{r\to 0}}

\providecommand{\norm}[1]{\lVert#1\rVert}

\nwc{\la}{\langle}
\nwc{\ra}{\rangle}
\nwc{\lp}{\left(}
\nwc{\rp}{\right)}

\nwc{\bequ}{\begin{equation}}
\nwc{\be}{\begin{equation}}
\nwc{\ben}{\begin{equation*}}
\nwc{\bea}{\begin{eqnarray}}
\nwc{\bean}{\begin{eqnarray*}}
\nwc{\bit}{\begin{itemize}}
\nwc{\bver}{\begin{verbatim}}

\newcommand{\cop}{{\mathbf{c}}^{\mathrm{op}}}

\nwc{\eequ}{\end{equation}}
\nwc{\ee}{\end{equation}}
\nwc{\een}{\end{equation*}}
\nwc{\eea}{\end{eqnarray}}
\nwc{\eean}{\end{eqnarray*}}
\nwc{\eit}{\end{itemize}}
\nwc{\ever}{\end{verbatim}}

\makeatletter
\newlength{\temp@wc@width}
\newlength{\temp@wc@height}
\newcommand{\widecheck}[1]{%
  \setlength{\temp@wc@width}{\widthof{$#1$}}%
  \setlength{\temp@wc@height}{\heightof{$#1$}}%
  #1\hspace{-\temp@wc@width}%
  \raisebox{\temp@wc@height+2pt}[\heightof{$\widehat{#1}$}]%
     {\rotatebox[origin=c]{180}{\vbox to 0pt{\hbox{$\widehat{\hphantom{#1}}$}}}}%
}
\makeatother

\newcommand{\Volwp}[1][g]{\mathrm{Vol}_{#1}^{\mathrm{\scriptsize{WP}}}}
\newcommand{\Pwpo}{\mathbb{P}_g^{\mathrm{\scriptsize{WP}}}}
\newcommand{\Ewpo}{\mathbb{E}_g^{\mathrm{\scriptsize{WP}}}}

\newcommand{\Pwp}[1]{\Pwpo \left( #1 \right)}
\newcommand{\Ewp}[2][g]{\mathbb{E}_{#1}^{\mathrm{\scriptsize{WP}}} \Bigg[ #2 \Bigg]}

\DeclarePairedDelimiter{\paren}{(}{)}

\DeclarePairedDelimiter{\abso}{|}{|}
\DeclarePairedDelimiter{\brac}{[}{]}

\let\div\relax
\newcommand{\div}[1]{\paren*{\frac{#1}{2}}}

\DeclareFontFamily{OMX}{yhex}{}
\DeclareFontShape{OMX}{yhex}{m}{n}{<->yhcmex10}{}
\DeclareSymbolFont{yhlargesymbols}{OMX}{yhex}{m}{n}
\DeclareMathAccent{\wideparen}{\mathord}{yhlargesymbols}{"F3}

\newcommand{\Dist}{\vec{\mathrm{d}}}

\renewcommand{\O}[2][ ]{\mathcal{O}_{#1} \left( #2 \right)}

\makeatletter
\newcommand{\smallbullet}{} % for safety
\DeclareRobustCommand\smallbullet{%
  \mathord{\mathpalette\smallbullet@{0.5}}%
}
\newcommand{\smallbullet@}[2]{%
  \, \vcenter{\hbox{\scalebox{#2}{$\m@th#1\bullet$}}} \,%
}
\makeatother

\newcommand{\1}[1]{\mathds{1}_{#1}}

\DeclareMathOperator{\sign}{Sign}
\DeclareMathOperator{\comp}{Comp}

\newcommand{\Lambdabeta}{\Lambda^\beta}
\newcommand{\LambdaFN}{\Lambda^{\Gamma}}
\newcommand{\LambdaBC}{\Lambda_{\mathrm{BC}}}
\newcommand{\LambdaBCbeta}{\Lambda_{\mathrm{BC}}^\beta}
\newcommand{\LambdaBCFN}{\Lambda_{\mathrm{BC}}^{\Gamma}}
\newcommand{\Lambdain}{\Lambda_{\mathrm{in}}}
\newcommand{\LambdainFN}{\Lambda_{\mathrm{in}}^{\Gamma}}
\newcommand{\Lambdainbeta}{\Lambda_{\mathrm{in}}^{\beta}}

\newcommand{\Wbeta}{W^\beta}
\newcommand{\WBCFN}{W^\Gamma}
\newcommand{\Vbeta}{V^\beta}
\newcommand{\VFN}{V^{\Gamma}}

\newcommand{\FR}{\mathcal{F}}
\newcommand{\FRrem}{\mathcal{R}}

\newcommand{\x}{\mathbf{x}}
\newcommand{\z}{\mathbf{z}}

\newcommand{\domain}{\mathfrak{D}}
\newcommand{\LocTFlc}{\mathrm{Loc}_{\Sf,\chic, \chitau}^{ \kappa, \tfL, L}}

\newcommand{\geod}{\mathcal{G}}

\newcommand{\Atf}{\mathrm{TF}_g^{\kappa,\tfL}}
\newcommand{\Sf}{\mathbf{S}}
\newcommand{\type}{\mathbf{T}}
\newcommand{\ftype}{\mathfrak{T}}
\newcommand{\curve}{\mathbf{c}}
\newcommand{\curv}{\mathbf{\zeta}}
\newcommand{\chic}{{\chi_+}}
\newcommand{\chitau}{\chi_+'}
\newcommand{\chictau}{{\chi_+''}}
\newcommand{\thetaAc}{\vec{\theta}_{\mathrm{ac}}}
\newcommand{\ThetaAc}{\Theta_{\mathrm{ac}}}
\newcommand{\thetaNe}{\vec{\theta}_{\mathrm{nt}}}
\newcommand{\ThetaNe}{\Theta_{\mathrm{nt}}}
\newcommand{\nac}{n_{\mathrm{ac}}}
\newcommand{\tausurv}{\vec{\tau}_{\mathrm{surv}}}
\newcommand{\Thetasurv}{\Theta_{\mathrm{surv}}}

\DeclareMathOperator{\step}{step}
\newcommand{\vi}{i_j}

\newcommand{\cc}{\mathfrak{n}}
\newcommand{\fz}{\chi_{\textrm{z}}}
\newcommand{\fcr}[1]{\chi_{\mathcal{C},#1}}
\newcommand{\fcell}[1]{\chi_{\Gamma,#1}}
\newcommand{\bfcell}[1]{\bar\chi_{\Gamma,#1}}
\newcommand{\CE}{l_0}
\newcommand{\Pol}{\mathrm{Pol}}
\newcommand{\dopt}{\bD^{\mathrm{opt}}}
\newcommand{\Sinn}{\Sf_{\mathrm{in}}}
\newcommand{\Sb}{\Sf_{\partial}}
\newcommand{\ord}{N}

\begin{document}

\maketitle

\begin{abstract}
  The core focus of this series of two articles is the study of the distribution of the length
  spectrum of closed hyperbolic surfaces of genus $g$, sampled randomly with respect to the
  Weil--Petersson probability measure. In the first article \cite{Ours1}, we introduced a notion of
  {\em{local topological type}}~${\mathbf{T}}$, and established the existence of a \emph{density
    function} $V_g^{\mathbf{T}}(\ell)$ describing the distribution of the lengths of all closed
  geodesics of type $\mathbf{T}$ in a genus $g$ hyperbolic surface. We proved that
  $V_g^{\mathbf{T}}(\ell)$ admits an asymptotic expansion in powers of $1/g$. We introduced a new
  class of functions, called \emph{Friedman--Ramanujan functions}, and related it to the study of
  the spectral gap $\lambda_1$ of the Laplacian.

  In this second part, we provide a variety of new tools allowing to compute and estimate the volume
  functions $V_g^{\mathbf{T}}(\ell)$. Notably, we construct new sets of coordinates on Teichm\"uller
  spaces, distinct from Fenchel--Nielsen coordinates, in which the Weil--Petersson volume has a
  simple form. These coordinates are tailored to the geodesics we study, and we can therefore prove
  nice formulae for their lengths. We use these new ideas, together with a notion of
  pseudo-convolutions, to prove that the coefficients of the expansion of $V_g^{\mathbf{T}}(\ell)$
  in powers of $1/g$ are Friedman--Ramanujan functions, for any local topological type
  $\mathbf{T}$. We then exploit this result to prove that, for any $\epsilon>0$,
  $\lambda_1 \geq \frac14 - \epsilon$ with probability going to one as $g \rightarrow + \infty$, or,
  in other words, typical hyperbolic surfaces have an asymptotically optimal spectral gap.
 \end{abstract}

\tableofcontents
 
\section{Introduction}
\label{s:intro}

\subsection{Main results and motivations}

For an integer $g \geq 2$, let $S_g$ be a fixed compact oriented surface of genus $g$ with no
boundary. This is the second part of a two-part article, studying the distribution of the length
spectrum of a random hyperbolic metric $X$ on $S_g$. Our probability space is the moduli space
$\cM_g$ of hyperbolic structures on $S_g$, endowed with the Weil--Petersson measure $\Pwpo$
normalized to be a probability measure.

\subsubsection{Local topological type and associated density}
In the first part \cite{Ours1}, we defined a notion of {\em{local topological type}} for a loop
\cite[\S 4]{Ours1}. Intuitively, a local type $\mathbf{T}$ corresponds to the topological datum of a
loop $\mathbf{c}$ filling a surface $\mathbf{S}$, so that we know the topology of $\mathbf{c}$ and
its regular neighbourhood, but not the way it is embedded in the ambient surface of genus
$g$. Examples of local topological types include the type $\mathbf{s}$ ``simple'' (loops with no
self-intersections), or the figure-eight (with exactly one self-intersection). We developed new
tools allowing to study the following averages.

\begin{defa}
  \label{def:av_type}
  Let ${\mathbf{T}}= \eqc{{\mathbf{S, c}}}$ be a local topological type. For any measurable function
  $F : \R_{\geq 0} \rightarrow \C$, bounded and compactly supported, we define the
  \emph{${\mathbf{T}}$-average} of $F$ to be
  \begin{equation} \label{e:avT}
    \av[{\mathbf{T}}]{F} := \Ewp{\sum_{\substack{\gamma \in  
  \geod(X) \\ \gamma\in {\mathbf{T}}}} F(\ell_X(\gamma))}
  \end{equation}
  where $\Ewpo$ denotes the expectation on the moduli space $\cM_g$ when the hyperbolic structure
  $X\in \cM_g$ is chosen randomly according to the normalized Weil--Petersson measure $\Pwpo$, and
  \begin{itemize}
  \item $\geod(X)$ is the set of primitive oriented closed geodesics for the metric $X$;
  \item for any $\gamma\in\geod(X)$ , $\ell_X(\gamma)$ is the length of $\gamma$ for
    the hyperbolic metric $X$;
%\item $F : \R_{\geq 0} \rightarrow \R$ is a measurable function with compact support.
\item $\gamma\in {\mathbf{T}}$ means that the sum is restricted to closed geodesics of local
  topological type ${\mathbf{T}}$.
 \end{itemize}
\end{defa}

Taking $F$ to be the indicator function $\bbbone_{[a, b]}$ of an interval $[a,b]$, the average
$\av[{\mathbf{T}}]{F}$ is the expected number of closed primitive geodesics of local topological
type ${\mathbf{T}}$ and length in the interval $[a, b]$.  In the case of the local topological type
``simple'', denoted as $\mathbf{s}$, the condition $\gamma \in \mathbf{s}$ means that the sum is
restricted to all simple primitive closed geodesics.

We proved in {\cite[Prop. 5.11]{Ours1}} that the averages $\av[\mathbf{T}]{F}$ can be written in an
integral form, against a density function $V_g^{\mathbf{T}}(\ell)$ which is, informally, the
``proportion'' of surfaces $X \in \cM_g$ containing a primitive closed geodesic of length $\ell$ and
local topological type $\mathbf{T}$. The formal definition is as follows.
\begin{prp}
  \label{prp:existence_density}
  For any local topological type ${\mathbf{T}}$, there exists a unique locally integrable function
  $V_g^{\mathbf{T}}: \R_{> 0} \rightarrow \R_{\geq 0}$ such that, for any measurable function $F$
  bounded with compact support,
  \begin{equation*}
    \av[{\mathbf{T}}]{F} = \frac{1}{V_g} \int_{0}^{+\infty} F(\ell) \,V_g^{{\mathbf{T}}}(\ell) \d \ell.
  \end{equation*}
\end{prp}

The functions $V_g^{{\mathbf{T}}}(\ell)$ are fundamental objects for the study of geodesics in
random hyperbolic surface.  However they have been little studied, except in the case of the local
topological type ``simple''. For simple closed geodesics, Mirzakhani showed in \cite{mirzakhani2007}
that the functions $V_g^{\mathbf{T}}(\ell)$ are polynomial in $\ell$ (hence they are called
\emph{volume polynomials}), and that they can be computed explicitly for all genus~$g$ by
topological recursion. Otherwise, we call $\ell \mapsto V_g^{{\mathbf{T}}}(\ell)$ the \emph{volume
  function} associated with local topological type~${\mathbf{T}}$ on surfaces of genus $g$.

The precise nature of the functions $V_g^{\mathbf{T}}$ is not known. One can deduce from
\cite[Theorem 8.1]{mirzakhani2016} that $V_g^{{\mathbf{T}}}(\ell)$ behaves like a power of $\ell$ to
leading order as $\ell\rightarrow +\infty$ with $g$ fixed. We believe that, for local topological
type other than ``simple'', the function $\ell \mapsto V_g^{\mathbf{T}}(\ell)$ is not polynomial,
but is exponentially close to one as $\ell\rightarrow + \infty$ with $g$ fixed.

\subsubsection{Large-genus asymptotic expansion}

The core focus of this paper is the large genus behavior of the volume functions
$V_g^{\mathbf{T}}$. We proved in the first article~\cite{Ours1} that they admit expansions to any
order in powers of $g^{-1}$, in the following sense.

 \begin{thm}[{\cite[Theorem 1.15]{Ours1}}]
  \label{thm:exist_asympt_type_intro}
  For any local topological type ${\mathbf{T}} = \eqc{{\mathbf{S, c}}}$, there exists a unique
  family of locally integrable functions $(f_k^{{\mathbf{T}}})_{k \geq 0}$ such that, for any integers
  $\ord \geq 0$ and $g \geq 2$, any $\epsilon>0$,
  \begin{equation}
    \label{eq:exist_asympt_type_intro}
    \Big|\av[{\mathbf{T}}]{F} - \sum_{k=0}^\ord
    \frac{1}{g^k} \int_0^{+ \infty} F(\ell) \,f_k^{{\mathbf{T}}}(\ell) \d \ell \Big| 
    \leq C_{\ord,\chi(\mathbf{S}),\epsilon}
    \frac{\sup_{\ell \geq 0} \{|F(\ell)| e^{(1+\epsilon)\ell}\}}{g^{\ord+1}}
  \end{equation}
  for any measurable function $F : \R_{\geq 0} \rightarrow \R$ bounded with compact support.
\end{thm}

As the notation indicates, the constant $C_{\ord,\chi(\mathbf{S}),\epsilon}$ depends on the order
$\ord$ of the expansion, the Euler characteristic of the filled surface ${\mathbf{S}}$ and the small
parameter $\epsilon$, but not on the specific loop ${\mathbf{c}}$ filling $\mathbf{S}$.  In
forthcoming applications of our results to the spectral gap of the Laplacian, this uniform behavior
with respect to all loops filling the same surface is important.

 \subsubsection{Statement of the main results}

 The main technical result of this paper may be summarised as follows, answering positively the
 question asked in  \cite{Ours1}.

\begin{thm}
  \label{con:FR_type}
  Let ${\mathbf{T}}$ be a local topological type other than ``simple''. For any $k \geq 0$, the
  function $\ell \mapsto f_k^{{\mathbf{T}}}(\ell)$ is a Friedman--Ramanujan function in the weak
  sense.
\end{thm}

The notion of Friedman--Ramanujan function was introduced in \cite[\S 3]{Ours1}. A locally
integrable function $f : \R_{\geq 0} \rightarrow \R$ is said to be a Friedman--Ramanujan function if
there exists a polynomial function $p$ such that $f(\ell) - p(\ell) e^{\ell}$ grows at most like a
polynomial multiple of $e^{\ell/2}$. Precise definitions follow in \S \ref{s:full_FR}, together with
a quantitative version of our main result (see Theorem~\ref{thm:FR_type}).

\begin{rem}
  The case of the local topological type $\mathbf{s}$ is excluded because, in this case, the
  function $f_k^{\mathbf{s}}$ have pole of order one at $0$. We prove that
  $\ell \mapsto \ell f_k^{\mathbf{s}}(\ell)$ are Friedman--Ramanujan functions
  in~\cite{anantharaman2022} (see \cite[Proposition 3.4]{Ours1} for more details).
\end{rem}

Theorem \ref{con:FR_type} has already been obtained when $\mathbf{T}$ is a figure-eight, or more
generally when the filled surface $\mathbf{S}$ has Euler characteristic $-1$, in~\cite{Ours1}. For
this relatively simple topological situation, we could use almost explicit expressions of the
functions $(f_k^{{\mathbf{T}}})_{k \geq 0}$.

The present paper deals with general local topological types and uses entirely new methods,
involving finding explicit expressions for the Weil--Petersson measure in coordinates other than the
Fenchel--Nielsen ones, and for the lengths of non-simple geodesics in those coordinates.

This work is largely motivated by the question of estimating the gap at the bottom of the spectrum
of the Laplacian, for typical hyperbolic surfaces, as detailed in \cite{Ours1,Expo}. In the first
part of this article \cite{Ours1}, we used the result when ${\mathbf{S}}$ has Euler characteristic
$-1$ to prove that typical surfaces have spectral gap at least $2/9-\eps$.  More generally, our
theorem (together with the variant presented in \S \ref{s:variantLC}) implies the optimal lower
bound $1/4-\eps$ on the spectral gap.

\begin{thm}\label{t:dream} Denote by $\lambda_1(X)$ the first non-trivial eigenvalue of Laplacian on
  the hyperbolic surface~$X$. Then, for any $\epsilon >0$,
\begin{align} \label{e:dream} \Pwp{\lambda_1(X) \geq \frac{1}{4}  - \epsilon} \Lim_{g\rightarrow +\infty} 1.
\end{align}
\end{thm}

\subsection{Full statement of the Friedman--Ramanujan property}
\label{s:full_FR}

Let us introduce a few definitions, and present our main result more precisely.

\subsubsection{Definitions}

We recall from \cite[\S 3]{Ours1} the \emph{strong} and \emph{weak} definitions of the class of
Friedman--Ramanujan functions, directly inspired by J. Friedman's seminal work on random graphs
\cite[Definition 2.1]{friedman2003} where he proves the Alon conjecture. 
\begin{defa}\label{def:FR}
  Let $\rN\geq 1$. We define the class of functions $\cR^\rN$ as the set of locally integrable
  functions $f : \R_{\geq 0} \rightarrow \C$ such that there exists $c>0$ satisfying
 \begin{equation}\label{e:strong}
    \forall \ell \geq 0, \quad
    \abso{f(\ell)} \leq c  (\ell + 1)^{\rN-1}  e^{\frac \ell 2}.
  \end{equation}
  We let $\cR=\bigcup_{\rN\geq 1} \cR^\rN$.
  A locally integrable function $f : \R_{\geq 0} \rightarrow \C$ is said to be a \emph{Friedman--Ramanujan
    function} if there exists a polynomial function $p$ such that
  \begin{equation*}
    f(\ell) - p(\ell) e^\ell \in \cR.
  \end{equation*}
  We denote as $\cF$ the set of Friedman--Ramanujan function.  For $f \in \cF$, the function
  $\ell\mapsto p(\ell) e^\ell $ is uniquely defined, we call it the \emph{principal term} of
  $f$.
 %, and we define the \emph{degree} of $f$ as a Friedman--Ramanujan function as the degree of the polynomial $p$. We denote as $\FR$ the class of Friedman--Ramanujan functions. 
\end{defa}

\begin{nota}
  For an integer $\rK \geq 0$, we denote as $\C_{\rK-1}[\ell]$ the space of polynomials of degree
  strictly less than $\rK$, and by $\C_{\rK-1}[\ell] \, e^\ell$ the class of functions of the form
  $p(\ell) \, e^\ell$ where $p\in \C_{\rK-1}[\ell]$. Note that these sets are reduced to the
  constant function equal to $0$ if $\rK=0$.
\end{nota}

\begin{defa}
  For $\rK\geq 0$, we denote as $\FR^\rK$ the subset of Friedman--Ramanujan functions of principal
  term in $\C_{\rK-1}[\ell] \, e^\ell$, and for $\rN \geq 1$,
  $\cF^{\rK,\rN} := \C_{\rK-1}[\ell] \, e^\ell \oplus \cR^\rN$.
\end{defa}

Note that the space $\FRrem$ can be identified with the subset $\cF^0$ of Friedman--Ramanujan
functions with principal term $p=0$, and similarily $\cR^\rN = \cF^{0,\rN}$.

In our proof of Theorem \ref{con:FR_type} we will rather use the following weak version of the
classes $\FRrem$ and~$\FR$. Remark that in the work of Friedman, where the functions are defined
on $\Z_{> 0}$ instead of $\R_{\geq 0}$, this distinction between weak and strong definition does not
exist.

 \begin{defa}\label{def:weakFR}
   Let $\rN \geq 1$. We define the class of functions $\FRrem_w^\rN$ as the set of locally
   integrable functions $f : \R_{\geq 0} \rightarrow \C$ such that there exists $c>0$ satisfying
   \begin{align}\label{e:weak}
     \forall \ell \geq 0, \quad
     \int_0^{\ell} |f(s)| \d s \leq c (\ell+1)^{\rN-1} e^{\frac \ell 2}.
   \end{align}
  We let $\cR_w=\bigcup_{\rN \geq 1} \cR_w^\rN$.
  A locally integrable function $f : \R_{\geq 0} \rightarrow \C$ is said to be a
  \emph{Friedman--Ramanujan function in the weak sense} if there exists a polynomial function $p$
  such that
  \begin{equation*}
    f(\ell) - p(\ell) e^\ell \in \cR_w.
  \end{equation*}
  For $\rK\geq 0$, we write $\FR_w^\rK := \C_{\rK-1}[\ell] \, e^\ell \oplus \cR_w$ and for
  $\rN \geq 1$, $\cF^{\rK,\rN}_w := \C_{\rK-1}[\ell] \, e^\ell \oplus \cR^\rN_w$.
\end{defa}

  As the name suggests, the strong definition
implies the weak one.

% \begin{rem}
%   The property $f\in\FRrem_w$ is equivalent to the definition given in \cite{Ours1}, that is,
%   the existence of $a>0$ such that
%    \begin{equation}
%    \label{eq:def_FR_w}
%     \int_{0}^{+ \infty} \frac{|f(\ell)|}{(\ell+1)^{a}
%       e^{\frac \ell 2}}
%     \d \ell < +\infty. 
%   \end{equation}
%   More precisely, if $f\in\FRrem^\rN_w$, then \eqref{eq:def_FR_w} holds for any exponent
%   $a>\rN$. Conversely if \eqref{eq:def_FR_w} holds, then $f\in\FRrem^{a+1}_w$.
%   \end{rem}

  \subsubsection{Norms on the set of Friedman--Ramanujan functions}

  Let us define a norm on the spaces $\cF^{\rK,\rN}$ and $\cF^{\rK,\rN}_w$, which captures the size
  of the principal term and the remainder of the Friedman--Ramanujan function.

  \begin{nota}
    For a polynomial function $p \in \C[\ell]$, we define $\|p\|_{\ell^\infty}$ to be the maximum modulus
    of the coefficients of $p$.
  \end{nota}

  \begin{defa} \label{d:normFR}
  For $\rN \geq 1$, we define a norm on $\cR^\rN$ and $\cR_w^\rN$ by
   \begin{align*}
  \norm{f}_{\cR^\rN}:=\sup_{\ell \geq 0}\frac{ |f(\ell)|}{  (\ell+1)^{\rN-1} e^{\frac \ell 2}}
     \quad \text{and} \quad
    \norm{f}^w_{\cR^\rN}:= \sup_{\ell \geq 0}\frac{\int_0^{\ell} |f(s)| \d s}{  (\ell+1)^{\rN-1}\, e^{\frac \ell 2}}.
 \end{align*} 
 If now $\rK \geq 0$, we equip $\cF^{\rK,\rN}$  with the norm
 \begin{equation*}
   \| f \|_{\FR^{\rK,\rN}} := \| p \|_{\ell^\infty} + \|f(\ell) - p(\ell) e^\ell\|_{\cR^\rN}
 \end{equation*}
 where $p(\ell) \, e^\ell$ is the principal term of $f$. The norm $\|f\|_{\cF^{\rK,\rN}}^w$ is defined
 the same way replacing the last norm by its weak version.
\end{defa}
  
It will be convenient to adopt the following convention.
 \begin{nota} \label{n:fn} We denote as $\fn ((u_i)_{i\in \mathcal{V}})$ any quantity that is a
   function solely of some set of variables $(u_i)_{i\in \mathcal{V}}$ (but for which we do not need
   to write an explicit expression).
\end{nota}

  \begin{rem} \label{rem:boundsFR}
    By definition, if $f\in \cR^{\rN}$, then for all $\ell \geq 0$,
    \begin{align}\label{e:normbound2}
      |f ( \ell)|\leq \norm{f}_{ \cR^{\rN}}\, (1+ \ell)^{\rN-1} e^{\ell/2}.\end{align}
    If we rather assume that $f\in \cF^{\rK,\rN}$, then  for $\ell \geq 0$,
    \begin{align}\label{e:normbound1}
      | f ( \ell)|
      \leq \fn(\rK,\rN) \norm{f}_{\cF^{\rK,\rN}}\, (1+\ell)^{\rK-1} e^{ \ell}\end{align}
    for some constant $\fn(\rK,\rN):=1+ \sup_{\ell \geq 0} (1+\ell)^{\rN-\rK} e^{- \ell/2}$.
  \end{rem} 
 
 \subsubsection{Precise statement}
 
The main technical result of this paper is the following precised version of Theorem \ref{con:FR_type}.
\begin{thm}\label{thm:FR_type}
  Let ${\mathbf{T}} = \eqc{{\mathbf{S, c}}}$ be a local type other than ``simple''. Then, for any
  $k \geq 0$, there exist $\rK, \rN$ and $C$, depending on $k$ and ${\mathbf{S}}$ but not on
  ${\mathbf{c}}$, such that $f_k^{{\mathbf{T}}}\in \cF_w^{\rK, \rN}$ and
    $\norm{f_k^{{\mathbf{T}}}}_{\cF^{\rK, \rN}}^w \leq C$.
\end{thm}

\begin{nota}
  \label{nota:S}
  We denote as $(\g,\n)$ the signature of $\mathbf{S}$, and $\chi(\mathbf{S}) := 2\g-2+\n$ its
  absolute Euler characteristic. We denote as $\partial \mathbf{S}$ the set of boundary components
  of $\mathbf{S}$, which is identified with $\{1, \ldots, \n\}$ through the labelling of the
  boundary of $\mathbf{S}$.
\end{nota}

 \begin{rem}\label{r:firstKN}
   The proof gives a seemingly optimal estimate on $\rK$, and a non-optimal estimate on $\rN$ and
   $C$: we may take $\rK= 2k + 2\chi(\Sf)$ and $\rN= \n (a_{k+1}+1)$, where $a_{k+1}$ is an explicit
   sequence coming from~\cite{anantharaman2022} which is a priori not optimal either.  Remark
   \ref{rem:improve} shows how we may optimise the value of $\rK$ further, depending on the topology
   of the multi-loop $\curve$.
 \end{rem}

\subsection{Outline of the proof}

Let ${\mathbf{T}} = \eqc{{\mathbf{S, c}}}$ be a local topological type, i.e. the topological datum
of a loop $\mathbf{c}$ filling a surface $\mathbf{S}$. We present prior results on the volume
function $V_g^{\mathbf{T}}$ and the coefficients $(f_k^{\mathbf{T}})_{k \geq 0}$ of its asymptotic
expansion in powers of $1/g$. We then explain the key ideas behind the proof of the main theorem.

\subsubsection{Teichm\"uller spaces and the Weil--Petersson measure}
\label{s:teich}

For  $\x =(x_1, \ldots, x_{\n}) \in \R_{>0}^{\n}$, we denote as
$\mathcal{T}_{\g,\n}(\x)$ the Teichm\"uller space of hyperbolic surfaces of genus $\g$ with $\n$
labelled geodesic boundary components of respective lengths $x_1, \ldots, x_{\n}$.  This space is a
real-analytic manifold of dimension $6\g-6+2\n$. It is equipped with the \emph{Weil--Petersson
  symplectic form}, which induces a volume that we call the \emph{Weil--Petersson measure} and
denote as $ \mathrm{Vol}^{\mathrm{WP}}_{\g,\n,\x}$.

 In this article, it will be more interesting to view the boundary lengths as free variables, and
 hence introduce the Teichm\"uller space 
 $$\mathcal{T}_{\g,\n}^*=\{(\x, Y) \, : \x \in \R_{>0}^{\n}, Y\in \mathcal{T}_{\g,\n}(\mathbf{x})\},$$
 which is the space of compact hyperbolic metrics on the surface $\mathbf{S}$.  The space
 $\mathcal{T}_{\g,\n}^*$ is a real-analytic manifold of dimension $6\g-6+3\n=3\chi(\mathbf{S})$. It
 may be endowed with the positive measure
 $\d\mathrm{Vol}^{\mathrm{WP}}_{\g,\n}(\x, Y):= \d \x \d\mathrm{Vol}^{\mathrm{WP}}_{\g,\n}(Y)$,
 where $\d \x$ is the Lebesgue measure on $\R_{>0}^{\n}$, which we also call \emph{Weil--Petersson
   measure}. We will sometimes omit the mention of $\x$ and simply write $Y \in \cT_{\g,\n}^*$ when it
 is not useful to emphasize the boundary length of $Y$.

 \subsubsection{Level-set desintegration} 

 The map $\mathcal{T}_{g_{\mathbf{S}},n_{\mathbf{S}}}^* \ni Y \mapsto \ell_Y({\mathbf{c}})$ is a
 non-constant analytic function on the Teichm\"uller space $\cT_{\g,\n}^*$. We can therefore
 desintegrate the Weil--Petersson measure $ \d\mathrm{Vol}^{\mathrm{WP}}_{\g,\n}$ on its
 level-sets. This yields a measure $\d\mathrm{Vol}^{\mathrm{WP}}_{\g,\n}(Y)/\d\ell$ defined by the
 property that, for any integrable function $F : \cT_{\g,\n}^* \rightarrow \R_{\geq 0}$, we have
\begin{align}\nonumber
\int_{\mathcal{T}_{g_{\mathbf{S}},n_{\mathbf{S}}}^*}  F(Y)  \d\mathrm{Vol}^{\mathrm{WP}}_{\g,\n}(Y) =
  \int_{0}^{+\infty}\Big( \int_{ \substack{ \ell_Y({\mathbf{c}})=\ell}} 
  F(Y)\frac{  \d\mathrm{Vol}^{\mathrm{WP}}_{\g,\n}(Y) }{\d\ell} \Big) \d \ell.
\end{align}
In particular, if $F(\x, Y)= \tilde F( \ell_Y({\mathbf{c}}))\phi(\x)$ is a multiplicative function
of $\ell_Y({\mathbf{c}})$ and $\x$, this equation takes the special form
\begin{align}\nonumber
  \int_{\mathcal{T}_{g_{\mathbf{S}},n_{\mathbf{S}}}^*}
 \tilde F( \ell_Y({\mathbf{c}})) \phi(\x) \d\mathrm{Vol}^{\mathrm{WP}}_{\g,\n}(\x, Y) =
\int_{0}^{+\infty}\tilde F(\ell)\Big( \int_{\ell_Y({\mathbf{c}})=\ell }\phi(\x)\frac{  \d\mathrm{Vol}^{\mathrm{WP}}_{g,n}(\x, Y) }{\d\ell} \Big) \d \ell.
\end{align}

\subsubsection{Description of the volume function and coefficients}

We have proven in \cite[\S 5]{Ours1} that the volume function $V_g^{{\mathbf{T}}}(\ell)$ and the
coefficients $f_k^{{\mathbf{T}}}(\ell)$ are all of the form
\begin{align}\label{e:disint_phi}
  \ell \mapsto \int_{\ell_Y({\mathbf{c}})=\ell }
  \phi(\x)\frac{  \d\mathrm{Vol}^{\mathrm{WP}}_{\g,\n}(\x, Y) }{\d\ell},\end{align}
with particular choices of $\phi(\x)$.

More precisely, the volume function $V_g^{{\mathbf{T}}}(\ell)$ is obtained by replacing $\phi$ with
the following explicit polynomial function:
\begin{align}\label{e:enumeration1}
  \Phi_g^{\mathbf{T}}(\x)
  :=\frac{1}{n(\mathbf{T})}\frac{x_1 \ldots x_{\n}}{V_g} \sum_{\mathfrak{R} \in R_g({\mathbf{S}})}
  V_{\mathfrak{R}}(\x),
\end{align}
which enumerates all possible embeddings of $\mathbf{S}$ in a surface of genus $g$ (the coefficient
$n(\mathbf{T})$ is a combinatorial factor, and $V_{\mathfrak{R}}(\x)$ are products of volumes of
moduli spaces, enumerated by configurations ${\mathfrak{R}} \in R_g(\mathbf{S})$).

From this expression, the functions $f_k^{{\mathbf{T}}}(\ell)$ are then obtained by performing the
asymptotic expansion of $ \Phi_g^{\mathbf{T}}(\x)$ in powers of $1/g$, using the results of
\cite{anantharaman2022, Ours1}, see \cite[Proposition 5.22]{Ours1}.  It follows that
$f_k^{{\mathbf{T}}}(\ell)$ has the form \eqref{e:disint_phi} for a $\phi(\x)$ which is a linear
combination of functions
 \begin{equation}\label{e:each_term}
     \prod_{i \in V_0} x_i^{k_i}
    \prod_{i \in V_+} x_i^{k_i} \cosh \div{x_i}
    \prod_{i \in V_-} x_i^{k_i} \sinh \div{x_i}  \prod_{j=1}^k x_{\vi} \delta(x_{\vi} - x_{\vi'}),
  \end{equation}
  where:
  \begin{itemize}
  \item $V_0$, $V_+$, $V_-$ are three disjoint subsets of $\partial \mathbf{S}$;
  \item $\mathrm{m} = \bigsqcup_{j=1}^k \{\vi, \vi'\}$ is a perfect matching of
    $\partial \mathbf{S} \setminus (V_0\sqcup V_+\sqcup V_-)$;
  \item $\delta$ is the usual Dirac mass.
  \end{itemize}
  The degree $k_i$ is odd if $i\in V_0\cup V_+$ and even if $i\in V_-$. In \cite{anantharaman2022}
  we proved\footnote{It is shown in \cite[\S 5.4]{anantharaman2022} that
    $\sum_{i\in V_+\cup V_-} k_i\leq 2k$. There is an imprecision in \cite[\S 5.4]{anantharaman2022}
    in the estimate of $\sum_{i\in V_+\cup V_-} m_i$ because $m_i$ is either $k_i$ or $k_i-1$
    depending on the parity of $k_i$, but this has no impact on the original estimate of
    $\sum_{i\in V_+\cup V_-} k_i$.}  that $\sum_{i\in V_+\cup V_-} k_i\leq 2k$, and that for
  $i\in V_0$, we have $k_i\leq a_{k+1}$ for a certain sequence $a_{k+1}$ (which is explicit but a
  priori not optimal).
 
  \subsubsection{Brief reminder about the method followed in \cite{anantharaman2022, Ours1}}

  It is worth now recalling the main steps in the proof of the asymptotic expansion of
  $\Phi_g^{\mathbf{T}}(\x)$.

\begin{nota}
  Throughout this article, we shall use the Landau notation $\cO$, i.e. we will write $A = \O{B}$ if
  there exists a constant $c>0$ such that, for any choice of parameters, $|A| \leq c B$. If the
  constant $c$ depends on some parameters, we shall write them as a subscript of the $\cO$.
\end{nota}
 
\begin{itemize}
\item First, we note that, although the cardinality of $R_g({\mathbf{S}}) $ in
  \eqref{e:enumeration1} grows with $g$, we can identify a set of embeddings $R_g({\mathbf{S}}, \ord)$
  of bounded cardinality, such that
  \begin{align*}
    \Phi_g^{\mathbf{T}}(\x)=\frac{1}{n(\mathbf{T})}\frac{x_1 \ldots x_{\n}}{V_g} \sum_{{\mathfrak{R}} \in R_g({\mathbf{S}}, \ord)} V_{\mathfrak{R}}(\x)  +  \O[\ord, \n]{\frac{(\norm{\x} +1)^{\alpha_\ord^{\Sf}} e^{\frac{x_1+\ldots+x_{\n}}2} }{g^{\ord+1}}}
  \end{align*}
  in a weak sense.
  This is done in \cite[Proposition 5.22]{Ours1}.
\item Then, for each realisation ${\mathfrak{R}}\in R_g({\mathbf{S}}, \ord)$, we observe that there exist integers
  $n_1\leq \n$ and $m\leq (\ord+\n)/2$ such that the ratio $V_{\mathfrak{R}}(\x)/V_{\mathfrak{R}}(0)$ coincides (modulo
  permutation of the variables) with $V_{g-m, n_1}(x_1, \ldots, x_{n_1})/V_{g-m, n_1}$
  multiplied by a fixed polynomial in $(x_{n_1+1}, \ldots, x_{\n})$. We recall that, for any integers
  $g,n$ with $2g-2+n>0$, $V_{g,n}$ denotes the constant coefficient of the volume polynomial
  $V_{g,n}(\cdot)$, i.e. its value at $(0, \ldots, 0)$. 
\item 
  The main result of \cite{anantharaman2022} says that
  $V_{g-m, n_1}(x_1, \ldots, x_{n_1})/V_{g-m, n_1}$ coincides, modulo a remainder in 
  $\O[\ord, n_1]{(\norm{\x} +1)^{\alpha_\ord^{n_1}} e^{\frac{x_1+\ldots+x_{n_1}}2}/g^{\ord+1}}$, with
  a linear combination of the functions \eqref{e:each_term}, with degree $k_i\leq a_{\ord+1}$, the
  coefficients being of the form
$$ A_{V_0, V_+, V_-}(n_1, m, \ord, k_1, \ldots, k_{n_1})\,\,\frac{c_{g-m, n_1}(\alpha)}{V_{g-m, n_1}}$$
and $\|\alpha\|_\infty \leq 2\ord+a_{\ord+1}$. Here $a_{\ord+1}$ is a sequence introduced in
\cite{anantharaman2022}, and $c_{g-m, n_1}(\alpha)$ is our notation for the coefficients of volume
polynomials.
\item
At the end we inject the result of Mirzakhani--Zograf \cite[Theorem 4.1]{mirzakhani2015}, showing
the existence of an asymptotic expansion in powers of $g^{-1}$ for
$c_{g-m, n_1}(\alpha)/V_{g-m, n_1}$ and for $V_{\mathfrak{R}}(0)/V_g$, and hence for
$c_{g-m, n_1}(\alpha)/V_{g}$.
\end{itemize}
 
 \begin{rem} The last step is only necessary if we insist on obtaining an expansion $V_g^{{\mathbf{T}}}(\ell)$ in powers of $g^{-1}$. The three first steps, together with the fact that $\frac{c_{g-m, n_1}(\alpha)}{V_{g-m, n_1}}$ and $\frac{V_{\mathfrak{R}}(0)}{V_g}$ are bounded functions of $g$, are actually sufficient to obtain the spectral gap (Theorem \ref{t:dream}).
 \end{rem}
 
 Until \S \ref{s:variantLC}, the surface $\mathbf{S}$ is fixed and we do not need to know how the
 quantity $A_{V_0, V_+, V_-}(n_1, m, \ord, k_1, \ldots, k_{n_1})$ depends on $n_1$. In \S
 \ref{s:variantLC}, we will need to come back to the proof of the asymptotic expansion, in order to
 have a very rough control of the dependency on $n_1, \n$ of those coefficients, as well as the
 constants involved in $\cO_{\ord, \n}{}$ and $\cO_{\ord, n_1}{}$.

 \subsubsection{Unshielded simple portions}
\label{sec:unsh-simple-port}

The following topological quantity, associated to a local topological type, will appear in our
abstract statement. 

\begin{defa}\label{d:double_fill}
  Let $\mathbf{T} = \eqc{\Sf,\mathbf{c}}$ be a local topological type.
  \begin{enumerate}
  \item A \emph{simple portion} of $\mathbf{c}$ is a maximal open subsegment of $\mathbf{c}$ which
    does not contain any self-intersection point of $\mathbf{c}$.
  \item A simple portion is said to be \emph{shielded} if it belongs to the boundary of a contractible
    component of $\mathbf{S}\setminus \mathbf{c}$, and \emph{unshielded} otherwise.
  \item We say $\mathbf{c}$ is  \emph{double-filling} if all simple portions are
    shielded.
  \end{enumerate}
  We denote by $\rho[{\mathbf{T}}]$ the number of unshielded simple portions of $\type$.
\end{defa}

The notion of simple portion and double-filling multi-loop appeared in our previous paper
\cite[Definition 8.3]{Ours1}.  We have the inequality $\rho[{\mathbf{T}}]\leq 2\chi(\Sf)$, which means
that the number of unshielded simple portions is bounded uniformly in $\curve$.

 \subsubsection{Main abstract statement} 
 \label{s:reduction_int}

 We can use \eqref{e:each_term} to reduce the proof of Theorem~\ref{thm:FR_type} to a general
 abstract statement, which we shall now state and comment.  Because our main theorem requires quite
 a few notation, we group them below.

\begin{nota}[Assumptions for main theorem]
  \label{nota:mt} \quad
  \begin{itemize}
  \item Let $\mathbf{T} = \eqc{\mathbf{S},\mathbf{c}}$ be a local type of signature $(\g,\n)$. We
    denote as $\rho[{\mathbf{T}}]$ be the number of unshielded simple portions of ${\mathbf{T}}$.
  \item Let $V$ be a subset of $\partial \mathbf{S}$ such that $\n- \# V=2k$ is even, and $\mathrm{m}$
    be a perfect matching of $\partial \mathbf{S}\setminus V$, that is, a partition
    $\{1, \ldots, \n\}\setminus V=\bigsqcup_{j=1}^k \{\vi, \vi'\}$.
  \item We fix a smooth function $\fz$ that vanishes at order $1$ at $x=0$ and is constant equal
    to~$1$ on $[1, +\infty)$, and such that
  \begin{equation}
   \Big|\frac{\fz(y)^2}{(1-e^{-y})^2}-1\Big|\leq 8 e^{-y}.\label{eq:condition_mu}
 \end{equation}
\item For $j \in V$, let $f_j$ be a function such that
  $\tilde g_j (x):= f_j(x) \sinh\div{x} / \fz(x)^2$ is continuous and lies in $\cF^{\rK_j, \rN_j}$
  for some integers $\rK_j$ and $\rN_j$. We assume w.l.o.g.  that $\rK_j\leq \rN_j$.  We set
    \begin{equation*}
      \rK= \sum_{j\in V} (\rK_j -1) + \rho[\mathbf{T}]
     \qquad  \text{and} \qquad 
      \rN=\sum_{j \in V}\rN_j + 3 \chi(\Sf)+1.
    \end{equation*}
  \end{itemize}
\end{nota}

\begin{thm}\label{t:main}
  With the assumptions from Notation \ref{nota:mt}, the function
 \begin{align}\label{e:mainint}
   \cJ = \cJ_{\mathbf{T}, V, \mathrm{m}, f}: \ell \mapsto
   \int_{\substack{(\x, Y)\in\mathcal{T}^*_{g_{\mathbf{S}},n_{\mathbf{S}}} \\ \ell_Y({\mathbf{c}})=\ell} }\, \,
\prod_{j\in V}  f_j(x_j)  \prod_{j=1}^k x_{\vi} \delta(x_{\vi} - x_{\vi'}) \frac{ \d \mathrm{Vol}^{\mathrm{WP}}_{g_{\mathbf{S}}, n_{\mathbf{S}}}(\x, Y) }{\d\ell}
\end{align}
is a Friedman--Ramanujan function in the weak sense. 
More precisely,   $\cJ \in \cF_w^{\rK, \rN}$ and
\begin{align}\label{e:multi_bound_intro}
  \norm{\cJ}^w_{\cF^{\rK, \rN}} \leq \fn(\g,\n,(\rK_j, \rN_j)_{j \in V})\,
  \prod_{j\in V} \norm{\tilde g_j}_{\cF^{\rK_j, \rN_j}}.
\end{align}
\end{thm}

\begin{rem}
  Since $\rho[{\mathbf{T}}]\leq 2\chi(\Sf)$,  our Friedman--Ramanujan exponents $\rK$ and $\rN$ can
  be made uniform amongst all loops filling $\mathbf{S}$.
\end{rem}

\begin{rem}
  \label{rem:improve}
  In the proof of Theorem \ref{t:main}, we have tried to optimise the value of $\rK$, but not the
  value of $\rN$.  One can further improve the value of $\rK$ depending on the loop $\mathbf{c}$,
  and prove that it can be taken to be $\sum_{j \in V'} (\rK_j -1)+\rho[\mathbf{T}]$, where
  $V' \subseteq V$ is the set of boundary components containing unshielded simple portions of
  $\mathbf{c}$. See details in Remark~\ref{r:opti_K}.

  We notice that in the special case when the loop $\mathbf{c}$ is double-filling, we have
  $V' = \emptyset$ and $\rho[\mathbf{T}]=0$. In other words, the function $\cJ$ is actually an
  element of $\cR_w$ as soon as $\mathbf{c}$ is double-filling.
\end{rem}

Let us now show how to go from Theorem \ref{t:main} to Theorem~\ref{thm:FR_type}.
 
\begin{proof}[Proof of Theorem \ref{thm:FR_type}]
  \label{rem:apply_main_to_vg}
  We simply use the expression
  \eqref{e:each_term} for the functions $f_k^{\mathbf{T}}$, and apply our main technical result to $
  V= V_0\sqcup V_+\sqcup V_-$ and the functions
  \begin{equation*}
    f_j(x) =
    \begin{cases}
      x_j^{k_j} & \text{if } j \in V_0 \\
      x_j^{k_j} \cosh \div{x_j} & \text{if } j \in V_+ \\
      x_j^{k_j} \sinh \div{x_j} & \text{if } j \in V_-
    \end{cases}
  \end{equation*}
  which clearly satisfy the hypotheses with
  \begin{equation*}
    \begin{cases}
      \rK_j=\rN_j=k_j+1 & \text{if } j\in V_+\cup V_- \\
      \rK_j=0, \rN_j=a_{k+1}+1 & \text{if } j\in V_0,
    \end{cases}
  \end{equation*}
  using the parity condition on the integers $k_j$ to obtain the needed
  cancellation at $0$.  
\end{proof}

\begin{rem}
  The proof above allows us to see that, in Theorem \ref{thm:FR_type}, we may  take the
  Friedman--Ramanujan parameters to be
  \begin{align*}
    \rK & = \sum_{j\in V_+\cup V_-} k_j+ \rho[{\mathbf{T}}]\leq 2k + \rho[{\mathbf{T}}]
    \leq 2k+2 \chi(\mathbf{S})\\
    \rN & =\sum_{j \in  V_0 \cup V_+ \cup V_- }\rN_j + 3 \chi(\mathbf{S})+1\leq \n (a_{k+1}+1)+3 \chi(\Sf)+1.
  \end{align*}
\end{rem}

The proof of Theorem \ref{t:dream} about the spectral gap of random surfaces requires variants of Theorem \ref{t:main}.
A first straightforward variant is that Theorem \ref{t:main} also applies in the case where $\curve$ is a multi-loop, defined as a collection of loops (\S \ref{s:definitions}).
Another variant will be given in Theorem \ref{t:main-kappa} (where some of the loops in $\curve$ are constrained to have bounded lengths).
The final variant needed for the spectral gap is described in \S \ref{s:variantLC}.

\subsubsection{Friedman--Ramanujan functions as solutions of an integral equation} 
\label{s:cL}
A first ingredient in the proof of Theorem \ref{t:main} is the fact that Friedman--Ramanujan
functions are fully characterized as solutions (modulo $\cR$) of a certain integral equation.  We
define two operators acting on locally integrable functions, $\cP=\cP_x=\int_0^x$ (taking the
primitive vanishing at $0$) and $\cL=\cL_x=\Id-\cP_x$ (where $\Id$ stands for the identity
operator). It is trivial, nevertheless useful, to note that $\cP$ and $\cL$ preserve
$\cF, \cF_w, \cR, \cR_w$. More precisely, $\cL$ is a continuous operator from $\cF^\rK$ to
$\cF^{\rK-1}$ and from $\cR$ to itself (and this is also true for the weak spaces).
  
  We will make use of the following characterization of Friedman--Ramanujan functions:
  \begin{prp}\label{p:charFR} For $\rK\geq 1$, any locally integrable function $f$,
    \begin{equation*}
      \begin{cases}
        f\in \FR^\rK \Leftrightarrow \cL^{\rK} f\in \cR \\
        f\in \FR^\rK_w \Leftrightarrow \cL^{\rK} f\in \cR_w.
      \end{cases}
    \end{equation*}
Furthermore, if  $\rN \geq 1$ is so that $\cL^{\rK} f\in \cR^{\rN}_w$, then
\begin{align}\label{e:FRnorm_estimate}
\norm{f}_{ \FR^{\rK, \rN}}^w\leq \fn(\rK, \rN)\, \norm{\cL^{\rK} f}^w_{\cR^\rN}.
\end{align}
\end{prp}

The ``only if'' part is a straightforward consequence of the fact that $\cL$ sends $\cR_w$ to
itself, $\cR$ to itself, and that $\cL^{\rK}$ sends $\C_{\rK-1}[t]e^t$ to
$\C_{\rK-1}[t]\subset \cR$.  The ``if'' part is proven by induction on~$\rK$, iterating the
following lemma:
\begin{lem} Let $\rK \geq 1$ be an integer. If $\cL f= b+c$, with $b\in \C_{\rK-1}[t]e^t$ and
  $c\in \cR$ or $\cR_w$, then there exists a constant $C$ such that
  \begin{equation}\label{e:solution}f(t)=C e^t +  b(t) + c(t)
    +\int_0^t b(s)e^{t-s}\d s
    + \int_{t}^{+\infty} c(s)e^{t-s}\d s.\end{equation}
As a consequence, if $c\in \cR^{\rN}_w$, then $f$ belongs in $\cF_w^{\rK, \rN}$ and
\begin{align} \label{e:supnorm}
\norm{f}_{ \FR^{\rK, \rN}_w}\leq \fn(\rK, \rN)\,(\norm{b}_\FR+ \norm{c}_{\cR_w^\rN}).
\end{align}
The same holds with the strong forms.
\end{lem}
\begin{proof}  
  The lemma is proven by the usual method for solving linear ODEs: let $F= \cP f$, then $F$ solves
  $F'-F=b+c$, so there exists a constant $C$ such that
  $$F(t)=C e^t
  +\int_0^t b(s)e^{t-s}\d s +\int_{t}^{+\infty} c(s)e^{t-s}\d s.$$ Here, because of the different
  behaviors of $b$ and $c$ at infinity, we have chosen different boundary conditions for the
  primitives of $b(s)e^{-s}$ and $c(s)e^{-s}.$ Since $f$ is the derivative of $F$, this
  implies~\eqref{e:solution}.  Assume that $c\in \cR^{\rN}_w$: then
  $\int_0^t b(s)e^{t-s}\d s\in \C_{\rK-1}[t]e^t$, $\int_{t}^{+\infty} c(s)e^{t-s}\d s\in \cR^{\rN}$,
  so $f \in \cF_w^{\rK, \rN}$, and
\begin{align*} 
\norm{f}^w_{ \FR^{\rK, \rN}_w}\leq \fn(\rK, \rN)\,\Big( \norm{c}^w_{\cR^\rN}  + |C|+ \norm{b}_\FR\Big).
\end{align*}

 Evaluation at $t=0$ yields $C= -\int_{0}^{+\infty} c(s)e^{-s} \d s $ so 
$$|C|\leq  \left( \int_{0}^{+\infty} (1+s)^{\rN-1}e^{-s/2}\d s \right)  \norm{c}^w_{\cR^{\rN}}.$$
Inequality \eqref{e:supnorm} follows.
\end{proof}

\begin{rem}
  \label{rem:boundsFRL}
  We can extend Remark \ref{rem:boundsFR} to the images of a Friedman--Ramanujan function $f \in
  \cF^{\rK,\rN}$ by $\cL$ and obtain a constant $M_{\rK,\rN}$ such that, for all $\ell \in \R_{\geq 0}$
  and any $0 \leq m < \rK$,
  \begin{align*}
    & | \cL^{m} f ( \ell)|\leq M_{\rK,\rN} \|f\|_{\cF^{\rK,\rN}} (1+ \ell)^{\rK-1} e^{ \ell} \\
    & | \cL^{\rK} f ( \ell)|\leq M_{\rK,\rN}  \|f\|_{\cF^{\rK,\rN}} (1+ \ell)^{\rN -1} e^{ \frac{\ell}2}.
  \end{align*}
  Here
  $M_{\rK, \rN}= \max_{m\leq \rK} ( \fn(\rK-m, \rN) \norm{\cL^m}_{\cF^{\rK, \rN}\rightarrow
    \cF^{\rK-m, \rN} })$ with $\fn(\rK-m, \rN)$ the constant in \eqref{e:normbound1}.
\end{rem}
 
\subsubsection{Stability under convolution}
\label{s:convolution}
Another idea behind our proof is the stability of the classes $\FR, \FR_w, \cR, \cR_w$ under convolution.
This was proven in \cite[Proposition 3.6]{Ours1} by direct algebraic calculation, following \cite[Theorem 7.2]{friedman2003}. Another way to understand this is to refer to the usual property of convolution (classically known for differential operators, but valid also for the operator~$\cL$): given two functions $f_1, f_2$ and integers $\rK_1, \rK_2$, we have
\begin{align}\label{e:easyconvo}
\cL^{\rK_1+\rK_2} (f_1 * f_2)= \cL^{\rK_1} f_1 * \cL^{\rK_2} f_2.
\end{align}
If $f_i\in \cF^{\rK_i}$, then $\cL^{\rK_i} f_i \in \cR$. Thus, using the stability
of~$\cR$ under convolution (a straightforward upper bound), we obtain that $\cL^{\rK_1+\rK_2} (f_1 * f_2)\in \cR$, implying
$f_1 * f_2\in \cF^{\rK_1+\rK_2}$. The same method holds with $\cF_w$ instead of $\cF$.

%$\cL^{m_1+m_2} (f_1 * f_2)\in \cR$, i.e.  $f_1 * f_2\in \cF^{m_1+m_2}$.

In the proof of Theorem \ref{t:main}, we show that the function $\cJ$ defined by the integral
\eqref{e:mainint} can be expressed in terms of ``convolutions'' of Friedman--Ramanujan functions. This
requires to express the measure  $\d \mathrm{Vol}^{\mathrm{WP}}_{g_{\mathbf{S}}, n_{\mathbf{S}}}(
Y) $ in a well chosen set of coordinates on
$\mathcal{T}^*_{g_{\mathbf{S}},n_{\mathbf{S}}}$, adapted to the loop~${\mathbf{c}}$. These
coordinates are not the well known Fenchel--Nielsen coordinates: their description will be the focus
of \S \ref{s:nc}.

The reason why \eqref{e:mainint} looks roughly like a convolution is that the loop ${\mathbf{c}}$
can be decomposed into a union of simple curves $\beta = (\beta_1, \ldots, \beta_n)$, by ``opening''
all the intersections (this operation is described in \S \ref{s:opening_int}).  In the case of
curvature $-\infty$ (i.e. graphs) we would have
$$\ell_Y(\mathbf{c})= \ell_Y(\beta_1)+ \ldots + \ell_Y(\beta_n),$$ and \eqref{e:mainint} would
exactly be a convolution. But we are in finite negative curvature, and $\ell_Y(\mathbf{c})$ differs
from $\ell_Y(\beta_1)+ \ldots + \ell_Y(\beta_n)$~: the expression \eqref{e:mainint} is not exactly a
convolution. We will nevertheless see that $\ell_Y(\mathbf{c})$ is close to
$\ell_Y(\beta_1)+ \ldots + \ell_Y(\beta_n)$ in a certain topology. In \S \ref{s:GC}, we define a
notion of {\em{pseudo convolutions}} and we develop a sophisticated variant of \eqref{e:easyconvo},
first in an abstract setting -- with the aim of later applying it to the study of the integral
\eqref{e:mainint}.

\subsubsection{Geometric comparison estimates}

A key ingredient of the proof of Theorem \ref{t:main} is the use of geometric comparison estimates,
which we shall now comment on.

The following straightforward bound, comparing the length of a loop $\mathbf{c}$ and the boundary of
the surface $\Sf$ it fills, is a simple case of a geometric comparison estimate.
\begin{prp} \label{p:comparison1}
  Let $\mathbf{T}=\eqc{\mathbf{S},\mathbf{c}}$ be a local type. For any $(\x,Y) \in
  \mathcal{T}_{\g,\n}^*$,
  \begin{equation}
    \label{eq:comparison1}
    \sum_{i=1}^{\n} x_i
    \leq \ell(\mathbf{c}) +\sum_{\substack{\cI \text{ unshielded} \\ \text{simple
          portion}}}\ell(\cI)
  \leq 2 \ell(\mathbf{c}).
  \end{equation}
\end{prp}

The estimate $\sum_{i=1}^{\n}x_i \leq 2 \ell(\mathbf{c})$ is a classical consequence of the fact
that $\mathbf{c}$ is assumed to fill the surface $\mathbf{S}$, see e.g. \cite[Lemma 4.4]
{Ours1}. In the special case of double-filling loops, Proposition \ref{p:comparison1} tells us that
this can be improved by dropping the factor of $2$, which was the crucial observation made in
\cite[Lemma 8.6]{Ours1}. Proposition \ref{p:comparison1} is a straightforward generalization of
this lemma.

Comparison estimates are a crucial step in the proof of our main results: they give upper bounds for
terms such as $\prod_{j\in V} f_j(x_j)$, appearing in the definition of $\cJ(\ell)$ in Theorem
\ref{t:main}.  Proposition \ref{p:comparison1} also holds for multi-loops, which will be defined as
collections of loops (\S \ref{s:definitions}).

The advantage of gaining a factor of $2$ in the comparison estimate can clearly be seen from the
statement of Theorem \ref{t:main}: the function $\prod_{j\in V} f_j(x_j)$ grows like
$\exp(\frac12\sum_{j \in V} x_j)$ (up to some polynomial factors, allowed in the definition of
Friedman--Ramanujan functions), so the bound $\sum_{j=1}^{\n}x_j \leq 2 \ell(\mathbf{c})$ implies
the upper bound
\begin{align}\label{e:ineq1}\prod_{j \in V} | f_j(x_j)|\leq C (1+\ell(\mathbf{c}))^{\alpha} e^{\ell(\mathbf{c})}\end{align}
for some $C,\alpha\geq 0$. This ensures that $|\cJ(\ell)|\leq C' (1+\ell)^{\alpha'} e^{\ell}$ in a
weak sense (see \eqref{e:weak-FR} for a precise bound). However, the Friedman--Ramanujan property
requires an \emph{exact expansion} rather than an exponential upper bound. Therefore, we cannot
directly deduce the Friedman--Ramanujan property from the trivial inequality
$\sum_{j=1}^{\n}x_j \leq 2 \ell(\mathbf{c})$.
The fundamental challenge adressed in this article consists in refining the approach to study
$\cJ(\ell)$, hence providing an exact expansion instead of the straightforward upper bound above.

In the special case when the loop $\mathbf{c}$ is {double-filling}, then \eqref{e:ineq1} may be
replaced by
\begin{align}\label{e:ineq2}\prod_{j\in V} | f_j(x_j)|\leq  (1+\ell(\mathbf{c}))^{\alpha} e^{\frac{\ell(\mathbf{c})}2},\end{align}
which is enough to ensure that $\cJ(\ell)$ is in the class $\cR_w$, and thus in particular is a weak
Friedman--Ramanujan function (see Corollary \ref{c:double_fills}).  In other words, Theorem
\ref{t:main} can be obtained by a straightforward upper bound in the case of double-filling
loops. In general, the presence of shielded segments should imply less work to prove Theorem
\ref{t:main}, thanks to Proposition \ref{p:comparison1}. On the other end of the spectrum, the worse
possible scenario should be the case when none of the simple portions are shielded: we call such
loops \emph{generalized eights}.  Most of the paper is dedicated to proving Theorem \ref{t:main} in
the case of generalized eights. For other multi-loops we will indicate how to modify our discussion
at the end of the paper (\S \ref{s:othercases}).

\subsubsection{Plan of the paper}

In \S \ref{s:loops} we ``open'' all the intersections of $\mathbf{c}$ to obtain a new
representative~$\cop$ of the same homotopy class, decomposed into simple curves, joined by ``bars''
replacing the intersections.  The whole discussion applies to multi-loops, defined as finite
collections of loops.  We define a class of multi-loops called {\em{generalized eights}}. They admit
a rather simple geometric description but offer the highest degree of difficulty, as far as the
proof of Theorem~\ref{t:main} is concerned.

In \S \ref{s:nc} we define a coordinate system on $\cT^*_{\g, \n}$, adapted to the study of the
function $Y\mapsto \ell_Y(\mathbf{c})$, and we express the Weil--Petersson measure in those new
coordinates. There are few results expressing the Weil--Petersson measure in coordinates other than
the Fenchel--Nielsen ones, so we believe that Theorem \ref{t:maindet} has an intrinsic interest, and
probably deserves further investigation.

In \S \ref{s:lengthf}, we use the representative $\cop$ to find an algebraic expression for
$\ell_Y({\mathbf{c}})$, as well as the lengths of all boundary curves of $\mathbf{S}$, in the new
coordinates. After this change of coordinates, it becomes apparent that the integral
\eqref{e:mainint} looks like a convolution of Friedman--Ramanujan functions, see \S
\ref{s:noncross}.  As already mentioned, the fact that we do not have exactly a convolution is
related to the fact that the curvature is finite.  As a result, we need to develop a theory of
{\em{pseudo-convolutions}} in an abstract setting, which is the focus of \S \ref{s:GC}.

In order to apply our theoretical results on pseudo-convolutions to our particular problem,
additional geometric and analytic estimates on the function $Y\mapsto \ell_Y(\mathbf{c})$ are
required. They are provided in \S \ref{s:lengthc} and \S \ref{s:lengthboundary}. Most of \S
\ref{s:lengthc} can be skipped in the purely non-crossing case (defined in \S \ref{s:def_crossing}).
The proof is finished in \S \ref{s:proof}, except for the ``Crucial comparison estimate'' which is a
variant of Proposition \ref{p:comparison1}, proven in \S \ref{s:horrible}. We indicate in \S
\ref{s:othercases} how to treat arbitrary topologies (i.e. multi-loops which are not generalized
eights). In \S \ref{s:variantLC} we state a variant of Theorem \ref{t:main}, and use it in \S \ref{s:wrapup}
to deduce Theorem~\ref{t:dream} about the optimal spectral gap.

\para{Acknowledgements} Discussions with J. Friedman were of invaluable help to surmount several
obstacles in this project.

This research has received funding from the European Research Council
(ERC) under the European Union’s Horizon 2020 research and innovation programme (Grant agreement
No. 101096550), the EPSRC Standard Grant EP/W007010/1 and the Royal Society Dorothy Hodgkin
Fellowship.

%%%Local Variables: 
%%% mode: latex
%%% TeX-master: "main"
%%% End: 

\section{Loops and generalized eights}
\label{s:loops}
This section starts by setting a number of conventions related to loops on surfaces, segments,
orientations and geodesics on the hyperbolic plane. We then construct in \S \ref{s:opening_int} a
procedure allowing to open the intersections of a loop, and following on this some associated
concepts (bars, simple portions, set $\Theta$, diagrams, cyclic orderings, u-turns and
crossings). Importantly, in \S \ref{s:gene_eight}, we introduce a class of loops called
\emph{generalized eights}, which will be the focus of the bulk of the paper (\S
\ref{s:nc}-\ref{s:horrible}).

\subsection{Definitions}
\label{s:definitions1}

Let ${\mathbf{S}}$ be a smooth surface (possibly with boundary).  Throughout this article, we will
assume that ${\mathbf{S}}$ is connected and has negative Euler characteristic. Indeed, the case
where~${\mathbf{S}}$ is a cylinder has already been addressed in \cite[Proposition 3.4]{Ours1}. In
applications, the surface $\mathbf{S}$ we apply Theorem \ref{t:main} to will often not be connected;
however, the proof for disconnected ${\mathbf{S}}$ only requires minor modifications from the proof
in the connected case.

\subsubsection{Multi-loops} \label{s:definitions}

Let us clarify what we mean by a loop or a multi-loop $\mathbf{c}$.

\begin{defa}
  A {\em{loop}} is a piecewise smooth map ${\mathbf{c}}:\R/\Z\rightarrow{\mathbf{S}}$ with nowhere
    vanishing derivative (we allow loops to be contractible).  A loop is called a {\em{curve}}
    if it is non-contractible and has no self-intersection.
\end{defa}

From now on we extend the discussion to {\em{multi-loops}}. This is not just for the sake of maximal
generality, but also because, when we later ``remove'' some self-intersections from ${\mathbf{c}}$,
this operation will inevitably turn a loop into a multi-loop.
  
\begin{defa} \label{def:mult_loop} A {\em{multi-loop}} is a collection of loops
  ${\mathbf{c}}=({\mathbf{c}}_1, \ldots, {\mathbf{c}}_{\cc}):\R/\Z\rightarrow {\mathbf{S}}^{\cc}$.  The
  loops ${\mathbf{c}}_i$ are called the {\em{components}} of ${\mathbf{ c}}$.
  \begin{enumerate}
  \item A multi-loop is said to be \emph{in minimal position} if it has only transverse
    self-intersections, and minimal number of self intersections in its free homotopy class.
  \item A multi-loop is said to be \emph{simple} if it has no self-intersections.
  \item A multi-loop is called a {\em{multi-curve}} if it is simple, its components are
    non-contractible and, for $i\not=j$, ${\mathbf{c}}_i$ is not freely homotopic to
    ${\mathbf{c}}_j$ nor ${\mathbf{c}}^{-1}_j$.
  \end{enumerate}
 \end{defa}
 
  \begin{rem}
    Although the multi-loop ${\mathbf{c}}$ is originally assumed to be in minimal position, the
    reader will notice that we will work with representatives of its homotopy class that do not have
    this property, but are $\cC^0$-limits of multi-loops in minimal position.
 \end{rem}  
 
 We will always assume that the multi-loop ${\mathbf{ c}}$ fills the surface ${\mathbf{S}}$, that is
 to say, if $\cN$ is a regular neighbourhood of ${\mathbf{ c}}$, then all the connected components
 of ${\mathbf{S}}\setminus \cN$ are disks or annular regions bordered by a boundary component of
 ${\mathbf{S}}$ (see \cite[Definition 2.3]{Ours1}).

 Multi-loops are parameterized and oriented, but often, we will only be interested in their
 geometric image as subsets of ${\mathbf{S}}$. Importantly, when we say that two multi-loops
 ${\mathbf{ c}}=({\mathbf{c}}_1, \ldots, {\mathbf{c}}_{\cc})$ and
 ${\mathbf{ c}}'=({\mathbf{c}}'_1, \ldots, {\mathbf{c}}'_{\cc'})$ are freely homotopic, we take the
 numberings and the orientations into account: we mean that they have the same number of components
 ($\cc=\cc'$) and that ${\mathbf{c}}_i$ is freely homotopic to ${\mathbf{c}}'_i$ for every
 $i$. Homotopies must respect orientation, and in particular a loop is not a priori homotopic to its
 inverse.

\subsubsection{Segments and concatenation} \label{s:nota_paths}

Throughout this article, we will decompose our loops in successions of segments.

\begin{defa}
  We call $J\subset {\mathbf{S}}$ an {\em{(oriented) segment}} if it is of the form $J=p([0, T])$,
  where $ p : [0, T] \rightarrow {\mathbf{S}}$ is a smooth oriented path with nowhere vanishing
  derivative.
\end{defa}

\begin{nota}[Origin and terminus] \label{n:origin}
  Let $J=p([0, T])$ an oriented segment.
  The \emph{origin} of $J$ is the point $o(J) := p(0)$, and its \emph{terminus} $t(J) := p(T)$.
\end{nota}
 
 %This terminology ``right and left'' extends to continuous piecewise smooth paths $p$ provided that 

\begin{nota}[Concatenation of paths] \label{n:concat} If $p : [0, T] \mapsto {\mathbf{S}}$ and
  $q : [0, T'] \mapsto {\mathbf{S}}$ are two piecewise smooth paths such that $p(T)=q(0)$, we define
  a piecewise smooth path $p\smallbullet q :[0, T+T']\rightarrow {\mathbf{S}}$ by
 \begin{align*}
 p\smallbullet q (t) :=
 \begin{cases}
   p(t) \mbox{ if } t\in [0, T] \\
   q(t-T) \mbox{ if } t\in [T, T+T'] .
   \end{cases}
 \end{align*}
 \end{nota}
 This is an associative and non-commutative operation. Most of the time, we shall only be interested
 in the geometric images of the paths, and not in their actual parametrizations.

 {
 \begin{defa}
   \label{defa:glide}
   Let $\beta$ be a multi-loop on $\mathbf{S}$, and $J=p([0,T])$ be an oriented segment with
   endpoints on $\beta$. We say an oriented segment $K=q([0,T])$ is \emph{homotopic to $J$ with
     gliding endpoints} if there exists a continuous function
   $h:[0,T] \times [0,1] \rightarrow \mathbf{S}$ such that $h(\cdot, 0)=p$, $h(\cdot, 1)=q$, and the
   endpoints $h(0,t)$ and $h(T,t)$ lie on $\beta$ for any $t \in [0,1]$.
 \end{defa}

 \begin{rem}
   \label{rem:beta_geod}
   In the following, $\beta$ will (for the most part) be a multi-curve. In this case, when we equip
   the surface $\mathbf{S}$ with a metric $Y$, we will systematically do so by picking a
   representative in the Teichm\"uller space $\cT_{\g,\n}^*$ so that $\beta$ is a simple
   multi-geodesic. In this case, any non-trivial homotopy class of segments $J$ with endpoints
   gliding on $\beta$ contains a uniquely defined length-minimizing element $\bar{J}$, which is
   an orthogeodesic (a geodesic segment orthogonal to $\beta$ at its endpoints). In pictures, in an
   effort to simplify the usual notation, we will often mark the endpoints of orthogeodesics with a
   black dot to emphasise the fact that there is a right angle there.
 \end{rem}}

 \subsubsection{Left and right}

 We take the following conventions in terms of the orientation of segments.

 \begin{defa} Let $J=p([0, T])$ an oriented segment on $\mathbf{S}$.
   \begin{itemize}
   \item A tangent vector $v\in T_x \mathbf{S}$ with $x=p(t)\in J$ is said to be \emph{on the left
       (resp. right) of $J$} if the angle from $\dot p(t)$ to $v$ belongs to $[0,\pi]$
     (resp. $[-\pi,0]$).
   \item Let $K=q([0, T'])$ be an oriented segment. We say that $K$ \emph{leaves $J$ from the left
       (resp. right)} if the origin of $K$ lies in $J$ and if the tangent vector $\dot q(0)$ is on
     the left (resp. right) of~$J$.  Similarly, we say that $K$ \emph{arrives on $J$ from the left
       (resp. right)} if the terminus of~$K$ lies on $J$ and $\dot{q}(T')$ is on the right
     (resp. left) of $J$.
   \end{itemize}
 \end{defa}

\subsubsection{Infinite geodesics on $\IH^2$}
 
 We now briefly focus on infinite geodesics in the hyperbolic plane $\IH^2$. For two infinite
 geodesics in the hyperbolic plane, with endpoints at infinity $(a, b)$ and $(c, d)$, the
 cross-ratio $(a :b:c:d)$ can be used as a convenient way of describing their respective positions:
 \begin{itemize}
 \item A cross-ratio $(a :b:c:d)\in (0, 1)$ means that the two geodesics intersect.
   \item Whenever $(a :b:c:d)\not\in (0, 1)$ we shall speak of \emph{parallel geodesics}.
     \begin{itemize}
     \item  A cross-ratio in $[1, +\infty)$ means that the two geodesics are parallel and
       oriented ``head-to-tail''.
     \item On the opposite, $(a :b:c:d)\leq 0$ means that the two geodesics are parallel and are
       oriented ``in the same direction''. We will then say that there are \emph{aligned}.
     \end{itemize}
   \end{itemize}
 
 Since the hyperbolic plane is oriented, an infinite oriented geodesic determines two half-spaces,
 one on its left and one on its right.

 \begin{defa}[Algebraic distance $\Dist$] \label{def:alg_d}
   If two points $x, y$ lie on an oriented geodesic $\gamma$, we will say \emph{$y$ is on the right
     of $x$ along $\gamma$} to mean that $y$ is after $x$ along $\gamma$.  We denote by $\Dist(x, y)$
   the signed distance between $x$ and $y$ along $\gamma$, measured positively if $y$ is on the right
   of $x$ along $\gamma$ and negatively otherwise.
 \end{defa}
 This notion is named according to the tradition of orienting lines from left to right.  Although it
 would be more rigourous to denote this distance by $\Dist_{\gamma}(x, y)$, the notation will always
 be used in the universal cover $\IH^2$ and the reference geodesic will always be clear from the
 context.

 The following key observation, illustrated in Figure \ref{fig:useful}, will be used on multiple
 occasions throughout this article.

 \begin{rem}[The Useful Remark]
   \label{r:useful}
      \begin{figure}[h!]
     \centering
     \includegraphics{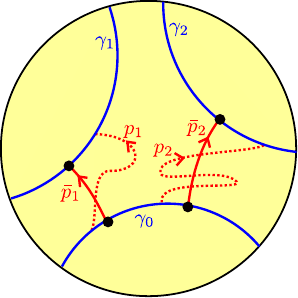}
     \caption{Illustration of the Useful Remark.}
     \label{fig:useful}
   \end{figure}

   Take $3$ distinct non-intersecting geodesics $(\gamma_i)_{0 \leq i \leq 2}$ in the hyperbolic
   plane $\IH^2$. Assume that for $i=1, 2$ we have a piecewise smooth oriented path $p_i$ going from
   $\gamma_0$ to $\gamma_{i}$, intersecting $\gamma_0$ and $\gamma_{i}$ only at its
   endpoints. Therefore, the complement $\IH^2 \setminus(\gamma_0\cup \gamma_i\cup p_i)$ has four
   connected components, that we respectively call north of $\gamma_i$, south of~$\gamma_0$ and
   east/west of $p_i$, following the compass.  Assume in addition that $p_{2}$ only visits the east
   of $p_1$, and $p_{1}$ only visits the west of $p_2$. Then:
   \begin{enumerate}
   \item the oriented orthogeodesics $\bar p_i$ from $\gamma_0$ to $\gamma_i$ are disjoint;
   \item $\bar p_{2}$ entirely lies east of $\bar p_1$;
   \item $\bar p_{1}$ entirely lies west of $\bar p_2$.
   \end{enumerate}
 \end{rem}

\subsubsection{Length of a path or a geodesic}
 \label{nota:length}
 
   If $\gamma$ is a 1-dimensional Riemannian manifold, we will denote by $\ell(\gamma)$ its
   length. If now $\gamma$ is a smooth path in a Riemannian manifold $Y$, we will denote by
   $\ell(\gamma)$ its length in $Y$ (which coincides with its length for the metric on $\gamma$
   induced by $Y$). If $\gamma$ is a loop in a hyperbolic surface $Y$, we will denote by
   $\ell_Y(\gamma)$ the minimal length of a representative of the free homotopy class of~$\gamma$
   (if $\gamma$ is contractible, $\ell_Y(\gamma)=0$, otherwise it is the length of the unique
   geodesic representative of the free homotopy class of $\gamma$). In general, $\ell(\gamma)$ and
   $\ell_Y(\gamma)$ are different, but they are of course equal if $\gamma$ is a closed geodesic in
   $Y$.
   % We try, as much as possible, to stick to the following notation : $\ell$ is the length just
   % defined, and $\ell_Z$ is the length of a preferred geodesic representative in a homotopy class
   % (periodic geodesic, orthogeodesic,... according to context).

If ${\mathbf{S}}$ is endowed with a hyperbolic structure $Y$, then each component ${\mathbf{c}}_i$
of $\mathbf{c}$ is freely homotopic to a periodic geodesic, which is unique if ${\mathbf{c}}_i$ is
non-contractible, but may be reduced to a point otherwise.  We denote
$\ell_Y({\mathbf{c}})=\sum_{i=1}^{\cc} \ell_Y({\mathbf{c}}_i)$ the total length of a geodesic
representative.

\subsection{Opening the intersections of ${\mathbf{ c}}$}
\label{s:opening_int}

In this paragraph, we explain how to \emph{open the self-intersections of ${\mathbf{ c}}$ to replace
  them by bars}. Our construction yields a new picture which consists of a simple multi-loop $\beta$
and a collection of bars $B$ with endpoints on $\beta$, replacing the intersections.

\begin{nota}
  \label{n:r}
  Throughout the paper, the letter $r$ stands for the number of self-intersection points of the
  multi-loop ${\mathbf{ c}}$. We identify the set of intersection points of the multi-loop
  $\mathbf{c}$ with the set $\{1, \ldots, r\}$ by picking an arbitrary numbering.
\end{nota}

\subsubsection{The construction}
\label{sec:open-inters}

If $r=0$ then ${\mathbf{ c}}$ is a simple loop and we simply take $\beta={\mathbf{ c}}$, and no
bars.  Now assume $r\geq 1$. Let $a_1, \ldots, a_r$ be the intersection points of $\mathbf{c}$ and
$D_1, \ldots, D_r$ be disjoint open disks around them. We can take the disks small enough
so that, for each $j$, ${\mathbf{ c}}\cap D_j$ is the union of two oriented segments
$I_j, I'_j$ intersecting transversally at $a_j$.  We open the intersection at each~$a_j$ and replace
the two intersecting segments $I_j, I'_j$ by two smooth disjoint segments, joined by a transversal
segment $B_j$ traversed twice, replacing the intersection.
      % leaving ${\mathbf{ c}}$ from the left and arriving at ${\mathbf{ c}}$ on the right, and by the same segment $B_j^-$ oriented to leave ${\mathbf{ c}}$ from the right and arrive at ${\mathbf{ c}}$ on its left
There are two distinct ways to perform this operation, represented in Figure \ref{fig:open_first} and
\ref{fig:open_second} respectively. 

\begin{figure}[h!]
  \centering
  \begin{subfigure}[b]{0.3\textwidth}
    \centering
    \includegraphics{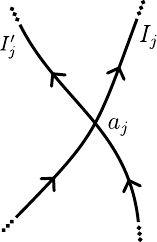}
    \caption{The initial intersection.}
    \label{fig:open_0}
  \end{subfigure}%
  \begin{subfigure}[b]{0.3\textwidth}
    \centering
    \includegraphics{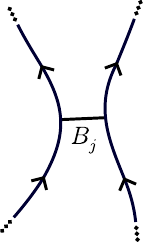}
    \caption{Opening the first way.}
    \label{fig:open_first}
  \end{subfigure}%
  \begin{subfigure}[b]{0.3\textwidth}
    \centering
    \includegraphics{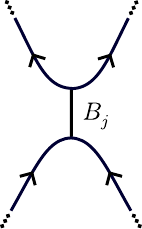}
    \caption{Opening the second way.}
    \label{fig:open_second}
  \end{subfigure}%
  \caption{Opening an intersection to replace it by a bar, two ways. % The orientation of the
    % bars $B_j^\pm$ depends on the type of opening: in the first way, $B_j^+$ and $B_j^-$ are
    % oriented in opposite ways whilst $B_j^+=B_j^-$ in the second way.
  }
  \label{fig:opening_int}
\end{figure}%
 
Opening all of the intersections (having made an arbitrary choice between the two ways for each
intersection) yields a new representative $\cop$ of the homotopy class of ${\mathbf{ c}}$,
represented in Figure \ref{fig:cop}. The operation of going from ${\mathbf{ c}}$ to the new
representative $\cop$ will be referred to as {\emph{opening the intersections of ${\mathbf{c}}$ (to
    replace them by bars)}}.

 \begin{figure}[h!]
   \centering
   \begin{subfigure}[b]{0.5\textwidth}
    \centering
   \includegraphics{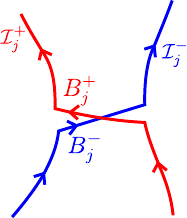}
    \caption{The first way.}
    \label{fig:cop_1}
  \end{subfigure}%
   \begin{subfigure}[b]{0.5\textwidth}
    \centering
   \includegraphics{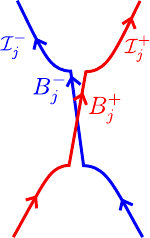}
    \caption{The second way.}
    \label{fig:cop_2}
  \end{subfigure}%
   \caption{A close deformation of the multi-loop $\cop$. Notice that each bar is
     traversed twice (in the directions we will later denote $B_j^+$ and $B_j^-$).}
   \label{fig:cop}
 \end{figure}

By convention, the bars $B_j$ are closed, i.e. contain their endpoints.  If $\dot{B_j}$ is the open
segment, i.e. the segment without its endpoints, we call $\beta$ the connected components of
$\cop \setminus ( \bigcup_{j=1}^r\dot{B_j})$.  We observe that $\beta$ is a simple multi-loop.  The
geometric image of $\cop$ is the simple multi-loop~$\beta$ joined by the collection of bars
$B=(B_1, \ldots, B_r)$.

Note that $\beta$ is not necessarily a multi-curve -- see Figure \ref{fig:ex_op_2} for an
example where one of the components of $\beta$ is contractible.  When we later specialise our
discussion to the class of multi-loops called \emph{generalized eights} in \S \ref{s:gene_eight},
the family $\beta$ will actually be a multi-curve (see Lemma~\ref{r:GE}).
% will later be called a \emph{diagram}.

\begin{figure}[h!]
  \centering
   \begin{subfigure}[b]{0.4\textwidth}
    \centering
     \includegraphics[]{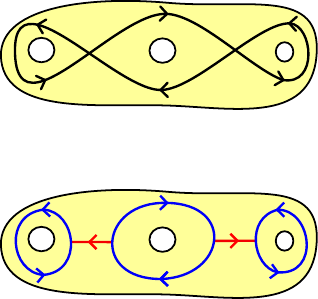}
    \caption{A generalized eight.}
    \label{fig:ex_op_1}
  \end{subfigure}%
  \begin{subfigure}[b]{0.6\textwidth}
    \centering
     \includegraphics[]{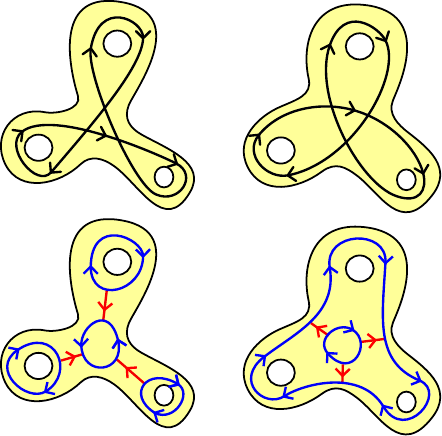}
     \caption{Two non-isotopic representatives of a loop.}
    \label{fig:ex_op_2}
  \end{subfigure}%  
  \caption{An example of the construction for three  loops filling a surface of signature
    $(0,4)$. Here, we opened all the intersections the first way.}
  \label{fig:ex_op}
\end{figure}

\begin{rem}
  \label{rem:lambda_beta}
  The bars are naturally numbered as $(B_1, \ldots, B_r)$ thanks to our choice of numbering of the
  intersection points of~$\mathbf{c}$. However, there is no canonical way to number the components
  of $\beta$ based on our information so far, and because picking a numbering is not very useful for
  our purposes, we do not do so. We will denote as $\Lambdabeta$ the set of components of $\beta$
  (this notation will make more sense after  \S \ref{s:ppdecomponumb}), so that $\beta =
  (\beta_\lambda)_{\lambda \in \Lambdabeta}$.
\end{rem}

\subsubsection{Orientation of the bars}
\label{s:orientation_bars}

Let us define orientations $B_j^\pm$ on the bars $B_j$ for $1 \leq j \leq r$. The convention will
depend on the type of opening.

We first observe that the endpoints of the bars $B_1, \ldots, B_r$ partition the multi-loop $\beta$
into~$2r$ segments, reflecting the partition of ${\mathbf{ c}}$ into the simple portions defined in
Definition~\ref{d:double_fill}. We will call \emph{simple portions} of $\cop$ these subsegments of
$\beta$.  Each simple portion is naturally oriented by the original orientation of $\cop$. Each bar
$B_j$ touches four (non necessarily distinct) simple portions: two at each endpoint.

We then settle on the following convention, illustrated in Figure~\ref{fig:cop}.
 \begin{enumerate}
 \item In the \emph{first way}, we choose the orientation $B_j^+$ of $B_j$ such that $B_j^+$ arrives
   on the right of the two simple portions touching its terminus $t(B_j^+)$: one of these simple
   portions originates in $t(B_j^+)$ and the other one terminates in $t(B_j^+)$. Similarly, $B_j^+$
   leaves from the left of the two simple portions touching its origin $o(B_j^+)$; one of them
   originates in $o(B_j^+)$ and the other one terminates in $o(B_j^+)$.  In this case we denote by
   $B_j^-$ the opposite orientation of $B_j$.
 \item In the \emph{second way}, we choose the orientation $B_j^+$ of $B_j$ such that the two simple
   portions touching its terminus $t(B_j^+)$ originate in $t(B_j^+)$, and the two simple portions
   touching its origin $o(B_j^+)$ terminate in $o(B_j^+)$.  In this case, we define the orientation
   $B_j^-$ to be identical to $B_j^+$.
 \end{enumerate}
%By default, we decide to choose the \emph{first way}. This choice is arbitrary and plays no major role in most of the paper: the discussions of  \S \ref{s:nc}, \ref{s:lengthf}, \ref{s:noncross},
%\ref{s:proof}, \ref{s:horrible} would remain valid had we opened the intersections the second way.

%Opening intersections the second way will mostly be useful in \S \ref{s:skeleton} to construct the so-called ``skeleton'', and in the discussion of \S \ref{s:othercases}.
 
%We denote $B_j$ the geometric image of $B_j^\eps$ for $\eps=\pm$: that is to say, we write $B_j$ when we forget about the parametrisation and the orientation of $B_j^\eps$. 

 \subsubsection{Labelling of simple portions} \label{s:labelling_msp}

 Let us label the simple portions and define a few useful notations. Our most important set of
 indices will be the set
 \begin{equation}
   \label{eq:1}
   \Theta=\{1, \ldots, r\}\times \{+, -\},
 \end{equation}
 which will naturally label the simple portions.  Indeed, for $(j, \eps)\in \Theta$, let
 $\cI_j^\eps$ be the oriented simple portion originating from the terminus of $B_j^\eps$, so that
 $B^{+}_j$ arrives on~$\cI_j^+$ from the right at the point $o(\cI_j^+)=t(B_j^+)$ and $B^{-}_j$
 arrives on $\cI_j^-$ from the left at the point~$o(\cI_j^-)=t(B_j^-)$, as represented in Figure
 \ref{fig:cop}.

\begin{rem}
  Note that, out of the four simple portions touching the bar $B_j$, only two are labelled
  $\mathcal{I}_j^\pm$ by the above convention. The other two simple portions are denoted as
  $\mathcal{I}_{j'}^\pm$ for another (possibly equal) integer $j'$.
\end{rem}

We shall most often refer to the elements of $\Theta$ by the letter $q$, so that if
$q = (j, \epsilon)$ is an elements of $\Theta$, then the bar $B_q$ stands for $B_j^\eps$, $\cI_q$
for $\cI_j^\eps$, etc.

The signs in $\Theta$ will play an important role throughout the paper.

\begin{nota} \label{def:sign}
  We define the sign function on $\Theta$ by:
  \begin{align} \label{e:defsign} \forall q = (j, \epsilon) \in \Theta, \quad \sign(q)= \eps.
  \end{align}
  We denote as $\iota : \Theta \rightarrow \Theta$ the involution defined by $\iota(j,\eps)=(j,-\eps)$.
\end{nota}

It will be helpful to relate the elements of $\Theta$ and the components of the multi-loop $\beta$.

\begin{nota}
  \label{nota:ind}
  For $q \in \Theta$, let $\lambda(q) \in \Lambdabeta$ denote the component of $\beta$ on which
  $B_q$ terminates, i.e. such that the terminus $t(B_q)$ belongs to the loop $\beta_{\lambda(q)}$.

  For $\lambda \in \Lambdabeta$ (i.e. for a component $\beta_\lambda$ of the multi-loop $\beta$), we
  define the set of indices
  \begin{equation} \label{eq:def_ind_j}
    \Theta_t(\lambda):=\{q \in \Theta : t(B_q) \in \beta_\lambda\} \subseteq \Theta.
  \end{equation}
  In other words, this is the set of indices $q \in \Theta$ such that $\lambda(q) = \lambda$, or,
  equivalently, such that the bar $B_q$ terminates on $\beta_\lambda$.  We extend this notation to
  any subset $W \subseteq \Lambdabeta$ by letting
  \begin{equation}
    \label{eq:def_ind_W}
    \Theta_t(W) := \bigcup_{\lambda \in W} \Theta_t(\lambda)
    = \{q \in \Theta : \lambda(q) \in W\}.
  \end{equation}
\end{nota}

\begin{rem}
  \label{rem:theta_disj}
  We make the key observation that, since the multi-loop $\beta$ is decomposed in the simple
  portions $\cI_q$ with $q \in \Theta$, the sets $(\Theta_t(\lambda))_{\lambda \in \Lambdabeta}$
  generate a partition of $\Theta$:
  $$\Theta = \bigsqcup_{\lambda \in \Lambdabeta}\Theta_t(\lambda).$$
\end{rem}

\begin{rem} \label{r:init} It will be useful to keep in mind from the construction that each
  component $\beta_\lambda$ of $\beta$ is freely homotopic to an ``initial'' piecewise smooth curve
  $\beta_{\lambda,\mathrm{init}}$, included in the original multi-loop ${\mathbf{c}}$, made of
  smooth segments joining some of the intersection points of $\mathbf{c}$. The partition of
  $\beta_\lambda$ into the union of a collection of segments $(\cI_q)_{q \in \Theta_t(\lambda)}$ is
  inherited from a partition of $\beta_{\lambda,\mathrm{init}}$ into segments
  $(\cI_{q,\mathrm{init}})_{q \in \Theta_t(\lambda)}$ which are the simple portions of Definition
  \ref{d:double_fill}.
   \end{rem}

\subsubsection{Diagrams}
\label{sec:diagrams}

The result of opening all of the intersections of the multi-loop $\mathbf{c}$ is what we call a
diagram.

\begin{defa}
  \label{defa:diagram}
   A \emph{diagram} $\mathbf{D}$ in $\mathbf{S}$ is the data (modulo isotopy) of a simple multi-loop
   $\beta$, together with a collection of disjoint bars $B=(B_1, \ldots, B_r)$ defined as closed,
   mutually disjoint segments, having their endpoints in $\beta$ and disjoint from $\beta$
   otherwise.
\end{defa} 

 Starting with a multi-loop $\mathbf{c}$ and opening all intersections (either way), we obtain a
 diagram~$\mathbf{D}$, which is said to \emph{represent} $\mathbf{c}$.  Conversely, the multi-loop
 $\mathbf{c}$ is said to \emph{stem from} $\mathbf{D}$ if $\mathbf{D}$ {represents}~$\mathbf{c}$.

Different representatives of the same homotopy class may give rise to non-isotopic diagrams: see
Figure \ref{fig:ex_op_2}.  The paper by Graaf and Schrijver \cite{graaf1997} precisely
describes when two multi-loops can be homotopic without being isotopic: this may happen if and only
if one can be sent to the other by an isotopy followed by a finite succession of Reidemeister moves
of type III.

\subsubsection{First v.s. second way}
\label{sec:first-v.s.-second}

Throughout this article, we will be in particular interested in the case when we open all of the
intersections the first way.  We denote by $\mathbf{D}^{(1)}$ the diagram obtained by opening all
intersections the first way.  One possible reason for preferring the first way is the fact that the
multi-loop $\beta$ then inherits a global orientation that fits with the orientation of the simple
portions $(\cI_q)_{q \in \Theta}$, in other words with the initial orientation of ${\mathbf{c}}$. If
we open the intersections the second way, the orientations of the simple portions
$(\cI_q)_{q \in \Theta}$ do not fit into a global orientation of $\beta$, which is not a problem,
but can sometimes be slightly cumbersome in terms of notations.

Unless explicitely stated, for the sake of simplicity, we invite the reader to consider that the
construction above has been done by \emph{opening every intersection the first way}. However, many
of our results, notably our change of variable presented in \S \ref{s:nc}, hold for whichever choice
of opening is made -- different openings yielding different interesting results. We will be careful
to point out the few places where the choice of opening actually does matter.

\subsection{Generalized eights} \label{s:gene_eight} We introduce a class of multi-loops, called
\emph{generalized eights}, on which we focus attention for most of the paper. Their name comes from
the fact that they are generalizations of the basic figure-eight loop, the loop with exactly one
self-intersection ($r=1$).
 
\begin{defa} \label{d:GE} Let $ \cN$ be a regular neighbourhood of ${\mathbf{ c}}$ in
  ${\mathbf{S}}$. We say that ${\mathbf{ c}}$ is a \emph{generalized eight} if no connected
  components of $\partial \cN$ is contractible in ${\mathbf{S}}$.
\end{defa}

\begin{rem} We recall that we have assumed at the beginning of \S \ref{s:definitions1} that
    $\Sf$ is a connected surface, and that it is not a cylinder. In particular, simple loops are not
    generalized eights. The definition above still holds verbatim in the case where $\Sf$ has
    several connected components (provided none of them is a cylinder).
\end{rem}

 This definition has the following consequence on the surface $\mathbf{S}$ and the multi-loop $\beta$.

 \begin{lem}
   \label{r:GE}
   If $\mathbf{c}$ is a generalized eight, then $\chi(\mathbf{S})=r$ and $\beta$ is a multi-curve.
 \end{lem}

 \begin{proof}
   {
   Let $\cN$ be a regular neighbourhood of $\mathbf{c}$ in $\mathbf{S}$. Since $\mathbf{c}$ fills
   $\mathbf{S}$ and no boundary component of $\cN$ is contractible, $\mathbf{S} \setminus \cN$ is a
   union of cylinders, and hence $\chi(\mathbf{S}) = \chi(\cN)$. We then conclude by observing that
   the regular neighbourhood $\cN$ retracts on the $4$-regular graph with vertices the $r$
   intersections points of $\mathbf{c}$, and edges its $2r$ simple portions, of Euler
   characteristic $r-2r=-r$.}
   
   The definition of generalized eight directly implies that none of the components of $\beta$ is
   contractible. To prove that $\beta_i$ cannot be freely homotopic to $\beta_j$ or its inverse for
   $i \neq j$ two components of $\beta$, we remark that if $\beta_i$ is freely homotopic to
   $\beta_j$ or its inverse, then there is a cylinder in ${\mathbf{S}}$ bordered by $\beta_i$ and
   $\beta_j$. This cylinder may contain in its interior other components $\beta_k$ homotopic to
   $\beta_i^{\pm 1}$. We pick $k$ such that the cylinder $C$ bordered by $\beta_i$ and
   $\beta_k$ does not contain in its interior any other component of the multi-loop $\beta$. Because
   ${\mathbf{ c}}$ fills ${\mathbf{S}}$, the cylinder $C$ must contain bars $B_l$ going from
   $\beta_i$ to $\beta_k$. These bars, together with $\beta_i$ and $\beta_k$, cut $C$ into
   connected components homeomorphic to disks, contradicting the fact that $\partial \cN$ does not
   have contractible components.
 \end{proof}

 \begin{rem}
   If $\mathbf{c}$ is a generalized eight, Reidemeister moves of type III are not possible and hence
   any multi-loop homotopic to $\mathbf{c}$ is also isotopic. As a consequence, the diagram
   $\mathbf{D}^{(1)}$ obtained by opening the intersections the first way is uniquely defined up to
   isotopy.
 \end{rem}

\subsection{Removing bars and intersections}\label{s:removing}

We define two ways of \emph{removing} a bar from a diagram: either by just deleting it (we call this
the \emph{first way} of removing a bar), or the \emph{second way}, shown on 
Figure~\ref{fig:removing_bar}.

\begin{figure}[h!]
    \begin{subfigure}[b]{0.5\textwidth}
    \centering
    \includegraphics{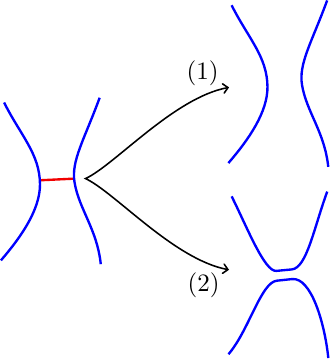} 
    \caption{A bar in a diagram.}
    \label{fig:removing_bar}
  \end{subfigure}%
    \begin{subfigure}[b]{0.5\textwidth}
    \centering
    \includegraphics{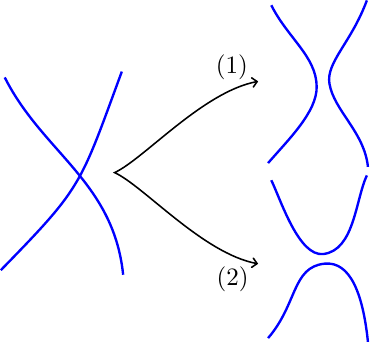} 
    \caption{An intersection on a multi-loop.}
    \label{fig:removing_int}
  \end{subfigure}%
  \caption{The two ways of removing a bar on a diagram or an intersection point on a multi-loop.}
    \label{fig:removing}
  \end{figure}%

  If the diagram was obtained by opening all intersections of a multi-loop ${\mathbf{ c}}$, removing
  a bar corresponds to removing an intersection from ${\mathbf{ c}}$.  Correspondingly, there are
  two ways to remove an intersection from ${\mathbf{ c}}$ (the first and the second way), shown on
  Figure~\ref{fig:removing_int}.

  We make the following straightforward observation, which will enable us to argue by induction on $r$ in the proof of Theorem \ref{t:maindet}.

\begin{lem}
  If ${\mathbf{ c}}$ is a generalized eight, then the multi-loop ${\mathbf{ c}}'$ obtained by
  removing an intersection point is still a generalized eight. It fills a surface of absolute Euler
  characteristic $r-1$.
\end{lem}

Note that the number of components of the multi-loop ${\mathbf{ c}}'$ is, a priori, different from
the number of components of $\mathbf{c}$; this is one of the motivations for studying multi-loops
instead of mere loops.

\subsection{Reconstitution of $\cop$ and relabelling of $\Theta$}
\label{s:cop_relabel}
 %  The construction of  \S \ref{s:coord} will give us a representative of the homotopy class of ${\mathbf{ c}}$, expressed as a concatenation of pairwise orthogonal geodesic segments. This will prove especially
 %  convenient to find an explicit expression $\ell({\mathbf{ c}})$ in the new coordinate system on Teichm\"uller space. 
Let us explain how to retrieve the multi-loop $\mathbf{c}$ from the diagram, and define some
convenient notations to express the representant $\cop$ of the homotopy class of $\mathbf{c}$ in
terms of a succession of bars and simple portions.

\subsubsection{Cycle-decomposition and rewriting of $\cop$}
\label{sec:cycle-decomp-rewr}

For $q \in \Theta$, we define
\begin{align}\label{e:psegment}p(q):=B_q \smallbullet \cI_q\end{align}
to be the concatenation of the oriented segments $B_q$ and $\cI_q$ (which do, by definition,
connect). The representative $\cop$ of the homotopy class of ${\mathbf{c}}$ is the concatenation of
the $2r$ paths $p(q)$ for $q\in \Theta$.  More precisely, there is a unique permutation $\sigma$ of
the set $\Theta$, decomposed into $\cc$ disjoint cycles
$\sigma = \prod_{i=1}^{\cc} (q_1^i \, q_2^i \, \ldots \, q_{n_i}^i)$ such that, for all
$i \in \{1, \ldots, \cc\}$, the $i$th component $\cop_i$ of $\cop$ coincides with the concatenation
\begin{align}\label{e:decompo-i} \cop_i=p(q_1^{i})\smallbullet  p(q_2^{i})\smallbullet\ldots \smallbullet p(q_{n_i}^{i})\end{align}
(or a circular permutation thereof).

The length of the $i$-th cycle is denoted as $n_i$, and we have $\sum_{i=1}^{\cc} n_i=2r$. Because
of the invariance of the free homotopy class of \eqref{e:decompo-i} under circular permutation, it
is convenient to consider that the elements $q_k^{i} = (j^i_k,\eps^i_k) \in \Theta$ are numbered by
$k\in \Z_{n_i}$ where $\Z_{n_i} := \Z \diagup n_i \Z$.

\begin{rem}
  By construction, we have the following orientation rules in \eqref{e:decompo-i}, see Figure
  \ref{fig:cop}:
  \begin{itemize}
  \item if $\eps^i_k=+$ then $B_{q^i_{k}}$ arrives at $\cI_{q^i_{k}}$ from the right and leaves
    from $\cI_{q^i_{k-1}}$ from the left;
  \item otherwise, $B_{q^i_{k}}$ arrives at $\cI_{q^i_{k}}$ from the left and leaves from
    $\cI_{q^i_{k-1}}$ from the right.
  \end{itemize}
\end{rem}

 \subsubsection{Re-labelling of $\Theta$} \label{s:relabel_Theta}

 For $q \in \Theta$, let $\comp(q)=i \in \{1, \ldots, \cc\}$ be the label of the
 component~${\mathbf{c}}_i$ of $\mathbf{c}$ in which $q$ appears in the decomposition
 \eqref{e:decompo-i}. We set
  $$\Theta^i=\{q\in \Theta, \comp(q)=i\}.$$
  Then, for any $q \in \Theta^i$, we can write using the permutation $\sigma$:
  \begin{align}\label{e:decompo-ibis} \cop_i= p(q) \smallbullet p(\sigma q)\smallbullet p(\sigma^2 q)\smallbullet \ldots \smallbullet p(\sigma^{n_i-1} q).
  \end{align}

  We now consider the set
  $$\Theta'=\bigsqcup_{i=1}^{\cc} \Z_{n_i}=\{(i, k), i \in \{1, \ldots, \cc \}, k\in \Z_{n_i}\}.$$
  The map
  \begin{align}\label{e:phi}
    \phi:
    \begin{cases}
      \Theta' & \rightarrow \Theta \\
      (i, k) & \mapsto q^i_{k} = (j^i_k,\eps^i_k)
    \end{cases}
 \end{align}
 is a bijection, and the cyclic ordering of $\Z_{n_i}$ is compatible with the cycling ordering of
 $q^i_{k}$ in the path \eqref{e:decompo-i}.  We shall identify $\Theta$ with $\Theta'$ through the
 map $\phi$, thus writing $(i, k)$ for $1 \leq i \leq \cc$, $k\in \Z_{n_i}$ to enumerate the
 elements of $\Theta$. Doing so, $\Theta^i$ can be identified with $\{(i, k), k\in \Z_{n_i}\}$, that
 is to say, with $\Z_{n_i}$.  After this identification, the permutation $\sigma$ is given by $ (i, k)\mapsto (i, k+1)$.

 \subsubsection{Periodic extensions} \label{s:periodic_Theta}
 Similarly define
 $${\mathbf{\Theta}}=\bigsqcup_{i=1}^{\cc} \Z=\{(i, k), i\in\{1, \ldots, \cc\}, k\in
 \Z\}=\bigsqcup_{i=1}^{\cc} {\mathbf{\Theta}}^i,$$ where ${\mathbf{\Theta}}^i=\{(i, k), k\in \Z\}$.
 When needed, the sequence $(i,k)\mapsto q^i_k$ can also be defined for
 $(i, k)\in {\mathbf{\Theta}}$ by pre-composing with the projection map $\Z\rightarrow \Z_{n_i}$.
 In accordance with the previous discussion, we will denote
 $\sigma: {\mathbf{\Theta}}\rightarrow {\mathbf{\Theta}}$ the shift map $ (i, k)\mapsto (i, k+1)$.

 \subsection{U-turns and crossings} \label{s:def_crossing}

 The following definition will be useful to our analysis.
 \begin{defa} \label{def:crossing} For $q \in \Theta$, we say that $q$ is a \emph{crossing index} if
   $\sign(q) \sign(\sigma q) = +1$, and a \emph{U-turn index} otherwise.  We denote by
   $\cC\subseteq \Theta$ the set of crossing indices and by $\cU \subseteq \Theta$ the set of U-turn
   indices. If $\Theta=\cU$, we say that we are in the {\emph{purely non-crossing case}}.
 \end{defa}

\begin{figure}[h!]
  \centering
   \begin{subfigure}[b]{0.5\textwidth}
    \centering
     \includegraphics{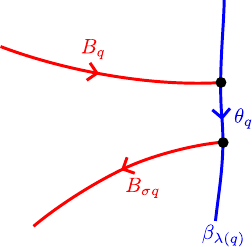}
    \caption{U-turn.}
    \label{fig:bars}
  \end{subfigure}%
  \begin{subfigure}[b]{0.5\textwidth}
    \centering
     \includegraphics{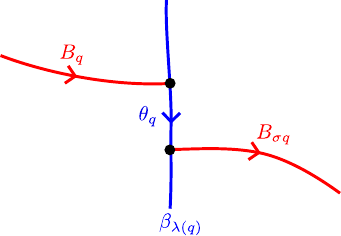}
    \caption{Crossing.}
    \label{fig:bars}
  \end{subfigure}%
  \caption{A U-turn v.s. crossing index (here $\sign(q)=+$).}
  \label{fig:crossing}
\end{figure}
The terminology is motivated by the fact that if $q$ is a crossing index, then to go from $B_{q}$ to
$B_{\sigma q}$ we have to cross the component $\beta_{\lambda(q)}$ of $\beta$, while in the other case,
we make a U-turn when arriving at $\beta_{\lambda(q)}$. See Figure \ref{fig:crossing}.

%%%Local Variables: 
%%% mode: latex
%%% TeX-master: "main"
%%% End: 

\section{Coordinate system adapted to a generalized eight}
\label{s:nc}
\textbf{ We assume until \S \ref{s:othercases} that the multi-loop ${\mathbf{ c}}$ is a generalized
  eight filling a surface~${\mathbf{S}}$.}  We open the intersections of ${\mathbf{ c}}$ by the
procedure defined in the previous section (in this section, we allow every intersection to be opened
in the first or second way). By hypothesis, the simple multi-loop $\beta$ is a multi-curve. Let us
show how the bars $B_1, \ldots, B_r$ and the simple portions $(\cI_q)_{q\in \Theta}$ can be used to
\emph{define new coordinates} on the Teichm\"uller space $\cT^*_{\g, \n}$ of marked hyperbolic
metrics on $\mathbf{S}$, well-adapted to expressing the length function of ${\mathbf{ c}}$.

First, we notice that the dimension of $\cT_{\g,\n}^*$ is equal to $6\g-6+3\n = 3
\chi(\mathbf{S})$. We proved in Lemma \ref{r:GE} that, for a generalized eight, $\chi(\mathbf{S})
= r$. As a consequence, the dimension of the Teichm\"uller space $\cT^*_{\g, \n}$ is $3r$, so this
is the number of coordinates we expect to need to describe a point on $\cT_{\g,\n}^*$.

\subsection{Coordinates $\vec{L}=(L_1, \ldots, L_r)$ and
   $\vec{\theta}=(\theta_1^{\pm}, \ldots, \theta_r^{\pm})$ on the Teichm\"uller
   space}\label{s:coord}

 Consider a point $Y\in \cT^*_{\g, \n}$, i.e. a compact hyperbolic metric on $\mathbf{S}$.  As
 explained in Remark~\ref{rem:beta_geod}, we take a representative of $Y$ in the Teichmüller space
 such that the multi-curve $\beta$ is a simple multi-geodesic.

 For $q \in \Theta$, we replace the bar $B_q$ by the representative $\overline{B}_q$ of minimal
 length in its homotopy class with endpoints gliding on $\beta$.  While the endpoints $t(B_q)$ glide
 to the new endpoints $t(\overline{B}_q)$ along $\beta$, it is important to note that the segments
 $\cI_q$ glide to new segments $\overline{\cI}_q$ of $\beta_{\lambda(q)}$.  Although the original
 segments $\cI_q$ were positively oriented along $\beta_{\lambda(q)}$, it can happen that
 $\overline{\cI}_q$ goes in the reverse direction, or that $\overline{\cI}_q$ wraps an arbitrary
 number of times around $\beta_{\lambda(q)}$.

We are now ready to define our new coordinates on the Teichm\"uller space $\cT_{\g,\n}^*$.
 
 \begin{defa}[Coordinates $(\vec{L},\vec{\theta})$] Let $Y \in \cT_{\g,\n}^*$. 
   \begin{enumerate}
   \item For $j\in\{1, \ldots, r\}$, we denote by $L_j = L_j(Y)$ the (positive) length of the
     orthogeodesic~$\overline{B}_j$.  We write $\vec{L}=(L_1, \ldots, L_r)$.
   \item For $q \in \Theta$, we denote by $\theta_q = \theta_q(Y)$ the {\em algebraic} length of
     $\overline{\cI}_q$: positive if it goes in the same direction as $\cI_{q}$, negative
     otherwise. We write $\vec{\theta}=(\theta_q)_{q\in \Theta}$. 
   \end{enumerate}
 \end{defa}

\begin{figure}[h]
  \centering
  \includegraphics{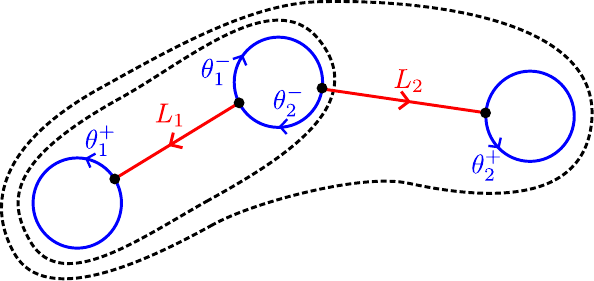}
  \caption{An example of coordinates $(\vec{L}, \vec{\theta})$ in a case when $r=2$. We have
    additionally represented the pair of pants decomposition constructed in \S \ref{s:ppdecompo} as
    dashed curves (here, $\# \Lambdabeta = 3$ and $\# \LambdaFN=2$).}
  \label{fig:expl_pop}
\end{figure}

 \begin{rem}[Crossing and U-turn variables]
   If $q\in \cC$ is a crossing index, then we say that $\theta_q$ is a \emph{crossing variable}; if
   $q\in \cU$ is a U-turn variable, then we say that $\theta_q$ is a \emph{U-turn variable}.  By the
   Useful Remark (Remark \ref{r:useful}), U-turn variables are always positive: if $q \in \cU$, then
   $\theta_q \geq 0$. Thus, the discussion concerning U-turn variables is a little simpler than for
   crossing variables.
 \end{rem}

   \begin{nota} \label{nota:signL} It will often be convenient to use the notation $L(q)=L_j$ for
     $q=(j, \eps) \in \Theta$ (note that $L(q)$ does not depend on the sign component of $q$). We
     will also benefit from introducing $\epsilon L(q) := \sign(q) L(q)$. The vector
     $(\epsilon L(q))_{q \in \Theta}$ then contains each length $L_j$ exactly twice, once positively
     and once negatively. This corresponds to the fact that each bar is explored exactly twice when
     going along $\cop$, following its two orientations, as illustrated in Figure \ref{fig:cop}.
 \end{nota}
 
 We can already note that $(\vec{L}, \vec{\theta})$ belongs to
 $\R_{> 0}^r\times \R^{\Theta}\subset \R^{3r}$, a space of same dimension as the Teichm\"uller space
 $\cT^*_{\g, \n}$. The remainder of this section is devoted to proving that
 $(\vec{L}, \vec{\theta})$ is a coordinate system on $\cT^*_{\g, \n}$, finding the expression of the
 Weil--Petersson measure in these coordinates, and describing its range.

 We do not start from scratch, but use the fact that the expression for the Weil--Petersson measure
 is already known in Fenchel--Nielsen coordinates, thanks to Wolpert's theorem
 \cite{wolpert1985}. In order to apply Wolpert's result, we use an auxiliary decomposition of
 ${\mathbf{S}}$ into pairs of pants, determined by the multi-curve $\beta$ and the bars
 $B_1, \ldots, B_r$. We find relations between the Fenchel--Nielsen coordinates and the parameters
 $(\vec{L}, \vec{\theta})$, and calculate the determinant.
 \begin{rem}
   This approach seems a bit artificial, and it would be interesting to study directly the
   Weil--Petersson measure in coordinates $(\vec{L}, \vec{\theta})$, without resorting to an
   auxiliary system of Fenchel--Nielsen coordinates.
 \end{rem}

\subsection{Decomposition into pairs of pants}
\label{s:ppdecompo}

Using the multi-curve $\beta$ and the bars $B_1, \ldots, B_r$, we define a decomposition of
${\mathbf{S}}$ into pairs of pants (3-holed spheres) that will serve as an auxiliary tool in several
places.

\subsubsection{Pair of pants determined by two simple curves joined by a segment}

Let $(C_+, C_-)$ be a multi-curve on $\mathbf{S}$, and $B$ be a simple segment going from $C_-$ to
$C_+$, not intersecting $C_-\cup C_+$ except at its endpoints. The topological pair of pants
determined by $C_-$, $C_+$ and $B$ is defined as a one-sided regular neighbourhood of
$C_- \cup C_+ \cup B$; that is to say, a regular neighbourhood $\mathcal{N}$ of
$C_- \cup C_+ \cup B$, from which we remove the cylinders bordered by one of the curves $C_\pm$ and
one component of $\partial \mathcal{N}$.

This also applies when $C = C_+=C_-$ as soon as $B$ is not homotopically trivial with gliding
endpoints along~$C$. We have either a pair of pants if $B$ leaves and arrives on the same side
of~$C$, or a once-holed torus otherwise.

The three cases are shown in Figure \ref{fig:ppC}.
  \begin{figure}[h!]
    \includegraphics[height=3.5cm]{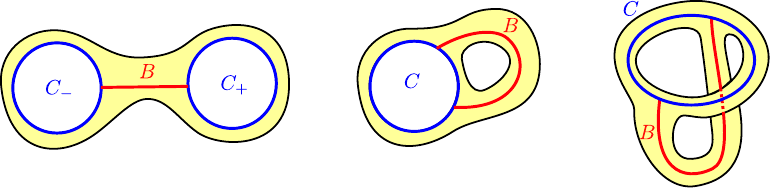}
    \caption{The three possible pairs of pants (or once-holed tori) determined by two simple curves
      joined by a segment.}
     \label{fig:ppC}
  \end{figure}

  \subsubsection{Construction of the pair of pants decomposition}

  The (arbitrary) numbering of $B_1, \ldots, B_r$ determines a decomposition
  $({\bP}_1, \ldots, {\bP}_r)$ of ${\mathbf{S}}$ into connected surfaces of Euler characteristic
  $-1$, defined inductively below. This decomposition is non-canonical, as it depends on the
  numbering of the bars $B_1, \ldots, B_r$.
  
  The construction  provides us with an increasing family of multi-curves
  $\Sigma_0\subset \Sigma_1\subset \ldots \subset \Sigma_r$, and an increasing family of surfaces
  ${\mathbf{S}}_{j}= {\mathbf{S}}_{j-1}\cup {\bP}_{j}$ for $1\leq j\leq r$. It will also
  produce an auxiliary family of sets $\widetilde{{\mathbf{S}}}_j$.

\begin{itemize}
\item To initiate, we let ${\mathbf{S}}_0=\emptyset$ and $\widetilde{{\mathbf{S}}}_0=\Sigma_0$ be
  the multi-curve $\beta$.
\item Let $j\in \{ 1, \ldots, r \}$. The intersection of the $j$-th bar $B_{j}$ with
  $\widetilde{{\mathbf{S}}}_{j-1}$ is the union of two non-empty closed segments $I_+$, $I_-$
  (possibly reduced to single points) containing the endpoints of~$B_{j}$.  The complement of
  $I_+\cup I_-$ in $B_{j}$ is a non-empty sub-segment $B'_{j}$ of $B_{j}$ which meets one or two
  components of $\partial \widetilde{{\mathbf{S}}}_{j-1}$ at its endpoints; call $C_+, C_-$ these
  components.  We let ${\bP}_{j}$ be the pair of pants determined by $B'_{j}$, $C_+$ and $C_-$.
\item We then let $\Sigma_{j}=\Sigma_{j-1}\cup \partial {\bP}_{j}$ (it is a multi-curve),
  $\widetilde{{\mathbf{S}}}_{j} =\widetilde{{\mathbf{S}}}_{j-1}\cup {\bP}_{j}$ (the disjoint union
  of a surface and a multi-curve) and
  ${\mathbf{S}}_{j}= {\mathbf{S}}_{j-1}\cup {\bP}_{j}={\bP}_1\cup\ldots \cup {\bP}_{j}$ (a bona-fide
  surface). They are related by
  $\widetilde{{\mathbf{S}}}_{j}={\mathbf{S}}_{j} \cup \beta$.
\end{itemize}
By an argument similar to Remark \ref{r:GE}, because $\mathbf{c}$ is a generalized eight, the
boundary of each pair of pants~$\mathbf{P}_j$ is a multi-curve in $\mathbf{S}$. The construction
finishes for $j=r$, and we have $\widetilde{{\mathbf{S}}}_{r}={{\mathbf{S}}}_{r}={\mathbf{S}}$,
because the multi-loop~$\mathbf{c}$ fills the surface~$\mathbf{S}$.

\begin{rem}
  \label{rem:remove_int_eight}
 For an integer $1 \leq j \leq r$, one can consider the multi-loop $\curve'_0$ obtained by
    removing the intersections $j, \ldots, r$ from $\mathbf{c}$ according to the procedure
    described in \S \ref{s:removing} (here we remove each intersection the same way as we did when
    replacing them by bars). This will be particularly useful for reasoning by induction, in the
    case $j=r$. Note that $\curve'_0$ is a disjoint union of a (possibly disconnected) generalized
    eight $\curve'$ and a (possibly empty) multi-curve. The generalized eight $\curve'$ fills the
    surface $\Sf_j$.  The pairs of pants decomposition $\mathbf{P}_1, \ldots, \mathbf{P}_j$
  associated to the multi-loops $\mathbf{c}$ and $\mathbf{c}'$ are the same.
\end{rem}

\subsubsection{Reduction to Case (a) and (b)}
\label{sec:case-a-b}

We shall assume (without loss of generality) that the numbering of $B_1, \ldots, B_r$ is such that
the first bars $B_1, \ldots, B_{r_0}$ are in the third situation of Figure \ref{fig:ppC} (that is,
depart and arrive on one unique component of $\beta$, on opposite sides) and the remaining
$B_{r_0+1}, \ldots, B_r$, either join two different components of $\beta$, or depart and arrive on
one component of $\beta$ on the same side, or leave from a component of $\beta$ already touched by
$B_1, \ldots, B_{r_0}$.  Then, ${\bP}_1, \ldots, {\bP}_{r_0}$ are disjoint once-holed tori, and
${\bP}_{r_0+1}, \ldots, {\bP}_r$ are \emph{bona fide} pairs of pants falling into one of the first
two cases of Figure \ref{fig:ppC}.

As a consequence, for $j > r_0$, there are only two possible scenarios that can happen at the $j$-th
step in the construction of the pair of pants decomposition, which we shall call Case (a) and Case
(b); they are represented in Figure \ref{fig:C1234b}. 
  \begin{enumerate}
  \item[\textbf{(a)}] $C_-$ and $C_+$ are distinct (i.e. the bar $B_{j}'$ has endpoints on different
    components of~$\partial \tilde{\mathbf{S}}_j$), and hence $\mathbf{S}_{j}$ has one more genus
    and one less boundary component than $\mathbf{S}_{j-1}$.
  \item[\textbf{(b)}] $C_- = C_+ =: C$ (i.e. the bar $B_j'$ has endpoints on one unique component of
    $\partial \tilde{\mathbf{S}}_j$), and hence $\mathbf{S}_j$ has one more boundary component than
    $\mathbf{S}_{j-1}$, and the same genus.
  \end{enumerate}
  We will refer to these two cases in proofs by induction.

  \begin{figure}[h!]
    \begin{subfigure}[b]{0.5\textwidth}
    \centering
     \includegraphics[height=4cm]{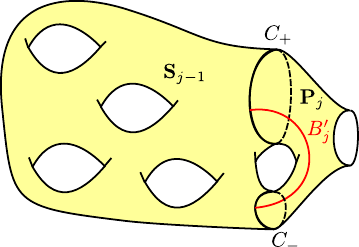}
    \caption{Case (a).}
    \label{fig:casea}
  \end{subfigure}%
  \begin{subfigure}[b]{0.5\textwidth}
    \centering
    \includegraphics[height=4cm]{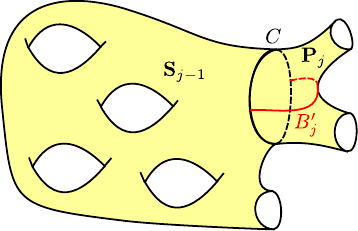}
    \caption{Case (b).}
    \label{fig:caseb}
  \end{subfigure}%  \begin{subfigure}[b]{0.33\textwidth}
  \caption{The two possible scenarios when constructing $\mathbf{P}_j$.}
    \label{fig:C1234b}
  \end{figure}

  \subsection{Set of boundary components of the pair of pants decomposition}
  \label{s:ppdecomponumb}

  \subsubsection{Definition} \label{sec:Lambda_def}

  Let us denote as $\Gamma = (\Gamma_\lambda)_{\lambda \in \Lambda}$ the multi-curve
  $\Sigma_r$ in the construction above, i.e. the multi-curve obtained by taking the boundary of the
  pair of pants decomposition $\textbf{P}_1, \ldots, \mathbf{P}_r$ of $\mathbf{S}$. It is indexed by
  a set~$\Lambda$ with $3\g-3+2\n$ elements.

  \begin{rem}
    \label{rem:pop_poly}
    The construction above provides a special representative of the homotopy class of each boundary
    component $\Gamma_\lambda$, as a concatenation of bars $\overline{B}_q$ and simple portions of
    $\beta$ (see the forthcoming Figure \ref{fig:Glambda} for an example). More precisely, we will
    show later that the $\Gamma_\lambda$ are what we call \emph{polygonal curves}, as defined in
    Definition \ref{d:pc}. This property will play a key role in \S \ref{s:bound} and \S
    \ref{s:lengthboundary}.
 \end{rem}

 \begin{rem} \label{r:initlambda} Extending Remark \ref{r:init}, it will also be convenient to keep
   in mind that each boundary component $\Gamma_\lambda$ is freely homotopic to a piecewise smooth
   curve $\Gamma_{\lambda,\mathrm{init}}$, included in the original ${\mathbf{ c}}$, made of smooth
   segments joining some of the intersection points of $\mathbf{c}$.
 \end{rem}  

 \subsubsection{Partitions of $\Lambda$} \label{sec:part_lambda}
  Some of the elements of $\Lambda$ are the boundary components of $\mathbf{S}$ -- we denote as
  $\LambdaBC$ this set, which is in natural bijection with
  $\partial \mathbf{S} = \{1, \ldots, \n\}$. The other elements are inner curves, and we denote this
  set as $\Lambdain := \Lambda \setminus \LambdaBC$. We have $\# \Lambdain= 3\g - 3 +\n$.
  
  We notice that the multi-curve $\beta$ is part of this pair of pants decomposition (because
  $\beta = \Sigma_0 \subset \Sigma_r$). To reflect this, we define $\Lambdabeta \subseteq \Lambda$
  to be the set of components of $\Gamma$ which are also a component of the multi-curve
  $\beta$. This motivates the writing of the multi-curve $\beta$ as
  $(\beta_\lambda)_{\lambda \in \Lambdabeta}$ introduced before, where
  $\Gamma_\lambda = \beta_\lambda$ for $\lambda \in \Lambdabeta$.
  
  The components of $\Gamma$ which are \emph{not} part of the multi-loop $\beta$, and have therefore
  been created in the construction above , will be denoted as
  $\LambdaFN := \Lambda \setminus \Lambdabeta$. The partition $\Lambda = \LambdaBC \sqcup \Lambdain$
  between boundary/inner components naturally induces two partitions of $\Lambdabeta$ and
  $\LambdaFN$, defined by
  \begin{align*}
   & \LambdaBCbeta := \LambdaBC \cap \Lambdabeta 
   & \LambdaBCFN := \LambdaBC \cap \LambdaFN \\
   & \Lambdainbeta := \Lambdain\cap \Lambdabeta 
   & \LambdainFN := \Lambdain \cap \LambdaFN.
  \end{align*}

  \subsubsection{Order on $\Lambda$} 
  
  We do not fix a numbering of the set $\Lambda$. However, it will be convenient to equip $\Lambda$
  with an order relation in the following way.
  \begin{defa}\label{def:order_Lambda}
    For $\lambda \in \Lambda$, we denote as $\mathrm{step}(\lambda)$ the step of the induction when
    the component $\Gamma_\lambda$ is created, i.e.
    \begin{equation*}
      \mathrm{step}(\lambda) := \min \{ j \in \{1, \ldots, r \}  : \,
      \Gamma_\lambda \text{ is a boundary component of } \mathbf{P}_j\}.
    \end{equation*}
    This induces an order on $\Lambda$, defined by writing, for $\lambda, \lambda' \in \Lambda$,
    \begin{equation*}
      \lambda < \lambda'  \quad \Leftrightarrow \quad
      \mathrm{step}(\lambda) < \mathrm{step}(\lambda').    
    \end{equation*}
  \end{defa}
  The relation $\lambda < \lambda'$ simply means that $\Gamma_\lambda$ appeared earlier than
  $\Gamma_{\lambda'}$ in our recursive construction of the pants decomposition.
  \begin{rem}
    This order is strongly connected (for any $\lambda$, $\lambda'$, either $\lambda \leq \lambda'$
    or $\lambda' \leq \lambda$) but a priori not anti-symmetric (there would typically be distinct
    elements $\lambda \neq \lambda'$ with $\mathrm{step}(\lambda)=\mathrm{step}(\lambda')$, since we
    often create several boundary components at a given step of the induction).
  \end{rem}

  \begin{rem} \label{r:convenient} In the proof by induction of \S \ref{sec:proof-theor-reft:m}, it
    will be convenient to assume that the arbitrary choice of the numbering $\{1, \ldots, \n\}$ of
    the boundary of $\mathbf{S}$ has been made so that the function
    $\lambda \mapsto \mathrm{step}(\lambda)$ is non-decreasing as a function of
    $\lambda \in \LambdaBC = \{1, \ldots, \n\}$, i.e. for every $1 \leq i < j \leq \n$, the $j$-th
    boundary component of $\partial \mathbf{S}$ appears after (or at the same time as) the $i$-th
    boundary component in the construction of the pair of pants decomposition.  In particular, the
    new components appearing in the last step, when we add the final pair of pants $\mathbf{P}_r$,
    are labelled $\n$ in Case (a), or $\n$ and $\n-1$ in Case (b).
 \end{rem}

\subsection{Expression of the Weil--Petersson measure} \label{s:FNc}

We are now ready to state and prove our new expression for the Weil--Petersson measure in the
coordinates $(\vec{L},\vec{\theta})$.

\subsubsection{Expression in Fenchel--Nielsen coordinates}

For given boundary lengths $\x$, once a decomposition of $\mathbf{S}$ into pairs of pants
${\bP}_1, \ldots, {\bP}_r$ is given, it is standard to consider Fenchel--Nielsen coordinates
$(y_\lambda, \alpha_\lambda)_{\lambda\in \Lambda_{\mathrm{in}}}$ on
$\cT_{g_{\mathbf{S}}, n_{\mathbf{S}}}(\x)$. The coordinates $y_\lambda$ are the lengths
$\ell_Y(\Gamma_\lambda)$ of the closed geodesics homotopic to $\Gamma_\lambda$ for the metric
$Y$. The coordinates $\alpha_\lambda$ are twist parameters, defined up to translation. We will
explain in the proof of Theorem \ref{t:maindet} how to adequately choose the origin of twist
parameters in relation with the multi-loop ${\mathbf{c}}$ that we want to study.  By work of
Wolpert~\cite{wolpert1985}, we know that
\begin{equation}\label{e:WPFN}\d\mathrm{Vol}^{\mathrm{WP}}_{g_{\mathbf{S}}, n_{\mathbf{S}}, \x}(Y) =\prod_{\lambda\in \Lambda_{\mathrm{in}}} \d y_\lambda \d \alpha_\lambda.
\end{equation}
To obtain coordinates on $\cT^*_{g_{\mathbf{S}}, n_{\mathbf{S}}}$, we simply multiply this by the
Lebesgue measure for the length parameters $\x=(x_1, \ldots, x_{\n})$, in other words the collection
of $y_\lambda$ for $\lambda\in \LambdaBC$. Since $\Lambda = \Lambda_{\mathrm{in}} \cup \LambdaBC$,
we obtain the expression
\begin{equation}
  \label{eq:WPFN2}
  \d\mathrm{Vol}^{\mathrm{WP}}_{g_{\mathbf{S}}, n_{\mathbf{S}}}(\x,Y) =\prod_{\lambda\in \Lambda} \d y_\lambda \prod_{\lambda\in \Lambda_{\mathrm{in}}}\d \alpha_\lambda.
\end{equation}

\subsubsection{Main statement}
 We shall prove the following.
\begin{thm}\label{t:maindet}
  Let $\mathbf{c}$ be a generalized eight. Then, the map
  \begin{equation}
    \label{eq:ch_var}
    \begin{cases}
      \cT^*_{\g, \n} & \rightarrow \R_{>0}^{r} \times \R^\Theta \\
      (\x, Y) & \mapsto (\vec{L}, \vec{\theta})
    \end{cases}
  \end{equation}
  is a $\mathcal{C}^1$-diffeomorphism onto its image $\domain$,
  and
  \begin{align}\label{e:maindet}
    2^{\n} \prod_{j=1}^{\n} \sinh\div{x_j}  \d \mathrm{Vol}^{\mathrm{WP}}_{\g, \n}(\x,Y)
    = 2^{2 |\Lambdabeta|} \prod_{\lambda \in \Lambda^\beta} \sinh^2\div{y_\lambda}
    \prod_{i=1}^r \sinh(L_i) \d^r \vec{L} \d^{2r}
    \vec{\theta}
  \end{align}
  where $y_\lambda = \ell_Y(\beta_\lambda)$, $\d^r \vec{L}=\prod_{i=1}^r \d L_i$ and
  $\d^{2r} \vec{\theta}=\prod_{i=1}^r \d \theta_i^+ \d\theta_i^-$.
\end{thm}
We insist on the fact that this formula holds for any coordinate system $(\vec{L},\vec{\theta})$
obtained by opening the intersections of $\mathbf{c}$, meaning that intersection can be opened in the
first or second way. This yields different new sets of coordinates.
We postpone to \S \ref{s:domain} the description of the image~$\domain$ of the change of variable.

\begin{rem}
  Earlier in the paper we assumed that $\Sf$ was connected, but since the formula is multiplicative,
  it will also hold for disconnected $\Sf$.
\end{rem}

\subsubsection{Dirac variant}
\label{sec:dirac-variant}

The integral that we want to evaluate, \eqref{e:mainint}, actually relates to the measures 
$$ \prod_{j=1}^k x_{\vi} \delta(x_{\vi} - x_{\vi'}) { \d \mathrm{Vol}^{\mathrm{WP}}_{g_{\mathbf{S}}, n_{\mathbf{S}}}(\x, Y)},$$
where $\mathrm{m} = \bigsqcup_{j=1}^k \{\vi, \vi'\}$ is a perfect matching of
$\partial \mathbf{S} \setminus V$ for a subset $V$ of $\partial \mathbf{S} = \{1, \ldots, \n\}$.
Note that, by our conventions, $\partial \mathbf{S}$ is identified with $\LambdaBC$, so that $V$ can
be viewed as a subset of $\LambdaBC$.  Theorem~\ref{t:maindet} allows us to write these measures in
the new coordinates.

\begin{cor} \label{c:maindet-dirac} For any set $V \subseteq \partial \mathbf{S}$ and any perfect
  matching $\mathrm{m} = \bigsqcup_{j=1}^k \{\vi, \vi'\}$ of $\partial \mathbf{S} \setminus V$,
  \begin{equation}
    \label{e:changemeasure}
    \begin{split}
      & \prod_{j=1}^k x_{\vi} \delta(x_{\vi} - x_{\vi'})
        \d \mathrm{Vol}^{\mathrm{WP}}_{\g, \n}(\x, Y) \\
      & = 2^{2 |\Lambdabeta|-\n} \frac{\prod_{\lambda \in \Vbeta} \sinh \div{y_\lambda}
        \prod_{\lambda \in \Lambdainbeta} \sinh^2 \div{y_\lambda} }
        {\prod_{\lambda\in \VFN}\sinh \div{y_\lambda}}
        T_{V,\mathrm{m}}(\vec{L}, \vec{\theta}) \prod_{i=1}^r \sinh(L_i) \d^r \vec{L}\, \d^{2r} \vec{\theta}
    \end{split}
  \end{equation}
  where $y_\lambda = \ell_Y(\Gamma_\lambda)$ for $\lambda \in \Lambda$,
  $\Vbeta = V \cap \LambdaBCbeta$, $V^\Gamma = V \cap \LambdaBCFN$, and
  $T_{V,\mathrm{m}}(\vec{L}, \vec{\theta})$ is the expression of the distribution
 $$\frac{\prod_{\lambda \in \LambdaBCbeta \setminus \Vbeta} \sinh \div{y_\lambda}}
        {\prod_{\lambda\in \LambdaBCFN \setminus\VFN}\sinh \div{y_\lambda}}\prod_{j=1}^k  x_{\vi} \delta(x_{\vi} - x_{\vi'})$$
 in the new coordinates. 
\end{cor}

\begin{proof}
  We simply divide the previous expression by
  $\prod_{j=1}^{\n} \sinh \div{x_j} = \prod_{\lambda \in \LambdaBC} \sinh \div{y_\lambda}$, and
  simplify the factors $\prod_{\lambda \in \LambdaBCbeta} \sinh \div{y_\lambda}$ which appear
  simultanously at the numerator and denominator of the resulting fraction.
\end{proof}

We shall never need to develop the  full explicit expression of $T_{V,\mathrm{m}}(\vec{L}, \vec{\theta})$.

Note that, in the expression above, the lengths $y_\lambda$ for
$\lambda \in \LambdaBC \cup \Lambdabeta$ should all be expressed in terms of the variables
$(\vec{L}, \vec{\theta})$. Obtaining explicit expressions will be the goal of \S
\ref{s:lengthf}. The expression of $y_\lambda$ for $\lambda \in \Lambdabeta$ is just a sum of the
variables $\theta_q$, but the expressions for $y_\lambda$ when $\lambda \in \LambdaFN$ will require more
work.

Whilst the lengths $y_\lambda$ for $\lambda \in \LambdainFN$ do not appear explicitely in the expression
\eqref{e:changemeasure}, they are implicitely present in the domain of definition $\domain$ of the
change of variable, as we will see in Proposition~\ref{p:super}.

\subsubsection{Proof of Theorem \ref{t:maindet}}
\label{sec:proof-theor-reft:m}

 The proof goes by induction on the number of intersections $r$.
  
 \begin{proof}[The case $r=1$] The case $r=1$ corresponds to three possible situations,
     represented in \cref{fig:jacobian_1}: ${\mathbf{c}}$ is either a figure-eight, opened the first
     or second way, or ${\mathbf{c}}=(c_1, c_2)$ where $c_1$ and $c_2$ are two curves intersecting
     once, filling a once-holed torus.

        \begin{figure}[h!]
     \begin{subfigure}[b]{0.33\textwidth}
    \centering
    \includegraphics[scale=0.9]{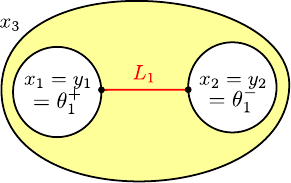}
    \caption{Figure-eight ($1$st way).}
    \label{fig:pp}
  \end{subfigure}%
     \begin{subfigure}[b]{0.33\textwidth}
    \centering
    \includegraphics[scale=0.9]{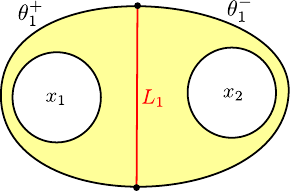}
    \caption{Figure-eight ($2$nd way).}
    \label{fig:pp2}
  \end{subfigure}%
     \begin{subfigure}[b]{0.33\textwidth}
    \centering
    \includegraphics[scale=0.9]{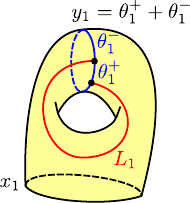}
    \caption{Once-holed torus.}
    \label{fig:jacobian_oht}
  \end{subfigure}%
    \caption{The three cases where $r=1$. }
    \label{fig:jacobian_1}
  \end{figure}
 
  The first case, of a figure eight opened the first way, is represented in Figure \ref{fig:pp}.
  The set of boundary components of the pair of pants decomposition is
  $\Lambda = \partial \mathbf{S} = \{1, 2, 3\}$, and $\Lambdabeta = \{1, 2\}$ so that
  $x_1=y_1=\theta_1^+$ and $x_2=y_2=\theta_1^-$. Noting that by cutting the pair of pants along its
  orthogeodesic $B_1$ we obtain a convex right-angled hexagon, the classic trigonometric formula
  \cite[Theorem 2.4.1 (i)]{buser1992} yields
  $$\cosh (L_1)=\frac{\cosh \div{x_1} \cosh \div{x_2}+\cosh \div{x_3}}{\sinh \div{x_1}\sinh \div{x_2}},$$
  hence $\sinh \div{x_3} \d x_3 =2  \sinh \div{x_1}\sinh \div{x_2} \sinh(L_1) \d L_1$, or equivalently
  \begin{align*}
   & 2^3 \sinh \div{x_1} \sinh \div{x_2}\sinh \div{x_3} \d x_1\d x_2 \d x_3 \\
   & = 2^{4} \sinh^2 \div{y_1}\sinh^2 \div{y_2} \sinh(L_1) \d L_1 \d\theta_1^{-} \d\theta_1^{+}
  \end{align*}
  as announced.

  The second case, the figure-eight opened the second way, is represented in Figure~\ref{fig:pp}. In
  this situation, $\Lambda = \{1,2,3\}$ once again, but now $\Lambdabeta = \{3\}$, and in particular
  $y_\lambda = x_3 = \theta_1^++\theta_1^-$. Cutting the two halves of the pair of pants in
  right-angled pentagons using orthogeodesics, we obtain
  \begin{equation}
    \label{eq:lengths_pop_second_way}
    \begin{cases}
      \cosh \div{x_1} = \sinh \div{\theta_1^+} \sinh \div{L_1} \\
      \cosh \div{x_2} = \sinh \div{x_3-\theta_1^+} \sinh \div{L_1}.
    \end{cases}
  \end{equation}
  We then compute the Jacobian of the change of variable $(x_1, x_2) \mapsto (\theta_1^+,L_1)$,
  where $x_3$ is fixed, by calculating a $2\times 2$ determinant. We obtain
  \begin{equation}
    \label{eq:det2_pop}
    2 \sinh \div{x_{1}} \sinh \div{x_{2}} \d x_{1} \d x_{2}
     = \sinh \div{x_3} \sinh(L_1) \d\theta_1^+ \d L_1
   \end{equation}
   which leads to the claim.

  In the last case, represented in Figure \ref{fig:jacobian_oht}, the surface $\mathbf{S}$ is a
  once-holed torus, $\beta$ is a curve on $\mathbf{S}$, and $\Lambda$ has two components, $\beta$
  and $\partial \mathbf{S}$. The length of $\beta$ is $y_1 = \theta_1^+ + \theta_1^-$, and we notice
  that the parameter $\theta_1^+$ can be chosen to be a twist parameter. The Weil--Petersson measure
  for a once-holed torus of boundary length $x_1$ can then be written
 $$\d \mathrm{Vol}^{\mathrm{WP}}_{1, 1, x_1}(Y)= \d y_1 \d\theta_1^{+}
 = \d\theta_1^{-} \d\theta_1^{+}.$$ We can conclude by observing that
 $\cosh \div{x_1}=\cosh^2 \div{y_1}-\sinh^2 \div{y_1} \cosh(L_1)$, as in the pair of pants case
 above, and hence
 $$\sinh \div{x_1} \d x_1 = 2 \sinh^2 \div{y_1} \sinh(L_1)\d L_1.$$
\end{proof}

Let us now perform the induction.

\begin{proof}[Proof of the induction of Theorem \ref{t:maindet}]
  Let $r \geq 2$, and $\mathbf{c}$ be a generalized eight with $r$ self-intersections filling a
  surface $\mathbf{S}$, with a choice of opening for each intersection point. Let ${\mathbf{c}}'_0$
  denote the multi-loop obtained by removing the last intersection point from ${\mathbf{c}}$ (as
  described in \cref{rem:remove_int_eight}). Then, the intersection $\mathbf{c}'$ of $\curve_0'$ with
    $\Sf_{r-1}$ is a generalized eight with $r-1$ intersections, filling the
    surface~$\mathbf{S}_{r-1}$. We assume that the proposition is true for the surface
  $\mathbf{S}_{r-1}$ and the multi-loop $\mathbf{c}'$ (with intersections opened as they are for
  $\mathbf{c}$), and want to prove it for $\mathbf{S}=\mathbf{S}_{r}$ by attaching the pair of pants
  $\mathbf{P}_{r}$.

  If $\mathbf{S}_{r}$ is a disjoint union of once-holed tori, then the property is a direct
  consequence of the result for $r=1$, thanks to its multiplicativity. We hence assume that it
  is not the case. We therefore have two cases to consider, Case (a) and Case (b).

  Let $\vec{L}'$ and $\vec{\theta}'=(\theta_1'^{\pm}, \ldots, \theta_{r-1}'^{\pm})$ be the
  parameters associated to the surface $\mathbf{S}_{r-1}$ and to ${\mathbf{c}}'$, as constructed in
  \S \ref{s:coord}.  We observe that $\vec{L}'=(L_1, \ldots, L_{r-1})$ are the lengths of
  $\overline{B}_1, \ldots, \overline{B}_{r-1}$, and in particular $\vec{L}'$ is the same as
  $\vec{L}$ with the last component removed.

  The parameters $(\theta_q')_{q \in \{1, \ldots, r-1\} \times \{\pm\}}$, are defined as the lengths
  of simple portions $\overline{\cI}_q'$ contained in $\beta$ and delimited by the endpoints of the
  bars $\overline{B}_1, \ldots, \overline{B}_{r-1}$.  We can therefore easily relate $\vec{\theta}'$
  and $\vec{\theta}$: among the $2(r-1)$ segments $\overline{\cI}_q'$, two of them contain the
  endpoints of $\overline{B}_r$. All the other $\overline{\cI}_{q}'$ are also intervals
  $\overline{\cI}_q$.  As a consequence, for $k\in\{1, \ldots, r-1\}$, $\eps=\pm$, we have
  $\theta_k^{\eps}=\theta_k'^{\eps}$ for all except two values of $(k, \eps)$, namely
  $(n_+, \eps_+), (n_-, \eps_-)$, for which we have the simple relations
  \begin{align}\label{e:sumtheta}\theta_{n_+}'^{\eps_+}= \theta_{n_+}^{\eps_+} + \theta_{r}^{+}
    \qquad \text{and}
    \qquad \theta_{n_-}'^{\eps_-}= \theta_{n_-}^{\eps_-} + \theta_{r}^{-}.\end{align}

  We now proceed differently depending on the case to consider, (a) or (b).
  
  \begin{proof}[Case (a)]
    In this case, $\mathbf{S}_{r}$ is obtained from $\mathbf{S}_{r-1}$ by gluing $\mathbf{P}_r$ to
    two of its boundary components. Let $\lambda_\pm \in \Lambda$ denote the indices of the two
    boundary components of $\mathbf{S}_{r-1}$ on which we glue the pair of pants $\mathbf{P}_r$,
    which satisfy $\mathrm{step}(\lambda_\pm)< r$. By Remark \ref{r:convenient}, these two boundary
    components of $\mathbf{S}_{r-1}$ are replaced by the $\n$-th boundary component of
    $\mathbf{S} = \mathbf{S}_{r}$ when ${\bP}_{r}$ is attached to $\mathbf{S}_{r-1}$.

    Because the endpoints of the orthogeodesic $\overline{B}^+_r$ are on the multi-loop $\beta$,
    there exist $i_\pm$ such that the bar $\overline{B}_r^+$ goes from the component $\beta_{i_-}$
    of $\beta$ to the component $\beta_{i_+}$ of $\beta$. It is homotopic (with endpoints gliding
    along $\beta$) to a succession of $5$ orthogonal segments of respective lengths
    $t_r^-, \alpha_r^-, T, \alpha_r^+, t_r^+$, as shown on Figure \ref{fig:jacobian_a}. More precisely,
    we can lift the four geodesics $\beta_{i_-}$, $\Gamma_{\lambda_-}$, $\Gamma_{\lambda_+}$ and
    $\beta_{i_+}$ to the hyperbolic plane, and orient them so that they are aligned and the bar
    $\overline{B}^+_r$ arrives on the right of $\beta_{i_+}$. Then, we define:
    \begin{itemize}
    \item $H^\pm$ to be the orthogeodesic between $\beta_{i_\pm}$ and $\Gamma_{\lambda_\pm}$, and
      $H$ between $\Gamma_{\lambda_-}$ and $\Gamma_{\lambda_+}$;
    \item $t_r^\pm$ and $T$ to be the respective (positive) lengths of $H^\pm$ and $H$;
    \item the quantities $\alpha_r^+$ and $\alpha_r^-$ to denote the algebraic distance between the
      endpoints of $H^+$ and $H$ along $\Gamma_{\lambda_+}$ and between the endpoints of $H$ and
      $H^-$ along $\Gamma_{\lambda_-}$ respectively.
    \end{itemize}
    It is important in our discussion to remark that $H^\pm$ are contained in $\mathbf{S}_{r-1}$, so
    that their lengths depend only on $Y_{\mathbf{S}_{r-1}}$, the restriction of the metric $Y$ to
    $\mathbf{S}_{r-1}$.
    \begin{figure}[h!]
      \includegraphics{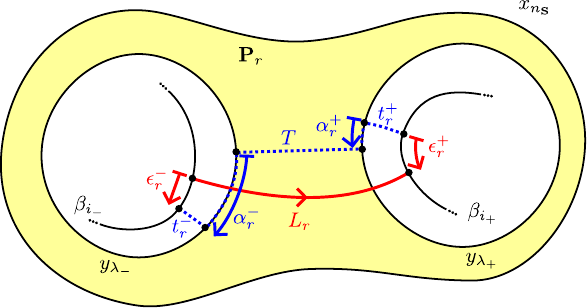}
      \caption{The geometric construction in Case (a). The black dots symbolize the right-angles
        between the different segments. }
      \label{fig:jacobian_a}
    \end{figure}

    In order to express the Weil--Petersson measure for $\mathbf{S}_r$, and relate it to the
    Weil--Petersson measure for $\mathbf{S}_{r-1}$, we need to define twist parameters along the
    glued components $\Gamma_{\lambda_\pm}$.  We choose them to be the algebraic lengths
    $\alpha_r^\pm$.  As a consequence, the expression of the Weil--Petersson measure in
    Fenchel--Nielsen coordinates recalled in \S \ref{s:FNc} implies the factorization
    \begin{align}\label{e:newold}
      \d \mathrm{Vol}^{\mathrm{WP}}_{\g, \n, \x}(Y) =
      \d\alpha_r^+\d\alpha_r^- \d y_{\lambda_+} \d y_{\lambda_-}
      \d \mathrm{Vol}^{\mathrm{WP}}_{\g-1, \n+1, \x'}(Y_{\mathbf{S}_{r-1}}),
    \end{align}
    where $\x = (x_1, \ldots, x_{\n})$ and
    $\x' = (x_1, \ldots, x_{\n-1}, y_{\lambda_+}, y_{\lambda_-})$ correspond to the respective
    length-vectors of $\mathbf{S}_r$ and $\mathbf{S}_{r-1}$.

    % To compare both sides of \eqref{e:newold}, the surface $\mathbf{S}_{r-1}$ has signature
    % $(g_{\mathbf{S}_{r-1}}, n_{\mathbf{S}_{r-1}})$, and the boundary lengths
    % $(x_{j})_{j=1,\ldots, \n}$ on the left-hand side have been replaced by
    % $(x_{j})_{j=1,\ldots, \n-1}$ and the two new components $y_{j, k}, y_{j', k'}$ on the
    % right-hand side.

    We can now, as in the case $r=1$, use the classic trigonometric formula for convex right-angled
    hexagons \cite[Theorem 2.4.1 (i)]{buser1992}, which yields
    \begin{align}\label{e:smile} \cosh(T)
      =\frac{\cosh \div{y_{\lambda_+}}\cosh \div{y_{\lambda_-}}+\cosh \div{x_{\n}}}
      {\sinh \div{y_{\lambda_+}}\sinh \div{y_{\lambda_-}}},
    \end{align}
    which implies that the map $x_{\n}\mapsto T$ is a $\mathcal{C}^1$-diffeomorphism onto its image,
    and
    \begin{equation}
      \label{eq:ind_WP_a_1}
      \sinh \div{x_{\n}} \d x_{\n} = 2 \sinh\div{y_{\lambda_+}}\sinh\div{y_{\lambda_-}} \sinh(T) \d T.
    \end{equation}
    It follows that
    $2^{\n}\prod_{j=1}^{\n} \sinh \div{x_j} \d \mathrm{Vol}^{\mathrm{WP}}_{\g, \n}(\x, Y) $ is equal to
    \begin{align*}
       \sinh(T) \d T \d\alpha_r^+\d\alpha_r^- \,
      2^{\n+1} \prod_{j=1} ^{\n -1} \sinh \div{x_j} \sinh \div{y_{\lambda_+}} \sinh \div{y_{\lambda_-}} 
      \d \mathrm{Vol}^{\mathrm{WP}}_{\g-1, \n+1}(\x', Y_{\mathbf{S}_{r-1}})
    \end{align*}
    which, using the induction hypothesis, implies that
    \begin{equation}
      \begin{split}
        & 2^{\n}\prod_{j=1}^{\n} \sinh \div{x_j}
          \d \mathrm{Vol}^{\mathrm{WP}}_{\g, \n}(\x, Y) \\
        & =\label{e:last}
          {2^{2 |\Lambdabeta|}} \sinh(T) \d T \d \alpha_r^+\d \alpha_r^-
          \prod_{\lambda \in \Lambdabeta} \sinh^2 \div{y_\lambda}
          \prod_{i=1}^{r-1} \sinh(L_i)
          \d L_i
          \prod_{k=1}^{r-1}
          \d\theta_k'^{+} \d\theta_k'^{-}
      \end{split}
    \end{equation}
    where we use the observation that the length-vector $\vec{L}'$ is equal to
    $(L_1, \ldots, L_{r-1})$.

    Let us now denote by $\eps_r^+$ and $\eps^-_r$ the algebraic distance from the endpoint of $H^+$
    to the endpoint of $\overline{B}_r$ along $\beta_{i_+}$, and from the endpoint of
    $\overline{B}_r$ to the endpoint of $H^-$ along $\beta_{i_-}$ respectively, as represented in
    Figure \ref{fig:jacobian_a}. Let us temporarily admit the following formula (which will be
    checked in \S \ref{s:goodformula1}):
    \begin{lem} \label{l:goodformula1} Given a fixed metric $Y_{\mathbf{S}_{r-1}}$, the map
      $ (\alpha_r^-, \alpha_r^+, T)\mapsto(\eps_r^-, \eps_r^+, L_r)$ is a
      $\mathcal{C}^1$-diffeomorphism onto its image, and
      \begin{align}\label{e:epsalpha}
        \sinh(T) \d T  \d\alpha_r^-\d\alpha_r^+ =\sinh(L_r) \d L_r
        \d\eps_r^-
        \d\eps_r^+.
      \end{align}
    \end{lem}

    We can now remark that
    \begin{align}\label{e:alphatheta} \theta_r^+ + \eps_r^+ = \fn(Y_{\mathbf{S}_{r-1}}) \qquad
      \mbox{and} \qquad \theta_r^- - \eps_r^- = \fn(Y_{\mathbf{S}_{r-1}}).\end{align}
    Indeed, for instance, the first quantity is the signed
    distance (measured along $\beta_{i_+}$) between the extremity of $H^+$ lying in $\beta_{i_+}$, and
    one of the bars $\overline{B}_{m_+}$ with $m_+<r$. See Figure \ref{fig:theta_casea}.

  \begin{figure}[h]
    \centering \includegraphics{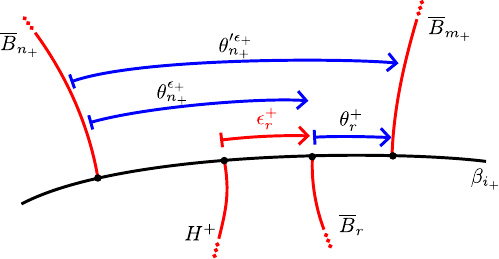}
    \caption{The relations between the algebraic lengths along~$\beta_{i_+}$.}
    \label{fig:theta_casea}
  \end{figure}

  The relation \eqref{e:alphatheta} implies that in \eqref{e:epsalpha}, we can change variables
  $(\eps_r^-, \eps_r^+)\mapsto (\theta_r^-, \theta_r^+)$ and replace
  $\d\eps_r^- \d\eps_r^+$ by $\d\theta_r^- \d\theta_r^+ $.  Our expression \eqref{e:last} then
  becomes
  $$ 2^r \sinh(L_r) \d L_r
  \d\theta_r^- \d\theta_r^+ \prod_{\lambda \in \Lambdabeta} \sinh^2 \div{y_\lambda}
  \prod_{i=1}^{r-1} \sinh(L_i) \d L_i \prod_{k=1}^{r-1} \d\theta_k'^{+} \d\theta_k'^{-}.$$

  Now recall the relations between $\theta_k^\epsilon$ and $\theta_k'^\epsilon$. For $k \leq r-1$
  and $\epsilon = \pm$, those parameters are equal, except from two sets of indices, for which these
  values differ by a translation by \eqref{e:sumtheta}. Hence, our final expression for the measure
  is the announced one.
\end{proof}

  \begin{rem}\label{e:special}
    Part of this discussion is irrelevant when $\beta_{i_-}= \Gamma_{\lambda_-}$, because in this
    case, $H^-$ and $\alpha_r^-$ are ill-defined. The previous discussion has to be modified as
    follows.
    \begin{itemize}
    \item If $\beta_{i_-}$ is a boundary curve, then there is no twist parameter attached to it and
      there is no bar $\overline{B}_k$ arriving at $\beta_{i_-}$ on the other side of
      $\overline{B}_r$: we have $\theta_r^-=x_{i_-}=y_{i_-}$.
    \item If $\beta_{i_-}$ is not a boundary curve, then we can take directly $\theta_r^-$ as twist
      parameter, instead of the ill-defined $\alpha_r^-$.
    \end{itemize}
    In all cases, the final outcome of the calculation is the same, with less intermediate steps.
    Similar remarks can be made when $\beta_{i_+}= \Gamma_{\lambda_+}$.  We skip details specific to
    these cases, similar and simpler than the one we treated.
  \end{rem}

  \begin{proof}[Case (b)] In this case, we attach $\mathbf{P}_r$ to one boundary component
    $\Gamma_{\lambda}$ of $\mathbf{S}_{r-1}$ (with $\mathrm{step}(\lambda)<r$), and get the two new
    boundary components of $\mathbf{S}_{r}$ labelled $\n$ and $\n-1$.

  \begin{figure}[h]
    \centering \includegraphics{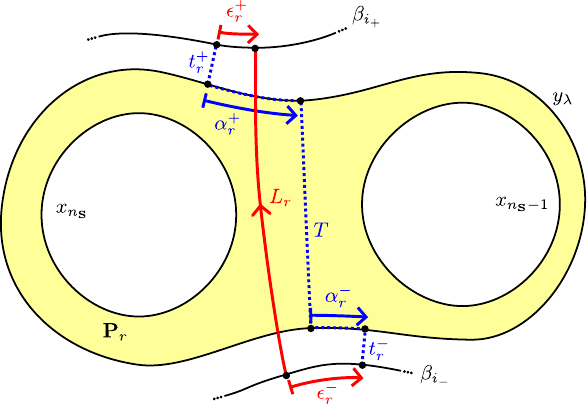}
    \caption{The geometric construction in Case (b).}
    \label{fig:jacobian_b}
  \end{figure}
 
  The situation is quite similar to Case (a) with a small difference, namely the fact that
  $\lambda_+=\lambda_-$; we therefore denote them $\lambda$. There exist $i_\pm$ such that the
  orthogeodesic $\overline{B}_r^+$ goes from $\beta_{i_-}$ to $\beta_{i_+}$. It is homotopic to a
  succession of $5$ orthogonal segments of respective lengths $t_r^-, \alpha_r^-, T, \alpha_r^+, t_r^+$,
  as shown on Figure \ref{fig:jacobian_b}.  Note that, when lifting Figure \ref{fig:jacobian_a} and
  \ref{fig:jacobian_b} in the universal cover, we obtain the same picture, the only difference being
  that the lifts of $\Gamma_{\lambda_\pm}$ now are two different lifts of the same periodic geodesic
  $\Gamma_{\lambda}$. We therefore define the orthogeodesics $H^\pm$ and $H$, as well as the lengths
  $t_r^\pm$, $\alpha_r^\pm$ and $T$ as in  Case (a).

  We choose the twist parameter to be $\alpha_r^+$.  As a consequence, the expression of the
  Weil--Petersson measure in Fenchel--Nielsen coordinates implies the factorization
  \begin{align}
    \d \mathrm{Vol}^{\mathrm{WP}}_{\g, \n, \x}(Y)
    =  \d\alpha_r^+  \d y_{\lambda}
    \d \mathrm{Vol}^{\mathrm{WP}}_{\g, \n -1, \x'}(Y_{\mathbf{S}_{r-1}}) 
  \end{align}
  where $\x' = (x_1, \ldots, x_{\n-2}, y_{\lambda})$.   The rest of the proof goes along the
    same lines as in Case (a), now using the formulae \eqref{eq:lengths_pop_second_way} and the
    determinant calculation \eqref{eq:det2_pop}.  
  
  % We can now use the following formulas:
  % \begin{align}
  %   \label{e:alpha1}
  %   & \cosh \div{x_{\n}}=\sinh \div{\alpha} \sinh \div{T} \\
  %   \label{e:alpha2}
  %   & \cosh \div{x_{\n-1}}=\sinh \div{y_{\lambda}-\alpha} \sinh \div{T},
  % \end{align}
  % where $\alpha$ is the distance between the two endpoints of the orthogeodesic $H$ measured along
  % $\Gamma_{\lambda}$, as shown on Figure \ref{fig:alpha}.
   
  %  \begin{figure}[h!]
  %    \includegraphics{Figures_alpha}
  %    \caption{The length $\alpha$ in the pair of pants $\mathbf{P}_r$.}
  %    \label{fig:alpha}
  %  \end{figure}

 \end{proof}

 Finally, to conclude, in all cases, the fact that the map
 $(\x, Y)\in \cT^*_{\g, \n}\mapsto (\vec{L}, \vec{\theta})$ is a $\mathcal{C}^1$-diffeomorphism onto
 its image can also obtained by induction, since at each step we checked that the maps were
 $\mathcal{C}^1$-diffeomorphisms onto their images.
\end{proof}

\subsubsection{Proof of Lemma \ref{l:goodformula1}} \label{s:goodformula1} To finish this section,
there remains to prove Lemma \ref{l:goodformula1}. We also start investigating the range of
definition of $\vec{L}$ and $\vec{\theta}$.

Let us consider a family of $M$ aligned geodesics $\gamma_1, \ldots, \gamma_M$ on the hyperbolic
plane, such that, for all $i$, $\gamma_{i+1}$ is on the left of $\gamma_i$ (for now, we will only
take $M=3$ and $4$, but more general situations will arise in \S \ref{s:lengthc}). For indices
$i \neq j$, let $H_{ij}$ denote the orthogeodesic between $\gamma_i$ and $\gamma_j$ and $z_{ij}$ the
endpoint of $H_{ij}$ lying on $\gamma_j$.
Let us first prove one intermediate lemma in a three-geodesic configuration, illustrated in Figure
\ref{fig:3geod}.

\begin{lem}\label{p:jacob}
  Let $M=3$. We denote as $L$, $t$, $T$ the following orthogeodesic lengths
  \begin{equation*}
    L = \ell(H_{13})
    \qquad t = \ell(H_{23})
    \qquad T = \ell(H_{12}) 
  \end{equation*}
  and $\eps$, $\alpha$ the algebraic distances
  \begin{equation*}
    \eps = \Dist(z_{23},z_{13}) \qquad \qquad \alpha = \Dist(z_{32},z_{12}).
  \end{equation*}
  Then, for fixed $t>0$, the map
  \begin{align*}
    \psi_{t}:
    \begin{cases}
      \R \times \R_{>0} & \rightarrow \R \times \R_{>0} \\
      (\alpha, T) & \mapsto(\eps, L)
    \end{cases}
  \end{align*}
  is a $\mathcal{C}^1$-diffeomorphism onto the set of real numbers
  $(\eps, L)$ such that
  \begin{equation}
    \label{eq:def_var_3}
    \begin{cases}
      \cosh(t) \cosh(L)-\cosh(\eps) \sinh(t) \sinh(L)> 1 \\
      L> t,
    \end{cases}
  \end{equation}
  and its Jacobian can be expressed as
  \begin{align}\label{e:det}\sinh(T) \d T
    \d\alpha = \sinh(L) \d L \d\eps.
  \end{align}
\end{lem}

\begin{figure}[h!]
  \includegraphics{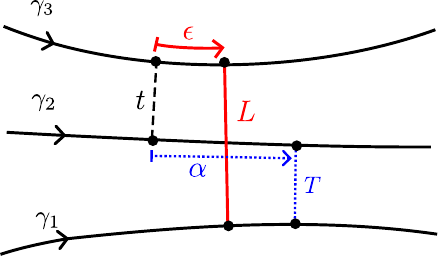}
  \caption{The three geodesics and the lengths we will study. The dots mark the points $z_{ij}$,
    where the geodesics meet perpendicularly.}
  \label{fig:3geod}
\end{figure}%

\begin{proof}
  The fact that the image of $\{\alpha\in\IR, T>0\}$ is included in the set of $(\eps, L)$
  satisfying the conditions \eqref{eq:def_var_3} can be read from standard formulas for right-angled
  hexagons in hyperbolic geometry, see e.g. \cite[Ch. 2 \S 4]{buser1992}:
  \begin{align}
    & \cosh (T) =\cosh (t)  \cosh (L)-\cosh (\eps) \sinh (t) \sinh (L) \label{e:cosinus} \\
     & \cosh (L) =\cosh (t) \cosh (T) + \cosh (\alpha) \sinh (t) \sinh (T). \label{e:hexagon}
  \end{align}

  Conversely, if \eqref{eq:def_var_3} holds for parameters $(L,\eps)$, then there exists $T\not= 0$
  such that \eqref{e:cosinus} holds. This defines $T$ up to the sign, and we choose the positive
  solution.  Another classical hyperbolic trigonometry formula states that, for this geometric
  situation to hold, we must have
  \begin{equation}
    \frac{\sinh (\eps)}{ \sinh(T)}=\frac{\sinh(\alpha)}{\sinh(L)}. \label{e:sinus}      \end{equation}
  One can then define $\alpha$ so that the relation \eqref{e:sinus} holds. Note that $\alpha$ has the
  same sign as~$\eps$.
  By using $\cosh^2 (L)=1 + \sinh^2 (L)$ and the two relations \eqref{e:cosinus} and \eqref{e:sinus},
  one then deduces that $$(\cosh (L)-\cosh (t) \cosh (T))^2= (\cosh (\alpha) \sinh (t) \sinh (T))^2$$ or,
  in other words, $$\cosh (L)=\cosh (t) \cosh (T)\pm \cosh (\alpha) \sinh (t) \sinh (T).$$ Because we
  assumed that $L>t$, we actually obtain that \eqref{e:hexagon} holds,
  which means that $L$ is indeed the length of the orthogeodesic $H_{13}$ on Figure
  \ref{fig:3geod}.

   The determinant calculation \eqref{e:det} follows from \eqref{e:hexagon} and \eqref{e:sinus}, but
   it is worth writing some details. Start writing the jacobian of the change of variable as the
   determinant of the matrix of partial derivatives:
   \begin{align*}
     \cosh (\eps) \, \frac{  \d L \d\eps}{ \d T \d\alpha}
     = \frac{\partial L}{\partial T} \frac{\partial (\sinh (\eps))}{\partial \alpha}
     -\frac{\partial L}{\partial \alpha} \frac{\partial (\sinh (\eps))}{\partial T}.
   \end{align*} 
Because of the relation \eqref{e:sinus}, we can rewrite this as
\begin{align*}
  \cosh (\eps) \, \frac{  \d L \d\eps}{ \d T \d\alpha}
  =
  \frac{\partial L}{\partial T} \frac{ \cosh (\alpha)  \sinh (T)}{  \sinh (L)}
  -\frac{\partial L}{\partial \alpha} \frac{\sinh (\alpha) \cosh (T)}{  \sinh (L)}.
 \end{align*} 
 An elementary calculation of $\partial L/\partial T$ and $\partial L/\partial \alpha$ from
 \eqref{e:hexagon} yields that
\begin{align*}
  \cosh (\eps) \, \frac{  \d L \d\eps}{ \d T \d\alpha}
  =
\frac{ \sinh (T)}{  \sinh^2 (L)} (\sinh (t) \cosh (T)+ \cosh (t) \sinh (T)  \cosh (\alpha)).
\end{align*} 
Using another classical trigonometric formula, we recognize the quantity in the bracket to be
$\cosh (\eps) \sinh (L)$, and hence this simplifies to
\begin{align*}
\frac{  \d L \d\eps}{ \d T \d\alpha}= \frac{ \sinh (T)}{  \sinh (L)} 
 \end{align*} 
 which is the desired relation.
 \end{proof}

 \begin{rem}
   We actually have $L\geq t+T$, so for any given $a>0$, the set $\{\alpha\in\IR, T> a\}$ is
   sent to a subset of the set of the set of values $(\eps, L)$ such that
   $$
   \begin{cases}
     \cosh (t) \cosh(L)-\cosh(\eps) \sinh(t) \sinh(L)\geq \cosh (a) \\
     L\geq t+ a.
   \end{cases} $$
 \end{rem}

 \begin{nota}
   This result provides us with a map $\phi_t : (\eps, L) \mapsto \Dist(z_{31},z_{21})$, which
   corresponds to the algebraic distance along the bottom geodesic.
 \end{nota}
 
 % \begin{nota}
 %   \label{nota:gamma}
 %   We denote $(\eps, t_2)=\phi_{t_1}(\alpha, L)$ the inverse map of
 %   $\psi_{t_1}$; we will also write $\Gamma=\Gamma(t_1, \alpha, L)$,
 %   $t_2=\cT(t_1, \alpha, L)$ when we want to write $\Gamma$ or $t_2$ as
 %   functions of $t_1, \alpha, L$, and $\Gamma=\gamma(t_1, \eps, t_2)$ when we
 %   want to write $\Gamma$ as function of $ t_1, \eps, t_2$.
 % \end{nota}

 In order to prove Lemma \ref{l:goodformula1}, we simply apply Lemma \ref{p:jacob} twice in a
 four-geodesic configuration, as done below. See also Figure \ref{fig:4geod}.

 \begin{figure}[h!]
     \includegraphics[height=5cm]{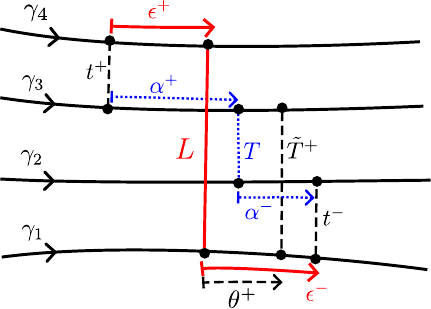}
     \caption{The four-geodesic configuration we consider. We have represented the quantities
       $\tilde{T}^+$ and $\theta^+$ appearing in the statement.}
    \label{fig:4geod}
  \end{figure}

 \begin{lem} \label{c:det2}
   Let $M=4$. We denote as $L$, $t^\pm$ and $T$ the orthogeodesic lengths
   \begin{equation*}
     L = \ell(H_{14})
     \qquad t^+ = \ell(H_{34})
     \qquad t^- = \ell(H_{12})
     \qquad T = \ell(H_{23})
   \end{equation*}
   and $\eps^\pm$, $\alpha^\pm$ the algebraic lengths
   \begin{equation*}
     \eps^+ = \Dist(z_{34},z_{14})
     \qquad \alpha^+= \Dist(z_{43},z_{23})
     \qquad \alpha^-= \Dist(z_{32},z_{12})
     \qquad \eps^- = \Dist(z_{41},z_{21}).
   \end{equation*}
   For fixed $t^-, t^+>0$, the map
   \begin{equation*}
     \begin{cases}
       \R^2 \times \R_{>0} & \rightarrow \R^2 \times \R_{>0} \\
       (\alpha^-, \alpha^+, T) & \mapsto (\eps^-, \eps^+, L)
     \end{cases}
   \end{equation*}
   is a $\mathcal{C}^1$-diffeomorphism onto the set of real numbers
 $(\eps^-, \eps^+, L)$ such that
 \begin{align}\label{e:Dloc}
   \begin{cases}
     \cosh(t^+)  \cosh(L)-\cosh(\eps^+) \sinh(t^+) \sinh(L)> 1 \\
     \cosh(t^-) \cosh(\tilde T^+)-\cosh(\eps^- -\theta^+) \sinh(t^-)\sinh(\tilde T^+)> 1 \\
     L> t^+ \quad \text{and} \quad
     \tilde T^+>t^-,
   \end{cases}
 \end{align}
 where $\tilde{T}^+$ is the second component of $\psi_{t^+}^{-1}(\eps^+,L)$ and
 $\theta^+ := \phi_{t^+}(\eps^+,L)$.  Its Jacobian can be expressed as
\begin{align}\label{e:det'}
  \sinh(T) \d T \d\alpha^- \d\alpha^+ =  \sinh(L) \d L \d\eps^- \d\eps^+.
 \end{align}
\end{lem}

\begin{proof}We simply apply the previous result twice, once to the three geodesics $\gamma_1,
  \gamma_3, \gamma_4$, and another time to $\gamma_3, \gamma_2, \gamma_1$ with reverse orientation.
\end{proof}
 
\begin{rem} \label{r:rka}We always have $L\geq t^++T+t^-$ and $\tilde{T}^+\geq t^-+T$. Hence, for
  any $a>0$, the set $\{\alpha^-, \alpha^+\in \IR, T>a\}$ corresponds to a subset of the set of real
  numbers $(\eps^-, \eps^+, L)$ such that
  $$
  \begin{cases}
    \cosh (t^+) \cosh (L)-\cosh(\eps^+) \sinh(t^+) \sinh(L)\geq \cosh(a) \\
    \cosh (t^-) \cosh (\tilde{T}^+)-\cosh(\eps^- - \theta^+)
    \sinh(t^-)\sinh(\tilde{T}^+) \geq \cosh(a) \\
    L> t^++a \quad \text{and} \quad
    \tilde{T}^+>t^-+a.
  \end{cases}$$
 \end{rem}

 \subsection{Domain of definition of the new coordinates}
  \label{s:domain}
  In this section, we discuss the range of definition $\domain$ of the coordinates
  $(\vec{L}, \vec{\theta})$, which will be the domain of the integral \eqref{e:mainint} in the new
  coordinates.  We prove the following statement, which describes the boundary of $\domain$.

\begin{prp}\label{p:super}
  The boundary of $\domain$ is contained in 
  $\bigcup_{\lambda\in \Lambda} \{ y_\lambda=0\}$.
\end{prp}

\begin{proof}
  We once again proceed by induction, closely following the proof of Theorem \ref{t:maindet} but
  making the domain of definition of the variables more explicit.

  The result is trivial for $r=1$. Let us fix a metric
  $Y_{\mathbf{S}_{r-1}} \in \cT_{g_{\mathbf{S}_{r-1}},n_{\mathbf{S}_{r-1}}}^*$ and describe the
  domain of definition of the three new coordinates $(\theta_r^-, \theta_r^+, L_r)$ defined when
  attaching the pair of pants~${\bP}_r$. We note that there exists a unique hyperbolic metric
  associated to any set of Fenchel--Nielsen coordinates. Hence, the constraints on the new variables
  are:
  \begin{itemize}
  \item $(\eps_r^+,\eps_r^-,L_r)$ need to belong to the image of the
    diffeomorphism in Lemma \ref{c:det2} -- we call this the \emph{local constraint};
  \item the new boundary components of the pair of pants need to be of positive
    length, i.e. $x_{\n}>0$ in Case (a) and $x_{\n}, x_{\n-1}>0$ in Case (b) -- we call this the
    \emph{global constraint}.
  \end{itemize}
  We prove Proposition \ref{p:super} by proving that the global constraints are stronger than the
  local constraints.
  Indeed, in Case (a), because of \eqref{e:smile}, the global constraint is equivalent to
  $$\cosh(T)> \frac{\cosh \div{y_{\lambda_+}}\cosh \div{y_{\lambda_-}}+1}
  {\sinh \div{y_{\lambda_+}}\sinh \div{y_{\lambda_-}}}.$$
  Note that the quantity
  $$a:=\argcosh
  \Big(\frac{\cosh \div{y_{\lambda_+}}\cosh \div{y_{\lambda_-}}+1}{\sinh \div{y_{\lambda_+}}\sinh \div{y_{\lambda_-}}}\Big)
  >0,$$
  only depends $y_{\lambda_+}$ and $y_{\lambda_-}$, i.e. on the metric $Y_{\mathbf{S}_{r-1}}$ in
  $\cT_{\g-1,\n+1}^*$ that we are considering to be fixed. Hence, the global constraint implies the
  $T>a$, where $a$ is fixed.  By Remark~\ref{r:rka}, this implies strictly stronger conditions than
  the original local constraints, which is what we wanted to check.

In Case (b), we have 
 $$\cosh(T)= \frac{\cosh^2 \div{\alpha}+ \cosh (x_{\n})}{\sinh^2 \div{\alpha}}
 \quad \mbox{and} \quad
 \cosh({{T}})= \frac{\cosh^2 \div{y_{\lambda}-\alpha}
   + \cosh (x_{\n-1})}{\sinh^2 \div{y_{\lambda}-\alpha}}$$
 where $\alpha$ is the length of the portion of $\beta_{\lambda}$ delimited by the orthogeodesic
 $H$, on the side of $x_{\n}$.
 Thus the new global constraints are equivalent to
 $$\cosh({{T}})>
 \frac{\cosh^2 \div{\alpha}+ 1}{\sinh^2 \div{\alpha}}
 =1 +\frac{2}{\sinh^2 \div{\alpha}}
 \quad \mbox{and} \quad
 \cosh({{T}})> 1 +\frac{2}{\sinh^2 \div{y_{\lambda}-\alpha}} .$$
   Define $a:=\argch(1+2/\sinh^2 \div{y_{\lambda}})$.  Since
   $0<\alpha< y_{\lambda}$, the two global constraints imply that $T>a>0$, which allows to
   conclude as in Case (a).
\end{proof}

   \subsection{A priori bounds on Fenchel--Nielsen parameters}
   \label{s:apriori} 
   
   We shall need some rough a priori bounds on the parameters $y_\lambda=\ell_Y(\Gamma_\lambda)$ for
   $\lambda \in \Lambda$ and $\theta_q$ for $q \in \Theta$, that grow at most linearly with respect
   to $\ell_Y({\mathbf{c}})$.
   
   \begin{prp}\label{p:apriori1} On the level set
     $\{ Y \in \mathcal{T}_{\g,\n}^* \, : \, \ell_Y({\mathbf{c}})=\ell\}$, we have
     \begin{align}
       \forall \lambda \in \Lambda, & \quad 0\leq y_\lambda \leq \ell \\
       \forall q \in \Theta, & \quad |\theta_q| \leq 2\ell
     \end{align}
     and the twist parameters $(\alpha_\lambda)_{\lambda \in \Lambdain}$ constructed in the proof of
     Theorem \ref{t:maindet} satisfy
     \begin{align}\label{e:twist_bound}
       \forall \lambda \in \Lambdain, \quad
       |\alpha_\lambda|\leq 3\ell.
    \end{align}    
   \end{prp}
   
   The proof relies on elementary hyperbolic trigonometry, and in particular the following lemma.
   
   \begin{lem}\label{l:basic_ortho}
     Let $Y$ be a compact hyperbolic surface with two simple closed geodesics $\gamma_1,
     \gamma_2$. We assume that $\gamma_1$ and $\gamma_2$ are either disjoint or equal. Let $H$ be a
     simple orthogeodesic segment going from $\gamma_1$ to~$\gamma_2$, not intersecting $\gamma_1$
     and $\gamma_2$ in its interior. Then, $\sinh (\ell(\gamma_1)/2) \sinh(\ell(H))\geq 1$.
   \end{lem}

   {
   \begin{proof}
     If $\gamma_1=\gamma_2$, this is a direct consequence of the collar lemma \cite[Theorem
     4.1.1]{buser1992}.  If the two geodesics are disjoint, then the surface filled by $\gamma_1$,
     $\gamma_2$ and $H$ is a pair of pants. Cutting this pair of pants along its orthogeodesics, we
     obtain a convex right-angled hexagon with consecutive sides of length $\ell(\gamma_1)/2$,
     $\ell(H)$ and $\ell(\gamma_2)/2$. We can further cut this hexagon and obtain a right-angled
     pentagon with consecutive side-lengths $\ell(\gamma_1)/2$ and $\ell(H)$, which implies the
     claim by \cite[Lemma 2.3.5]{buser1992}.
   \end{proof}}
   
   We use this lemma to prove the following.
   \begin{lem}\label{l:niceeight}
     Let $Y$ be a compact hyperbolic surface with:
     \begin{itemize}
     \item three simple oriented closed geodesics $\gamma_1, \gamma_2, \gamma_3$;
     \item an orthogeodesic $H_{12}$ leaving from the right of $\gamma_1$ and arriving at the left
       of $\gamma_2$;
     \item an orthogeodesic $H_{23}$ leaving from the right of $\gamma_2$ and arriving at the left
       of $\gamma_3$.
     \end{itemize}
     Let $J\subset \gamma_2$ be the segment joining $t(H_{12})$
     to $o(H_{23})$ and $H$ be the orthogeodesic from $\gamma_1$ to $\gamma_3$, homotopic with
     gliding endpoints to $H_{12}\smallbullet J \smallbullet H_{23}$. Then, if $\gamma$ denotes the
     geodesic representative of the homotopy class
     $\gamma_1 \smallbullet H \smallbullet \gamma_3\smallbullet H^{-1}$, we have
     \begin{align}
       \cosh \div{\ell(\gamma)}\geq \cosh(\ell(J)).
     \end{align}
\end{lem}

Of course, the orientations of $\gamma_1$, $\gamma_2$ and $\gamma_3$ are irrelevant in the length
estimate, but serve to describe the isotopy class of the picture
$\gamma_1\cup\gamma_2\cup \gamma_2\cup H_{12}\cup H_{23}$.

\begin{proof}
  The geodesic $\gamma$ is a figure eight, going around the two closed geodesics $\gamma_1$ and
  $\gamma_3$, connected by the orthogeodesic $H$. It then follows that
  \begin{align*}
\cosh \div{\ell(\gamma)}&= \cosh \div{\ell(\gamma_1)}\cosh \div{\ell(\gamma_3)}+
\cosh(H) \sinh \div{\ell(\gamma_1)}\sinh \div{\ell(\gamma_3)}\\
&\geq \cosh(H) \sinh \div{\ell(\gamma_1)}\sinh \div{\ell(\gamma_3)}.
  \end{align*}
  Expressing the length of $H$ in terms of the lengths of $J$ and $H_{12}$, $H_{23}$, we obtain 
  \begin{align*}
    \cosh(H)
    & = \cosh(\ell(H_{12})) \cosh(\ell(H_{23})) + \cosh(\ell(J))\sinh(\ell(H_{12}))
      \sinh(\ell(H_{23})) \\
    & \geq \cosh(\ell(J))\sinh(\ell(H_{12}))
      \sinh(\ell(H_{23})).
  \end{align*}
  By Lemma \ref{l:basic_ortho}, $\sinh (\ell(\gamma_1)/2)\sinh(\ell(H_{12}))\geq 1$ and $\sinh
  (\ell(\gamma_3)/2)  \sinh(\ell(H_{23}))\geq 1$, which
  leads to our claim.
\end{proof}

We are now ready to proceed to the proofs of the bounds.

{
\begin{proof}[Proof of Proposition \ref{p:apriori1}]
  The upper bound on $y_\lambda$ is a direct consequence of Remark \ref{r:initlambda}, which states
  that $\Gamma_\lambda$ is freely homotopic to a path $\Gamma_{\lambda, \mathrm{init}}$ on
  $\mathbf{c}$, and hence
  $$y_\lambda=\ell_Y(\Gamma_{\lambda}) \leq \ell(\Gamma_{\lambda, \mathrm{init}})
  \leq  \ell({\mathbf{c}})=\ell.$$

  Let us now prove the bound on $\theta_q$ for an index $q \in \Theta$.
  \begin{itemize}
  \item If $q$ is a U-turn parameter, then we simply observe that
    $0\leq \theta_q\leq \ell(\beta_{\lambda(q)})\leq \ell$.
  \item If $q$ is a crossing parameter, we apply Lemma \ref{l:niceeight} to the simple closed
    geodesics $\gamma_1=\beta_{\lambda( \sigma^{-1} q)}$, $\gamma_2=\beta_{\lambda( q)}$ and
    $\gamma_3= \beta_{\lambda( \sigma q)}$, with the orthogeodesics $H_{12}=\overline{B}_q$ and
    $H_{23}=\overline{B}_{\sigma q}$, so that $J= \overline{\cI}_q$. We check that the figure-eight
    $\gamma$ produced by the lemma is freely homotopic to a loop $\gamma_{\mathrm{init}}$ drawn on
    $\mathbf{c}$, going along $\beta_{\lambda(\sigma^{-1}q),\mathrm{init}}$,
    $\beta_{\lambda(\sigma q), \mathrm{init}}$ and twice $\cI_{q, \mathrm{init}}$. We conclude that
    \begin{equation*}
      \cosh(\theta_q)\leq  \cosh(\ell(\gamma)) \leq \cosh(2 \ell).
    \end{equation*}
  \end{itemize}

  The bounds on the twist parameters are obtained the same way, constructing for each case (Case (a)
  and (b)) a figure-eight $\gamma$ such that $|\alpha_\lambda| \leq \ell(\gamma) \leq 3 \ell$ (the
  factor of $3$ being there for Case (b), due to repeated portions of $\mathbf{c}$ used to describe
  the homotopy class of $\gamma$). 
\end{proof}}
 
\subsection{Polynomial bounds on volumes and the case of double-filling loops}

One can straightforwardly deduce from the apriori bound on Fenchel--Nielsen parameters,
Proposition~\ref{p:apriori1}, a rough polynomial upper bound on the volume of the set of hyperbolic
surfaces $Y \in\cT^*_{\g, \n}$ such that $\ell_Y({\mathbf{c}})\leq \ell$.
A similar statement was obtained when $\mathbf{S}$ is a once-holed torus in
\cite[eq. (8.2)]{Ours1}.  

\begin{cor}\label{c:pvolume} For any family of distinct indices $(\vi, \vi')_{1 \leq j \leq k}$ in
  $\partial \mathbf{S}$, any $\ell \geq 0$,
  \begin{align} 
    \int_{\ell_Y({\mathbf{c}})\leq \ell}    
    \prod_{j=1}^k  \delta(x_{\vi} - x_{\vi'})
    \d \mathrm{Vol}^{\mathrm{WP}}_{\g, \n}(\x, Y)  
    \leq (3\ell)^{3\chi(\mathbf{S})}. \label{e:pol_bound}
  \end{align} 
\end{cor}

\begin{proof}
  We simply write the expression of
  $ \d \mathrm{Vol}^{\mathrm{WP}}_{g_{\mathbf{S}}, n_{\mathbf{S}}}(\x,Y) $ in Fenchel--Nielsen
  coordinates as in~\eqref{e:WPFN}, and use the bounds given by Proposition \ref{p:apriori1} on the
  lengths and twist angles.
\end{proof}

\begin{rem}
  We have assumed at the beginning of \S \ref{s:nc} that the loop $\mathbf{c}$ is a generalized
  eight, and therefore, as of now, we have only proven Corollary \ref{c:pvolume} for generalized
  eights.  For the sake of the following application, it will be useful to observe that this result
  can be extended to any multi-loop $\mathbf{c}$. This proof relies on definitions from Section
  \ref{s:othercases} and can be omitted at first read; it is only presented here for completeness
  and to allow to state the interesting Lemma \ref{c:double_fills} below.

  \begin{proof}[Proof of Corollary \ref{c:pvolume} for general loops]
    Let $\mathbf{c}'$ be a skeleton of $\mathbf{c}$, and $\mathbf{S}'$ be the surface filled by
    $\mathbf{c}'$. If $\mathbf{S}'=\mathbf{S}$, then by Proposition \ref{p:remove},
    $\ell_Y({\mathbf{ c}}')\leq \ell_Y({\mathbf{ c}})$, from which we deduce the claim. Otherwise,
    $\mathbf{S}$ is obtained from $\mathbf{S}'$ by gluing a finite number of cylinders and adding
    cylindrical bars. We add the corresponding twist parameters to our coordinate system and remove
    the $\theta_q$ that have become redundant. Corollary \ref{c:pvolume} persists because we can
    choose those twist parameters so that they still satisfy~\eqref{e:twist_bound}, by the same proof.
  \end{proof}
\end{rem}

As in \cite[\S 8.1]{Ours1}, we can use polynomial volume bounds to prove the following special case
of our main result, Theorem \ref{t:main}, for double-filling loops (defined in Definition
\ref{d:double_fill}).

\begin{lem}\label{c:double_fills}
  \label{lem:df}
  If $\mathbf{c}$ is double-filling, then $\cJ \in \cR_w^{\rN}$ and
  $$\norm{\cJ}_{\cR_w^{\rN}} \leq \fn(\g,\n,\rK,\rN) \prod_{j \in
    V}\|\tilde{g}_j\|_{\cF^{\rK_j,\rN_j}}.$$
\end{lem}

\begin{proof}[Proof of Lemma \ref{lem:df}]
  Let us write
  \begin{align*} \int_0^\ell| \cJ(s)| \d s \leq \int_{ \ell_Y({\mathbf{c}})\leq \ell} \, \prod_{j\in
      V} | f_j(x_j)| \prod_{j=1}^k x_{\vi} \delta(x_{\vi} - x_{\vi'}) \d
    \mathrm{Vol}^{\mathrm{WP}}_{g_{\mathbf{S}}, n_{\mathbf{S}}}(\x, Y).
  \end{align*} 
  By hypothesis on the functions $f_j$ and using equation \eqref{e:normbound1} on $\tilde{g}_j$, we
  have for each $j \in V$,
  $$| f_j(x_j) |\leq \fn(\rK,\rN) \|\tilde{g}_j\|_{\cF^{\rK_j,\rN_j}} (1+x_j)^{\rK_j-1} e^{x_j/2}$$
  and for each $1 \leq j \leq k$, $x_{\vi} \leq e^{x_{\vi}/2}$.
  By Proposition \ref{p:comparison1}, if $\ell_Y({\mathbf{c}})\leq \ell$ and $\mathbf{c}$ is
  double-filling, then $\sum_{j=1}^{\n} x_j \leq \ell$ and hence
  \begin{align*} \prod_{j\in V}  |f_j(x_j)|     \prod_{j=1}^k x_{\vi} 
    \leq \fn(\rK,\rN) \prod_{j \in V}\|\tilde{g}_j\|_{\cF^{\rK_j,\rN_j}}
    (1+\ell)^{\sum_{j \in V}(\rK_j-1)} e^{\ell/2}.
  \end{align*}   
  With Corollary \ref{c:pvolume}, using the hypothesis $\rK_j \leq \rN_j$ and the definition
  $\rN=\sum_{j \in V} \rN_j+3\chi(\Sf)+1$, this implies that
  $$\int_0^\ell |\cJ(s) | \d s \leq
  \fn(\g,\n,\rK,\rN) \prod_{j \in V}\|\tilde{g}_j\|_{\cF^{\rK_j,\rN_j}}
  (1+\ell)^{\rN-1} e^{\ell/2},$$
  which allows us to conclude and to obtain the claimed bound.
\end{proof}
        
\begin{rem} Had we not assumed that $\mathbf{c}$ was double-filling, Proposition \ref{p:comparison1}
  would only have implied that
  \begin{align}\label{e:weak-FR} \int_0^\ell | \cJ(s)| \d s
    \leq   \fn(\g,\n,\rK,\rN) \prod_{j \in V}\|\tilde{g}_j\|_{\cF^{\rK_j,\rN_j}}
    (1+\ell)^{\rN-1} 
    e^{\ell},\end{align}
  i.e. a similar upper bound but without the factor $1/2$ in the exponential, which is not enough to
  conclude to Theorem~\ref{t:main}. 
\end{rem}

%%%Local Variables: 
%%% mode: latex
%%% TeX-master: "main"
%%% End: 

\section{Expression of length functions in the new coordinate system}
\label{s:lengthf}
The aim of this section is to give explicit expressions in the coordinates $(\vec{L}, \vec{\theta})$
for the length $\ell_Y({\mathbf{c}})$, but also for all the lengths $y_\lambda$ of the multi-curve
$(\Gamma_\lambda)_{\lambda\in \Lambda}$ defined in \S \ref{s:ppdecompo}.

\subsection{Trajectories in $T^1 {\mathbf{S}}$ and multiplication in $\mathrm{SL}(2, \R)$}
\label{s:sl2}
 
It is a classical fact that $\cT^*_{\g, \n}$ may be seen as the space of discrete faithful
representations of the fundamental group $\pi_1(\mathbf{S})$ into $\mathrm{PSL}(2, \R)$ (modulo
conjugacy). Moreover, if $Y\in \cT^*_{\g, \n}$ corresponds to the representation
$\rho : \pi_1(\mathbf{S}) \To \mathrm{PSL}(2, \R)$, we can express all lengths thanks to the trace
of $\rho$:
\begin{align} \label{e:trace_cosh}2 \cosh \div{\ell_Y(\gamma)}= |\Tr\, \rho(\gamma)|\end{align}
(where $\rho(\gamma)\in \mathrm{PSL}(2, \R)$ is the image under $\rho$ of the homotopy class of
$\gamma$).  Our method to express lengths in the new coordinates consists in providing an explicit
expression for $\rho(\gamma)$ in~\eqref{e:trace_cosh}, by writing a representative of the homotopy
class of $\gamma$ as a succession of orthogonal segments.

We recall that the unit tangent bundle $T^1 \IH^2$ itself can be identified with
$\mathrm{PSL}(2, \R)$. We can then identify different geometric actions on $T^1\IH^2$ with the
right-multiplication by a specific matrix: $a^t$ is the action of the geodesic flow at time $t$,
$w^t$ the conjugate of the geodesic flow by a rotation of angle $\pi/ 2$, and $k^\theta$ the
rotation of angle $\theta$ by $k^\theta$, where
\begin{align} \label{e:matmult}
  a^t =
  \begin{pmatrix}
    e^{\frac{ t}{2}} & 0 \\ 0 & e^{-\frac{t}{2}}
  \end{pmatrix}
  \qquad
   w^t= 
  \begin{pmatrix}
    \cosh \div{t} & \sinh \div{t} \\ \sinh \div{t} & \cosh \div{t}
      \end{pmatrix}
  \qquad
 k^\theta= 
  \begin{pmatrix}
    \cos \div{\theta} & \sin \div{\theta} \\- \sin \div{\theta} & \cos \div{\theta}
      \end{pmatrix}.
\end{align}
      These moves are represented on Figure \ref{fig:flows}.

 \begin{figure}[h!]
     \includegraphics{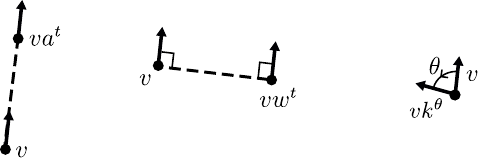}
     \caption{The moves corresponding to the matrices $a^t$, $w^t$ and $k^\theta$.}
     \label{fig:flows}
   \end{figure}

It is important to note that, while $ a^t$ always has non-negative entries, this is true of $w^t$ only if $t\geq 0$. In what follows, a lot of effort
is devoted to using $w^t$ only for $t\geq 0$.

\subsection{Expression of $\ell_Y({\mathbf{c}})$ in the new coordinates\label{s:expell}}

Let us now focus on finding an expression for $\ell_Y(\mathbf{c})$ in terms of
$(\vec{L}, \vec{\theta})$.  We shall do so component by component, i.e. provide an expression for
the $i$th component $\mathbf{c}_i$ of $\mathbf{c}$ for $i \in \{1, \ldots, \cc\}$. Oftenwise, we
will focus on one component $\mathbf{c}_i$ at a time, in which case we will drop the mention of $i$
to lighten notations. This shall always be mentioned, to avoid any confusion.

\subsubsection{Piecewise geodesic representative of ${\mathbf{ c}_i}$ \label{s:straightening2}}

First, we recall that we saw in \eqref{e:decompo-ibis} that the homotopy class of $\cop_i$ can be
written, for any index $q \in \Theta^i$, as
\begin{equation*}
  \cop_i = p(q) \smallbullet p(\sigma q) \smallbullet \ldots \smallbullet p(\sigma^{n_i-1}q)
\end{equation*}
where, for any $q'$, $p(q')$ is the concatenation of the bar $B_{q' }$ and the simple portion
$\cI_{q'}$.

Consider a point $Y\in \cT^*_{\g, \n}$. As suggested in Remark \ref{rem:beta_geod}, we can pick our
representative so that $\beta$ is a multi-geodesic. We then consider the representative of the
homotopy class of ${\mathbf{c}}_i$ given by a cyclic succession of orthogonal geodesics segments:
 \begin{align}\label{e:decompo-ig}\bar p(q)\smallbullet \bar p(\sigma(q)) \smallbullet \ldots
   \smallbullet \bar p(\sigma^{n_i-1}q)\end{align}
 for a $q\in \Theta^i$, where $\bar p(q):=\overline{B}_q \smallbullet \overline{\cI}_q.$ Each bar
 $ \overline{B}_q$ is now an orthogeodesic from $\beta_{\lambda(\sigma^{-1}q)}$ to
 $\beta_{\lambda(q)}$.

 For the rest of \S \ref{s:expell}, it will be convenient to use the relabelling of $\Theta$
 introduced in \S \ref{s:relabel_Theta}, which allows us to identify $\Theta^i$ with the cyclic set
 $\Z_{n_i} = \Z \diagup n_i \Z$. Through this identification, going through the elements
 $k \in \Z_{n_i}$ corresponds to going through the elements $\sigma^k q$ in \eqref{e:decompo-ig}.
 We will therefore write $\overline{B}(i,k)=\overline{B}(\sigma^kq)$,
 $\overline{\cI}(i,k) = \overline{\cI}(\sigma^kq)$, $\theta(i,k) = \theta(\sigma^kq)$ etc. Following
 Notation~\ref{nota:signL}, we will also write $L(i,k) = L(\sigma^k q)$ and
 $\epsilon L(i,k) = \epsilon L(\sigma^kq)$. As a consequence, the consecutive lengths of all the
 segments encountered in \eqref{e:decompo-ig} are (up to cyclic permutation)
$$L(i,1), \theta(i,1), \ldots, L(i,n_i), \theta(i,n_i).$$
The sign of $\epsilon L(i,k)$ tells us in which direction the bar $\overline{B}(i,k)$ is taken at the
step $k$.

\subsubsection{Lift of ${\mathbf{ c}}$ to the universal cover}
\label{s:lift}

 We now consider a lift of the closed path \eqref{e:decompo-ig} to the universal cover of $\Sf$,
 identified with a subset of the hyperbolic plane. For each index $q\in {\mathbf{\Theta}}^i$, let
 $\tilde \beta_q \subset \IH^2$ be a lift of the geodesic $\beta_{\lambda(q)}$ such that, if
 $\tilde B_{\sigma^{-1}q, q}$ denotes the oriented orthogeodesic going from
 $\tilde \beta_{\sigma^{-1}q}$ to $\tilde \beta_q$, then $\tilde B_{\sigma^{-1}q, q}$ projects
 down to $ \overline{B}_q$ on the surface.

 {Whilst there was a priori no natural orientation for the multi-curve $\beta$ on
   $\mathbf{S}$, we can now orient the lifts $\tilde{\beta}_q$ so that the bar
   $\tilde B_{\sigma^{-1}q, q}$ arrives on the right of $\tilde \beta_q$ for each
   $q \in \mathbf{\Theta}^i$ with $\sign(q)=+$ and on the left otherwise. We then make the following
   crucial observation.
   \begin{lem}
     \label{lem:beta_al}
   The geodesics $(\tilde \beta_q)_{q \in \mathbf{\Theta}^i}$ are disjoint and, for any
   $q \in \mathbf{\Theta}^i$, $\tilde{\beta}_q$ and $\tilde{\beta}_{\sigma q}$ are
   aligned. $\tilde{\beta}_{\sigma q}$ is on the left of $\tilde{\beta}_{q}$ if and only if $\sign(q)=+$.
 \end{lem}}

We denote by $\tilde{\cI}_q \subset \tilde{\beta}_q$ the segment between the endpoint of
$\tilde B_{\sigma^{-1}q, q}$ and the origin of $\tilde B_{q, \sigma q}$. Then the segment
$\tilde{\cI}_q$ projects down to the segment $\overline{\cI}_q$ on the surface.  If we call
$\tilde p(q)=\tilde B_{\sigma^{-1}q, q} \smallbullet \tilde{\cI}(q)$, then the infinite
concatenation $\smallbullet_{q\in {\mathbf{\Theta}}^i}\; \tilde p(q)$ of orthogonal geodesics
segments in the hyperbolic plane projects down to \eqref{e:decompo-ig} on the surface. We shall
denote by $\tilde{\mathbf{c}}_i$ the infinite geodesic homotopic to
$\smallbullet_{q\in {\mathbf{\Theta}}^i}\; \tilde p(q)$ in the hyperbolic plane.

\begin{nota}
  The notation $\tilde B_{mn}$, for the moment only defined for $n=\sigma m$, may be
  extended to arbitrary $m, n\in {\mathbf{\Theta}}^i$. We call $\tilde B_{mn} \subset \IH^2$ the
  oriented orthogeodesic segment from $\tilde \beta_m$ to~$\tilde \beta_n$.  We denote by
  $L_{mn}>0$ its length and $z_{mn} \in \tilde\beta_n$ its terminus.
\end{nota}

\subsubsection{Expression of the length of $\mathbf{c}_i$  as a trace}
\label{sec:expression-length-as}

Starting from \eqref{e:decompo-ig} we now find a formula expressing each
$\ell_Y({\mathbf{c}}_i)$ as the trace of a product of matrices.

\begin{lem}\label{p:trace1} For $i \in \{1, \ldots, \cc\}$, the length of the component
  ${\mathbf{c}}_i$ is given by the formula
\begin{equation}
  \label{eq:l_gamma_product_trace}
  \cosh \div{\ell_Y({\mathbf{c}}_i)} = \frac{1}{2} \,  \Tr \paren*{a_1^ib_1^i \ldots a_{n_i}^ib_{n_i}^i}
\end{equation}
where the matrices $a_k^i, b_k^i$ are defined by
\begin{equation} \label{e:ajbj}
  a_k^i = a^{\epsilon L(i,k)}=
  \begin{pmatrix}
    e^{\frac{\epsilon L(i,k)}{2}} & 0 \\ 0 & e^{-\frac{\epsilon L(i,k)}{2}}
  \end{pmatrix}
  \qquad 
 b_k^i = w^{\theta(i,k)}=
  \begin{pmatrix}
    \cosh \div{\theta(i,k)} & \sinh \div{\theta(i,k)} \\ \sinh \div{\theta(i,k)} & \cosh \div{ \theta(i,k)}
  \end{pmatrix}.
\end{equation}
\end{lem}

\begin{proof}
  For the proof, let us fix a value of $i$. As explained at the beginning of this section, the
  elements of $\Theta^i$ are labelled $(i, k)$ with $k\in \Z_{n_i}$.  To lighten notations, we shall
  omit the mention of $i$, for instance writing $n_i=n$, $L_k = L(i,k)$, $\theta(i, k) =\theta_k$,
  etc.

  Let $v$ denote the unit vector tangent to the bar $\overline{B}_1 = \overline{B}(i,1)$ based at
  its origin. We consider a lift $\tilde v$ of $v$ in $T^1 \IH^2$ identified with
  $\mathrm{PSL}(2, \R)$.  The decomposition \eqref{e:decompo-ig} into a succession of orthogonal
  geodesic segments precisely tells us that
  $\tilde v a_1b_1 \ldots a_{n}b_{n}=\rho({\mathbf{c}}) \tilde v$.  This implies that
\begin{align}\label{e:conjugacy}\rho({\mathbf{c}})=\tilde v a_1b_1 \ldots a_{n}b_{n}\tilde
  v^{-1}.\end{align}
 In light of \eqref{e:trace_cosh},  we obtain
$\cosh (\ell_Y({\mathbf{c}})/2) = \frac{1}{2} \, |\Tr \paren*{a_1b_1 \ldots a_{n}b_{n}}|$, which
is our claim up to the presence of the absolute value.
  
  To prove the proposition, we need to argue that we can remove the absolute
  value. In order to do so, we realise the matrix $a_1b_1 \ldots a_{n}b_{n}$ as
  the endpoint of a continuous path
  $$\tilde{a}_1 \diamond \tilde{b}_1 \diamond \ldots \diamond \tilde{a}_n
  \diamond \tilde{b}_n$$ in $\mathrm{SL}(2, \R)$ starting at the identity
  matrix.  More precisely, for any $k$, let us define the paths \begin{align*}
    \tilde a_k : [0,L_k] &\To \mathrm{SL}(2, \R)
    & \tilde b_k : [0, \theta_k] &\To \mathrm{SL}(2, \R)\\
    t &
        \longmapsto
        a^{\sign(k) t} & t &
                  \longmapsto
                  w^t,
\end{align*}
which respectively correspond to the geodesic motion along the bar $\overline{B}_k$ (in positive or
negative time depending on $\sign(k)$), and the parallel transport of a vector orthogonal to
$\overline{\cI}_k$.  Note that the endpoint of $\tilde{a}_k$ and $\tilde{b}_k$ are the matrices
$a_k$ and $b_k$. We then define the path
$\tilde{a}_1 \diamond \tilde{b}_1 \diamond \ldots \diamond \tilde{a}_n \diamond \tilde{b}_n$ as the
concatenation of the successive paths with matching endpoints $t \mapsto \tilde{a}_1(t)$,
$t \mapsto a_1 \tilde{b}_1(t)$, $t \mapsto a_1 b_1 \tilde{a}_2(t)$,
$t \mapsto a_1 b_1 a_2 \tilde{b}_2(t)$, $\ldots$, and $t \mapsto a_1 b_1 \ldots a_n \tilde{b}_n(t)$.

Thanks to our prescription of orientation rules, the path $\cop_i$ traced in $\Sf$ may be lifted to
a path ${\mathbf{c}}^{\mathrm{nl}}_i$ traced in the unit tangent bundle $T^1 {\mathbf{ S}}$,
projecting down to $\cop_i$, satisfying the following: ${\mathbf{c}}^{\mathrm{nl}}_i$ is made of
unit vectors that point to the left of each oriented segment $\overline{\cI}_k$, and are tangent to
the bars $\overline{B}_k$.  We call this a \emph{nice lift} of $\cop_i$.  To
${\mathbf{c}}^{\mathrm{nl}}_i$, we may apply the following statement.

\begin{lem}
  \label{lem:nicelift}
  Suppose we have a piecewise smooth closed path $(\gamma, v)$ drawn in $T^1 {\mathbf{S}}$, where
  the path $t \mapsto \gamma(t)\in {\mathbf{S}}$ is non-contractible, in minimal position, and
  $v(t)\in T_{\gamma(t)}^1 {\mathbf{S}}$ is always on the left of $\gamma$. Then there is a homotopy
  from $(\gamma, v)$ to $(\gamma_{\mathrm{g}}, v_{\mathrm{n}})$, where $\gamma_{\mathrm{g}}$ is the
  closed geodesic homotopic to $\gamma$ and $v_{\mathrm{n}}(t)$ is the left-normal vector to
  $\gamma_g(t)$.
 \end{lem}
 % The latter is a periodic trajectory of the flow $w^t$ followed in {positive} time.  

 We can then conclude that the path $c_i^{\mathrm{nl}}$ in $T^1\mathbf{S}$ can be continuously deformed to
\begin{align*}
[0, \ell_Y(\mathbf{c}_i)]&\rightarrow T^1 {\mathbf{S}}\\
t &\mapsto  v_0 w^t
\end{align*} where $v_0$ is normal to the geodesic representative of $\mathbf{c}_i$ and lies on its left.
As a consequence, the path $v \tilde{a}_1 \diamond \tilde{b}_1 \diamond \ldots \diamond \tilde{a}_n
  \diamond \tilde{b}_n$ can also be continuously deformed
to $t \mapsto  v_0 w^t$.
Each path in this continuous family leads to the same value of $\Tr\, \rho({\mathbf{c}}_i)$.
Since  
the trace of $w^{\ell(\mathbf{c}_i)}$ is positive, we must have $\Tr (a_1b_1 \ldots a_{n}b_{n})>0$ as well, proving the proposition.
\end{proof}

\subsubsection{The formula for the length of $\mathbf{c}_i$}

  Expanding \eqref{eq:l_gamma_product_trace}, we can deduce the following expression for the length
  of ${\mathbf{c}}_i$. 
  We shall use a shorthand notation for hyperbolic trigonometric functions.
  \begin{nota}
    \label{nota:hypdelta}
    We denote $\hyp_+ := \cosh$, $\hyp_- := \sinh$.
    For $\delta=(\delta_1, \ldots, \delta_n)\in \{ \pm 1 \}^{n} $ and
    $t=(t_1, \ldots, t_n)\in \R^n$, we set
    \begin{align}\label{e:hyp_eps}
      \hyp_\delta(t) :=\prod_{i=1}^n \hyp_{\delta_i}(t_i).
    \end{align}
  \end{nota}

We prove the following.

\begin{thm}
  \label{prop:form_l_gamma_try}
   For any $Y \in \mathcal{T}_{\g,\n}^*$ of coordinates
  $(\vec{L},\vec{\theta})$, any $i \in \{1, \ldots, \cc\}$,
  \begin{align*}
    \cosh \div{\ell_Y({\mathbf{c}}_i)}
    & = \sum_{\substack{\delta \in \{ \pm 1 \}^{n_i} \\ \delta_1 \ldots \delta_{n_i} = +1}}
    \hyp_{\delta} \div \theta \cosh
    \div{\rho^\delta \cdot \eps L} 
  \end{align*}
  where $\rho^\delta_k = \prod_{1 \leq j<k} \delta_{j}$ (for $k=1$ this empty product is equal to
  $1$), $\eps L$ is the vector $(\eps L(i,1), \ldots, \eps L(i,n_i))$, and $\cdot$ denotes the
  scalar product, so that $\rho^\delta \cdot \eps L=\sum_{k=1}^{n_i} \eps_k L_k \rho^\delta_k$.
\end{thm}

Note that this result holds for any choice of openings of the intersections of $\mathbf{c}$, and
provides different expressions for the length of $\mathbf{c}$ in different coordinate systems
depending on that choice.

\begin{proof}
  As in the previous proof, we momentarily drop all references to the index $i$.
  By equation~\eqref{eq:l_gamma_product_trace}, the desired quantity can be expressed as
  $\frac 1 2 \, \Tr \paren*{A_1 \ldots A_{n}}$ where for any $k$, $A_k=a_k b_k$, and $a_k , b_k$ are defined in \eqref{e:ajbj}. Calculating the product, we find
  \begin{align*}
    A_k
    & = a_k b_k
      =  \begin{pmatrix} \cosh \div{\theta_k} e^{\eps_k L_k/2} & \sinh \div{\theta_k} e^{\eps_k L_k/2} \\
      \sinh \div{\theta_k} e^{-\eps_k L_k/2} & \cosh \div{\theta_k} e^{-\eps_k L_k/2} \end{pmatrix} \\
    &  = \hyp_+ \div{\theta_k} E_{+1} \paren*{\frac{\eps_k L_k}{2}}
    + \hyp_- \div{\theta_k} E_{-1} \paren*{\frac{\eps_k L_k}{2}}
  \end{align*}
  with the matrices $E_{\pm 1}$ defined by
  \begin{equation}
    \label{e:epm}
    E_{+1}(s) = \begin{pmatrix} e^{s} & 0 \\  0 & e^{-s} \end{pmatrix}
    \qquad \qquad
    E_{-1}(s) = \begin{pmatrix} 0 & e^{s} \\ e^{-s} & 0\end{pmatrix}.
  \end{equation}
We need to understand products of these matrices. Using the fact that
\begin{align*}
  E_{+1}(s) E_{+1}(s') = E_{+1}(s+s') \qquad \qquad
  E_{+1}(s) E_{-1}(s') = E_{-1}(s+s') \\
  E_{-1}(s) E_{+1}(s') = E_{-1}(s-s') \qquad \qquad
  E_{-1}(s) E_{-1}(s') = E_{+1}(s-s')
\end{align*}
one proves by a straightforward induction that for any $s \in \R^{n}$ and $\delta \in \{ \pm 1 \}^{n}$,
\begin{equation*}
  E_{\delta_1}(s_1) \ldots E_{\delta_{n}}(s_{n}) = E_{\delta_1 \ldots \delta_{n}}(\rho^\delta \cdot s)
\end{equation*}
where $\rho^\delta_k$ is defined in the statement of the Lemma. Therefore, expanding the product of
the matrices $A_k$, we obtain
\begin{equation} \label{e:lauras}
  \cosh \div{\ell_Y({\mathbf{c}}_i)}
  = \frac{1}{2} \, 
    \sum_{\delta \in \{\pm 1\}^{n}} \hyp_\delta \div{\theta} \Tr \paren*{E_{\delta_1 \ldots \delta_{n}} \div{\rho^\delta \cdot \eps L}}
\end{equation}
which allows us to conclude, since $\Tr (E_{+1}(s)) = 2 \cosh (s)$ and $\Tr(E_{-1}(s)) = 0$.

 \end{proof}  

Motivated by Proposition \ref{prop:form_l_gamma_try}, we introduce the following
function $M_n$.
\begin{nota}
  \label{nota:Mn}
    For an integer $n \geq 1$, $x = (x_1, \ldots, x_n)$ and $t = (t_1, \ldots,
    t_n) \in \R^n$, we define
  \begin{align}\label{e:Mn}
    M_n(x,t)   & =2\argcosh  \sum_{\substack{\delta \in \{ \pm 1 \}^{n} \\ \delta_1 \ldots \delta_{n} = +1}}  \hyp_{\delta} \div t \cosh
    \div{\rho^\delta \cdot x}.
  \end{align}
\end{nota}

Then, Theorem \ref{prop:form_l_gamma_try} tells us that
\begin{align} \label{e:oldc} \ell_Y({\mathbf{c}}_i)= M_{n_i}\left( \epsilon L(i, 1), \ldots,
  \epsilon L(i, n_i), \theta(i, 1), \ldots,\theta(i, n_i) \right)
\end{align}
where we recall that $\epsilon L(i, k)$ is short-hand for $\epsilon(i, k) L(i, k)$.  Adding the
lengths of all components of ${\mathbf{c}}$, we obtain
\begin{align} \label{e:outcome2}
   \ell_Y({\mathbf{c}})= \sum_{i=1}^{\cc}M_{n_i}\left( \epsilon L(i, 1), \ldots,\epsilon L(i, n_i), \theta(i, 1), \ldots,\theta(i, n_i)    \right).
\end{align}

{
\subsection{Parameters $\tau$}\label{s:tau}

In the following section, we shall apply similar techniques to find the expression of the lengths
$(\ell_Y(\Gamma_\lambda))_{\lambda\in \Lambda}$ of the pair of pants decomposition constructed in \S
\ref{s:ppdecompo}. As in the case of $\mathbf{c}$, we can realise the path $\Gamma_\lambda$ by going
along a succession of bars $B_q$ and simple portions~$\mathcal{I}_q$, now with different orientation
rules (this was already observed in Remark \ref{rem:pop_poly}).  This motivates the following
definitions.

\textbf{
For the rest of this section (and hence until the end of \S \ref{s:horrible}), we shall assume that
\emph{all intersections are opened the first way}.} As a consequence, the multi-curve $\beta$ is
oriented. This convention will lighten notations for the expressions of $\ell_Y(\Gamma_\lambda)$;
however, the methods we develop can be adapted to other situations.

\subsubsection{Definition}

  Let $q, q'\in \Theta$ be two indices such that $\lambda(q)=\lambda(q')$ and
  $\sign(q) = \sign(q') =: \epsilon$, i.e. both bars $B_q$ and
  $B_{q'}$ arrive on the same component of $\beta$ and on the same side.  We denote as
  $\cK_{qq'}$ the oriented segment on $\beta$ going from the terminus point
  $t(B_q)$ to the terminus point $t(B_{q'})$, following the orientation of $\beta$ if $\epsilon =
  +$, and its reverse orientation otherwise.}

  Once the bars ${B}_q, {B}_{q'}$ have glided to the orthogeodesics
  $\overline{B}_q, \overline{B}_{q'}$, the segment $\cK_{qq'}$ glides to a new segment
  $\overline{\cK}_{qq'}$, and we denote by $\tau_{q q'}$ its length.
 Thanks to the Useful Remark (Remark \ref{r:useful}), $\tau_{q q'}$ is always strictly positive.

 \begin{figure}[h!]
     \begin{subfigure}[b]{0.5\textwidth}
    \centering
  \includegraphics{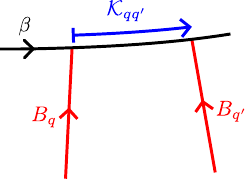}
    \caption{$\eps = +$.}
    \label{fig:tau+}
  \end{subfigure}%
     \begin{subfigure}[b]{0.5\textwidth}
    \centering
  \includegraphics{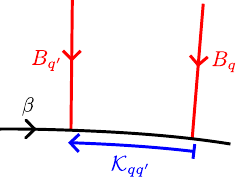}
    \caption{$\eps = -$.}
    \label{fig:tau-}
  \end{subfigure}%
  \label{fig:tauij}
  \caption{The segment $\mathcal{K}_{qq'}$.}
\end{figure}

In the special case when a segment $\cK_{qq'}$ does not contain any other point $t(B_{q''})$ in its
interior, with $B_{q''}$ also arriving on the same side as $B_q$ and $B_{q'}$, then we simply denote
by $\cK_q$ this segment. This notation (where $q'$ has been removed) is justified by the remark that
there is exactly one~$q'$ with this property. We then denote $\tau_q:=\tau_{q q'}$ (or, when
convenient, $\tau(q)$).

\subsubsection{Relation between $\tau$ and $\theta$} 
Each parameters $\tau_{q q'}$ is a sum of a certain number of parameters $\theta_{q''}$: there
exists a (uniquely defined) set of indices $\Theta_\tau(q q')\subseteq \Theta$ such that
\begin{align} \label{e:tautheta}\tau_{q q'}=\sum_{q''\in \Theta_\tau(q q')}\theta_{q''},
 \end{align} 
 reflecting a partition of $\cK_{q q'}$ into
 $\cK_{q q'} =\bigsqcup_{q''\in \Theta_\tau(q q')} \cI_{q''}$.
 
 In the case of $\tau_q$, we adopt as slightly simpler notation, writing that there exists a set of
 indices $\Theta_\tau(q)\subseteq \Theta$ such that
 \begin{align} \label{e:tautheta2}\tau_q=\sum_{q'\in \Theta_\tau(q)}\theta_{q'}.
 \end{align} 
 For a set $W \subseteq \Theta$, we let $\Theta_\tau(W) := \bigcup_{q \in W} \Theta_\tau(q)$.
  Note that this union is not necessarily disjoint.
  
%The following lemma will be important in \S \ref{e:?}.
%\begin{lem}
%For any $q\in\Theta$,
%\begin{itemize}
%\item either there are two distinct $q_1, q_2\in \Theta$, $\sign(q_1)\sign(q_2)=-1$, such that $q\in \Ind(q_1)\cap \Ind(q_2)$,
%\item or there exists $q_1\in \Theta$ and $n\in \Lb 1, N_o\Rb$  (i.e. $\beta_n$ is a boundary curve) such that $q\in \Ind(q_1)\cap \ind(n)$.
%\end{itemize}

% \end{lem}
%This expresses the fact that each segment $\cI_q$ has two sides in the surface $\mathbf{S}$, and on each side there is a boundary component of ${\mathbf{S}}$.
% The first item corresponds to the case where none of the boundary components adjacent to $\cI_q$ is in the family $(\beta_n)_{1\leq n\leq N}$. When one of the boundary components is a curve $\beta_n$, we fall into the second item.

 \subsection{Expression of the boundary lengths in the new coordinates\label{s:bound}}

 We are now ready to find the expression of the length of the curves $\Gamma_\lambda$ for
 $\lambda\in \Lambda$. We recall that we assumed, for the sake of simplicity, that the intersections
 have been opened the first way (see Remark \ref{rem:lambda_both} for the general case).

 \subsubsection{Case where $\Gamma_\lambda$ is a component of $\beta$}
 First, we note that in the special case where the curve $\Gamma_\lambda$ is a component of the
 multi-curve $\beta$, i.e. when $\lambda \in \Lambdabeta$, then we can simply write
 $y_\lambda=\ell_Y(\beta_\lambda)$, which has the simple expression in terms of the parameters
 $\vec{\theta}$:
 \begin{align}\label{e:lnsum}y_\lambda=\sum_{q: \lambda(q) = \lambda} \theta_q
   = \sum_{q\in \Theta_t(\lambda)}  \theta_q,
\end{align}
where we recall the definition of the set $\Theta_t(\lambda)$ from Notation \ref{nota:ind}. 

\subsubsection{Polygonal curve associated to $\Gamma_\lambda$}
\label{s:poly_curve_glambda}

{
Let us now fix $\lambda\in\LambdaFN = \Lambda \setminus \Lambdabeta$ (i.e. $\Gamma_\lambda$ is not a
component of $\beta$). The length of $\Gamma_\lambda$ does not depend on its orientation; we orient
it so that, if $j :=
\step(\lambda)$ is the step of the construction of the pair of pants decomposition at which
$\Gamma_\lambda$ appears, then $\mathbf{P}_j$ lies on the \emph{left} of $\Gamma_\lambda$.

\begin{figure}[h!]
  \centering \includegraphics{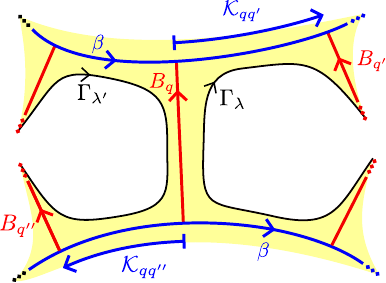}
  \caption{Following two (non-necessarily distinct) boundary components $\Gamma_\lambda$ and
    $\Gamma_{\lambda'}$ of $\mathbf{S}$ near a bar $B_q$ with $\sign(q)=+$.}
  \label{fig:N_pop}
\end{figure}

Let us first assume that $\lambda \in \LambdaBC$, i.e. $\lambda$ is a boundary component of
$\textbf{S}$. Examining the local picture at each intersection and following the boundary of the
regular neighbourhood of the loop~$\mathbf{c}$ (as represented in Figure \ref{fig:N_pop}), we obtain
that we can write the boundary component $\Gamma_\lambda$ as an alternating succession of
bars $B_q$ and segments $\cK_{q}$ defined in \S \ref{s:tau}, i.e. as
 \begin{equation}
  \bar p_1^\lambda \smallbullet \ldots \smallbullet \bar p_{m(\lambda)}^\lambda
  \label{e:decompo-j}
\end{equation}
for some integer $m(\lambda)$, where each piece $\bar p_k^\lambda$ is a concatenation
$\overline{B}_{q} \smallbullet \overline{\cK}_{q}$ for some index $q=q_{k}^\lambda$ in $\Theta$.

\begin{rem}
  We also notice that $\sign q_{k}^\lambda =(-1)^k$, and in particular the integer $m(\lambda)$ is
  even, but this will not play any role in the discussion. 
\end{rem}

More generally, if $\lambda \in \Lambdain$, then we realise $\Gamma_\lambda$ as a boundary component
of the surface $\mathbf{S}_j$ explored at the step $j$ of the construction of the pair of pants
decomposition. We then obtain that $\Gamma_\lambda$ can, too, be written under the form
\eqref{e:decompo-j}, but now with bars
$\bar p_k^\lambda = \overline{B}_{q} \smallbullet \overline{\cK}_{qq'}$ for some $q = q_k^\lambda$
and $q' = {q_k'^\lambda}\in \Theta$. See Figure \ref{fig:Glambda} for an example.

\begin{figure}[h!]
  \includegraphics{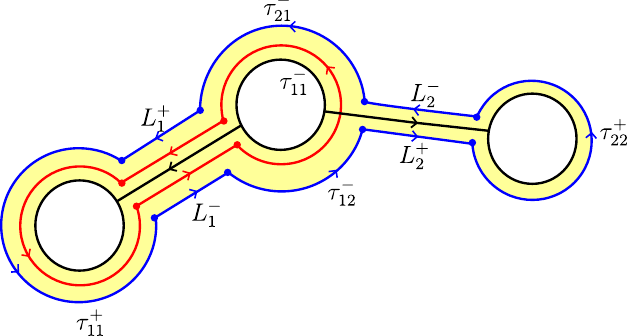}
  \caption{An example of the two new boundary components $\Gamma_\lambda$ for a case where $r=2$,
    and their description in terms of a succession of bars and segments~$\cK$. We use the shorthand
    notation $\tau_{ij}^\epsilon := \tau_{(i,\epsilon),(j,\epsilon)}$.}
  \label{fig:Glambda}
\end{figure}

This expression for the homotopy class of $\Gamma_\lambda$ is the equivalent of \eqref{e:decompo-ig}
for the curve $\mathbf{c}_i$ and will, as before, be our starting point to compute the length of
$\Gamma_\lambda$.  The lengths appearing in this decomposition are the lengths of the bars
$\overline{B}_{q}$, which are the same lengths $L_q$ as before, now together with the lengths
$\tau_{q q'}$ of the segments $\overline{\cK}_{q q'}$.

The important observation here is that, as we go along the curve $\Gamma_\lambda$, we follow a
succession of geodesic segments, only making \emph{right turns} (as seen on Figure \ref{fig:N_pop},
and in the orientation convention of $\tau_{qq'}$ in \S \ref{s:tau}).  Such a path will be called a
{\em{polygonal curve}} in Definition \ref{d:pc}.

\begin{rem}
  \label{rem:lambda_both}
  If we open the intersections both ways, we can still obtain a description of $\Gamma_\lambda$ as a
  polygonal curve (a succession of orthogonal geodesic segments making only right turns), but the
  notations for the orientations of the bars and the parameters $\tau$ need to be adapted. In
  particular, we no longer have alternate signs for the bars, and $m(\lambda)$ is not necessarily
  even (see for instance a figure-eight opened the second way).
\end{rem}}

\subsubsection{Re-labelling}
  
Let us now relabel the segments and lengths appearing in the decomposition~\eqref{e:decompo-j}
following the order in which they occur in $\Gamma_\lambda$. We shall therefore denote
$\overline{B}^\lambda_k := \overline{B}_{q_{k}^\lambda}$ and
$\cK^\lambda_k := \cK_{q_{k}^\lambda q_{k}'^{\lambda}}$, and call $L^\lambda_k$, $\tau^\lambda_k$
their respective lengths. Note that both these lengths are always positive quantities, by definition
for $L^\lambda_k$, and in virtue of the Useful Remark for $\tau^\lambda_k$.
We denote as $\vec{L}^\lambda$ and $\vec{\tau}^\lambda$ the vectors
$\vec{L}^\lambda = (L_1^\lambda, \ldots, L^\lambda_{m(\lambda)})$ and
$\vec{\tau}^\lambda = (\tau^\lambda_1, \ldots, \tau^\lambda_{m(\lambda)})$.
 
%For each $q\in \Theta$, there is exactly one $(\lambda, n)$ with $\lambda\in \LambdaBC$, such that $\overline{B}^\lambda(n)= \overline{B}_q $. This translates the fact that each oriented bar ${B}_q$
% has exactly one boundary component on its right.

% To alleviate notation in the following calculation, we fix $\lambda\in\Lambda$ and remove it temporarily from the labelling. We write $\Gamma$ for $\Gamma_\lambda$, and $m(\lambda)=m$, $L(n)$ for $L^\lambda(n)$, and so on (we shall restore the dependence on $\lambda$ when needed).
%  Thus we have represented the homotopy class of $\Gamma$ by an alternance  $ J_1\smallbullet K_1\smallbullet \ldots J_{m} \smallbullet K_{m} $, where $J_i$ corresponds to a geodesic segment of length $L(i)$, and $K_i$ to a geodesic segment length $\tau(i)$, each segment forming an angle $-\pi/2$ with the previous one. 

 \subsubsection{Expression of the length}
 We now follow quite closely the discussion of \S \ref{s:sl2}.

\begin{lem} \label{p:trace2}
For any $Y \in \cT_{\g,\n}^*$, any $\lambda \in \LambdaFN$, we have
\begin{equation}
  \label{eq:l_beta_product_trace}
  \cosh \div{\ell_Y(\Gamma_\lambda)}
  = \frac{1}{2} \,  {\Tr \paren*{a_1^\lambda b_1^\lambda \ldots a_{m(\lambda)}^\lambda b_{m(\lambda)}^\lambda}}
\end{equation}
where, for any index $k$, $a_k^\lambda := A(L_k^\lambda)$ and $b_k^\lambda := A(\tau_k^\lambda)$ with
\begin{equation} \label{e:ajbj_beta}
  A(s) :=  w^s k^{-\frac \pi 2}
  = \frac{\sqrt 2}{2}
  \begin{pmatrix}
  e^{\frac s2}  & -e^{-\frac s2} \\ e^{\frac s2} & e^{-\frac s2}
  \end{pmatrix}.
\end{equation}
\end{lem}

\begin{proof}
  The proof is the same as the proof of Lemma \ref{p:trace1}. The new path
  in $T^1 \mathbf{S}$ that we consider is represented on
  Figure \ref{fig:deformation}.
 \begin{figure}[h!]
   \includegraphics{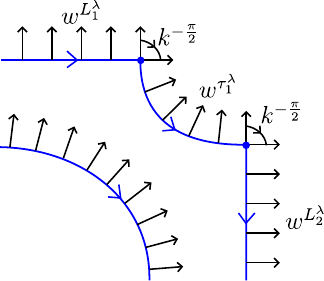}
   \caption{The path in $T^1\mathbf{S}$ followed in the proof of Lemma
     \ref{p:trace2}.}
      \label{fig:deformation}
  \end{figure}
  The key argument to remove the absolute value is that this path has the
  property required by Lemma \ref{lem:nicelift}, namely that the vectors always
  point to its left.
\end{proof}
 
Expanding the trace, we can deduce the following formula giving the lengths of the
curves~$\Gamma_\lambda$.

\begin{thm}
  \label{prp:length_g_lambda}
For any $Y \in \cT_{\g,\n}^*$, any $\lambda \in \LambdaFN$, we have
    \begin{align*}
    \cosh \div{\ell_Y(\Gamma_\lambda)}
      = \frac12 \sum_{\alpha \in \{ \pm 1 \}^{m(\lambda)} }
      {\alpha_1\alpha_2\ldots \alpha_{m(\lambda)}}  \hyp_{-\nabla \alpha} \div{L} \exp\div{\alpha\cdot \tau}
  \end{align*}
  where %  $\rho^{\delta}$ was introduced in Proposition
  % \ref{prop:form_l_gamma_try} and
  $(\nabla \alpha)_k:= \alpha_k/{\alpha_{k-1}}= \alpha_k{\alpha_{k-1}}$ with the
  convention $\alpha_{0}=\alpha_{m(\lambda)}$.
\end{thm}

\begin{rem}\label{r:better}
  For $m=1$ this reduces to
  \begin{align*}
    \cosh \div{\ell_Y(\Gamma)}
    =   \sinh \div{L_1}  \sinh \div{\tau_1}.
  \end{align*} 
\end{rem}

\begin{proof}
  To simplify notations, we shall remove the dependency on $\lambda$ from the
  name of the matrices and lengths.
  By equation \eqref{eq:l_beta_product_trace}, the announced quantity can be
  expressed as
  $$\frac 1 2 \, {\Tr \paren*{A (L_1) A (\tau_1) \ldots A(L_m) A(\tau_m)}}$$
  where $A(s)$ is the matrix introduced in \eqref{e:ajbj_beta}. Calculating the
  product, we find
  \begin{align*}
    A (L) A(\tau)
    =
    \begin{pmatrix}
      e^{\frac \tau 2} \sinh \div L & -e^{-\frac \tau 2} \cosh \div L \\
      e^{\frac \tau 2} \cosh \div L & -e^{-\frac \tau 2} \sinh \div L
    \end{pmatrix}
      = \cosh \div{L} e_{-1} \paren*{\frac{\tau}{2}}
    + \sinh \div{L} e_{+1} \paren*{\frac{\tau}{2}}
  \end{align*}
  with the matrices $e_{\pm 1}$ defined by
  \begin{equation*}
    e_{+1}(s) =\begin{pmatrix} e^{s} & 0 \\  0 & -e^{-s} \end{pmatrix}
    \qquad \qquad
    e_{-1}(s) = \begin{pmatrix} 0 & -e^{-s} \\ e^{s} & 0\end{pmatrix}.
  \end{equation*}
We can use once again the matrices $E_{-1}, E_{+1}$ introduced in equation \eqref{e:epm}, and note the relations
\begin{align*}
  e_{+1}(s) e_{+1}(s') = E_{+1}(s+s') \qquad \qquad
  e_{+1}(s) e_{-1}(s') = -E_{-1}(s-s') \\
  e_{-1}(s) e_{+1}(s') = E_{-1}(-s-s') \qquad \qquad
  e_{-1}(s) e_{-1}(s') = -E_{+1}(s'-s)
\end{align*}
and
\begin{align*}
 & E_{+1}(s) e_{+1}(s') = E_{-1}(s) e_{-1}(s') = e_{+1}(s+s') \\
 & E_{+1}(s) e_{-1}(s') = E_{-1}(s) e_{+1}(s') = e_{-1}(s'-s).
\end{align*}
A straightforward induction then shows that for any $s \in \R^{2k}$ and
$\delta \in \{ \pm 1 \}^{2k}$,
\begin{equation*}
  e_{\delta_1}(s_1)e_{\delta_2}(s_2) \ldots e_{\delta_{2k}}(s_{2k})
  = \delta_2\delta_4 \ldots\delta_{2k} \, E_{\pi(\delta)}(\alpha^\delta \cdot s),
\end{equation*}
where $\pi(\delta) = \prod_k \delta_k$ and  $\alpha^\delta_k := \prod_{j \leq k} \delta_j$,
 and that for any $s \in \R^{2k+1}$ and $\delta \in \{ \pm 1 \}^{2k+1}$,
 \begin{equation*}
   e_{\delta_1}(s_1)e_{\delta_2}(s_2) \ldots e_{\delta_{2k+1}}(s_{2k+1})
   = \delta_2\delta_4 \ldots\delta_{2k} \, e_{\pi(\delta)}(\pi(\delta) \, \alpha^\delta \cdot s).
\end{equation*}
 
Therefore, in the case $m=2k$,
\begin{align*}
  \cosh \div{\ell(\Gamma)}
  = \frac{1}{2} \,  
  \sum_{\delta  \in \{\pm 1\}^{2k}} \delta_2\delta_4\ldots\delta_{m}
  \hyp_{-\delta} \div{L} \Tr \paren*{E_{\pi(\delta)} \div{\alpha^\delta \cdot \tau}}
\end{align*}
which implies that
\begin{align*}
    \cosh \div{\ell(\Gamma)}
    & =  \sum_{\substack{\delta \in \{ \pm 1 \}^{m} \\ \pi(\delta) = +1}}
  \delta_2\delta_4\ldots\delta_{m} \hyp_{-\delta} \div{L} \cosh\div{{\alpha^{\delta} \cdot \tau}}
  \end{align*}
since $\Tr (E_{+1}(s)) = 2 \cosh (s)$ and $\Tr(E_{-1}(s)) = 0$.
By expanding $\cosh\div{{\alpha^{\delta} \cdot \tau}}$ into exponential functions, and letting
 $\alpha_{k-1}\alpha_k = \delta_k =\alpha_k / \alpha_{k-1}$, we obtain the
 claimed formula.

 In the case $m=2k+1$, we conclude similarly, now using
 \begin{align*}
    \cosh \div{\ell(\Gamma)}
    & =  \sum_{\substack{\delta \in \{ \pm 1 \}^{m} \\ \pi(\delta) = +1}}  \delta_2\delta_4 \ldots\delta_{m-1} \hyp_{-\delta} \div{L}\sinh\div{{ \alpha^{\delta} \cdot \tau}}.
  \end{align*} 
\end{proof}

\subsubsection{Reformulation and final outcome}
Motivated by Theorem \ref{prp:length_g_lambda}, we now introduce the
  following notation.

  \begin{nota} \label{nota:QFm}
    For an integer $m \geq 1$ and $x = (x_1, \ldots, x_m)$, $t = (t_1, \ldots,
    t_m)$, we define
    \begin{align}\label{e:Bn}
      Q_m(x, t) =
      2\argcosh \frac12 \sum_{\alpha \in \{ \pm 1 \}^{m} }
        {\alpha_1\alpha_2\ldots \alpha_{m}}
        \hyp_{-\nabla \alpha} \div{x} \exp\div{\alpha\cdot t}
    \end{align}
    and 
\begin{align} \label{e:Fn}
  F_m(x, t)
  & =  \frac{\cosh\Big(\frac{1}{2}Q_m(x, t) \Big)}{\exp \Big(\frac{1}{2}\sum_{n=1}^m
    t_n \Big)}
    = \frac12 \sum_{\alpha  \in \{0, 1\}^{m}} (-1)^{\alpha_1+\ldots+\alpha_m}
    \hyp_{-\partial \alpha} \div{x} \exp(-\alpha \cdot t)
  \end{align}
 with $(\partial \alpha)_n:=1$ if $\alpha_{n-1}=\alpha_n$ and $-1$ if
 $\alpha_{n-1}\not=\alpha_n$. 
\end{nota}
\begin{rem}
  Note that we go from \eqref{e:Bn} to \eqref{e:Fn} by the change of indices
  $(\alpha_n)\rightsquigarrow (\frac{\alpha_n+1}2)$.
\end{rem}

We have proved that
\begin{align}\label{e:outcome}
  \ell_Y(\Gamma_{\lambda})
  = Q_{m(\lambda)}(\vec{L}^\lambda, \vec{\tau}^\lambda).
\end{align}
We observe that the relation between the functions $Q_m$ and $F_m$ can be rewritten
 \begin{align}\label{e:BFn}
 U( Q_m(x, t))=
 \sum_{n=1}^m t_n    + 2\log F_m(x, t),
 \end{align}
 for the function
 \begin{align}\label{e:function_U}
   U(x)=2\log  \cosh \div{x}=  x -2 \log 2 +\cO(e^{-x}).
  \end{align}
The final outcome of the calculation then is
  \begin{align}\label{e:outcome3}  \ell_Y(\Gamma_{\lambda})=
    \sum_{n=1}^{m(\lambda)} \tau^\lambda_n -2\log 2 +2\log
    F_{m(\lambda)}(\vec{L}^\lambda, \vec{\tau}^\lambda)
   +\cO(e^{- \ell_Y(\Gamma_{\lambda})}).
  \end{align}

  \subsubsection{Set of indices associated to $\lambda$}
  Recall from \eqref{e:tautheta2} (the labelling by $q$ being replaced by $(\lambda, k)$) that the
  lengths $\tau^\lambda_k$ can be expressed by
  \begin{align} \label{e:indlambdan} \tau^\lambda_k
    =\sum_{q\in \Theta_\tau(\lambda, k)} \theta_q,\end{align} for some uniquely defined set of indices
  $\Theta_\tau(\lambda, k)\subseteq \Theta$.
  Now, we define the following index notation.
  \begin{nota} \label{nota:indlambda}
    For $\lambda \in \LambdaFN$, let 
    \begin{equation}
      \label{Slambda}
      \Theta_{\ell}(\lambda) := \bigcup_{k=1}^{m(\lambda)} \Theta_\tau(\lambda,k) \subseteq \Theta.
    \end{equation}
    We extend the notation $\Theta_\ell(\lambda)$ to the indices $\lambda \in \Lambdabeta$ by
    letting $\Theta_\ell(\lambda) := \Theta_t(\lambda)$ the set of indices introduced in
    Notation~\ref{nota:ind}.  If $W$ is a subset of $\Lambda$ we denote
    \begin{align}\label{e:indW}\Theta_\ell(W)
      :=\bigcup_{\lambda\in  W} \Theta_\ell(\lambda).
    \end{align}
  \end{nota}
  For any $\lambda \in \Lambda$, $\Theta_\ell(\lambda)$ denotes the set of indices $q\in \Theta$
  necessary to express the length $\ell_Y(\Gamma_\lambda)$.
  \begin{rem}
    Keep in mind that the sets $\Theta_\ell(\lambda)$ are not disjoint: in fact, each $q\in \Theta$
    may belong to two sets $\Theta_\ell(\lambda)$ with $\lambda\in \LambdaBC$, expressing the fact
    that each simple portion $\cI_q$ has a boundary component of $\mathbf{S}$ on its left and one on
    its right. As a consequence, one has an inclusion $(\Theta_\ell(W))^c\subseteq \Theta_\ell(W^c)$
    (where $V^c$ denotes the complement of the set $V$), which in general is strict.
  \end{rem}

%%%Local Variables: 
%%% mode: latex
%%% TeX-master: "main"
%%% End: 

\section{Expression of the volume functions as pseudo-convolutions of FR functions}
\label{s:noncross}
  Let us now, in this brief section, come back to our initial objective,
  i.e. proving Theorem~\ref{t:main}, and reformulate it using
  the new notions and formulae presented in \S \ref{s:loops} to
  \ref{s:lengthf}.   
  More precisely, we shall explain how to compute the volume functions defined
  in \eqref{e:mainint} by
  \begin{equation*}
    \cJ: \ell \mapsto
    \int_{\ell_Y({\mathbf{c}})=\ell}
\prod_{j\in V}  f_j(x_j)  \prod_{j=1}^k x_{\vi} \delta(x_{\vi} - x_{\vi'}) \frac{ \d \mathrm{Vol}^{\mathrm{WP}}_{g_{\mathbf{S}}, n_{\mathbf{S}}}(\x, Y) }{\d\ell}
  \end{equation*}
  as \emph{pseudo-convolutions} of functions in the class $\cF$ -- a notion that will be formalized
  in~\S \ref{s:GC}. 

   We recall that, in the integral above, $V$ is a subset of $\partial \mathbf{S}$. By our
  conventions introduced in \S \ref{s:ppdecomponumb}, we identify $\partial \mathbf{S}$ with
  $\LambdaBC$, and hence $V$ with a subset of $\LambdaBC \subseteq \Lambda$.
  We write $\Vbeta := V \cap \Lambdabeta$ and $\VFN := V \cap \LambdaFN$, so that $\Vbeta$ are the
  components of $V$ that are components of the multi-curve~$\beta$, and $\VFN$ those that are not.  

  \subsection{Change of variables}
  \label{s:ch_var_applied}
 
  Let us change variables using our Jacobian computation, and more precisely its Dirac variant,
  Corollary \ref{c:maindet-dirac}. We obtain that:
 \begin{multline*}
  \cJ(\ell) = 2^{2|\Lambdabeta|-\n} \int_{\ell_Y({\mathbf{c}})=\ell} 
  \bbbone_{\domain} (\vec{L}, \vec{\theta})
  \prod_{\lambda \in V} f_\lambda(y_\lambda)
  \frac{\prod_{\lambda \in \Vbeta}\sinh \div{y_\lambda}
    \prod_{\lambda \in \Lambdainbeta} \sinh^2 \div{y_\lambda}}
  {\prod_{\lambda\in \VFN}\sinh \div{y_\lambda}} \\
     T_{V,\mathrm{m}}(\vec{L}, \vec{\theta})
\prod_{i=1}^r \sinh(L_i)\frac{ \d^r \vec{L} \d^{2r}\vec{\theta}}{\d \ell}
\end{multline*}
where $\domain$ is the domain of variation of the new variables $(\vec{L}, \vec{\theta})$, studied
in \S \ref{s:domain}, and $T_{V,\mathrm{m}}$ is an explicit distribution which we shall leave as
such.  The integral $\cJ(\ell)$ can be rewritten under the simpler form
 \begin{equation}
   \label{e:THEintegral2}
    2^{2|\Lambdabeta|-\n} \int_{\ell_Y({\mathbf{c}})=\ell}  \bbbone_{\domain} (\vec{L}, \vec{\theta})
\prod_{\substack{\lambda \in \Vbeta_+}} g_\lambda(y_\lambda)  
\prod_{\lambda\in \VFN} u_\lambda(y_\lambda) \, T_{V,\mathrm{m}}(\vec{L}, \vec{\theta})
\prod_{i=1}^r \sinh(L_i) \frac{ \d^r \vec{L} \d^{2r}\vec{\theta}}{\d \ell},
\end{equation}
by setting $\Vbeta_+:= \Vbeta \cup \Lambdainbeta$ and, for $\lambda \in \Lambda$,
 \begin{align*}
 g_\lambda(y):=f_\lambda(y) \sinh \div{y} \qquad \text{and} \qquad
 u_\lambda(y):=\frac{f_\lambda(y)}{\sinh \div{y}}
 \end{align*}
 where $f_\lambda$ is already defined for $\lambda \in V \subseteq \LambdaBC$, and is extended to
 $\lambda \in \Lambda$ by letting
 \begin{equation*}
   \begin{cases}
     f_\lambda(y):=1 & \text{for } \lambda \in \LambdaBC \setminus V \\
     f_\lambda(y) := \sinh \div{y} & \text{for } \lambda \in \Lambdain.
   \end{cases}
 \end{equation*}

\subsection{Expansion of the product in the integral}

 Let us now focus on the factor
 \begin{equation}
   \label{eq:factor_J}
   \prod_{\substack{\lambda \in \Vbeta_+}} g_\lambda(y_\lambda)
   \prod_{\VFN} u_\lambda(y_\lambda)
 \end{equation}
 appearing in the integral $\mathcal{J}$.

 \subsubsection{The functions $g_\lambda$}

 For technical considerations related to cancellations at zero in the integral, let us first rewrite
 for each $\lambda \in \Vbeta_+$
 $$g_\lambda(y) := \tilde{g}_\lambda(y) \fz(y)^2,$$
 where $\fz$ is the function introduced in Notation \ref{nota:mt}.  We  notice that we can apply
 the Friedman--Ramanujan hypothesis to every term $\tilde{g}_\lambda$ appearing in the product
 above.  Indeed, for $\lambda \in \Vbeta$, by the hypothesis on the functions
 $(f_\lambda)_{\lambda \in V}$ in Notation \ref{nota:mt}, we have that
 $\tilde{g}_\lambda \in \cF^{\rK_\lambda, \rN_\lambda}$ and, if $p_\lambda$ denotes its principal
 term, 
  \begin{equation}
   \label{eq:replace_gj_tilde}
   |\tilde{g}_\lambda(y) - p_\lambda(y) \, e^{y}|
   \leq \|\tilde{g}_\lambda\|_{\mathcal{F}^{\rK_\lambda,\rN_\lambda}} (1+y)^{\rN_\lambda -1 } e^{y/2}
 \end{equation}
 where $p_\lambda$ is a polynomial of degree $\leq \rK_\lambda-1$.
 For $\lambda \in \Lambdainbeta$, we simply notice that
 $\tilde{g}_\lambda(y) = \sinh^2 (y/2)/\fz(y)^2$ which belongs to $\cF^{\rK_\lambda, \rN_\lambda}$
 for the integers $(\rK_\lambda,\rN_\lambda):=(1,1)$, thanks to the fact that $\sinh^2(y/2)$ has a
 zero of order $2$ at~$0$. Its principal term is $p_\lambda(y) e^{y}$ with $p_\lambda(y):=1/4$. We
 can write the Friedman--Ramanujan hypothesis explicitely for these terms and obtain  
 \begin{align}
   \label{eq:replace_gj_tildein}
   |\tilde{g}_\lambda(y) - p_\lambda(y) \, e^{y}|
   \leq \|\tilde{g}_\lambda\|_{\cF^{\rK_\lambda,\rN_\lambda}} e^{y/2}
 \end{align}
 where $\|\tilde{g}_\lambda\|_{\cF^{\rK_\lambda,\rN_\lambda}}$ is a fixed universal constant. 
 
 We pick a set of indices $\Wbeta \subseteq \Vbeta_+$ for which we take the remainder term in the
 approximations \eqref{eq:replace_gj_tilde} and \eqref{eq:replace_gj_tildein}. This allows us to expand the product as a sum of terms of the form
 \begin{equation*}
   \prod_{\lambda \in \Vbeta_+} \fz(y_\lambda)^2
   \prod_{\lambda \in \Vbeta_+ \setminus \Wbeta} p_\lambda(y_\lambda) e^{y_\lambda}
   \prod_{\lambda \in \Wbeta} \|\tilde{g}_\lambda\|_{\cF^{\rK_\lambda,\rN_\lambda}} \, r_\lambda(y_\lambda)
 \end{equation*}
 where for all $\lambda \in \Wbeta$,
 $|r_\lambda(y_\lambda)|\leq (1+y_\lambda)^{\rN_\lambda -1 } e^{y_\lambda/2}$.

 We then use the fact that, for $\lambda \in \Lambdabeta$, $y_\lambda$ is a linear combination of
 the components of the vector $\vec{\theta}$, given by the simple expression \eqref{e:lnsum}:
 \begin{equation}
   \label{eq:lnsumbis}
   y_\lambda=\sum_{q\in \Theta_\ell(\lambda)}\theta_q.
 \end{equation}
 Note that the sets $\Theta_\ell(\lambda)$ for $\lambda \in \Vbeta_+ \setminus \Wbeta$ are disjoint and
 their union is $\Theta_\ell(\Vbeta_+ \setminus \Wbeta)$ (note that this is not a priori the case for
 the sets $\Theta_\ell(\lambda)$ for $\lambda \in \Lambda$, and we here use the fact that
 $\lambda \in \Lambdabeta$). We obtain that we can further expand the product as a sum of
 contribution of the form
\begin{equation}
   \label{e:prod_g}
   \fn((\tilde{g}_\lambda)_{\lambda \in \Vbeta})
    \prod_{\substack{\lambda \in \Vbeta_+}}\fz(y_\lambda)^2
  \prod_{\lambda \in \Vbeta_+\setminus \Wbeta}  \prod_{q \in \Theta_\ell(\lambda)}
  \theta_q^{\rK_q^{\lambda}}
  \prod_{q \in \Theta_\ell(\Vbeta_+ \setminus \Wbeta)}e^{\theta_q}
   \prod_{\lambda \in \Wbeta} r_\lambda(y_\lambda)
 \end{equation}
 where for $\lambda \in \Vbeta_+ \setminus \Wbeta$,
 $\sum_{q \in \Theta_\ell(\lambda)}\rK_q^{\lambda} \leq \rK_{\lambda}-1$. In particular, for
 $\lambda\in \Lambdainbeta$, $\rK_q^{\lambda}=0$.
 We also note that the constant in front can be bounded by $\prod_{\lambda \in \Vbeta} \|\tilde{g}_\lambda\|_{\cF^{\rK_\lambda,\rN_\lambda}}$.
 
\subsubsection{The functions $u_\lambda$}
 
Similarly, applying the Friedman--Ramanujan hypothesis to the functions
$\tilde{g}_\lambda(y) = f_\lambda(y) \sinh (y/2)/\fz(y)^2 \in \cF^{\rK_\lambda,\rN_\lambda}$ for
$\lambda \in \VFN$ allows us to replace the functions $u_\lambda$ by an approximation,
this time by the polynomial function $p_\lambda$ coming from the principal term of
$\tilde{g}_\lambda$, up to an exponentially decaying error. More precisely, using the relation
\begin{equation*}
  u_\lambda(y) = \frac{4 \fz(y)^2}{(1-e^{-y})^2} e^{-y} \tilde{g}_\lambda(y)
\end{equation*}
together with the hypothesis \eqref{eq:condition_mu} on the function $\fz$, we see that
\begin{equation*}
  |u_\lambda(y) - p_\lambda(y)| \leq 32
  \|\tilde{g}_\lambda\|_{\cF^{\rK_\lambda,\rN_\lambda}}(1+y)^{\rN_\lambda-1} e^{-y/2}.
\end{equation*}

We now replace each term in the product by this expansion, and obtain it can be rewritten as a sum
of terms of the form
\begin{equation*}
  \prod_{\lambda \in \VFN \setminus \WBCFN_0} p_\lambda(y_\lambda)
  \prod_{\lambda \in \WBCFN_0} 32\|\tilde{g}_\lambda\|_{\cF^{\rK_\lambda,\rN_\lambda}}
  \, r_\lambda^0(y_\lambda)
\end{equation*}
where $\WBCFN_0 \subseteq \VFN$ and $|r_\lambda^{0}(y)| \leq (1+y)^{\rN_\lambda-1} e^{-y/2}$ for all
$\lambda \in \WBCFN_0$.

We now express $p_\lambda(y_\lambda)$ as a linear combination of monomials, and use the expression
for $y_\lambda$ in terms of $(\vec{L},\vec{\theta})$ obtained in \S \ref{s:bound}. Note that, for
these terms which are not components of $\beta$, the expression is more complex, and not a simple
linear combination. More precisely, by \eqref{e:outcome3},
\begin{equation}
  \label{eq:ylambdaapp}
  y_\lambda
  = \sum_{n=1}^{m(\lambda)} \tau_n^\lambda - 2 \log 2 + 2 \log
  F_{m(\lambda)}(\vec{L}^\lambda,\vec{\tau}^\lambda) + \O{e^{-y_\lambda}}
\end{equation}
where $\vec{L}^\lambda = (L_1^\lambda, \ldots, L_{m(\lambda)}^\lambda)$ and
$\vec{\tau}^\lambda = (\tau_1^\lambda, \ldots, \tau_{m(\lambda)}^\lambda)$ are the lengths of the
bars and segments~$\cK$ appearing in the description of $\Gamma_\lambda$ as a polygonal curve, and
$F_{m(\lambda)}$ is the function defined in~\eqref{e:Fn}. We can further express the parameters
$\tau$ appearing in the sum in \eqref{eq:ylambdaapp} in terms of $\vec{\theta}$ using
\eqref{e:indlambdan} which states that
\begin{equation*}
  \tau_n^{\lambda} = \sum_{q \in \Theta_\tau(\lambda, n)} \theta_q.
\end{equation*}
The set of indices involved in writing $\vec{\tau}^\lambda$ is
$\Theta_\ell(\lambda)=\bigcup_{n=1}^{m(\lambda)} \Theta_\tau(\lambda,n)$.
Altogether, we obtain that the product of the $u_\lambda$ terms can be expressed as a sum of terms
of the form 
 \begin{equation}
   \label{e:prod_g3}
   \fn((\tilde{g}_\lambda)_{\lambda \in \VFN})
   \prod_{\lambda\in \VFN \setminus \WBCFN}
   \Big[(\log F_{m(\lambda)}(\vec{L}^\lambda, \vec{\tau}^\lambda))^{d_\lambda}
   \prod_{q \in \Theta_\ell(\lambda)} \theta_q^{\rK_q^{\lambda}} \Big]
   \prod_{\lambda \in \WBCFN} r_\lambda(\vec{L},\vec{\theta}^\lambda)
 \end{equation}
 where:
 \begin{itemize}
 \item $\WBCFN \subseteq \VFN$ is the set of indices with exponential decay (it
   contains the set $\WBCFN_0$ above, but can contain more elements due to the exponentially
   decaying term in \eqref{eq:ylambdaapp});
 \item for $\lambda \in \WBCFN$, $\vec{\theta}^\lambda = (\theta_q)_{q \in \Theta_\ell(\lambda)}$ is
   the set of $\theta$ parameters required to express $y_\lambda$
   and
   \begin{equation}
     |r_\lambda(\vec{L},\vec{\theta}^\lambda)|
     \leq (1+y_\lambda)^{\rN_\lambda-1}e^{-y_\lambda/2};
   \end{equation}
 \item for $\lambda \in \VFN \setminus \WBCFN$, 
   $\sum_{q \in \Theta_\ell(\lambda)} \rK_q^{\lambda} +\d_\lambda \leq \rK_{\lambda}-1$;
 \item the constant factor is bounded by $\prod_{\lambda \in \VFN}
   \|\tilde{g}_\lambda\|_{\cF^{\rK_\lambda,\rN_\lambda}}$.
 \end{itemize}

 \subsubsection{Conclusion and substitution in the integral}
 \label{s:conc_subtit_all}
 
 Taking equations \eqref{e:THEintegral2}, \eqref{e:prod_g} and \eqref{e:prod_g3}, we obtain that the
 integral $\cJ(\ell)$ can be rewritten as a fixed sum of terms of the form
 \begin{equation}
   \label{e:developing_factor}
   \begin{split}
     &\fn((\tilde{g}_\lambda)_{\lambda \in V}, r)
     \int_{\ell_Y(\mathbf{c})=\ell}
      \bbbone_{\domain} (\vec{L}, \vec{\theta})
       \prod_{\lambda \in \Vbeta_+} \fz(y_\lambda)^2
       \prod_{q \in \Theta_\ell(V_+ \setminus W)} \theta_q^{\rK_q}
       \prod_{q \in \Theta_\ell(\Vbeta_+ \setminus \Wbeta)} e^{\theta_q} \\
     &\prod_{\lambda \in \VFN  \setminus \WBCFN} (\log F_{m(\lambda)}(\vec{L}^\lambda,\vec{\tau}^\lambda))^{d_\lambda}
       \prod_{\lambda \in W} r_\lambda(\vec{L},\vec{\theta}^\lambda)
       \, T_{V,\mathrm{m}}(\vec{L}, \vec{\theta})
       \prod_{i=1}^r \sinh(L_i) \frac{ \d^r \vec{L} \d^{2r}\vec{\theta}}{\d \ell},
   \end{split}
 \end{equation}
 where:
 \begin{itemize}
 \item $V_+ := V \cup \Lambdainbeta$ and $W := \Wbeta \cup \WBCFN \subseteq V_+$;
 \item $\rK_q := \sum_{\lambda \in V_+\setminus W, \Theta_\ell(\lambda) \ni q} \rK_q^\lambda$ (with
   $\rK_q^\lambda$ defined to be $0$ if non already defined); 
 \item $|r_\lambda(\vec{L},\vec{\theta}^\lambda)| \leq (1+y_\lambda)^{\rN_\lambda-1}e^{\alpha_\lambda
     y_\lambda/2}$ with $\alpha_\lambda=1$ if $\lambda \in \Wbeta$ and $-1$ if
   $\lambda \in \WBCFN$
 \item the constant factor is bounded by $\fn(r) \prod_{\lambda \in V}
   \|\tilde{g}_\lambda\|_{\cF^{\rK_\lambda,\rN_\lambda}}$.
 \end{itemize}
 Note that we have
 \begin{equation}
   \label{eq:sum_K}
   \sum_{q \in \Theta_\ell(V_+ \setminus W)} \rK_q + \sum_{\lambda \in V^\Gamma \setminus W^\Gamma}
   d_\lambda
   \leq \sum_{\lambda \in V} (\rK_\lambda-1).
 \end{equation}
 
\subsection{Active and neutral parameters} \label{s:act_ntr}
 
Our final aim is to prove that the integral $\cJ$ is a Friedman--Ramanujan function following the
method sketched in \S \ref{s:convolution} to prove that a convolution $f_1 \ast f_2$ is a
Friedman--Ramanujan function. We recall that, in this case, the idea is to observe that
$\cL^{\rK_1+\rK_2}(f_1 \ast f_2) = \cL^{\rK_1} f_1 \ast \cL^{\rK_2}f_2$. Hence, if we choose the
indices $\rK_1$ and $\rK_2$ so that $\cL^{\rK_i} f_i$ are elements of $\cR$ (which we can do if
$f_i \in \cF$), then we can easily deduce by an upper bound that
$\cL^{\rK_1+\rK_2}(f_1 \ast f_2) \in \cR$, which in turn implies that $f_1 \ast f_2 \in \cF$.

Now, if a function $f_i$ is already an element of $\cR$, we do not need to cancel its principal
term; in other words, we do not need to apply the operator $\cL$ to $f_i$ in order to establish the
Friedman--Ramanujan property for $f_1 \ast f_2$ (we can take $\rK_i=0$). In our language, this means
that the variable $x_i$ corresponding to $f_i$ in the convolution
$f_1 \ast f_2(x) = \int_{x_1+x_2=x}f_1(x_1)f_2(x_2) \frac{\d x_1 \d x_2}{\d x}$ is a \emph{neutral
  parameter}. To the contrary, a variable which appears with a principal term in the integral (and
will therefore need to be cancelled by applying the operator $\cL$) will be called an \emph{active
  parameter}.

\begin{nota}
  \label{nota:ac_ne}
   We partition the set of $\Theta$ into sets of \emph{active and neutral parameters}
   $\Theta = \ThetaAc \sqcup \ThetaNe$ where
   \begin{align}\label{e:cV}
     \ThetaAc := \{ q \in \Theta_\ell(\Vbeta_+) \, : \, q \notin \Theta_\ell(W)\}.
   \end{align}
   We denote as $\thetaAc = (\theta_q)_{q \in \ThetaAc}$ and
   $\thetaNe = (\theta_q)_{q \in \ThetaNe}$ the components of
   $\vec{\theta} = (\theta_q)_{q \in \Theta}$ in those two sets of indices. We will call
   \emph{neutral parameters} the components of $\thetaNe$ together with the vector
   $\vec{L}=(L_1, \ldots, L_r)$, and \emph{active parameters} the components of $\thetaAc$.
 \end{nota}
 In the following, the convolution argument will concern the variables $\thetaAc$, while $\vec{L}$ and
 $\thetaNe$ are kept fixed. We will not need to apply the operator $\cL$ with respect to the neutral
 variables, because their contribution does not yield a principal term in the product \eqref{e:developing_factor}.  

 \begin{rem}
   The reason why we do not need to work on the neutral parameters is that they will appear without
   a factor $2$ in our comparison estimate, Proposition \ref{p:comparison}. We have already
   commented on the importance of factors of $2$ in Proposition~\ref{p:comparison1}, a simple
   comparison estimate. We saw in Lemma \ref{lem:df} how double-filling loops (for which all
   parameters are neutral) can easily be proven to be Friedman--Ramanujan remainders, by direct
   inequalities, without any need to apply the operator $\cL$.
 \end{rem}
 
 We now set aside the neutral parameters in \eqref{e:developing_factor} and consider them as
 constants. The focus of the remainder of the paper is thus on integrals of the form
\begin{align} \label{e:int_ell}
  \mathbf{Int}(\ell)=
  \int_{h(\thetaAc)=\ell}
  \varphi_0(\thetaAc)
  \prod_{q \in \ThetaAc} \theta_q^{\rK_q} e^{\theta_q}
  \,  \frac{  \d \thetaAc}{\d \ell}
\end{align}
where:
\begin{itemize}
\item all neutral variables $\vec{L}, \thetaNe$ are fixed;
\item we have the upper bound
  \begin{equation}
    \label{eq:boundKlambda}
    \sum_{q\in \ThetaAc}\rK_q
    \leq \sum_{\substack{\lambda \in \Vbeta_{+}\\ \Theta_\ell(\lambda) \not\subseteq\Theta_\ell(W)}}( \rK_\lambda-1)
    \leq \sum_{\lambda \in \Vbeta_{+}\setminus W} (\rK_\lambda-1);
  \end{equation}
\item the weight function $\varphi_0$ is defined as
  \begin{align} \label{e:spec_phi}
    \varphi_0(\thetaAc)
    = \bbbone_{\domain}(\vec{L}, \vec{\theta})
    \prod_{\lambda \in \Vbeta_+} \fz(y_\lambda)^2
    \prod_{\lambda\in \VFN  \setminus \WBCFN}
    (\log F_{m(\lambda)}(\vec{L}^\lambda, \vec{\tau}^\lambda))^{d_\lambda}\end{align}
  and, for each $\lambda \in \VFN \setminus \WBCFN$, $d_\lambda \leq \rK_\lambda-1$;
\item the level function $h$ is the length function of $\mathbf{c}$,
\begin{align} \label{e:spec_h}
h(\thetaAc)=\ell_Y({\mathbf{c}})=\sum_{i=1}^{\cc}  \ell_Y({\mathbf{c}}_i).
\end{align}
\end{itemize}

It will be useful in the following to keep track of the quantities we have taken out of the integral
due to depending only on neutral variables:
\begin{equation}
  \label{e:takenout}
  \fn((\tilde{g}_\lambda)_{\lambda \in V}, r)
  \prod_{q \in \Theta_\ell(V_+) \cap \ThetaNe} \theta_q^{\rK_q}
  \prod_{\lambda \in W} r_\lambda(\vec{L},\thetaNe)
  \, T_{V,\mathrm{m}}(\vec{L}, \thetaNe)
  \prod_{i=1}^r \sinh(L_i).
\end{equation}

Comparing with the forthcoming definition \eqref{e:defconv}, this appears as the
$(h, \varphi_0)$-convolution of the Friedman--Ramanujan functions
$(\theta_q^{\rK_q}e^{\theta_q})_{q\in \ThetaAc}$.  The reason why we use the word \emph{convolution}
is that the function $ h(\thetaAc)$ behaves in many respects like the sum
$\sum_{q\in \ThetaAc} \theta_q$ (plus a function of the neutral variables), and the function
$\varphi_0$ will be shown to have ``small'' derivatives in some sense, so it can be thought of as
close to being constant.  The purpose of \S \ref{s:GC} is to make these notions precise.

%%%Local Variables: 
%%% mode: latex
%%% TeX-master: "main"
%%% End: 

\section{Pseudo-convolutions and stability of the FR hypothesis}
\label{s:GC}
In this section, we introduce a notion of \emph{pseudo-convolution}
$f_1\star \ldots \star f_n|^h_\varphi$ of functions $(f_i)_{1 \leq i \leq n}$ with the level
function $h$ and the weight $\varphi$, which generalizes the usual convolution
$f_1 \ast \ldots \ast f_n$, and is motivated by the form of \eqref{e:int_ell}. We then study the way
the operator $\cL = \id - \cP$ acts on pseudo-convolutions in \S
\ref{s:action_L_pseudo-convolution}. This will provide statements of the following kind: if
$f_1, \ldots, f_n$ are Friedman--Ramanujan functions, and if $(h, \varphi)$ satisfy adequate
assumptions, then $f_1\star \ldots \star f_n|^h_\varphi$ is also a Friedman--Ramanujan function; see
Theorem \ref{t:intermediate}. Roughly speaking, the function $h(x_1, \ldots, x_n)$ is assumed to be close to
$x_1+\ldots+x_n+C$ and $\varphi$ to a constant (in a certain topology defined in \S \ref{s:Evarphi}).

\subsection{Definition of the notion of pseudo-convolution}

We make the following definition.
 \begin{defa}
   Take $n \geq 1$. Let $\varphi :\IR^n\rightarrow \IR$ be a locally bounded, measurable function,
   and $h$ be a real-valued mesurable function defined on a set containing the support
   $\supp(\varphi)$ of $\varphi$. We assume that for every compact subset $K\subset \IR$, the set
   $h^{-1}(K) \cap \supp (\varphi)$ has finite measure.

   We call \emph{$(h, \varphi)$-convolution} of $n$ locally integrable functions
   $f_1, \ldots, f_n: \IR \rightarrow \IR$ the pushforward of the product measure
   $\varphi(x_1, \ldots, x_n)\prod_{i=1}^n f_i(x_i) \d x_i$ by the map $h$. In other words, it is
   the distribution $\nu$ on $\IR_{\geq 0}$ defined by the fact that, for every continuous function
   $F:\R_{\geq 0} \rightarrow \R$ with compact support,
 \begin{align*}
   \int_{ \IR} F(x)\nu(\mathrm{d} x)
   = \int_{\IR^n}F(h(x_1, \ldots, x_n)) \, \varphi(x_1, \ldots, x_n) \prod_{i=1}^n f_i(x_i) \d x_i.
 \end{align*}
 \end{defa}
 
 In our applications, the functions $h$ and $\varphi$ will always be much more than measurable. More
 precisely, throughout this article, we suppose the following.

 \begin{nota}[Assumptions on $(\varphi,h)$] \label{as:h0}
   Take $a > 0$. We assume the following.
   \begin{itemize}
   \item $\varphi : \R^n \rightarrow \R$ is a $\mathcal{C}^1$ function of support included in
     $(a, +\infty)^n$.
   \item $h: [a, +\infty)^n \rightarrow \R$ is a $\cC^1$ function.
   \item For any index $1 \leq i \leq n$ and any fixed
     $\hat{\x}_i = (x_j)_{j\not=i}\in [a, + \infty)^{n-1}$, the function
   \begin{align*}h_i^{\hat{\x}_i}: \,x_i\mapsto h(x_1, \ldots, x_n)
   \end{align*}
   is a $\cC^1$-diffeomorphism from $[a, +\infty)$ onto its image; we denote
   $\ell\mapsto h_i^{-1}(\ell, \hat{\x}_i)$ its inverse.
 \item  The derivative
   $(\ell, \hat{\x}_i) \mapsto \partial_{\ell}h_i^{-1}(\ell, \hat{\x}_i)$ is bounded on
   its domain of definition.
   \end{itemize}
 \end{nota}

The following proposition follows directly from the change of variable formula:
 
 \begin{prp} 
   Under the hypotheses of Notation \ref{as:h0}, the $(h, \varphi)$-convolution~$\nu$ of any~$n$
   locally integrable functions $f_1, \ldots, f_n$ admits a density with respect to the Lebesgue
   measure, with the following expression
\begin{align} \label{e:hconv}
  \frac{d\nu}{\d\ell}(\ell)= \int
  \varphi \big(x_1, \ldots, h_i^{-1}(\ell, \hat{\x}_i), \ldots, x_n\big)
  f_i(h_i^{-1}(\ell,  \hat{\x}_i))
        \prod_{j\not=i} f_j( x_j) \,
       \partial_{\ell}h_i^{-1}(\ell, \hat{\x}_i)
 \d \hat{\x}_i
\end{align}
where $\d \hat{\x}_i := \prod_{j\not=i} \d x_j$. In particular this is independent of the choice of
the index $i$. 
\end{prp}

\begin{nota}
  \label{nota:pc}
  The expression \eqref{e:hconv} will be denoted $f_1\star \ldots \star f_n|^h_\varphi$.
\end{nota}

If $x = ( x_1, \ldots, x_n)$ are related by the implicit relation $h(x)=\ell$, we will denote
$\frac{\d x_1\ldots \d x_n}{\d\ell} = \frac{\d x}{\d \ell}$ any of the equivalent expressions of the
measure
\begin{align}\label{e:density}\frac{\d x_1\ldots \d x_n}{\d\ell}
  = \frac{\d x}{\d \ell}
  := \partial_{\ell}h_i^{-1}(\ell, \hat{\x}_i)
\d \hat{\x}_i.
\end{align}
This corresponds to the
disintegration of the Lebesgue measure $\d x = \d x_1\ldots \d x_n$ on the level-sets of~$h$.  Thus,
we can write
\begin{align}\label{e:defconv}
  f_1\star \ldots \star f_n|^h_\varphi (\ell)
  =&\int_{h(x)=\ell} 
     \varphi(x)
     \prod_{i=1}^n f_i(x_i) \, \frac{\d x}{\d\ell}.
\end{align}

\begin{exa}
  The usual convolution corresponds to $h(x)= x_1+\ldots+ x_n$, and either $\varphi\equiv 1$, or
  $\varphi(x)=\bbbone_{ x_1\geq 0, \ldots, x_n \geq 0}$ (in the case of convolutions of functions
  defined on $\R_{\geq 0}$).
\end{exa}

\begin{rem}[Associativity] The associativity of the usual convolution follows from associativity of
  the addition.  For a general function $h$, the associativity rule may be replaced by formulae such
  as
  \begin{align*}
    f_1\star \ldots \star f_n|^h_\varphi(\ell)
    = \int f_2\star \ldots \star f_n|^{h_{ x_1}}_{\varphi_{ x_1}}(\ell)  f_1( x_1) \d x_1
  \end{align*}
  where $\varphi_{ x_1}$ (and similarly $h_{ x_1}$) is the function of $n-1$ variables defined by
  freezing the first variable to $ x_1$:
$$\varphi_{ x_1}( x_2,  x_3, \ldots,  x_n):=\varphi( x_1, x_2,  x_3, \ldots,  x_n).$$
This simply expresses the fact that we can do the convolution
$f_2\star \ldots \star f_n|^{h_{ x_1}}_{\varphi_{ x_1}} $ with respect to $x_2, \ldots, x_n$ (while
$x_1$ is kept fixed) and then integrate with respect to $f_1( x_1) \d x_1$, to obtain the
pseudo-convolution $f_1\star \ldots \star f_n |^h_\varphi$.
\end{rem}
In the rest of the article, we will need to use this associativity rule for multiple variables and
pairs $(h,\varphi)$. In order to avoid cumbersome notations, we will often omit the index $x_1$
in the freezed functions $h_{x_1}$ and $\varphi_{x_1}$ when the context allows to avoid any confusion.

\subsection{Action of the operator $\cL_\ell$ on pseudo-convolution}
\label{s:action_L_pseudo-convolution}

In \S \ref{s:cL} we defined two operators acting on locally integrable functions, $\cP_x=\int_0^x$ (taking a primitive) and $\cL_x= \id-\cP_x$. Since we are dealing with functions of several variables $(x_1, \ldots, x_n)$, the subscript
$x_i$ will serve to indicate that the primitive is taken with respect to the variable $x_i$. 
 
The goal of this section is to compare
$\cL^{\rK_1+\ldots+\rK_n}_\ell(f_1\star \ldots \star f_n|^h_\varphi)$ with
$(\cL_{ x_1}^{\rK_1}f_1)\star \ldots \star (\cL_{ x_n}^{\rK_n }f_n)|^h_\varphi$. With the usual
convolution on functions defined on $\R$, we have the exact identity
\begin{equation}
\cL^{\rK_1+\ldots+\rK_n}_\ell(f_1* \ldots * f_n) = (\cL_{ x_1}^{\rK_1}f_1)*\ldots * (\cL_{
  x_n}^{\rK_n }f_n).\label{eq:L_convolution}
\end{equation}
Note that, in the case of functions defined on $\IR_{\geq 0}$ only, there are additional boundary
terms coming from $x_i=0$ in this relation (we leave this observation as an exercise to the reader).

For general functions $h$ and $\varphi$, the two expressions
$\cL^{\rK_1+\ldots+\rK_n}_\ell(f_1\star \ldots \star f_n|^h_\varphi)$ and
$(\cL_{ x_1}^{\rK_1}f_1)\star \ldots \star (\cL_{ x_n}^{\rK_n }f_n)|^h_\varphi$ will not be
equal. We provide in Theorem \ref{t:thebigone} a detailed expression for the error term produced by
replacing one by the other.

We saw in \S \ref{s:convolution} how the identity \eqref{eq:L_convolution} can be exploited to prove
the stability of the Friedman--Ramanujan hypothesis by convolution. In \S \ref{s:technicalsection},
we will use the results of this section to provide a similar stability result, under certain
hypotheses on the level function $h$ and the weight $\varphi$.

\subsubsection{Basic calculus in one variable} We gather in the following lemmas the behavior of $\cP$ and $\cL$ under change of variables and multiplication. The change of variable formula yields:
\begin{lem} \label{l:Icompo} Let $h: (0, + \infty) \rightarrow (h_{\inf}, + \infty)$ be a
$\cC^1$-diffeomorphism with $h_{\inf}\geq 0$, and $h^{-1}$ be its inverse.  Then, for any continuous
function $G:\R_{\geq 0} \rightarrow \R$, 
 \begin{align*} \cP_\ell (G\circ h^{-1}) = (\cP_{x}  G) \circ h^{-1} + R_x G \circ h^{-1}   
  \end{align*}
  where $R_x G(x) := \cP_{x} (G (\partial_x h-1))$.
\end{lem}
\begin{proof}
This is just a way of writing
$$\int_{h_{\inf}}^{\ell} G(h^{-1}(u)) \d u
= \int_0^{h^{-1}(\ell)} G(t)\d t
+\int_0^{h^{-1}(\ell)} G(t) (h'(t)-1) \d t.$$
 It is important to note that $\ell \mapsto G(h^{-1}(\ell))$ is only defined on $(h_{\inf}, +\infty)$, so the operator $ \cP_\ell$ should, a priori, be defined as $\int_{h_{\inf}}^\ell$. However, if we extend $G(h^{-1}(\ell))$ to be $0$ if $\ell\leq h_{\inf}$, the formula remains true taking $ \cP_\ell=\int_0^\ell$.
 \end{proof}
  
Integration by parts can be written in the following form:  
    \begin{lem}
      \label{l:product}
      For any continuous function $G : \R_{\geq 0} \rightarrow \R$, any $\mathcal{C}^1$ function
      $\varphi : \R_{\geq 0} \rightarrow \R$,
   \begin{align*} 
  \cP_{x}(\varphi G) =\varphi\, \cP_{x}G   -\cP_x \big(\partial_x{\varphi} \, \cP_x G\big).
   \end{align*}
       \end{lem}
\begin{proof}
This just says that
$$ \int_0^x \varphi(t) G(t) \d t=\varphi(x) \int_0^x  G(t) \d t
-\int_0^x \varphi'(t)\Big(\int_0^t G(s)\d s \Big) \d t$$
which is indeed true by integration by parts.
\end{proof}
  
By combining both previous formulas and iterating them, we obtain:
       
\begin{prp}\label{p:LK}
  With the notations of Lemma \ref{l:Icompo} and \ref{l:product}, for any $\rK\geq 1$,
  \begin{align*} 
    \cL^{\rK}_\ell((\varphi G)\circ h^{-1}) = \,
    &  
      (\varphi \, \cL^{\rK}_{x}G ) \circ h^{-1} 
    \\
 &  -\sum_{t=0}^{\rK-1}\cL_\ell^{\rK-1-t} \big((\varphi \, R_x  ( \cL_{x}^t  G) )\circ h^{-1} \big)\\
    &  +\sum_{t=0}^{\rK-1}\cL_\ell^{\rK-1-t}
      \cP_\ell \big(\partial_\ell (\varphi\circ h^{-1})\,   \cP_\ell (\cL_x^{t} G \circ h^{-1}) \big).
            \end{align*}            
          \end{prp}
   \begin{proof}  
     For $\rK=1$, a combination of Lemmas \ref{l:Icompo} and \ref{l:product} (w.r.t. the variable
     $\ell$) yields
   \begin{align*} 
     \cL_\ell((\varphi G)\circ h^{-1}) 
     % & = (\varphi G)\circ h^{-1} - \cP_\ell\Big(\varphi \circ h^{-1} \, G \circ h^{-1}\Big) \\
& = (\varphi G)\circ h^{-1} - \cP_\ell(\varphi \circ h^{-1} \cdot G \circ h^{-1}) \\     
& \overset{\ref{l:product}}{=} 
       (\varphi G)\circ h^{-1} - \varphi \circ h^{-1} \cdot \cP_\ell ( G \circ h^{-1})
       +  \cP_\ell\big( \partial_\ell (\varphi \circ h^{-1}) \,   \cP_\ell (G\circ h^{-1})\big) \\
     & \overset{\ref{l:Icompo}}{=}   (\varphi \, \cL_{x}G) \circ h^{-1}
       - (\varphi \, R_xG) \circ h^{-1} 
       +  \cP_\ell\big( \partial_\ell (\varphi \circ h^{-1}) \,   \cP_\ell (G\circ h^{-1})\big)
         \end{align*}     
         which gives the expression of the commutator $ [L, T] $ for   
 with $L=\cL$ and $TG= (\varphi G) \circ h^{-1}$:
  \begin{align*}   
    [L, T] G=  
    -  ( \varphi R_x G   )\circ h^{-1}
  +  \cP_\ell \big(\partial_\ell (\varphi\circ h^{-1})\,  \cP_\ell (G\circ h^{-1}) \big).
  \end{align*}            
   We then use the algebraic identity
  \begin{align*}
  L^{\rK}T=TL^{\rK}+ \sum_{t=0}^{\rK-1}L^{\rK-1-t}[L, T] L^t.
  \end{align*}
  to obtain the claimed formula for general $\rK$.
  \end{proof}

  \subsubsection{Application to the first variable} Proposition~\ref{p:LK} allows us to compare the
  action $\cL_\ell^{\rK}(f_1\star \ldots \star f_n|^h_\varphi)$ of $\cL^\rK$ on a pseudo-convolution
  with $(\cL_{x_1}^{\rK}f_1)\star f_2 \ldots \star f_n|^h_\varphi$, which is a first step towards
  the aim of this subsection. We prove the following.
  \begin{prp}
    \label{thm:L_pseudo_conv_one_var}
    For any family of continuous functions $(f_j)_{1 \leq j \leq n}$ and any $\rK \geq 1$, the
    quantity $   \cL_\ell^{\rK}\big(f_1\star \ldots \star f_n|^h_\varphi\big)(\ell)$ can be
    expressed as 
 \begin{align}
   \label{e:LK_1_mult_var}
  & (\cL_{ x_1}^{\rK} f_1)\star f_2 \star \ldots \star f_n|^h_{\varphi \frac{\partial h}{\partial
     x_1}} (\ell) 
   +\cL_\ell^{\rK}\big(f_1\star \ldots \star f_n |^h_{\varphi (1-\frac{\partial h}{\partial  x_1})}\big)(\ell)\\
\label{e:LK_3_mult_var} & -\sum_{t=0}^{\rK-1} \cL_{\ell}^{\rK-1-t} \Big( \int f_2\star \ldots \star
                          f_n|^{h_{ x_1}}_{\varphi \frac{\partial h}{\partial  x_1} R_{ x_1}
                          (\cL^t_{ x_1}f_1 ) } (\ell) \d x_1\Big) \\
   \label{e:LK_2_mult_var}
   &+ \sum_{t=0}^{\rK-1}\cL_\ell^{\rK-1-t} \cP_\ell
     \Big(\int f_2\star \ldots \star f_n|^{h_{ x_1}}_{\frac{\partial \varphi}{\partial  x_1}\cP_{
     x_1}(\cL_{ x_1}^t f_1 \frac{\partial h}{\partial  x_1})}(\ell) \d x_1\Big).
\end{align}
\end{prp} 
          
\begin{rem} For the usual convolution, we have $h( x_1, \ldots, x_n)= x_1+\ldots + x_n$ and
  $\varphi=1$, so $\frac{\partial\varphi}{\partial x_1} = 0$ and $\frac{\partial h}{\partial x_1}=1$
  (and hence $R_{x_1} = 0$).  We see that only the first term remains:
 \begin{align*}
\cL_\ell^{\rK}(f_1\ast \ldots \ast f_n ) = 
(\cL_{ x_1}^{\rK} f_1)\ast f_2 \ast \ldots \ast f_n .
\end{align*}
In the case $\varphi( x_1, \ldots , x_n)=\bbbone_{ x_1\geq 0, \ldots, x_n\geq 0}$, we get a boundary
term with a Dirac mass at $ x_1=0$, it corresponds to the term \eqref{e:LK_2_mult_var}.
 \end{rem} 
 
 \begin{proof}
  We start by writing our pseudo-convolution as:
  \begin{align}\label{e:start}
    f_1\star \ldots \star f_n|^h_\varphi (\ell) =
    \int \varphi(h_1^{-1}(\ell, \hat{\x}_1), \hat{\x}_1)
    \, f_1(h_1^{-1}(\ell,  \hat{\x}_1))
   \prod_{i=2}^n f_i( x_i) \,
   \partial_{\ell}h_1^{-1}(\ell, \hat{\x}_1)
   \d x_2\ldots \d x_n.
\end{align}
 We break \eqref{e:start} into two terms, writing
 \begin{align*}
   \partial_{\ell}h_1^{-1}(\ell, \hat{\x}_1) = 1
   + \partial_{\ell}h_1^{-1}(\ell, \hat{\x}_1) \Big( 1- \frac{\partial h}{\partial x_1}
   (h_1^{-1}(\ell, \hat{\x}_1))\Big).
\end{align*}
We observe that the second term on the right-hand-side yields exactly the second announced term
in \eqref{e:LK_1_mult_var}. We are therefore left to apply the operator $\cL_\ell^\rK$ to the function
\begin{equation*}
  \int \varphi(h_1^{-1}(\ell, \hat{\x}_1), \hat{\x}_1) \, f_1(h_1^{-1}(\ell,  \hat{\x}_1))
   \prod_{i=2}^n f_i( x_i)  \d x_2\ldots \d x_n
\end{equation*}
which we do by applying Proposition \ref{p:LK} with respect to the variable $x_1$ (while
$ x_2, \ldots , x_n$ are frozen parameters) and the function $h=h_1$. We obtain:
\begin{align}
  \cL_\ell^{\rK} \big( (\varphi f_1)\circ h_1^{-1}
  \big) 
  =\, & \label{e:LK_1_mult_var_proof}
        \varphi \circ h_1^{-1} \, \big(\cL^{\rK}_{x_1}f_1\big) \circ h_1^{-1}
  \\ & \label{e:LK_3_mult_var_proof} -\sum_{t=0}^{\rK-1}\cL_\ell^{\rK-1-t}
       \big(\varphi\circ h_1^{-1} \, \big(R_{x_1}  ( \cL_{x_1}^t  f_1) \big)\circ h_1^{-1}  \big) \\
  & \label{e:LK_2_mult_var_proof} +\sum_{t=0}^{\rK-1}\cL_\ell^{\rK-1-t}   \cP_\ell \big(\partial_\ell (\varphi\circ h_1^{-1})\,   \cP_\ell (\cL_{x_1}^{t} f_1\circ h_1^{-1} ) \big).
\end{align}
Integrating the terms \eqref{e:LK_1_mult_var_proof} and \eqref{e:LK_3_mult_var_proof} along
$ x_2, \ldots , x_n$ respectively yield the first term of \eqref{e:LK_1_mult_var} and the term
\eqref{e:LK_3_mult_var} in the theorem. For the last term \eqref{e:LK_2_mult_var_proof}, we use the
chain rule
 $$\partial_\ell (\varphi\circ h_1^{-1})
 = \Big(\frac{\partial \varphi}{\partial x_1} \circ h_1^{-1} \Big)\,
 \partial_{\ell}h_1^{-1},$$
and the change of variable
 $$  \cP_\ell (\cL_{x_1}^{t} f_1\circ h_1^{-1} )
 = \cP_{ x_1}\Big(\cL_{ x_1}^t f_1 \, \frac{\partial h}{\partial  x_1}\Big) \circ h_1^{-1} $$
which leads to the claim once again by integration.
 \end{proof}

 \begin{nota}
   \label{nota:FGHI}
   For $1 \leq j \leq n$, $x = (x_1, \ldots, x_n)$ and $0 \leq t\leq \rK-1$, introduce the functions
   \begin{align*} 
     & F^{\rK}_j (x)
     = \cL_{ x_j}^{\rK} f_j( x_j) \, \frac{\partial h}{\partial x_j} (x) 
     & G_j (x)
     = f_j( x_j) \Big(1-\frac{\partial h}{\partial  x_j} (x)\Big)  \\
     & H^ t_j(x)
     =- R_{ x_j}( \cL^t_{ x_j}f_j  )(x_j) \,  \frac{\partial h}{\partial  x_j} (x) 
     & I_j^t (x)
     = \cP_{ x_j}\Big(\cL_{ x_j}^t f_j \frac{\partial h}{\partial  x_j}\Big)(x).
 \end{align*} 
\end{nota}

 We can rephrase Proposition \ref{thm:L_pseudo_conv_one_var} as follows:
 \begin{cor}\label{c:firstconv}
   For any family of continuous functions $(f_j)_{1 \leq j \leq n}$ and any integer $\rK \geq 1$,
 $\cL_\ell^{\rK}(f_1\star \ldots \star f_n |^h_\varphi)$ is the sum of the following functions, or
 of the images of these functions under some power $\cL_\ell ^t$ where $0\leq t\leq \rK-1$:
\begin{itemize}
\item $ \int f_2\star \ldots \star f_n|^{h_{ x_1}}_{T_{1} \varphi} (\ell) \d x_1 $ where the
  operator $T_1$ is the multiplication by the function $F^{\rK}_1, G_1 $ or $H_1^t$ with
  $0\leq t\leq \rK-1$; or
\item $\cP_\ell \big(\int f_2\star \ldots \star f_n|^{h_{ x_1}}_{T_{1} \varphi} (\ell) \d x_1\big)$
  where $T_{1} =I_1^t \partial_1$ with $0 \leq t \leq \rK-1$, i.e. $T_1$ is the derivative with
  respect to $x_1$ followed by multiplication by $I_1^t$.
\end{itemize}
%The coefficients of this linear combination are combinatorial coefficients which depend only on $\rK$.
\end{cor}

\begin{proof}
  We observe that by the associativity rule, Theorem \ref{thm:L_pseudo_conv_one_var} can also be
  written
   \begin{align*}
     \cL_\ell^{\rK}\big(f_1\star \ldots \star f_n|^h_\varphi\big)(\ell)=
     &\; \int f_2\star \ldots \star f_n |^{h_{ x_1}}_{F_1^K \varphi} (\ell)\d x_1
     +\cL_\ell^{\rK}\Big(\int f_2\star \ldots \star f_n |^{h_{ x_1}}_{G_1 \varphi} (\ell)\d x_1\Big) \\
     &+\sum_{t=0}^{\rK-1} \cL_{\ell}^{\rK-1-t}  \int f_2\star \ldots \star f_n (\ell)|^{h_{
       x_1}}_{H_1^t \varphi}  \d x_1\\
     &+ \sum_{t=0}^{\rK-1}\cL_\ell^{\rK-1-t} \cP_\ell \Big(  \int f_2\star \ldots \star f_n
       |^{h_{ x_1}}_{I_1^t \partial_1 \varphi} (\ell)\d x_1 \Big).
   \end{align*}
   The conclusion then follows directly.
 \end{proof}

 \subsubsection{The multi-variable result}

We shall now successively apply Corollary \ref{c:firstconv} in the variables $ x_1, \ldots, x_n$. In order to do so, the following notations for multi-indices will be helpful.
    
\begin{nota}[Multi-indices]
  If $\alpha=(\alpha_1, \ldots, \alpha_n)\in \Z_{\geq 0}^n$ is a multi-index, we use the standard notations
  $|\alpha|=\sum_{i=1}^n \alpha_i$, $\partial^\alpha=\prod_{i=1}^n \partial_i^{\alpha_i}$ (with
  $\partial_i= \frac{\partial}{\partial x_i}$) and $\alpha \cdot x=\sum_{i=1}^n \alpha_i x_i$.
\end{nota}

  Most of the time, all of the coefficients will be $\alpha_i=0$ or $1$, that is to say, we will
  manipulate differential operators of order at most one in each variable. This motivates the
  following convention.
  \begin{nota}
    If $\pi$ is a subset of $\{1, \ldots, n\}$, we identify $\pi$ with the multi-index
    $\alpha = \bbbone_\pi$ (the indicator function of $\pi$), and hence denote
    $\partial^\pi=\partial^{\bbbone_\pi}=\prod_{i\in \pi} \partial_i$. Then,
    $\bbbone_\pi \cdot x= \sum_{i\in \pi} x_i$.
  \end{nota}

  We now have the following.

  \begin{thm}\label{t:thebigone}
    For any family of continuous functions $(f_j)_{1 \leq j \leq n}$ and any integers
    $(\rK_j)_{1 \leq j \leq n}$ with $\rK_j \geq 1$, if $\rK=\sum_{j=1}^n \rK_j$, then the function
    $\cL_\ell^{\rK}( f_1\star \ldots \star f_n|^{h}_{\varphi })$ can be expressed as the sum of the
    following functions, or their images under $\cL_\ell^{t}\cP_\ell^{t'}$ for $0\leq t\leq \rK-n$
    and $0\leq t'\leq n$,
 \begin{align}\label{e:Tgood2}
   \mathbf{Int}_{\pi}(\ell)
   :=   \int_{\substack{h(x)=\ell}}    \partial^{\pi_0} \varphi(x)
   \prod_{j=1}^n \partial^{\pi_j}  \Phi_j(x)
   \frac{\d x}{\d\ell}
 \end{align} 
 where:
 \begin{itemize}
 \item $\pi = (\pi_j)_{0 \leq j \leq n}$ is a family of disjoint subsets of $\{1, \ldots, n\}$ such
   that $j \notin \pi_j$ for all $j$;
 \item if $B := \bigsqcup_{j=0}^n \pi_j$, then for any $1 \leq j \leq n$,
   \begin{equation} \label{e:thebigphi}
     \Phi_j      =
     \begin{cases}
       F^{\rK_j}_j, G_j \text{ or } H_j^{t''} \text{ with } 0 \leq t'' \leq \rK_j-1
       & \text{if } j\not \in B \\
       I_j^{t''} \text{ with } 0 \leq t'' \leq \rK_j-1 & \text{if } j\in B.
     \end{cases}
   \end{equation}
 \end{itemize}
  \end{thm}

\begin{rem}
  \label{rem:bad}
  The set $B$ corresponds to the set of ``bad'' parameters, i.e. parameters for which $\Phi_j=I_j^{t''}$
  (we will see in \S \ref{sec:estim-funct-f_jk_j} why these terms are more problematic). It is also
  the set of variables w.r.t. which a derivative appears in \eqref{e:Tgood2}.
\end{rem}

\begin{rem}The number of possible terms of form \eqref{e:Tgood2} is $\fn(n, \rK)$.
\end{rem}

The proof is an induction on the number of variables to which the operators are applied (Corollary
\ref{c:firstconv} corresponding to one variable). More precisely, we prove the following result, and
deduce Theorem \ref{t:thebigone} by taking $k=n$.

\begin{lem}
  With the notations of Theorem \ref{t:thebigone}, for any $k\in \{1, \ldots, n \}$, the function
  $\cL_\ell^{\rK_1+\ldots+\rK_k}( f_1\star \ldots \star f_n|^{h}_{\varphi })$ can be expressed as
  the sum of the following functions, or their images under $\cL_\ell^{t}\cP_\ell^{t'}$ for
  $0\leq t\leq \sum_{j=1}^k (\rK_j-1)$ and $0\leq t'\leq k$,
 \begin{align}\label{e:Tgood2k}
   \mathbf{Int}_{\pi}^k(\ell)
   :=   \int_{\substack{h(x)=\ell}}    \partial^{\pi_0} \varphi(x)
   \prod_{j=1}^k \partial^{\pi_j}  \Phi_j(x)
   \prod_{j=k+1}^n f_{j}( x_{j})  \frac{\d x}{\d\ell}
 \end{align} 
 where:
 \begin{itemize}
 \item $\pi = (\pi_j)_{0 \leq j \leq k}$ is a family of disjoint subsets of $\{1, \ldots, k\}$ such
   that $\pi_j$ contains only integers $\geq j+1$ for all $j$;
 \item if $B := \bigsqcup_{j=0}^k \pi_j$, then for any $1 \leq j \leq k$,  $\Phi_j$ satisfies
   \eqref{e:thebigphi}. 
 \end{itemize}
\end{lem}

% \begin{rem}
% It is not strictly necessary to include $B$ in the set of indices, since it can be recovered from $B=\tilde \pi\sqcup \pi_1\sqcup \pi_2\ldots \sqcup \pi_k$. We prefer nevertheless to keep track of $B$, which labels
% the set of variables $x_i$ such that the derivative $\partial_{x_i}$ appears under the integral \eqref{e:Tgood2}.  

%  The condition that $\pi_j$ contains only integers $\geq j+1$ implies that in \eqref{e:Tgood2},
%  $\Phi_j$ is only differentiated in the variables $ x_s$ with $s\geq j+1$. Note that it implies that
%  $\pi_k=\emptyset$, and $\pi_{k-1 }$ can be either $\emptyset$ or $\{k\}$. This condition also
%  implies that if $1\not\in  \tilde \pi$, then $1\not\in B$ (which means that given $ \tilde \pi$,
%  not all subsets $B$ are allowed). 
% \end{rem}

\begin{proof} As announced, this is done by induction on $k$.  For $k=1$, we see that the announced
  expression fits with Corollary \ref{c:firstconv}. If $T_1$ is a multiplication operator by
  $\Phi_1 \in \{F_1^{K_1}, G_1, H_1^{t''}\}$, we take $\pi_0=\pi_1 =\emptyset$.  If
  $T_1= I_1^{t''} \, \partial_1$, we take $\Phi_1=I_1^{t''}$, $\pi_0=\{1\}$ and $\pi_1=\emptyset$.
 
  Assume the result is known for $k$. We apply Corollary \ref{c:firstconv} with respect to the
  variable $ x_{k+1}$, to express the image under $\cL^{\rK_{k+1}}$ of the function
  \eqref{e:Tgood2k}, corresponding to a given data set $\pi=(\pi_0, \ldots, \pi_k)$. Application of
  Corollary \ref{c:firstconv} gives rise to four different kind of terms, depending on the nature of
  an operator $T_{k+1}$ that modifies the function under the integral. For the various situations,
  we examine how the data evolves to a new set of data
  $\pi'=(\pi'_0, \pi'_1,\ldots, \pi'_k, \pi'_{k+1})$, for which we check that the conditions of the
  theorem are still satisfied.
 
 First, for the terms for which $T_{k+1}$ is a multiplication operator by a function $\Phi_{k+1}$
 equal to $F^{\rK_{k+1}}_{k+1}$, $G_{k+1}$ or $H_{k+1}^{t''}$, we take
 $\pi'_0=\pi_0$, $\pi_j'=\pi_j$ for $j \leq k$ and $\pi_{k+1}'=\emptyset$.

 The last terms are those for which $T_{k+1}= I_{k+1}^{t''} \partial_{{k+1}}$.  We need to examine
 the effect of applying this operator to an expression of the form
 $\partial^{\pi_0} \varphi \prod_{j=1}^k \partial^{\pi_j} \Phi_j$.  We use the Leibniz rule to
 write
 \begin{equation}
   \label{eq:Leibniz_term_I}
   I_{k+1}^{t''} \partial_{{k+1}} \Big( \partial^{\pi_0} \varphi \prod_{j=1}^k \partial^{\pi_j}  \Phi_j
   \Big)
   = I_{k+1}^{t''} \partial_{{k+1}} \partial^{\pi_0} \varphi \prod_{j=1}^k \partial^{\pi_j}  \Phi_j
   + I_{k+1}^{t''} \partial^{\pi_0} \varphi \, \partial_{x_{k+1}}\Big( \prod_{j=1}^k \partial^{\pi_j}  \Phi_j\Big).
 \end{equation}
 Then, we observe that those two terms can be expressed in the announced form, letting
 $\Phi_{k+1}=I_{k+1}^{t''}$, $\pi'_{k+1}=\emptyset$, and
 \begin{itemize}
 \item for the first term, $\pi'_0= \pi_0 \cup\{k+1\}$ and $\pi'_j= \pi_j$ for $1\leq j\leq k$;
 \item for the second term, we take $\pi'_0=\pi_0$ and apply the Leibniz rule to the product
   $\prod_j \partial^{\pi_j} \Phi_j$, to write it as a sum of different contributions; then, for
   $1\leq j\leq k$, we take one of the $\pi'_j$ to be equal to $\pi_j\cup\{k+1\}$ (corresponding to
   the index which is hit by the Leibniz rule) and leave all the others are untouched
   ($\pi'_j=\pi_j$).
 \end{itemize} 
 \end{proof}

\subsection{The classes of functions $E^{(a)}$ and $\cE^{(a)}$}
\label{s:Evarphi}

Let us introduce classes of functions which will be helpful to quantify the distance between a
$(h,\varphi)$ pseudo-convolution and the usual convolution. 

\subsubsection{Definition}
\label{sec:definition}

For now, and the rest of the paper, we fix $a>0$.
  
\begin{defa} \label{def:Ena}
  We let $E_n^{(a)}$ be the vector space of functions $h : [a, +\infty)^n \rightarrow \R$
  such that:
\begin{itemize}
\item $h$ is of class $\cC^\infty$;
\item for every multi-index $\alpha\in \{0, 1\}^n$ such that $\alpha\not=(0,\ldots, 0)$, 
\begin{align*}
\sup_{x \in [a, +\infty)^n}\Big\{ e^{\alpha \cdot x}|\partial^\alpha h(x)|\Big\} <+\infty.
\end{align*}
\end{itemize}
This space is endowed with the seminorms
$\norm{h}_{\alpha, a}:=\sup \{ e^{\alpha \cdot x}|\partial^\alpha h(x)|, x \in [a, + \infty)^n\} $.
\end{defa}

Notice that, in this definition, we only control partial derivatives of $h$ that are of
order at most one in each variable.

\begin{nota}
  Define $L_n(x_1,\ldots,x_n)=x_1+\ldots+x_n$ to be the linear functional associated to the usual
  convolution.
\end{nota}
In what follows, we will prove results for pseudo-convolutions assuming that $h$ is close to $L_n$
in a certain sense; the definition below is used to quantify this. 

\begin{defa}\label{def:cE}
  We let $\cE_n^{(a)}$ be the affine space of functions $h$ such that $h -L_n \in E_n^{(a)}$.
Equivalently, $h$ belongs in $\cE_n^{(a)}$ if and only if
\begin{itemize}
\item $h$ is of class $\cC^\infty$;
\item for every multi-index $\alpha\in \{0, 1\}^n$ such that $|\alpha|\geq 2$, 
\begin{align*} 
\sup_{x \in [a, +\infty)^n}\Big\{ e^{\alpha \cdot x}|\partial^\alpha h(x)|  \Big\} <+\infty
\end{align*}
\item for every $j\in \{1, \ldots, n\}$,  
\begin{align*} 
\sup_{x \in [a, +\infty)^n}\Big\{ e^{  x_j}|\partial_j h(x)-1| \Big\} <+\infty.
\end{align*}
\end{itemize}
\end{defa}

\subsubsection{Core examples}
\label{sec:core-examples}

Our main examples of such functions will be the following.

\begin{nota}
  For $\beta=(\beta_\eps)_{\eps\in \{\pm\}^n}$ a family of real numbers, we define
 $$h_\beta(x)= 2\log\Big(\sum_{\eps\in \{\pm\}^n} \beta_\eps \hyp_\eps\div{x}\Big).$$
\end{nota}
 
\begin{prp}\label{p:mainclass}
  We assume that $\beta$ is not identically $0$ and that, for all $\epsilon$, $\beta_\epsilon \geq 0$.
  Then, $h_\beta \in \cE_n^{(a)}$ for any $a>0$.  
 Moreover, there exists $C=C(a, n)$, independent of $\beta$,  such that 
 \begin{align}\label{e:defEM}
   \sup_{\alpha\not= (0, \ldots, 0)}\norm{h_\beta-L_n}_{\alpha, a} \leq C.
 \end{align}
\end{prp}

\begin{cor} \label{r:argcosh} The same result holds for the family
   \begin{equation}
   \tilde h_\beta(x_1, \ldots, x_n)= 2 \argcosh\Big(\sum_{\eps\in \{\pm\}^n} \beta_\eps
   \hyp_\eps\div{x}\Big)\label{eq:def_h_tilde_beta}
 \end{equation}
 if we impose the additional condition that $\beta_{(1, \ldots, 1)}\geq 1$.
\end{cor}

 \begin{proof}[Proof of Proposition \ref{p:mainclass}] 
   For $y\in \IR$, introduce $$g_\pm(y)= e^{-y/2}\hyp_{\pm}(y)=\frac{1\pm e^{-y}}{2}.$$
   Let $ g_\eps(x)=\prod_{j=1}^n g_{\eps_j}(x_j)$.  We have
\begin{align*}
h_\beta(x)=L_n(x)+2\log \Big(\sum_{\eps\in \{\pm\}^n} \beta_\eps g_\eps(x)\Big).
\end{align*}
  Note that
  $|\partial g_+(y)| = \frac{e^{-y}}{2}\leq  e^{-y}g_+(y)$
  and, if $y\geq a>0$,
  $$|\partial g_-(y)| =\frac{e^{-y}}2 \leq \frac{1}{1-e^{-a}} e^{-y}g_-(y).$$ It follows that the
  constant $C =(1-e^{-a})^{-n}$ satisfies, for all $\alpha\in \{0,1\}^n$ and $x \in [a, + \infty)^n$,
\begin{align}\label{e:propg}|\partial^\alpha g_\eps(x)|\leq C\, e^{-\alpha \cdot x}g_\eps(x).
  \end{align}

  For $\alpha\in \{0, 1\}^n$ with $|\alpha|=k \geq 1$, the partial derivative
  $\partial^\alpha  \log \Big(\sum_{\eps\in \{\pm\}^n} \beta_\eps g_\eps(x)\Big)$ is a linear combination of terms of the form
 \begin{align}\label{e:prod_j_k}
 \frac{\prod_{j=1}^k\Big(\sum_{\eps\in \{\pm\}^n} \beta_\eps \partial^{\pi_j}g_\eps(x)\Big)}{\Big(\sum_{\eps\in \{\pm\}^n} \beta_\eps g_\eps(x)\Big)^k}
 \end{align}where $\pi_1, \ldots, \pi_k$ form a partition of $\{i, \alpha_i=1\}$.
 For each $j \in \{1, \ldots, k\}$, \eqref{e:propg} implies
  \begin{align*}
 \frac{ \sum_{\eps\in \{\pm\}^n} \beta_\eps |\partial^{\pi_j}g_\eps(x) |}{\sum_{\eps\in \{\pm\}^n}
    \beta_\eps g_\eps(x)} \leq C \exp \Big( -\sum_{i\in \pi_j} \alpha_i x_i \Big).
 \end{align*}
 As a result, the term \eqref{e:prod_j_k} is bounded by $C^k \exp \paren{-\sum_{i=1}^n \alpha_i x_i}$,
proving the proposition. 
 \end{proof}

Let us now prove the corollary.

 \begin{proof}[Proof of Corollary \ref{r:argcosh}]
   We write $ \tilde h_\beta = V ( h_\beta)$ where
   $$V(x)= 2\argcosh(e^{\frac x 2})=x+2\log 2+\omega(e^{-x})$$
   is the reciprocal of the function $U$ that appeared in \eqref{e:BFn}.  By $\omega(e^{-x})$ we
   mean a function whose derivatives are all $\cO(e^{-x})$.  It is straightforward to check that
   $\tilde h$ satisfies the same derivative estimates as $h$, provided we know beforehand that
   $\sum_{\eps\in \{\pm\}^n} \beta_\eps \hyp_\eps\div{x}$ stays bounded away from $1$ on
   $[a, +\infty)^n$ (to avoid the singularity of $\argcosh$ at $1$). The condition
   $\beta_{(1, \ldots, 1)}\geq 1$ implies that
   $$\sum_{\eps\in \{\pm\}^n} \beta_\eps \hyp_\eps\div{x}
   \geq \exp \paren*{\frac 12 \sum_{i=1}^n x_i}
   \geq \exp\Big(\frac{na}{2} \Big)>1$$
   which is enough to conclude.  
\end{proof}

\subsubsection{Useful properties}
\label{sec:useful-properties}

First, we make the following immediate observation.

  \begin{rem}\label{r:partition} If $\pi=(\pi_1, \ldots, \pi_N)$ is a partition of $\{1, \ldots,
    n\}$, and if for all $j \in \{1, \ldots, N\}$, the functions
    $(x_i)_{i\in \pi_j}\mapsto h_j((x_i)_{i\in \pi_j})$ are in $\cE_{|\pi_j|}^{(a)}$, then following
    function is an element of $\cE_n^{(a)}$: 
    $$(x_1, \ldots, x_n)\mapsto \sum_{j=1}^N h_j((x_i)_{i\in \pi_j}).$$
  \end{rem}
  
  It may also be useful to note the following simple implication of the definition of~$\cE_n^{(a)}$:
  
  \begin{prp} \label{p:existlim} Assume that \eqref{e:defEM} holds for all $\alpha\in\{0, 1\}^n$
    such that $|\alpha|=1$. Then there exists $\CE\in \R$ such that, for all
    $x=(x_1, \ldots, x_n) \in [a, +\infty)^n$,
$$|h(x)-L_n(x)-\CE|
\leq C \sum_{i=1}^n e^{-x_i}.$$
\end{prp}
\begin{proof} First, we observe that if a function $g:[a,+\infty) \rightarrow \R$ satisfies
  $|g(t)-g(s)|\leq C e^{-t}$ for all $t < s$, by Cauchy's rule, it admits a limit $g(+\infty)$ at
  infinity, which satisfies 
  $$|g(t)- g(+\infty)| \leq C e^{-t}.$$
  This property holds for instance if $|g'(t)|\leq C e^{-t}$, and is stable under taking limits of
  sequences of functions. Applied to $g= h-L_n$ successively in each variable, denoting
  $\CE = g(+\infty, \ldots, +\infty)$, this yields
\begin{align*}
|g(x)- \CE|
 &\leq \sum_{i=1}^n | g(x_1, \ldots, x_{i-1}, x_i, +\infty, \ldots, +\infty) -g(x_1, \ldots, x_{i-1}, +\infty, +\infty, \ldots, +\infty)| 
\end{align*}
which implies that $|g(x)- \CE| \leq C \sum_{i=1}^n e^{-x_i}$.
\end{proof}

 \subsection{Stability of the Friedman--Ramanujan hypothesis by pseudo-convolution}
  \label{s:technicalsection}

 We now have the ingredients to state a general result, according to which $   f_1\star \ldots \star f_n|^{h}_{  \varphi } $ is a Friedman--Ramanujan function if $f_1, \ldots, f_n$ are so, under certain assumptions on $h $ on
 $\varphi$. In particular, we will prove Theorem \ref{t:intermediate} -- which is not yet the result used to prove Theorem \ref{t:main}, but may be considered as a preliminary version, and provides a nice abstract statement:
 
 \begin{thm} \label{t:intermediate} Let $a>0$. We assume that $h\in \cE_n^{(a)}$ and that
   $\varphi\in E_n^{(a)}$ with $\supp \varphi\subset (a, +\infty)^n$.  If $(f_j)_{1 \leq j \leq n}$
   are continuous functions in $\FR$, then $ f_1\star \ldots \star f_n |^h_\varphi\in \FR$.
 \end{thm}
 
 Recall that we always assume that the conditions of Notation \ref{as:h0} hold.

\subsubsection{Quantitative hypotheses on the functions $h$, $\varphi$ and $f_j$}

As we prove Theorem \ref{t:intermediate}, we will be able to give estimates on the
Friedman--Ramanujan norms of $ f_1\star \ldots \star f_n|^{h}_{ \varphi} $ as a function of those of
$f_1, \ldots, f_n$ (the precise bound is given in \eqref{e:controlnorm}). To this end, we introduce
more specific notation which will be useful to track everything in the proof.

\begin{nota} \label{nota:assum_1} Let $n \geq 1$ and $a >0$. We assume the following.
  \begin{enumerate}
  \item[\textbf{($h$)}] The function $h$ lies in $\cE_n^{(a)}$. We denote $C_1, C_2 \geq 1$ two
    constants such that
    \begin{align*}
      \forall \alpha \in \{0,1\}^n \setminus \{(0, \ldots, 0)\}, 
       &\quad \norm{h-L_n}_{\alpha, a} \leq C_1 \\
      \forall j \in \{1, \ldots, n\}, 
        & \quad \Big|\frac{\partial h}{\partial  x_j}\Big| \leq C_2.
    \end{align*}
    By the hypothesis from Notation \ref{as:h0} there exists a constant $C_2'$ such that
    \begin{equation*}
      \forall j \in \{1, \ldots, n\}, \quad
   0 < \Big|\frac{\partial x_j}{\partial h}\Big| \leq C_2'.
 \end{equation*}
 We let $C_0=\max(C_1, C_2,C_2')$.
    By Proposition \ref{p:existlim}, there exists a $\CE \geq 1$ such that
 \begin{align}\label{e:comparisonZ}
   \forall x = (x_1, \ldots, x_n) \in [a, + \infty)^n, \quad
   \sum_{i=1}^n  x_i\leq h(x)+ \CE.
 \end{align}
\item[\textbf{($\varphi$)}] The function $\varphi$ is supported in $(a,+\infty)^n$ and
  $\varphi\in E_n^{(a)}$. Let $C_3\,\geq 0$ be a constant such that, for every multi-index
  $\alpha\in \{0, 1\}^n$ with $\alpha\not=(0,\ldots, 0)$, $\norm{\varphi}_{\alpha, a} \leq C_3$.
\item[\textbf{($f$)}] For $j \in \{1, \ldots, n\}$, $f_j \in \cF^{\rK_j, \rN_j}$ is a continuous
  function. We assume w.l.o.g. that $\rK_j \leq \rN_j$ and $\norm{f_j}_{\cF^{\rK_j, \rN_j}}=1$.  By
  Remark \ref{rem:boundsFRL}, for $x_j \geq 0$ and $0 \leq m < \rK_j$,
  \begin{align*}
    & | \cL^{m} f_j ( x_j)|\leq M_{\rK_j,\rN_j} (1+ x_j)^{\rK_j-1} e^{ x_j}\\
     & | \cL^{\rK_j} f_j ( x_j)|\leq M_{\rK_j,\rN_j} (1+ x_j)^{\rN_j -1} e^{ \frac{x_j}2}.
  \end{align*}
  We let $\rK = \sum_{j=1}^n \rK_j$, $\rN = \sum_{j=1}^n \rN_j$ and $M_{\rK,\rN}^\times := \prod_{j=1}^n M_{\rK_j,\rN_j}$.
\end{enumerate}
\end{nota}

\begin{rem}\emph{Comparison estimates} such as \eqref{e:comparisonZ}, comparing $\sum_{i=1}^n x_i$
  with $h( x_1, \ldots, x_n)$, are crucial throughout the paper.
\end{rem}

\subsubsection{Estimates on the functions $F_j^{\rK_j}$, $G_j$, $H_j^t$ and $I_j^t$}
\label{sec:estim-funct-f_jk_j}

We can deduce from our hypotheses some bounds on the functions introduced in Notation
\ref{nota:FGHI}, which appear when computing $\cL^\rK_\ell(f_1 \star \ldots f_n|_\varphi^h)$. This
will allow us to foresee some key elements of the proof of Theorem \ref{t:intermediate}.

\begin{lem}
  \label{lem:themiracle}
  Under the assumptions \textbf{($h$)} and \textbf{($f$)} from Notation \ref{nota:assum_1}, for
  $j \in \{1, \ldots, n\}$, for any set $\pi_j\subset \{ 1, \ldots, n \}$ such that $j \notin \pi_j$
  and any $x \in [a, + \infty)^n$,
  \begin{itemize}
  \item if $\Phi_j = F^{\rK_j}_j, G_j$ or $H_j^t$ with $0 \leq t < \rK_j$, 
    \begin{equation}
      \label{e:themiracle}
      |\partial^{\pi_j}\Phi_j(x)|
      \leq   C_0^{2} M_{\rK_j,\rN_j} (1+ x_j)^{\rN_j} e^{\frac{x_j}{2}} \prod_{i \in
    \pi_j} e^{-x_i}
    \end{equation}
  \item if $\Phi_j = I_j^t$ with $0 \leq t < \rK_j$,  
    \begin{equation}
      \label{e:bI}
      |\partial^{\pi_j}\Phi_j(x)|
      \leq   C_0^{2} M_{\rK_j,\rN_j} (1+ x_j)^{\rN_j} e^{x_j} \prod_{i \in
    \pi_j} e^{-x_i}.
    \end{equation}
  \end{itemize}
\end{lem}

\begin{proof}
  If $\pi_j = \emptyset$, it is a straightforward computation that
  \begin{equation*}
    \begin{cases}
      |F^{\rK_j}_j(x)|
      \leq C_2 M_{\rK_j, \rN_j} (1+ x_j)^{\rN_j-1} e^{ \frac{x_j}2},\\
      |G_j(x)|
      \leq C_1 M_{\rK_j, \rN_j} (1+ x_j)^{\rK_j-1} ,\\
      |H^ t_j(x)|
      \leq C_1 C_2 M_{\rK_j, \rN_j} (1+ x_j)^{\rK_j}
    \end{cases}
  \end{equation*}
  and
    \begin{align*}
   |I^ t_j(x)|\leq C_2 M_{\rK_j, \rN_j}  (1+ x_j)^{\rK_j-1} e^{ x_j}.
  \end{align*}  
  We then use the fact that $\rK_j \leq \rN_j$ and $C_1,C_2 \leq C_0$. When $\pi_j$ is non-empty,
  the additional exponential factors come from the hypothesis on the derivatives of $h$.
\end{proof}
  
We here see the crucial phenomenon at play for the proof of Theorem \ref{t:intermediate}: compared
to the original upper bound $|f_j(x_j)|\leq M_{\rK_j, \rN_j} (1+ x_j)^{\rK_j-1} e^{x_j}$, we have
gained in \eqref{e:themiracle} a factor $e^{- x_j/2}$ when replacing the function $f_j$ by the
function $F^{\rK_j}_j$, $G_j$ or $H^ t_j$.
Unfortunately, the estimate we obtain on the function $I^ t_j$ does not constitute an improvement
compared to the initial growth-rate of $f_j$.

In this case however, an exponential factor $e^{- x_j/2}$ will be gained elsewhere, from the
derivatives of the other functions.  Remember that we have noticed in Remark \ref{rem:bad} that the
set of ``bad'' indices $B$ for which the function $I_j^t$ appears corresponds to the set of
derivatives that appear elsewhere in the integral. We have seen in the lemma above that such
derivatives will create a decay reducing the growth of $I_j^t$.

We obtain the following bound on the quantity
$\cL^{\rK} (f_1\star \ldots \star f_n \|^{h}_{ \varphi } )$.

\begin{prp} \label{p:finalupperbound}
  Under the assumptions \textbf{($h$)} and \textbf{($f$)} from Notation \ref{nota:assum_1},
  \begin{equation}
    \label{eq:bound_LK_int}
     | \cL^{\rK}  (f_1\star \ldots \star f_n \|^{h}_{  \varphi } )(\ell) |
     \leq   \sum_{\pi}
     \sum_{  0\leq t+t' \leq \rK }  | \cL_\ell^t \cP_\ell^{t'}
     \mathbf{Int}_{\pi}(\ell)|
\end{equation}
and  $| \mathbf{Int}_{\pi}(\ell)|$ is bounded above by
\begin{equation}
  \label{e:finalupperbound}
    C_0^{2n} M_{\rK,\rN}^\times
  \int_{\substack{h(x)=\ell}}  |\partial^{\pi_0} \varphi (x)|
  \prod_{j\in B\setminus \pi_0}   e^{-  x_j} 
  \prod_{j\in B}  (1+ x_j)^{\rN_j} e^{ x_j}
  \prod_{j\not\in B}  (1+ x_j)^{\rN_j} e^{ \frac{x_j}2}
  \frac{\d x}{\d\ell}.
\end{equation}  
\end{prp}

\begin{proof}
  The bound \eqref{eq:bound_LK_int} is a direct consequence of the expression \eqref{e:Tgood2}
  together with the triangle inequality. Recall that 
  \begin{align*}  \mathbf{Int}_{\pi}(\ell)= \int_{h(x)=\ell}
    \partial^{\pi_0} \varphi (x) \prod_{j=1}^n \partial^{\pi_j} \Phi_j(x)
    \, \frac{\d x}{\d\ell}
  \end{align*} 
  where $\pi_0, \pi_1, \ldots, \pi_n$ are disjoint subsets of $\{1,\ldots,n\}$ such that
  $j \notin \pi_j$, and $B = \bigsqcup_{j=0}^n \pi_j$. We then use Lemma \ref{lem:themiracle} on
  each term, remembering that $\Phi_j$ is in the first situation if $j \notin B$ and the second if
  $j \in B$. We conclude by noticing that
  $\prod_{j=1}^n \prod_{i \in \pi_j}e^{-x_i} = \prod_{j \in B \setminus \pi_0}e^{-x_j}$.
\end{proof}

\begin{rem}\label{r:help} The previous statement is not completely optimal, but we have decided to
  line up all the estimates on the worst possible case.  For later use, it will also be useful to
  lose again in optimality in \eqref{e:finalupperbound}, and write that:
  \begin{equation} \label{e:finalupperbound-} 
    \begin{split}
      % & \leq  C_0^{2n}  \prod_{j=1}^{n}M(\rK_j, \rN_j)  
      %   \int_{\substack{x \in [a, + \infty)^n \\ h(x)=\ell}}  |\partial^{\tilde \pi} \varphi (x)|
      %   \prod_{j\in B\setminus \pi_0}   e^{-  \frac{x_j}2}  \prod_{j\not\in B}  e^{-  \frac{x_j}2}  
      %   \prod_{j=1}^n    (1+ x_j)^{\rN_j} e^{ x_j}
      %   \frac{\d x}{\d\ell}
      % \\
      % &
| \mathbf{Int}_{\pi}(\ell)|
        \leq  C_0^{2n}  M^\times_{\rK, \rN}
        \int_{\substack{h(x)=\ell}}  |\partial^{\pi_0} \varphi (x)|
        \prod_{\substack{j \notin \pi_0}}
        e^{-  \frac{x_j}2}\prod_{j=1}^n (1+ x_j)^{\rN_j} e^{ x_j}
        \frac{\d x}{\d\ell}.
    \end{split}
  \end{equation} 
  If we compare with \eqref{e:finalupperbound}, we note that we have accepted to waste a factor
  $e^{-\frac{x_j}2}$ for $j\in B\setminus \pi_0$. The resulting upper bound only depends on $\pi_0$
  and not on $B$ (both coincide if $\pi_0 = B$).
\end{rem}

\subsubsection{End of the proof of Theorem \ref{t:intermediate}}
\label{sec:end-proof-theorem}

We are now ready to finish the proof of Theorem \ref{t:intermediate} by using Assumption
\textbf{($\varphi$)} to bound \eqref{e:finalupperbound}.
  
 \begin{proof}  
   Given that we now assume that
   $|\partial^{\pi_0} \varphi (x)|\leq C_3 e^{-\sum_{j\in \pi_0} x_{j}}$, we obtain that for any
   term~$\pi$ in \eqref{eq:bound_LK_int},
   \begin{equation}
     \label{eq:in_1}
     | \mathbf{Int}_{\pi}(\ell)| \leq C_3\,C_0^{2n}   M^\times_{\rK, \rN}
     \int_{ h(x)=\ell}
 \prod_{j=1 }^n  (1+ x_j)^{\rN_j} 
  e^{\frac{x_j}2 }
     \frac{\d x}{\d\ell}.
   \end{equation}
   We now explicit the level-set integration and use the comparison estimate
   \eqref{e:comparisonZ}. We obtain that
   \begin{equation*}
     | \mathbf{Int}_{\pi}(\ell)| \leq C_3\,C_0^{2n}   M^\times_{\rK, \rN}
     \int_{\sum_{j=2}^n x_j \leq \ell+\CE}
     \prod_{j=1}^n  (1+ x_j)^{\rN_j} 
     e^{\frac{x_j}2 } \Big|\frac{\partial x_1}{\partial h}\Big|
     \prod_{j=2}^n \d x_j
   \end{equation*}
   which, by assumption on the inverse derivatives of $h$, implies
   \begin{equation*}
     | \mathbf{Int}_{\pi}(\ell)| \leq C_3\,C_0^{2n} C_2'   M^\times_{\rK, \rN}
     (1+\ell+\CE)^{\rN} e^{\frac{\ell+\CE}{2}}
     \int_{\sum_{j=2}^n x_j \leq \ell+\CE}
\prod_{j=2}^n \d x_j. 
\end{equation*}
As a consequence,
   \begin{equation*}
     | \mathbf{Int}_{\pi}(\ell)| \leq C_3\,C_0^{2n+1}  M^\times_{\rK, \rN}
     (1+\ell+\CE)^{\rN+n-1} e^{\frac{\ell+\CE}{2}}.
\end{equation*}
   
   We have thus proven that the functions $\mathbf{Int}_{\pi}$ are in the class $\cR^{\rN+n}$.  From
   Proposition~\ref{p:finalupperbound},
   $\cL_\ell^{\rK}( f_1\star \ldots \star f_n |^{h}_{\varphi})$ is a sum of such functions
   (or images thereof under powers of $\cP_\ell$ or $\cL_\ell$), hence
   $\cL_\ell^{\rK}( f_1\star \ldots \star f_n|^{h}_{\varphi }) \in \cR^{\rN+n}$.  Proposition
   \ref{p:charFR} then implies that $f_1\star \ldots \star f_n|^{h}_{\varphi }$ belongs in
   $\FR^{\rK, \rN+n}$. The proof furthermore gives
   \begin{align}\label{e:controlnorm}\norm{f_1\star \ldots \star f_n|^{h}_{\varphi}}_{\cF^{\rK,
     \rN+n}}
     \leq C_3\,C_0^{2n+1}\fn(n, \rK) (1+\CE)^{\rN+n-1}
     e^{\frac{\CE}2}\prod_{j=1}^n \norm{f_j}_{\cF^{\rK_j, \rN_j}}.
   \end{align}
 
 \end{proof} 

 \subsection{Problems in the geometric application of Assumption \textbf{($h$)} and
   \textbf{($\varphi$)}}
 \label{s:problems}

 Having noticed at the end of \S \ref{s:noncross} that the integral \eqref{e:int_ell} is the
 $(h, \varphi_0)$-convolution of Friedman--Ramanujan functions (with $\varphi_0$ and $h$ specified
 in \eqref{e:spec_phi} and \eqref{e:spec_h} respectively), we now need to examine the possibility to
 apply Theorem \ref{t:intermediate}. Unfortunately, the functions $h$ and $\varphi_0$ do not satisfy
 the required assumptions, which explains why we will need two additional sections \S
 \ref{s:lengthc} and \S \ref{s:lengthboundary} to solve these issues.

 \subsubsection{Problem with Assumption \textbf{($h$)}} The explicit expression of
 $\ell_Y({\mathbf{ c}})$ was obtained in \eqref{e:outcome2}.  By virtue of Corollary \ref{r:argcosh}
 and Remark \ref{r:partition}, for each parameter $\vec{L}$, we have that the function
 $$h : \vec{\theta} \mapsto \ell_Y({\mathbf{ c}}) $$
 lies in the set $\cE^{(a)}_{2r}$ {\em{provided we restrict all the $\theta_q$ to be $\geq a$}}.  If
 we now fix the neutral parameters $\thetaNe \in \R_{\geq 0}^{\ThetaNe}$ and consider
 \begin{align}
   h: \thetaAc \mapsto  \ell_Y({\mathbf{c}}),
\end{align}
then it is easy to check that the restriction of this function to $\R_{\geq a}^{\ThetaAc}$ belongs
to $\cE^{(a)}_{\# \ThetaAc}$.

We have already remarked that the non-crossing variables $\theta_q$ stay positive, as consequence of
the Useful Remark \ref{r:useful}. As a consequence, in the purely non-crossing case (when all
variables are non-crossing), Assumption \textbf{($h$)} is satisfied.

If there exists some crossing variables, then they may take positive or negative values, in which
case we are out of the domain of application of Corollary \ref{r:argcosh}.

{\em{$\rightsquigarrow$Section \ref{s:lengthc} will be devoted to solving this issue by choosing a
    different representative of the homotopy class of ${\mathbf{c}}$}}. This will provide an
expression of $h$ which falls in a class $\cE$ with respect to newly defined variables. The reader
might wish to focus upon first read on the purely non-crossing case, and hence skip \S
\ref{s:fam_aligned} and \S \ref{s:cross_int}.
  
\subsubsection{Problem with Assumption \textbf{($\varphi$)}} The coefficients arising in the
expression of $Q_{m(\lambda)}$, obtained in \eqref{e:Bn} and \eqref{e:outcome}, are equal to
$\pm 1$, thus not of constant sign.  As a consequence, we cannot use Proposition \ref{p:mainclass}
nor Corollary \ref{r:argcosh} to deduce that $Q_{m(\lambda)}$ belongs to the class
$\cE_{m(\lambda)}$ with respect to the variables
$(\tau^\lambda_1, \ldots, \tau^\lambda_{m(\lambda)})$. We cannot either deduce that the functions
$\log F_{m(\lambda)}$ are in the class $E_{m(\lambda)}$.
  
 {\em{$\rightsquigarrow$  Section \S \ref{s:lengthboundary} will be devoted to estimating the derivatives of $F_{m(\lambda)}$.}}
 Again this will necessitate changing the representative of the homotopy class of $\Gamma_\lambda$.
%If some negative terms in \eqref{e:Bn} are not negligible compared to the positive ones, we use this information to choose a better representative of the homotopy class of
%  $\cG_{\lambda}$, giving a new expression of $\ell_Y(\cG_{\lambda})$, for which the negative terms are negligible.

 \subsection{Variants of Theorem \ref{t:intermediate}}
\label{sec:vari-theor-reft}

In the geometric applications of these techniques, we shall not use Theorem \ref{t:intermediate}
directly, but several variants.  Because these variants become quite technical, we present them by
gradually modifying the assumptions and conclusions. The reader is invited to skip this section at
first read, and come back to it when the necessity for these variants becomes clear.

\subsubsection{First variant: comparison estimates}
\label{sec:first-vari-comp}

First, we loosen the decay hypothesis on $\varphi$, and replace it by a new assumption. 

\begin{nota}[Assumption \textbf{($\varphi$-CE)}]
  \label{nota:phi_CE}
  Assume that the smooth function $\varphi$ is supported on $(a, + \infty)^n$ and there exists
  positive functions $Z_j : \R_{>0} \rightarrow \R_{>0}$, $1 \leq j \leq n$, such that:
  \begin{itemize}
  \item for all $\pi_0\subseteq \{1, \ldots, n\}$, all $x = (x_1, \ldots, x_n) \in (a, + \infty)^n$,
    $$ |\partial^{\pi_0} \varphi (x)| \leq C_3\, \prod_{j\in \pi_0} \frac{1}{Z_j^{1/2}(x_j)};$$
  \item there exists $\CE>0$ satisfying \eqref{e:comparisonZ} and such that, for any
    $\pi_0\subseteq \{1, \ldots, n\}$,
    \begin{align} \label{e:comparisonbaby}2\sum_{j=1}^n x_j\leq h(x) +\sum_{j\not \in \pi_0} x_j
      +\sum_{j\in \pi_0} \log Z_j +\CE.\end{align}
  \end{itemize}
  \end{nota}

  The important point, here, is that the functions $(Z_j)_{1 \leq j \leq n}$ controlling the decay
  of the derivatives of $\varphi$ can be incorporated in a comparison estimate for the function $h$.

  \begin{rem}
    The initial result, Theorem \ref{t:intermediate}, may be recovered by taking
    $Z_j=e^{ x_j}$, in which case the comparison estimate \eqref{e:comparisonbaby} coincides with \eqref{e:comparisonZ}.
  \end{rem}
  
  \begin{thm} \label{p:yet} Let $a>0$.  We assume that the assumptions \textbf{($h$)},
    \textbf{($f$)} and \textbf{($\varphi$-CE)} hold. Then,
    $ f_1\star \ldots \star f_n |^h_\varphi\in \FR^{\rK, \rN+n}$, and the estimate
    \eqref{e:controlnorm} holds.
\end{thm}

\begin{proof}
  We do not reproduce the proof, completely identical to the proof of Theorem \ref{t:intermediate}.
  A crucial point is again the {{comparison estimate}} \eqref{e:comparisonbaby}, which may be
  rewritten on the level-set $\{h(x)=\ell\}$ in exponential form
  \begin{align*} \prod_{j=1}^n e^{x_j} \prod_{j\in \pi_0} \frac{1}{Z_j ^{1/2}} \prod_{j\not \in
      \pi_0} e^{-\frac{x_j}2}\leq e^{\frac{\CE}2}e^{\frac{\ell}2} .\end{align*} This means that the
  gain $ \prod_{j\in \pi_0} Z_j ^{-1/2} \prod_{j\not \in \pi_0} e^{-x_j/2}$ on the left-hand side
  results in a growth of size $e^{\ell/2}$ instead of the expected $e^\ell$ on the right-hand side,
  leading to the fact that the resulting function is a Friedman--Ramanujan remainder.
  % The analogous property in our geometric application is stated as Proposition \ref{p:comparison},
  % and its proof occupies the whole \S \ref{s:horrible}.
\end{proof}

\subsubsection{Second variant: Linear forms} \label{s:variant} In our geometric applications,
$\varphi(x_1, \ldots, x_n)$ will actually be split into a product involving linear forms, as
follows:
\begin{nota}
  Let $\varphi$ be of the form:
  \begin{align}\label{e:special}
    \varphi(x_1, \ldots, x_n)=\phi(\tau_1, \ldots, \tau_m) \, \psi(x_1, \ldots, x_n)
  \end{align}
  where, for all $1 \leq i \leq m$, $\tau_i$ is a linear form, and more precisely
  $\tau_i(x)=\sum_{j\in \Theta(i)} x_j$ for a subset $\Theta(i)\subseteq \{1, \ldots, n\}$. For
  $\tilde\pi \subseteq \{1, \ldots, m\}$, we denote
  $\Theta(\tilde\pi) := \bigcup_{i \in \tilde\pi} \Theta(i)$.
\end{nota}
  \begin{rem}
    The original result corresponds to the case where $m=n$, $\Theta(i)=\{i\}$ so that $\tau_i(x)=x_i$,
    and $\psi=1$.
  \end{rem}
 
  Let us adapt Theorem \ref{t:thebigone} to this new form of the density $\varphi$. The motivation
  for this upcoming variant may be roughly explained as follows: suppose that the derivatives
  $\partial^{ \tilde\pi} \phi ( \tau_1, \ldots, \tau_m)$ are bounded by
  \begin{align}\label{e:derphi}|\partial^{\tilde\pi} \phi ( \tau_1, \ldots,  \tau_m)|
    \leq C_3\, e^{-\sum_{i\in  \tilde\pi}  \tau_{i}} = C_3\,e^{-\sum_{i\in  \tilde\pi} \sum_{j\in \Theta(i)} x_j}\end{align}
  for any $\tilde\pi \subseteq \{1, \ldots, m\}$.
  Then, once the term $\partial^{\tilde\pi} \phi$ appears in the application of
  Corollary~\ref{c:firstconv}, there is no further need to gain exponential decay with respect to the variables $x_j$ with $j\in \Theta(\tilde\pi) $.
  We wish to avoid this extra work, not only to optimize the proof, but also because in the geometric application the higher order derivatives of $\phi$ will be difficult to calculate. Thus, we copy the proof of Theorem \ref{t:thebigone}, but once
  $\phi$ has been differentiated with respect to a variable $x_j$ (say, for $j\in \Theta(i)$), we
  stop applying our operators to the other variables $x_{j'}$ with $j'\in \Theta(i)$.
  This leads to the following statement.

  \begin{thm}\label{t:thebigone'}
    For any family of continuous functions $(f_j)_{1 \leq j \leq n}$ and any integers
    $(\rK_j)_{1 \leq j \leq n}$ with $\rK_j \geq 1$, if $\rK=\sum_{j=1}^n \rK_j$, then the function
    $\cL_\ell^{\rK}( f_1\star \ldots \star f_n|^{h}_{\varphi })$ can be expressed as the sum of the
    following functions, or their images under $\cL_\ell^{t}\cP_\ell^{t'}$ for $0\leq t\leq \rK-n$
    and $0\leq t'\leq n$,
 \begin{align}\label{e:Tgood4}
   \mathbf{Int}_{\tilde \pi_0, \tilde\pi,\cV, \pi}(\ell)
   :=   \int_{\substack{h(x)=\ell}} \partial^{\tilde \pi} \phi (\tau_1, \ldots, \tau_m) \,
   \partial^{\pi_0}\psi(x)
   \prod_{j \in \cV} \partial^{\pi_j}  \Phi_j(x)
   \prod_{j \notin \cV} f_j(x_j)
   \frac{\d x}{\d\ell}
 \end{align} 
 where:
 \begin{itemize}
 \item $\tilde \pi_0 \subseteq \{1, \ldots, n\}$ and $\tilde \pi \subseteq \{1, \ldots, m\}$ are two
   sets of same cardinal, with a bijection $\tilde\pi_0 \ni j \mapsto i_j \in \tilde\pi$ such that
   $j \in \Theta(i_j)$ for all $j \in \tilde\pi_0$;
 \item $\cV$ is a subset of $\{1, \ldots, n\}$ containing $\tilde{\pi}_0$, and so that
   $\cV \cup \Theta(\tilde\pi)=\{1, \ldots, n\}$;
 \item $(\pi_j)_{j \in \cV \cup \{ 0 \}}$ is a family of disjoint subsets of
   $\cV \setminus \tilde\pi_0$ such that $j \notin \pi_j$ for all $j$;
 \item if $B := \tilde\pi_0 \sqcup \pi_0 \sqcup \bigsqcup_{j\in \cV} \pi_j$, then for any $j \in
   \cV$,  $\Phi_j$ satisfies \eqref{e:thebigphi}.
 \end{itemize}
  \end{thm}

\begin{rem} 
  To go from Theorem \ref{t:thebigone'} to Theorem \ref{t:thebigone}, let
  $\tilde\pi_0=\tilde\pi=\emptyset$, $\cV = \{1, \ldots, n\}$.
\end{rem} 

We prove, once again, this result by proving the following (stronger) lemma for by induction on the
integer $0 \leq k \leq n$. This leads to the claim taking $k=n$.

  \begin{lem}
    For any $0 \leq k \leq n$, the function
    $\cL_\ell^{\rK}( f_1\star \ldots \star f_n|^{h}_{\varphi })$ can be expressed as the sum of the
    following functions, or their images under $\cL_\ell^{t}\cP_\ell^{t'}$ for $0\leq t\leq \rK-k$
    and $0\leq t'\leq k$,
 \begin{align}\label{e:Tgood4'}
   \mathbf{Int}_{\tilde \pi_0, \tilde\pi,\cV, \pi}^k(\ell)
   :=   \int_{\substack{h(x)=\ell}} \partial^{\tilde \pi} \phi (\tau_1, \ldots, \tau_m) \,
   \partial^{\pi_0}\psi(x)
   \prod_{j \in \cV} \partial^{\pi_j}  \Phi_j(x)
   \prod_{j \notin \cV} f_j(x_j)
   \frac{\d x}{\d\ell}
 \end{align} 
 where:
 \begin{itemize}
 \item $\tilde \pi_0 \subseteq \{1, \ldots, k\}$ and $\tilde \pi \subseteq \{1, \ldots, m\}$ are two
   sets of same cardinal, with a bijection $\tilde\pi_0 \ni j \mapsto i_j \in \tilde\pi$ such that
   $j \in \Theta(i_j)$ for all $j \in \tilde\pi_0$;
  \item $\cV$ is a subset of $\{1, \ldots, k\}$ containing $\tilde{\pi}_0$, and so that
   $\{1, \ldots, k\} \subseteq \cV \cup \Theta(\tilde\pi)$;
 \item $(\pi_j)_{j \in \cV \cup \{ 0 \}}$ is a family of disjoint subsets of
   $\cV \setminus \tilde\pi_0$ such that $\pi_j$ contains only integers $\geq j+1$ for all $j$;
 \item if $B := \tilde\pi_0 \sqcup \pi_0 \sqcup \bigsqcup_{j\in \cV} \pi_j$, then for any $j \in
   \cV$,  $\Phi_j$ satisfies \eqref{e:thebigphi}.
 \end{itemize}
  \end{lem}

  \begin{proof}
    For $k=0$, we take $\tilde\pi=\tilde\pi_0=\pi_0=\cV=\emptyset$. Now assume the formula known at
    rank $k$ and let us prove it for $k+1$, by applying $\cL^{\rK_{k+1}}$ to an expression of the
    form \eqref{e:Tgood4'}.

    First, let us observe that, if $k+1 \in \Theta(\tilde\pi)$, then we can simply apply
    $\cL^{\rK_{k+1}}$ to the integral \eqref{e:Tgood4} and leave it unchanged; the result will
    satisfy the requirements of the theorem at the rank $k+1$, with the same $\tilde{\pi}_0$,
    $\tilde\pi$, $\cV$, $(\pi_j)_{j \in \cV}$.

    Let us now assume that $k+1 \notin \Theta(\tilde\pi)$. As before, we apply Corollary
    \ref{c:firstconv} to \eqref{e:Tgood4'} with respect to the variable $x_{{k+1}}$. For each term,
    we explicit how the data $(\tilde \pi_0, \tilde \pi, \cV, (\pi_j)_{j \in \cV})$ evolves to a new
    set of data $(\tilde \pi_0', \tilde \pi', \cV', (\pi_j')_{j \in \cV'})$ which satifies the
    conditions at the next rank. Note that, here, we will always take $\cV' = \cV \cup \{k+1\}$ and
    $\pi_{k+1}'=\emptyset$.

    In the three first cases, $T_{k+1}$ is a multiplication operator by some function
    $\Phi_{k+1}$ equal to $F^{\rK_{{k+1}}}_{{k+1}}$, $G_{{k+1}}$ or $H_{{k+1}}^t$. We take  $\tilde
    \pi_0' = \tilde\pi_0$, $\tilde \pi' = \tilde \pi$, and $\pi_j'=\pi_j$ for
    $j \in \cV$.
    
    Otherwise, $T_{k+1}= I_{{k+1}}^t \partial_{{k+1}} $. In this case, we let
    $\Phi_{{{k+1}}}=I_{{k+1}}^t$.  When we apply $\partial_{k+1}$ to an expression of the form
    $ \partial^{\tilde \pi} \phi (\tau_1, \ldots, \tau_m) \, \partial^{\pi_0}\psi(x) \prod_{j \in
      \cV} \partial^{\pi_j} \Phi_{j}(x)$ and then multiply it by
    $I_{k+1}^t = \partial^{\pi_{k+1}'}\Phi_{k+1}$, by the Leibniz rule we obtain three types of
    terms.
    \begin{itemize}
    \item Terms of the form 
      \begin{align}\label{e:vk+1}
         \partial^{\tilde \pi} \partial_{{i}}\phi ( \tau_1, \ldots,  \tau_m) \,
        \partial^{\pi_0}\psi(x)
        \prod_{j \in \cV'} \partial^{\pi_j}  \Phi_{j}(x) 
      \end{align}
      for $i\in \{1, \ldots, m\}$ with $k+1 \in \Theta(i)$. We note that by hypothesis
      $i \notin \tilde\pi$. We take $\tilde\pi_0'=\tilde\pi_0\cup\{k+1\}$,
      $\tilde\pi'=\tilde\pi \cup \{i\}$, and leave the $(\pi_j)_{j \in \cV}$ unchanged.
    \item A term of the form
      \begin{align}\label{e:vk+2}
        \partial^{\tilde \pi} \phi ( \tau_1, \ldots,  \tau_m) \;
        \partial^{\pi_0}\partial_{{k+1}}\psi(x)
        \prod_{j \in \cV'} \partial^{\pi_j}  \Phi_{j} (x)
      \end{align}
      for which we set $\tilde\pi_0'=\tilde\pi_0$, $\tilde\pi'=\tilde\pi$, $\pi_0'=\pi_0\cup
      \{k+1\}$, and leave the $(\pi_j)_{j \in \cV}$ unchanged.
    \item Finally, a term of the form
      \begin{align}\label{e:vk+11}
        \partial^{\tilde \pi} \phi ( \tau_1, \ldots,  \tau_m) \;
        \partial^{\pi_0}\psi(x) \;
        \partial^{{k+1}} \Big(\prod_{j \in \cV} \partial^{\pi_j}  \Phi_{j}\Big)(x) \Phi_{k+1}(x) 
      \end{align}
      which is itself a sum of terms coming from re-applying the Leibniz rule. For each of these
      terms, we leave $\tilde\pi_0$, $\tilde\pi$ and $\pi_0$ unchanged, add $k+1$ to exactly one of
      the $\pi_j$ with $j \in \cV$ (corresponding to the index which is hit by the Leibniz rule),
      leaving the others unchanged.
    \end{itemize}
 \end{proof}

 We deduce from Theorem \ref{t:thebigone'} the following proposition replacing Proposition
 \ref{p:finalupperbound}:
 
\begin{lem}\label{p:finalupperbound'} 
  Under the assumptions \textbf{($h$)} and \textbf{($f$)}, 
  \begin{align*}
    |\cL^{\rK} (f_1\star \ldots \star f_n|^{h}_{  \varphi })(\ell)|
    \leq   \sum_{\tilde \pi_0, \tilde\pi,\cV, \pi}
    \sum_{  0\leq t+t' \leq \rK +n}  | \cL^t \cP^{t'}
\mathbf{Int}_{\tilde \pi_0, \tilde\pi,\cV, \pi}(\ell)|
\end{align*}
and $|\mathbf{Int}_{\tilde \pi_0, \tilde\pi,\cV, \pi}(\ell)|$ can be bounded above by
\begin{align*}
 C_0^{2n} M^\times_{\rK, \rN}
  \int_{\substack{h(x)=\ell}} |\partial^{\tilde \pi} \phi (\tau_1, \ldots, \tau_m) |
  |\partial^{\pi_0}\psi(x)|
  \prod_{j \in \cV \setminus (\tilde\pi_0\cup\pi_0)} e^{-\frac{x_j}{2}}
  \prod_{j =1}^n (1+x_j)^{\rN_j}e^{x_j}
        \frac{\d x}{\d\ell}. 
\end{align*}
\end{lem}

The following proposition inputs an assumption about $\varphi$ to find upper bounds on the terms
appearing in Lemma \ref{p:finalupperbound'}.

\begin{nota}[Assumption \textbf{($\varphi$-CE-form)}] \label{nota:CEform}
  Assume that $\varphi(x) = \phi(\tau_1, \ldots, \tau_m)\psi(x)$ where:
  \begin{itemize}
  \item $\psi$ satisfies Assumption \textbf{($\varphi$)}, namely it is supported in $(a, +
    \infty)^n$ and there exists a constant $C_3$ such that, for any $\pi \subseteq \{1, \ldots, n\}$,
    $$\sup_{x \in (a, +\infty)^n} e^{\sum_{j \in \pi}x_j}|\partial^\pi \psi|  \leq C_3.$$
  \item $\phi$ satisfies Assumption \textbf{($\varphi$-CE)} w.r.t. its variables
    $(\tau_i)_{1 \leq i \leq m}$, i.e. there exists positive functions
    $Z_i : \R_{>0} \rightarrow \R_{>0}$, $1 \leq i\leq m$, such that:
  \begin{itemize}
  \item for all $\tilde \pi\subseteq \{1, \ldots, m\}$,
    \begin{align}\label{e:partial_phi} |\partial^{\tilde\pi} \phi (\tau_1, \ldots, \tau_m)| \leq
      C_4\,
      \prod_{i\in \tilde\pi} \frac{1}{Z_i^{1/2}(\tau_i)};
    \end{align}
  \item there exists a function $\cH(x)$ such that for all $\tilde \pi\subseteq \{1, \ldots, m\}$,
    we have the comparison estimate
    \begin{align} \label{e:comparison}2\cH(x)\leq h(x) 
      +\sum_{j\not \in \Theta(\tilde\pi)} x_j + \sum_{i\in \tilde \pi} \log Z_i(\tau_i).
    \end{align}
  \end{itemize}
  \end{itemize}
\end{nota}

\begin{rem}
  On the level-set $\{h(x)=\ell\}$, the comparison estimate can be rewritten under the
  multiplicative form
  \begin{align}
    \prod_{j\not \in \Theta(\tilde\pi)} e^{-\frac{x_j}{2}}
    \prod_{i\in \tilde \pi} \frac{1}{Z_i^{1/2}(\tau_i)}
    \leq e^{\frac \ell 2} e^{-\cH(x)}.      
    \end{align}
\end{rem}

\begin{prp} \label{l:yet2} Assume that \textbf{($f$)}, \textbf{($h$)} from Notation
  \ref{nota:assum_1}  and Assumption
  \textbf{($\varphi$-CE-form)} are satisfied. Then a term of the form \eqref{e:Tgood4} may be bounded by
  \begin{align*}  C_0^{2n}C_3C_4\, M^\times_{\rK, \rN}
    \prod_{j=1}^n \|f_j\|_{\cF^{\rK^j,\rN^j}}
    (1+ \ell+\CE)^{\rN} e^{\frac{\ell}2}
&  \int_{h(x)=\ell}  e^{-\cH(x)}
  e^{ \sum_{j=1}^n x_j}
      \frac{\d x}{\d\ell}.
       \end{align*}  
 \end{prp}  

As a consequence, in order to prove the Friedman-Ramanujan property, it will be enough to bound the
terms appearing in Proposition \ref{l:yet2}.

%%%Local Variables: 
%%% mode: latex
%%% TeX-master: "main"
%%% End: 

\section{Erasing of crossing parameters in the length of $\mathbf{c}$}
\label{s:lengthc}
In this section, a hyperbolic metric $Y$ on $\mathbf{S}$ is fixed.  We consider our multi-loop
${\mathbf{ c}} =({\mathbf{c}}_1, \ldots, {\mathbf{c}}_{\cc})$.  The section is devoted to solving
the first issue raised in \S \ref{s:problems}.
 
We recall that, for $q\in { {\Theta}}$, if $\sign(q)\sign(\sigma q)=-1$, we call $q$ a U-turn index,
and otherwise a crossing index. We denote as $\mathcal{C} \subseteq \Theta$ the set of crossing
indices. We shall solve issues related to situations when some crossing parameters $\theta_q$, with
$q \in \mathcal{C}$, assume negative values.

 \subsection{Idea and organization of this section}
 
 An example of a crossing variable is represented in Figure \ref{fig:crossing_ex}. In Figure
 \ref{fig:crossing_ex_pos}, the crossing parameter $\theta_q$ is positive, and the length-function
 $\ell_Y(\mathbf{c})$ belongs to an appropriate class $\mathcal{E}$, which means that the argument
 we have set up in \S \ref{s:GC} can be used to establish Theorem \ref{thm:FR_type}. This is not the
 case in Figure \ref{fig:crossing_ex_neg} unfortunately, when $\theta_q \leq 0$. 

 \begin{figure}[h!]
   \centering
   \begin{subfigure}[b]{0.5\textwidth}
    \centering
     \includegraphics[scale=0.9]{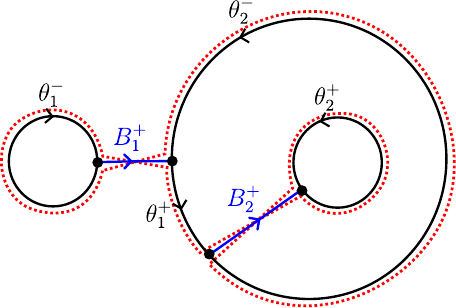}
    \caption{When $\theta_1^+ >0$.}
    \label{fig:crossing_ex_pos}
  \end{subfigure}%
   \centering
   \begin{subfigure}[b]{0.5\textwidth}
    \centering
     \includegraphics[scale=0.95]{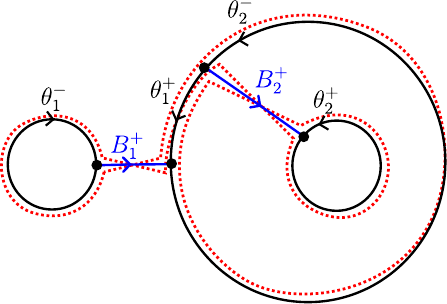}
    \caption{When $\theta_1^+ <0$.}
    \label{fig:crossing_ex_neg}
  \end{subfigure}%
  \\
  \caption{A simple example with a crossing parameter $\theta_1^+$. The dotted curve is the element
    of the homotopy class of $\mathbf{c}_i$ constructed in \S \ref{s:nc} to express
    $\ell_Y(\mathbf{c}_i)$.}
   \label{fig:crossing_ex}
 \end{figure}
 
 In order to solve this issue, depending on the metric $Y$, we shall find a new representative of
 the homotopy class of a component ${\mathbf{c}}_i$ of $\mathbf{c}$, and introduce new
 \emph{positive} variables $\tilde \theta_q$, such that the expression of $\ell_Y({\mathbf{c}}_i)$
 in those new variables now belongs to the desired class of functions $\cE$. The idea is that, when
 a crossing variable is negative, we can rectify the problem by a basic move: if
 $\sign(q)\sign(\sigma q)=+1$ and $\theta_q<0$, the \emph{basic move} consists in replacing
 $\overline{B}_{\sigma^{-1} q, q}\smallbullet \overline{\cI}_q \smallbullet \overline{B}_{q, \sigma
   q}$ by the homotopic orthogeodesic $\overline{B}_{ \sigma^{-1} q, \sigma q}$ from
 $\beta_{\sigma^{-1}q}$ to $\beta_{\sigma q}$. The segment $\overline{\cI}_q$ is then said to be
 \emph{erased}, and the variable $\theta_q$ is an erased variable.  On Figure \ref{fig:basicm} we
 show a basic move in the case $\sign(q)=\sign(\sigma q)=+1$.  We then perform these basic moves
 iteratively to get rid of all negative crossing variables.

 \begin{figure}[h!]
   \includegraphics{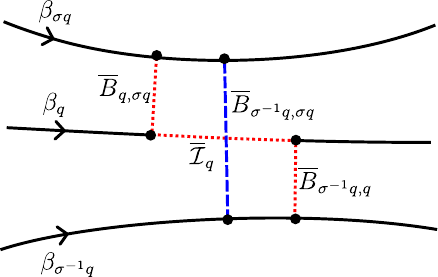}
   \caption{The basic move consists in replacing the dotted path with the dashed line in the
     homotopy class of $\mathbf{c}_i$.}
   \label{fig:basicm}
 \end{figure}%

 We start this section by studying families of aligned geodesics in $\IH^2$ and their
 orthogeodesics. We prove a key lemma called the \emph{staircase inequality} in \S
 \ref{s:staircase}. We then construct the erasing procedure in \S \ref{s:erasing}; see Proposition
 \ref{p:pluscross}. Once we lift our closed geodesic $\mathbf{c}_i$, loops $\beta$ and bars $B$ to
 the hyperbolic plane $\IH^2$ as done in \S \ref{s:lift}, we can apply Proposition~\ref{p:pluscross}
 to them.  This yields a partition of the Teichm\"uller space, as well as a new set of parameters
 $\tilde\theta_q$ on each element of this partition, as detailed in \S \ref{s:cross_int}. The new
 length $\ell_Y(\mathbf{c}_i)$ is expressed in these new parameters in \S \ref{s:newexp}. We then
 introduce a partition of unity adapted to this construction in \S \ref{s:cutoffQ}.

 For readers who decide to focus on the purely non-crossing case, it suffices to read \S
  \ref{s:purelync}-\ref{s:cutoffQ}.
 
\subsection{Families of aligned geodesics and their orthogeodesics}
\label{s:fam_aligned}

Let us group together a few key observations about families of aligned geodesics and their
orthogeodesics, which will be key to this section.
In this whole subsection, we shall consider the following set-up.  For an integer $M \geq 2$,
$\gamma_0, \ldots, \gamma_{M}$ are $M+1$ oriented geodesics in the hyperbolic plane. We assume
that they are pairwise aligned, and that $\gamma_{i+1}$ is on the left of $\gamma_i$. We define
$\overline{B}_{ij}$ to be the oriented orthogeodesics from $\gamma_i$ to $\gamma_j$, of length
denoted by $L_{ij}>0 $, and $z_{ij}\in \gamma_j$ (resp. $z_{ji}\in \gamma_i$) the endpoint
(resp. origin) of $\overline{B}_{ij}$.  

\subsubsection{The $3$-geodesic lemma}

First, we prove the following elementary observation.

\begin{lem}
  \label{lem:lemma_hexa}
  Let $j \in \{1, \ldots, M-1\}$. Then,
  \begin{equation}
    \label{eq:signs_hexa}
    \Dist(z_{j,j+1},z_{j-1,j+1}) \geq 0
    \quad \Leftrightarrow \quad
    \Dist(z_{j+1,j},z_{j-1,j}) \geq 0
    \quad \Leftrightarrow \quad
    \Dist(z_{j+1,j-1},z_{j,j-1}) \geq 0
  \end{equation}
  and, in terms of absolute distances,
  \begin{equation}
    \label{eq:ineq_hexa}
    |\Dist(z_{j+1,j},z_{j-1,j})|
    > |\Dist(z_{j+1,j-1},z_{j,j-1})| + |\Dist(z_{j,j+1},z_{j-1,j+1})|.
  \end{equation}
\end{lem}

\begin{proof}
  These results stem from the elementary study of the polygon of consecutive vertices $z_{j-1,j+1}$,
  $z_{j,j+1}$, $z_{j+1,j}$, $z_{j-1,j}$, $z_{j,j-1}$, $z_{j+1,j-1}$ represented in Figure
  \ref{fig:lemma_hexa}.
  \begin{figure}[h!]
    \centering
    \includegraphics{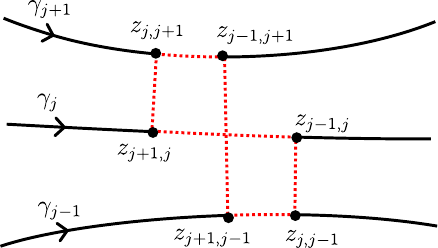}
    \caption{Illustration of the proof of Lemma \ref{lem:lemma_hexa}.}
    \label{fig:lemma_hexa}
  \end{figure}
  Equation~\eqref{eq:signs_hexa} is obtained by observing that there exists no
  hyperbolic hexagon with $5$ inner angles equal to $\pi/2$ and one equal to $3\pi/2$, and hence the
  polygon is a self-intersecting right-angled hexagon. Then, \eqref{eq:ineq_hexa} is a direct
  consequence of the classic trigonometric formula for self-intersecting right-angled hexagons of
  consecutive side-lengths $(\ell_i)_{1 \leq i \leq 6}$ (see e.g. \cite[Theorem 2.4.4]{buser1992}):
  \begin{align*}
    \cosh (\ell_3)
    =  \cosh (\ell_1) \cosh (\ell_5) + \cosh(\ell_6) \sinh (\ell_1) \sinh (\ell_5) 
     > \cosh (\ell_1+\ell_5).
  \end{align*}
\end{proof}

 \subsubsection{The staircase inequality}
 \label{s:staircase}
 
 Let us now prove an inequality which we will refer to as the staircase inequality,
 illustrated in Figure \ref{fig:magick}.

 \begin{figure}[h!]
   \includegraphics{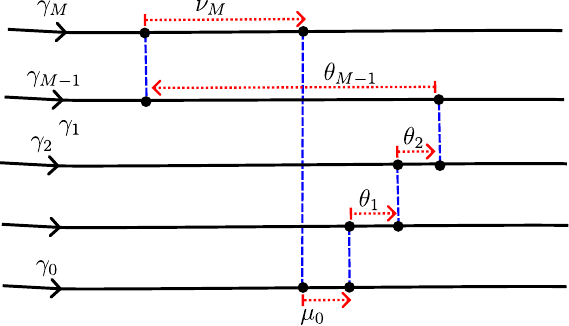}
   \caption{The geometric situation for the staircase inequality.}
    \label{fig:magick}
  \end{figure}
 
  \begin{lem} \label{l:magicstaircase} Let $M\geq 2$.
    We introduce the distances
  \begin{equation*}
    \theta_j=\Dist(z_{j-1, j}, z_{j+1, j}), \quad
    \mu_{j}=\Dist(z_{M, j}, z_{j+1, j}), \quad \nu_{M}=\Dist(z_{M-1, M}, z_{0, M}).
  \end{equation*}
  Assume that $\theta_{j}\geq 0$ for $1 \leq j \leq M-2$, that $\theta_{M-1}<0$, and that $\mu_{j}>0$
  for $0 \leq j \leq M-2$. Then,
    \begin{align}\label{e:magicineq}
      |\theta_{M-1}| >\sum_{j=1}^{M-2}\theta_{j} +\mu_{0} + \nu_{M}.
    \end{align}
\end{lem}

\begin{rem}
  The inequality \eqref{e:magicineq} can also be written in the form
  \begin{align}\label{e:theta-sum}
    - \sum_{j=1}^{M-1} \theta_{j} >  \mu_{0} + \nu_{M}.
  \end{align}
  % or, with the notation $\Dist$,
  % \begin{align*}
  %   \Dist( z_{N, N-1}, z_{N-2, N-1}) > \sum_{j=1}^{N-3} \Dist(z_{j, j+1}, z_{j+2, j+1}) + \Dist(z_{N, 1}, z_{2, 1})+\Dist(z_{N-1, N}, z_{1, N}).
  % \end{align*}
\end{rem}

\begin{proof}
  We apply Lemma \ref{lem:lemma_hexa} to the geodesics $\gamma_{j-1}$, $\gamma_{j}$ and
  $\gamma_{M}$. We obtain that
  \begin{align}
    \label{e:ineq'}
    u_j := \Dist(z_{M, j}, z_{j-1, j} )
    > \Dist(z_{M, j-1}, z_{j, j-1}) + \Dist(z_{j, M}, z_{j-1, M}),
  \end{align}
  all of these quantities being positive thanks to the hypothesis $\mu_{j-1}>0$ and
  \eqref{eq:signs_hexa}.

  For $2\leq j\leq M$, going along $\gamma_{j-1}$, we have the identity
  \begin{align*}\Dist(z_{M, j-1},z_{j, j-1})
    = \Dist(z_{M, j-1}, z_{j-2, j-1})+ \Dist(z_{ j-2, j-1}, z_{j, j-1})
    = u_{j-1}+\theta_{j-1},
  \end{align*}
  and for $j=1$,
  $\Dist(z_{M, 0}, z_{1, 0})= \mu_{0}$.
  As a consequence, \eqref{e:ineq'} can be rewritten in terms of the sequence $(u_j)_j$ as:
  \begin{equation*}
    \begin{cases}
      u_j > u_{j-1} + \theta_{j-1} + \Dist(z_{j, M}, z_{j-1, M}) & (j \geq 2) \\
      u_1> \mu_{0}+ \Dist(z_{1, M}, z_{0, M}) & (j=1).
    \end{cases}
  \end{equation*}
  The result follows by induction, using that $u_{M-1}=|\theta_{M-1}|$ and
  $$\nu_{M}= \sum_{j=1}^{M-1} \Dist(z_{j, M}, z_{j-1, M}).$$
 \end{proof}

 \subsubsection{Erasing the negative crossings}
 \label{s:erasing}

 We are now ready to prove the main result of this section, the contruction of the erasing
 procedure. We denote as $\vec{\theta} \in \R^{M-1}$ the vector
 $\theta_j = \Dist(z_{j-1,j},z_{j+1,j})$ for all $1 \leq j \leq M-1$, and $\vec{L} \in \R_{>0}^{M}$
 the vector of components $L_{j,j+1}$, $0 \leq j < M$.  We note that the geometry of our family of
 aligned geodesics is entirely determined by the values of the parameters $\vec{L}, \vec{\theta}$
 (which belong in a subset of $\R_{>0}^{M} \times \R^{M-1}$).
 
 \begin{prp} \label{p:pluscross} Let $M \geq 2$ and $\vec{L} \in \R_{>0}^{M}$ be a fixed
   length-vector. There exists a partition $(P_\xi)_{\xi\in \Xi_M}$ of $\R^{M-1}$ and a family
   $(I(\xi))_{\xi \in \Xi_M}$ of subsets
   $$I(\xi) = \{ 0 = \varphi_{\xi}(0) < \varphi_\xi(1) < \ldots < \varphi_\xi(K_\xi+1) = M \}$$
   such that the following holds.  Let $\xi \in \Xi_M$. For any $\vec{\theta} \in P_\xi$,
   \begin{enumerate}
   \item[(a)] for any $1 \leq j \leq K_\xi$,
     $\Dist(z_{\varphi_\xi(j-1), \varphi_\xi(j)}, z_{\varphi_\xi(j+1), \varphi_\xi(j)}) \geq 0$;
   \item[(b)] for $0 \leq j \leq K_\xi$, if we set
     \begin{align*}
       & \mu_{\varphi_\xi(j)}  :=\Dist(z_{ \varphi_\xi(j+1), \varphi_\xi(j)}, z_{\varphi_\xi(j)+1, \varphi_\xi(j)}) \\
       & \nu_{\varphi_\xi(j+1)}:=\Dist(z_{\varphi_\xi(j+1)-1, \varphi_\xi(j+1)}, z_{\varphi_\xi(j), \varphi_\xi(j+1)})
     \end{align*}
     then both of these distances are positive;
   \item[(c)] for $0 \leq j \leq K_\xi$, 
     \begin{equation}
       \label{e:magick}
       -\sum_{\varphi_\xi(j)< k<\varphi_\xi(j+1)} \theta_{k} \geq  \mu_{\varphi_\xi(j)} +\nu_{\varphi_\xi(j+1)}.
     \end{equation}
   \end{enumerate}
 \end{prp}

 Intuitively, this statement uses a combination of basic moves, together with the staircase
 inequality, to extract a path homotopic to the initial path, but with only positive turns.  This is
 illustrated in Figure \ref{fig:example_varphi}. Moving forward, we will use Proposition
 \ref{p:pluscross} to \emph{erase} the variables $\theta_k$ with $k \notin I(\xi)$. This includes
 all of the negative $\theta_k$, but also the $\theta_k$ that are ``not positive enough'' to
 compensate for the existence of negative ones. Inequality \ref{e:magick} says that the sum of
 erased $\theta_k$ is negative, and gives a lower bound on its absolute value.
 
 \begin{figure}[h!]
   \includegraphics{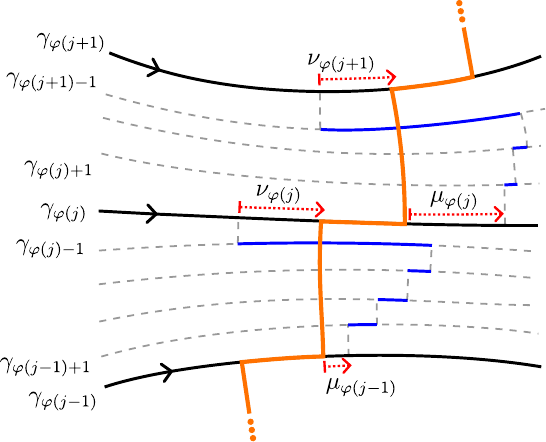}
   \caption{Illustration of the statement of Proposition \ref{p:pluscross}.}
   \label{fig:example_varphi}
  \end{figure}%

 \begin{rem}
   The partition is indexed by some set $\Xi_M$, the definition of which will be determined in the
   proof, but is of no particular interest. It is worth noting from the proof that the elements
   $P_\xi$ of the partition are products of intervals.
 \end{rem}

 \begin{rem}
   We have fixed the lengths $\vec{L}$ because the whole discussion depends implicitly on the value
   of $\vec{L}$, although not apparent in the notation.
   The dependency on $\vec{L}$ shall not be relevant to our purposes.
 \end{rem}
 
 \begin{proof}
   We shall prove the proposition by induction on the integer $M$, using a succession of basic
   moves.  If $M=2$, the partition of $\R$ is $\R_{<0} \sqcup \R_{\geq 0}$. More precisely, the sets
   $I$ are as follows.
   \begin{itemize}
   \item If $\theta_1 \geq 0$, there is nothing to be done, and we take $I = \{0,1,2\}$ (i.e.
     $K = M-1 = 1$ and $\varphi$ is the identity).
   \item If $\theta_{1} <0$, then we perform a basic move, that is, we take $I =
     \{0,2\}$. Inequality \eqref{eq:ineq_hexa} then implies the claim that
     $|\theta_{1}| > \Dist(z_{2, 0}, z_{1, 0}) + \Dist(z_{1, 2}, z_{0, 2}) = \mu_0+\nu_2$.
   \end{itemize}
   
   Assume now that the construction has been achieved at the rank $M$, and extend it to the next
   rank $M+1$, i.e. to a family of aligned geodesics $\gamma_0, \ldots, \gamma_{M+1}$. We apply the
   construction to the geodesics $\gamma_0, \ldots, \gamma_{M}$ and obtain a partition
   $(P_\xi)_{\xi \in \Xi_{M}}$ of $\R^{M-1}$ and a family of sets
   $$I(\xi) = \{0 = \varphi_\xi(0) < \varphi_\xi(1) < \ldots < \varphi_\xi(K_\xi+1) = M \}$$
   satisfying the properties (a), (b), (c). Let us now construct a partition
   $(P_{\xi'}')_{\xi' \in \Xi_{M+1}}$ of $\R^{M}$ and a family of sets $I'(\xi')$ (associated to a
   function $\varphi_{\xi'}'$ and an integer $K_{\xi'}'$) satisfying (a), (b), (c) at the rank
   $M+1$.
   
   We shall partition $$\R^M = \R^{M-1} \times \R = \bigsqcup_{\xi \in \Xi_{M}} (P_\xi \times \R)$$
   by partitionning each individual $P_\xi \times \R$.  Let $\xi \in \Xi_{M}$ be fixed -- to
   simplify notations, we drop the mention of $\xi$ in $\varphi_\xi$ and $K_\xi$. Let us cut the
   $\R$-component of $P_\xi \times \R$ (related to the value of~$\theta_M$) depending on the
   relative positions of the points $z_{\varphi(K), M}$ and $z_{M+1, M}$ along $\gamma_{M}$.
   \begin{itemize}
   \item If the segment $[z_{\varphi(K), M}, z_{M+1, M}]$ is positively oriented, there is
     nothing to be done, and we simply let $I' = I(\xi) \cup \{M+1\}$ (in which case $K'=K+1$).
   \item Otherwise, using \eqref{eq:signs_hexa}, we observe that for any $k \in \{1, \ldots, K\}$,
     \begin{align*}
       &\Dist(z_{M+1,\varphi(k)},z_{\varphi(k+1),\varphi(k)}) \geq 0
       \quad \Rightarrow \quad
       \forall j \geq k, \quad \Dist(z_{M+1,\varphi(j)},z_{\varphi(j-1),\varphi(j)}) > 0 \\
       &\Dist(z_{\varphi(k-1),\varphi(k)}, z_{M+1,\varphi(k)}) \geq 0
       \quad \Rightarrow \quad
       \forall j \leq k, \quad \Dist(z_{\varphi(j+1),\varphi(j)}, z_{M+1,\varphi(j)}) > 0.
     \end{align*}
     We use this to establish the existence of a uniquely defined $K' \in \{0, \ldots, K\}$,
     represented in Figure \ref{fig:erasing_constr}, such that the following two conditions hold:
     \begin{equation}
       \label{eq:cond_theta_step}
       \begin{split}
         \forall 0 \leq j \leq K, \quad
         \big(\Dist(z_{M+1,\varphi(j)},z_{\varphi(j+1),\varphi(j)}) \geq 0
         & \quad \Leftrightarrow \quad j \geq K'\big) \\
         \forall 1 \leq j \leq K, \quad
         \big(\Dist(z_{\varphi(j-1),\varphi(j)},z_{M+1,\varphi(j)}) > 0
         & \quad \Leftrightarrow \quad  j \leq K'\big).
       \end{split}
     \end{equation}
     \begin{figure}[h!]
       \includegraphics[scale=0.9]{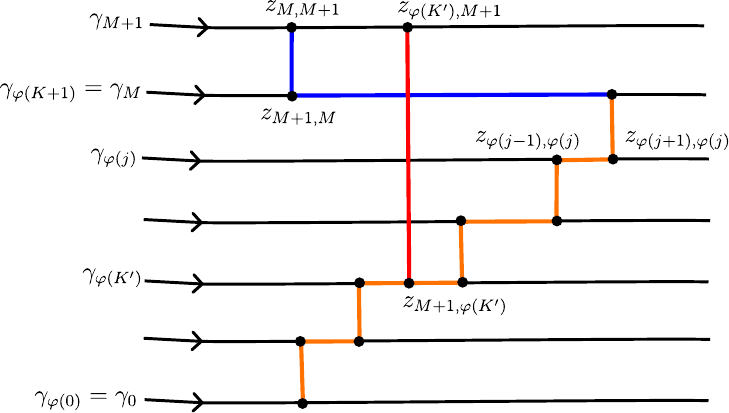}
       \caption{In order to highlight the construction of the integer $K'$, we have only represented
         the geodesics $\gamma_{\varphi(j)}$ for $0 \leq j \leq K+1$ as well as $\gamma_{M+1}$, and
         the relevant orthogeodesics.}
       \label{fig:erasing_constr}
     \end{figure}
     We then let $I' = \{\varphi(0), \ldots, \varphi(K'), M+1\}$, a set associated to
     the integers $K'$ and the function $\varphi'$ which coincides with $\varphi$ for $j \leq K'$
     and is equal to $M+1$ at $K'+1$.
   \end{itemize}

   Let us now construct the partition in more detail. We temporarily revert to notations depending
   on $\xi$ for clarity. For any $K' \in \{0, \ldots, K_\xi+1\}$, we denote as $E_{K'}$ the set of
   values of $\theta_M \in \R$ such that:
   \begin{itemize}
   \item $\Dist(z_{\varphi_\xi(K_\xi), M}, z_{M+1, M}) \geq 0$ if $K'=K_\xi+1$;
   \item the set of conditions \eqref{eq:cond_theta_step} holds if $K'$ is an integer in
     $\{0, \ldots, K_\xi\}$.
   \end{itemize}
We can then write
   $\R = \bigsqcup_{K'=0}^{K_{\xi}+1} E_{K'}$ for each individual $\xi \in \Xi_{M}$. This yields the
   partition
   $$\R^M = \bigsqcup_{\xi' \in \Xi_{M+1}} P'_{\xi'},
   \quad \text{where} \quad \Xi_{M+1} = \{ (\xi, K') \, : \, \xi \in \Xi_{M}, 0 \leq K' \leq
   K_\xi+1\}$$
   and $P'_{\xi'} := P_\xi \times E_{K'}$  for $\xi' = (\xi, K') \in \Xi_{M+1}$.
   
   Our choice of $K'$ straighforwardly implies that (a) and (b) are satisfied. The condition (c) is
   automatic in the case $K'=K+1$, and is the last thing that we need to check when $K' \leq K$.

   By the induction hypothesis and the construction of $I'$, we already know that \eqref{e:magick}
   holds for $j=1, \ldots, K'-1$, and we just need to check it for $j=K'$. That is, we need to prove
   that
   \begin{equation*}
     -\sum_{\varphi(K')< k<M+1} \theta_{k} \geq
     \Dist(z_{M+1, \varphi(K')}, z_{\varphi(K')+1, \varphi(K')}) 
     + \Dist(z_{M, M+1}, z_{\varphi(K'), M+1}).
   \end{equation*}
   In order to do so, we write
   \begin{equation}
     \label{eq:theta_decomp}
     \sum_{\varphi(K')< k<M+1} \theta_{k}
     = \sum_{j=K'+1}^{K+1} \theta_{\varphi(j)}
     + \sum_{j=K'}^K \sum_{\varphi(j) < k < \varphi(j+1)} \theta_k.
   \end{equation}

   In order to bound the first term, we apply the staircase inequality, Lemma \ref{l:magicstaircase}
   (under the form \eqref{e:theta-sum}), to the geodesics $\gamma_{\varphi(j)}$ for
   $K' \leq j \leq K$ and $\gamma_{M+1}$. The conditions (a) and (b) allow to verify the
   hypotheses. We obtain
  \begin{equation}\label{e:point0}
    \begin{split}
      & -\sum_{j=K'+1}^{K} \Dist(z_{\varphi(j-1), \varphi(j)}, z_{\varphi(j+1),\varphi(j)})
        - \Dist(z_{\varphi(K),M}, z_{M+1,M})
      \\ & >  \Dist(z_{M+1, \varphi(K')}, z_{\varphi(K'+1),\varphi(K')}) 
          + \Dist(z_{M, M+1}, z_{\varphi(K'), M+1}).
    \end{split}
  \end{equation}
  We then relate each of the terms on the left-hand-side to a $\theta_{\varphi(j)}$ by writing for
  $j \leq K$
    \begin{equation*}
    \begin{split}
      \theta_{ \varphi(j)}
      & = \Dist(z_{\varphi(j)-1, \varphi(j)}, z_{\varphi(j)+1, \varphi(j)} ) \\
      & = \Dist(z_{\varphi(j)-1, \varphi(j)}, z_{\varphi(j-1), \varphi(j)})
        + \Dist(z_{\varphi(j-1), \varphi(j)}, z_{\varphi(j+1), \varphi(j)} )
        + \Dist(z_{\varphi(j+1), \varphi(j)}, z_{\varphi(j)+1, \varphi(j)} ) \\
      & = \nu_{\varphi(j)}
        + \Dist(z_{\varphi(j-1), \varphi(j)}, z_{\varphi(j+1), \varphi(j)} )
        + \mu_{\varphi(j)}
    \end{split}
  \end{equation*}
  and similarly
  \begin{equation*}
       \theta_{\varphi(K+1)} = \theta_M
       = \Dist(z_{M-1, M}, z_{M+1, M} ) 
       = \nu_{\varphi(K+1)} + \Dist(z_{\varphi(K), M}, z_{M+1, M} ).
  \end{equation*}
  Equation \eqref{e:point0} then becomes
  \begin{equation}\label{e:point0'}
    \begin{split}
      & -\sum_{j=K'+1}^{K+1} \theta_{ \varphi(j)}
        +  \sum_{j=K'+1}^{K}\mu_{\varphi(j)} + \sum_{j=K'}^{K} \nu_{\varphi(j+1)} \\
       & >  \Dist(z_{M+1, \varphi(K')}, z_{\varphi(K'+1),\varphi(K')}) 
         + \Dist(z_{M, M+1}, z_{\varphi(K'), M+1}) .
    \end{split}
  \end{equation}

  Now, to bound the second sum in \eqref{eq:theta_decomp}, we apply the induction hypothesis to $K'
  \leq j \leq K$, which yields:
  \begin{align} \label{e:point1}
    - \sum_{\varphi(j)< k<\varphi(j+1)} \theta_{ k}
    >    \mu_{\varphi(j)} + \nu_{\varphi(j+1)}.
  \end{align}
  Summing \eqref{e:point0'} with \eqref{e:point1} for all $K' \leq j \leq K$ allows to obtain, due
  to \eqref{eq:theta_decomp}, that
  \begin{align*}
    -\sum_{\varphi(K')< k<M+1} \theta_{k}
    > \Dist(z_{M+1, \varphi(K')}, z_{\varphi(K'+1),\varphi(K')}) 
         + \Dist(z_{M, M+1}, z_{\varphi(K'), M+1})  
       + \mu_{\varphi(K')}.
  \end{align*}
  (we notice that all but one of the terms $\mu_{\varphi(j)}$ and $\nu_{\varphi(j+1)}$ vanish).  The
  conclusion then follows, observing that
  \begin{equation*}
    \Dist(z_{M+1, \varphi(K')}, z_{\varphi(K'+1),\varphi(K')})
    + \mu_{\varphi(K')}
    =     \Dist(z_{M+1, \varphi(K')},  z_{\varphi(K')+1, \varphi(K')} ).
  \end{equation*}
\end{proof}

\begin{rem}
  \label{rem:neg_cross}
  We can apply a similar construction to families of aligned geodesics ordered in the opposite
  order, i.e. such that $\gamma_{i+1}$ is now on the \emph{right} of $\gamma_i$. Importantly, in
  this case, we do {\em{not}} reverse the direction of the hypotheses (b) (otherwise, the
  construction would be exactly the same and of no additional interest). See Figure
  \ref{fig:straight}.
  \begin{figure}[h!]
    \includegraphics[scale=0.9]{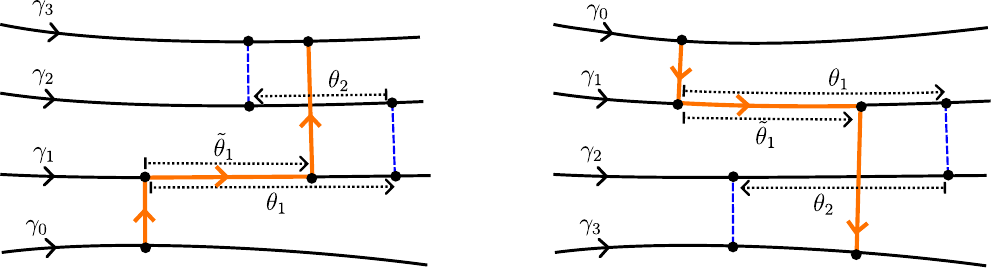}
    \caption{Two examples with $M=3$ and positive or negative
   crossings.}
    \label{fig:straight}
  \end{figure}
   \end{rem}

  \subsection{Crossing intervals and erasing of negative crossing parameters}
  \label{s:cross_int}

  In this section, we lift the geodesics $\beta$ to the hyperbolic plane, and apply Proposition
  \ref{p:pluscross} to erase all negative crossing variables.  This yields newly defined \emph{positive}
  variables $(\tilde \theta_q)_{q \in I(\xi)}$. We will rely on the periodic re-indexation of
  $\Theta$ introduced in \S \ref{s:cop_relabel} as well as the lifting from \S \ref{s:lift}.
  
  The forthcoming discussion, in the case of purely non-crossing eights, is essentially empty.

  \subsubsection{Decomposition of $\Theta^i$ in crossing intervals}
  \label{sec:decomp-thet-cross}

  Let $i \in \{1, \ldots, \cc\}$.
  We decompose the set $\Theta^i\sim \Z_{ n_i}$ according to the signs of its elements.
  
  \begin{nota} \label{nota:cross_int}
    We let $I_1^+, \ldots, I_{p}^+ \subseteq \Theta^i$ denote the connected components (maximal
    subintervals) of $\{q\in\Theta^i, \sign(q)=+\}$, and $I_1^-, \ldots, I_{p}^- \subseteq \Theta^i$
    the connected components of the complement.  We call the intervals $I_j^\delta$ \emph{crossing
      intervals} ($j=1, \ldots, p$, $\delta=\pm$).  \emph{Actual crossings} correspond to intervals
    $I_j^\delta$ of length $\geq 2$.
  \end{nota}
 %For the moment let us assume that $|I_j^+| <n_1$ and $|I_j^-| <n_1$ for all $j$, in other words we assume that $\sign(k)$ is not constant for $k\in \Theta^1$.
 %We will see in Remark \ref{e:allequal} that the case where the $\sign(k)$ are all equal is not
 %very different.

% We note that the previously defined $\theta(i, n)$ coincides with the signed distance
% $\Dist(z_{\sigma^{-1}n, n}, z_{\sigma n, n})$, and that the notations introduced above are
% consistent with the ones used in \S \ref{s:fam_aligned}, now studying the family of aligned
% geodesics $\tilde \beta_q$ for $q \in \mathbf{\Theta}^i$.

%Note as well that $L_{\sigma^{-1}n, n}$ coincides with $L(i, n) $.

  \begin{rem}
    \label{rem:crossing_const_sign}
    If the function $\sign$ is constant on $\Theta^i$, then $p=1$ (there is only one crossing
    interval), and we either have $I_1^+ = \emptyset$ or $I_1^- = \emptyset$. In all other cases,
    due to the cyclicity of $\Theta^i\sim \Z_{ n_i}$, the crossing intervals alternate in signs, and
    hence the number of positive and negative crossing intervals are equal.
  \end{rem}
  
  This decomposition of $\Theta^i$ in crossing intervals yields the partition
  $$\Theta^i=\bigsqcup_{j=1}^p \bigsqcup_{\delta=\pm} I_j^\delta,$$
  which  allows us to identify
  $$\R^{\Theta^i} \simeq \prod_{j=1}^p \prod_{\delta=\pm} \R^{I_j^\delta}.$$
 
  Note that crossing intervals $I_j^\delta$ ($j\in\Z$) can also be defined as subsets of
  ${\mathbf{\Theta}}^i$,
 %(invariant under translation by $n_1\Z$)
  after pre-composing with the projection map $\Z\longrightarrow \Z_{n_i}$. In that case we have
  $I_{j+p}^\delta=I_{j}^\delta+ n_i$, where $p$ is the number of positive crossing intervals
  (equivalently, of negative crossing intervals).

  \subsubsection{Application of the erasing procedure to one crossing interval}
  \label{sec:appl-eras-proc}

  Let $j \in \{1, \ldots, p\}$ be an actual crossing. We write
  $I_j^\delta = \{q_0 +1, \ldots, q_0 + M \}$, where $M = \#I_j^\delta \geq 2$. By Lemma
  \ref{lem:beta_al}, we can apply Proposition~\ref{p:pluscross} (and Remark \ref{rem:neg_cross} if
  $\delta = -$) to the family of $M+1$ aligned geodesics $\tilde{\beta}(q)$ with
  $q_0 \leq q \leq q_0+M$, and obtain:
    \begin{itemize}
    \item a partition $(P_\xi^o)_{\xi \in \Xi_M}$ of $\R^{\{q_0+1, \ldots, q_0+M-1\}}$;
    \item subsets $\{q_0 = \varphi_\xi(0) < \varphi_\xi(1) < \ldots < \varphi_\xi(K_\xi+1) = q_0+M\}$ of
      $I_j^\delta \cup \{q_0\}$
    \end{itemize}
    satisfying the properties (a), (b), (c).  We then define:
    \begin{itemize}
    \item $\Xi(I_j^\delta)$ as a copy of $\Xi_M$;
    \item for $\xi = \xi_j^\delta \in \Xi(I_j^\delta)$, $P_{\xi_j^\delta} = P_\xi^o \times \R$, so
      that $(P_{\xi_j^\delta})_{\xi_j^\delta \in \Xi(I_j^\delta)}$ is a partition of
      $\R^{I_j^\delta}$;
    \item for $\xi \in \Xi(I_j^\delta)$,
      $I_j^\delta(\xi) := \{\varphi_\xi(1) < \ldots < \varphi_\xi(K_\xi+1)\}$, which is a
      subinterval of $I_j^\delta$.
    \end{itemize}

    We extend this definition to the case where the crossing is not an actual crossing, by setting
    $(P_{\xi_j^\delta})_{\xi_j^\delta \in \Xi(I_j^\delta)}$ to be the trivial partition
    $\R^{I_j^\delta} = \R$ indexed by a singleton set $\Xi(I_j^\delta)$, and
    $I_j^\delta(\xi) = I_j^\delta$ in that case.

  \begin{rem}
    Keep in mind that by construction, for any $\xi\in \Xi(I_j^\delta)$, the interval
    $I_j^\delta(\xi)$ always contains the right endpoint $q_0+M$ of $I_j^\delta$, which is by
    definition a U-turn. The sets $P_{\xi_j^\delta} \subseteq \R^{I_j^\delta}$ all contain a full copy
    of $\R$ for the rightmost point of $I_j^\delta$, because the U-turn variables do not need to be
    controlled, by virtue of the Useful Remark, Remark \ref{r:useful}. Indeed, it is automatic that
    the segments $ [z_{q_0-1, q_0}, z_{\varphi(1), q_0}] $ and $[z_{\varphi(K), q_0+M}, z_{q_0+M+1, q_0+M}]$ are
    positively oriented.
  \end{rem}

  \begin{rem} \label{e:allequal1} In the case described in Remark \ref{rem:crossing_const_sign}, the
    construction is slightly different at the last step of the induction. If
    $[z_{\varphi(K), n_i}, z_{n_i+1, n_i}]$ is positively oriented, we proceed as normal, but if
    $[z_{\varphi(K), n_i}, z_{n_i+1, n_i}]$ is negatively oriented, then we simply take
    $I(\xi)=\emptyset$.  
\end{rem}

 \begin{rem} 
   The construction is done in the universal cover. Since the covering group acts by isometries, if
   we replace $I$ by $I+n_i$, then the picture is translated by a hyperbolic isometry, namely the
   translation of axis $\tilde{\mathbf{c}}_i$ and of translation length $\ell({\mathbf{c}}_i)$.
% There is a natural idenfication of $\R^{I}$ and $\R^{I+n_1}$ by translation of indices.
% We can take $\Xi(I)=\Xi(I+n_1)$, the same partition $(P_\xi)_{\xi\in\Xi(I)}$, and $(I+n_1)(\xi)=I(\xi)+n_1$.
   This remark implies that the construction done in the universal cover projects down to
   ${\mathbf{S}}$, and that the subsets $I_j^\delta(\xi_j^\delta)$ satisfy the invariance property
   $I_{j+p}^\delta(\xi_{j+p}^\delta)=I_j^\delta(\xi_j^\delta)+n_i$.
 \label{r:lift}
\end{rem}

\subsubsection{New representative of the homotopy class of $\mathbf{c}_i$}
Let us show how our construction will be used to resolve the issue caused by the crossing interval
$I_j^\delta$, raised in \S \ref{s:problems}.

\begin{nota}
  \label{nota:GD}
  Let $\xi \in \Xi(I_j^\delta)$.  For $q = \varphi_\xi(k) \in I_j^\delta(\xi)$, we consider the points
  $$G_{q} := z_{\varphi_\xi(k-1), \varphi_\xi(k)} \quad \text{and} \quad
  D_{q} := z_{\varphi_\xi(k+1), \varphi_\xi(k)}.$$ We denote as $\cI'_{q}$ the segment
  $[G_{q}, D_{q}]$ on $\tilde{\beta}(q)$.
\end{nota}

By property (a), the segment $\tilde{\cI}'_q$ is positively oriented for every
$q \in I_j^\delta(\xi)$.  We observe that the portion
$\tilde p_{q_0+1} \smallbullet \ldots \smallbullet \tilde p_{q_0+M}$ of the infinite path
$\smallbullet_{q\in {\mathbf{\Theta}}^i}\; \tilde p(q)$ is homotopic, with endpoints gliding along
$\tilde\beta(q_0)$ and $\tilde \beta(q_0+M)$, to the path
\begin{align}
  \label{eq:homotopy_one_cross}
   \tilde{B}_{q_0, \varphi_\xi(1)}
   \smallbullet \cI'_{\varphi_\xi(1)}
   \smallbullet
   \ldots
   \smallbullet
   \tilde{B}_{\varphi_\xi(k-1), \varphi_\xi(k)}
   \smallbullet \cI'_{\varphi_\xi(k)}
   \smallbullet
   \tilde{B}_{\varphi_\xi(j), \varphi_\xi(j+1)}
   \smallbullet
   \ldots
   \smallbullet \cI'_{\varphi_\xi(K)}
   \smallbullet
   \tilde{B}_{\varphi_\xi(K),q_0+M}.
 \end{align}
 This is the path represented in Figure \ref{fig:erasing_constr}.  We therefore see that Proposition
 \ref{p:pluscross} allows to replace our original representation of $\tilde {\mathbf{c}}_i$, where
 the segments $\tilde{\mathcal{I}}(q)$ could be oriented in any direction, to a new one where the
 segments $\cI'_{\varphi_\xi(k)}$ are positively oriented.

\subsubsection{Erasing of all crossing intervals in $\Theta^i$}
\label{sec:erasing-all-crossing}

Let us now apply the erasing procedure simultanously to all crossing intervals in $\Theta^i$.

 \begin{nota} \label{nota:erasing_i}
   We define for the component $\mathbf{c}_i$ the set
   \begin{equation}
     \label{eq:partition_crossing_i}
     \Xi^i:=\prod_{1 \leq j \leq p} \prod_{\delta=\pm}\Xi(I_j^\delta).
   \end{equation}
   We denote as $\xi^i=(\xi_j^\delta)_{\substack{1 \leq j \leq p \\ \delta=\pm}}$ the elements of
   $\Xi^i $, and for such a $\xi^i$, we let
   \begin{equation}
     P_{\xi^i}:=\prod_{j=1}^p \prod_{\delta=\pm}  P_{\xi_j^\delta} \subseteq \prod_{j=1}^p
     \prod_{\delta=\pm} \R^{I_j^\delta} =\R^{\Theta^i}\label{eq:partition_crossing_i_2}
   \end{equation}
   and
   \begin{equation}
     I(\xi^i)=  \bigsqcup_{j=1}^p \bigsqcup_{\delta=\pm} I_j^\delta(\xi_j^\delta),
     \quad \quad 
     \tilde I(\xi^i)=  \bigsqcup_{j\in \Z} \bigsqcup_{\delta=\pm} I_j^\delta(\xi_j^\delta).
     \label{eq:partition_crossing_i_3}
   \end{equation}
   We denote $n_{\xi^i} = \# I(\xi^i)$.
\end{nota}

The collection $(P_{\xi^i})_{\xi^i\in \Xi^i}$ then forms a partition of $\R^{\Theta^i}$.  The set
$I(\xi^i)$ is a subset of $\Theta^i$ while $\tilde I(\xi^i)$ is the corresponding $n_i\Z$-periodic
subset of ${\mathbf{\Theta}}^i$. Because $I_j^\delta(\xi_j^\delta)$ always has the same right
endpoint as $I_j^\delta$, $ I(\xi^i)$ contains all the U-turn variables in $\Theta^i$.

\subsubsection{New variables $\tilde{\theta}_q$ and crossing inequality}

For $\xi^i \in \Xi^i$, the sets $I(\xi^i)$ and $\tilde{I}(\xi^i)$ can be described by an increasing
function $(\varphi_k)_{k \in \Z}$ with values in $\mathbf{\Theta}^i$ such that
$\varphi_{k+ \# I(\xi^i)} = \varphi_k + n_i$. 

\begin{nota}
  \label{nota:erase_lengths}
  Let $\xi^i \in \Xi^i$.  For $q = \varphi_k \in I(\xi^i)$, we denote as:
  \begin{itemize}
  \item   $G_{q} := z_{\varphi_{k-1}, \varphi_k}$ and
  $D_{q} := z_{\varphi_{k+1}, \varphi_k}$;
  \item $\cI'_{q}$ the segment $[G_{q}, D_{q}]$ on $\tilde{\beta}(q)$, and $\tilde \theta_{q}$ its length;
  \item $\tilde L_q= L_{\varphi_{k-1}, \varphi_k}$ the length of the orthogeodesic
    $\tilde B_{\varphi_{k-1}, \varphi_k}$.
  \end{itemize}
\end{nota}
By property (a), the segment $\cI'_{q}$ is always positively oriented and hence $\tilde{\theta}_q
\geq 0$.

We can now adapt \eqref{eq:homotopy_one_cross} and now simultaneous erase all crossing
intervals. Indeed, the infinite concatenation
$\smallbullet_{k \in \Z}\big(\tilde B_{\varphi_{k-1}, \varphi_k}\smallbullet \cI'_{\varphi_k}\big)$
gives an infinite curve in the hyperbolic plane, which projects down to a curve homotopic to the
lift~$\tilde{\mathbf{c}}_i$ of $\mathbf{c}_i$ to $\IH$.  The fact that the new path does no longer
go along $\tilde{\cI}(q)$ for $q\in \tilde I(\xi^i)$ motivates the following terminology.
\begin{defa}The variables $\theta_q$ for $q \in \Theta^i \setminus \tilde I(\xi^i)$ are said to be
  \emph{erased}.
\end{defa}

We can then restate the property (c) from Proposition \ref{p:pluscross} in the following equivalent
way.

\begin{lem}
  For any component $i$, 
  \begin{align} \label{e:funnier} \sum_{q \in \Theta^i} \theta_q
    < \sum_{q\in I(\xi^i)} \tilde \theta_q.
\end{align}
\end{lem}
\begin{proof}
  We simply notice that property (c) states that
  \begin{align} \label{e:funny}
  \sum_{q\in \Theta^i \setminus I(\xi^i)} \theta_q
    < - \sum_{k=1}^{\# I(\xi^i)} (\nu_{\varphi_{k+1}} + 
    \mu_{\varphi_{k }}) 
  \end{align}
  with the notation $\mu, \nu$ from Proposition \ref{p:pluscross}, and observe that by definition,
  \begin{align} \label{e:chasles}
    \nu_{\varphi_k}  + \tilde \theta_{\varphi_k}+
    \mu_{\varphi_k} = \theta_{\varphi_k}.
  \end{align}
\end{proof}

\begin{rem}
 
  In the case described in Remarks \ref{rem:crossing_const_sign} and \ref{e:allequal1}, when the sign function is constant on
  $\Theta^i$, if $\xi$ is such that $I(\xi) = \emptyset$, then \eqref{e:funnier} means that
  \begin{align*} 
    \sum_{q \in \Theta^i} \theta_q <0.
  \end{align*}
  In this case it is sufficient to consider the geodesic representative of ${\mathbf{c}}_i$, and we
  won't need to express the length $\ell_Y({\mathbf{c}}_i)$ in coordinates.
\end{rem}

\subsubsection{Product over all components $\mathbf{c}_i$}
\label{sec:product-over-all}

Let us now group our constructions for the different components
$\mathbf{c}_1, \ldots, \mathbf{c}_{\cc}$ of the multiloop $\mathbf{c}$.  In order to do so, we let
$\Xi := \prod_{i=1}^{\cc}\Xi^i$. We shall consider the cartesian product of the partitions
$P_{\xi^i}$, that is, for $\xi=(\xi^1, \ldots, \xi^{\cc})\in \Xi$, we let
\begin{align}\label{e:pxi}P_\xi = \prod_{i=1}^{\cc} P_{\xi^i}
  \quad  \text{and} \quad
  I(\xi)=\bigsqcup_{i=1}^{\cc} I(\xi^i)\subseteq \Theta.
\end{align}
The sets $(P_\xi)_{\xi \in \Xi}$ then form a partition of
$\prod_{i=1}^{\cc}\R^{\Theta^i}=\R^\Theta$. Each $P_\xi $ is a product of intervals.

We shall say that the variables $\theta_q$ for $q\not\in I(\xi)$ are \emph{erased}. The set $I(\xi)$
contains all U-turn variables. As a consequence, in the purely non-crossing case, we have
$I(\xi)=\Theta$ for all $\xi$, and all the previous discussion is empty.

\subsection{New expression for the length} \label{s:newexp}

The erasing procedure immediately yields a new expression for $\ell_Y({\mathbf{c}})$, true for
$\vec{\theta} \in P_\xi$, which replaces \eqref{e:outcome2}.

\begin{prp}
  Let $\xi \in \Xi$. As soon as $\vec{\theta} \in P_\xi$, we have
  \begin{align}\label{e:newc}
    \ell_Y({\mathbf{c}})
    % &= M_n\Big(( \sign(\varphi_j)L_{\varphi_{j-1}, \varphi_j})_{j=1}^n, (\tilde \theta_{\varphi_j})_{j=1}^n\Big)\\
= \sum_{i=1}^{\cc} M_{n_{\xi^i}}( (\eps \tilde L_q)_{q\in I(\xi^i)}, (\tilde \theta_q)_{q\in I(\xi^i)} )
  \end{align}
  where $n_{\xi^i}= \# I(\xi^i)$, $\eps \tilde L_q := \sign(q) \tilde{L}_q$ and $M_n$ is defined in
  \eqref{e:Mn}.\end{prp}

\begin{proof}
  For each $1 \leq i \leq \cc$, on the subset $P_{\xi^i}$ of $\R^{\Theta^i}$, we apply Proposition
  \ref{prop:form_l_gamma_try} to the  representative 
  $\smallbullet_{k \in \Z}\big(\tilde B_{\varphi_{k-1}, \varphi_k}\smallbullet
  \cI'_{\varphi_k}\big)$ of the homotopy class of $\tilde{\mathbf{c}}_i$.
\end{proof}

Note that, as long as $\vec{\theta} \in P_{\xi}$, each variable $\tilde \theta_{q}$ with
$q\in I(\xi)$ is positive. As a consequence, this new formula solves the first issue raised in \S
\ref{s:problems}.

When using formula \eqref{e:newc} later, we shall consider all the variables $\theta_q $ with
$q\not \in I(\xi)$ as additional neutral parameters. The following observation will then be useful.
\begin{lem}
  \label{lem:new_var_erase}
  Let $\xi \in \Xi$. On the set $P_\xi$, for any $q \in I(\xi)$,
  \begin{equation}
    \tilde{L}_q = \fn((\theta_{q'})_{q'\not \in  I(\xi)}, \vec{L})
    \quad \text{and} \quad
    \tilde \theta_q= \theta_q +\fn((\theta_{q'})_{q'\not \in I(\xi)},
    \vec{L}).
\end{equation}
\end{lem}
\begin{proof}
  Let $\xi \in \Xi$.  For $1 \leq i \leq \cc$, let us represent $\xi^i \in \Xi^i$ by a function
  $(\varphi_k)_{k \in \Z}$. For $q \in I(\xi^i)$, the length $\tilde{L}_q$ is the length of the
  orthogeodesic from $\tilde{\beta}(\varphi_{k-1})$ to $\tilde{\beta}(\varphi_k)$, and therefore can
  be expressed as a function of $\vec{L}$ and the $\theta_{q'}$ with $q' \in \Theta^i$ such that
  $\varphi_{k-1} < q' < \varphi_k$. This implies that $q' \notin I(\xi)$, which was our claim.

  Similarly, the quantities $\nu_{\varphi_{k}} , \mu_{\varphi_{k}}$ can be expressed solely in terms
  of the $\theta_{q'}$ with $q'\not \in I(\xi)$.  However, we have seen in \eqref{e:chasles} that
  $\theta_{\varphi_k} = \nu_{\varphi_k} + \tilde{\theta}_{\varphi_k} + \mu_{\varphi_k}$, which
  allows us to conclude.
\end{proof}

As a consequence, if we fix $(\theta_{q})_{q\not \in I(\xi^i)}$ as well as $\vec{L}$, and view the
expression \eqref{e:newc} as a function of $(\tilde \theta_q)_{q\in I(\xi^i)}$, then Proposition
\ref{p:mainclass} holds for this new length function.

\begin{rem}
  An intuition as to why the variables $\theta_q$ taking negative values yield additional neutral
  parameters can be found by examining Figure \ref{fig:crossing_ex}. Recall that we have motivated
  in \S\ref{s:act_ntr} the fact that the lengths $\vec{L}$ are neutral parameters by observing that,
  when we describe the path~$\mathbf{c}_i$ as the concatenation
  $\mathbf{c}_i^{\mathrm{op}} = \smallbullet_{q\in {{\Theta}}^i}\; p(q)$, each bar is visited twice, which mean
  that there is a factor of $2$ in front of each $L_j$ in our comparison estimates.

  We see in the example of Figure \ref{fig:crossing_ex} that, when $\theta_1^+ < 0$, the portion
  $\mathcal{I}_1^+$ is visited three times by the path $\mathbf{c}_i^{\mathrm{op}}$, instead of once
  when $\theta_1^+\geq0$. This means that we do not need to apply the operator $\cL$ to cancel
  exponential growth related to the parameter $\theta_1^+$, or, in other words, that it is a neutral
  parameter.
\end{rem}

%  The same construction applies to all components ${\mathbf{c}}_1, \ldots, {\mathbf{c}}_{\cc}$, with $\Theta^1, {\mathbf{\Theta}}^1$ replaced everywhere by $\Theta^i, {\mathbf{\Theta}}^i$ for each $i=1, \ldots, \cc$, $\Xi^1$ replaced by $\Xi^i$, etc. We obtain, for each $i=1, \ldots, \cc$, 
% a partition $(P_{\xi^i})_{\xi^i\in \Xi^i}$ of $\Theta^i$. The construction also gives us a subset of indices $ I(\xi^i) \subset \Theta^i $, as well as $ \tilde I(\xi^i) \subset   {\mathbf{\Theta}}^i\simeq\Z$.
% For each $k\in \tilde I(\xi^i)$, we have a positively oriented segment $\cI^{'}_k= [G_k, D_k]$ in $\tilde \beta_k$, of length $\tilde \theta_k$.
%  If $\varphi$ is an increasing enumeration of all elements of $\tilde I(\xi^i)$, then the concatenation of all $\tilde B_{\varphi_{j-1}, \varphi_j}\smallbullet \cI^{'}_{\varphi_j}$ gives a bi-infinite curve in the hyperbolic plane, which projects down to a curve homotopic to ${\mathbf{c}}_i$. For $k\in \tilde I(\xi^i)$, we denote by $\nu_k$ the distance between $z_{k-1, k}$ and $G_k$, and by $\mu_k$ the distance between $D_k$ and $z_{k+1, k}$.
% For $k=\varphi_j\in  I(\xi^i)$, we shall denote (in analogy with Lemma \ref{l:magicstaircase}) by 
%$\rr_k=\rr_{\varphi_j}$ the portion of $\tilde B_{\varphi_{j}, \varphi_{j+1}}$ lying between $\tilde \beta_{k}$ and $\tilde \beta_{{k+1}}$ (starting at the point $D_{k}$).

\subsection{The purely non-crossing case}
\label{s:purelync}
    
The reader only interested in the purely non-crossing case can read the rest of the paper,
considering that $\Xi$ is a set with one element denoted $\xi=(\xi^1, \ldots, \xi^{\cc})$, and that
in what follows we let $I(\xi)=\Theta$, $I(\xi^i)=\Theta^i$ for $i \in \{1, \ldots, \cc\}$. In this
case we have $\tilde \theta_q=\theta_q$ for all $q\in \Theta$, as well as
$\cI'_q=\overline{\cI}_q=[G_q, D_q]$.

\subsection{A partition of unity on Teichm\"uller space} \label{s:cutoffQ}

We recall that the space of interest in this article is the Teichm\"uller space
$\mathcal{T}_{\g,\n}^\ast$, which we have equipped with a new set of coordinates
$(\vec{L}, \vec{\theta}) \in \R_{>0}^r \times \R^\Theta$.  We have introduced a partition
$\R^\Theta = \bigsqcup_{\xi \in \Xi} P_\xi$ that we wish to use on the parameter
$\vec{\theta}$. Because our methods require some smoothness, it will be more convenient to partition
the values of $\vec{\theta}$ using a partition of unity, which we shall now construct.  Note that
the entire discussion above is dependent on the length vector $\vec{L}$, although this has not been
made explicit in the notation.

\begin{nota}
  \label{nota:fcr}
  Let $\fcr{0}$ and $\fcr{\infty}$ denote two smooth functions on $\R_{\geq 0}$ such that:
  \begin{itemize}
  \item $\fcr{0}+\fcr{\infty}\equiv 1$;
  \item $\fcr{0}$ is supported in $[0, 2\log 2]$ and identically equal to $1$ on $[0, \log 2]$.
  \end{itemize}
\end{nota}

\begin{defa}
  \label{defa:part_unity_cross}
  For $\xi\in \Xi$ and $\cQ\subseteq I(\xi)$, we define a cut-off function on Teichm\"uller space,
  \begin{align}\label{e:PsiQ}
    \Psi_{\xi, \cQ}(\vec{L}, \vec{\theta})
    = \bbbone_{P_\xi}(\vec{L}, \vec{\theta})
    \prod_{q\in \cQ} \fcr{\infty}(\tilde \theta_q) \prod_{q \not\in \cQ} \fcr{0}(\tilde \theta_q).
  \end{align}
\end{defa}

 By construction,
 \begin{align*}
   \sum_{\substack{\xi\in \Xi \\ \cQ\subseteq I(\xi)}} \Psi_{\xi, \cQ}\equiv 1.
 \end{align*}

 As announced in \S \ref{s:newexp}, we associate to each partition element some additional neutral
 variables.
 \begin{nota}[Neutral variables associated with $(\xi, \cQ)$]
   \label{nota:neutr_cr}
   For $\xi\in \Xi$ and $\cQ\subseteq I(\xi)$, we let 
   \begin{align}\label{e:neutralxiQ}
   \ThetaNe(\xi, \cQ)= \Theta\setminus \cQ.
    \end{align}  
 \end{nota}

 Note that, on the support of the cut-off function $\Psi_{\xi, \cQ}$, the variables
 $(\tilde\theta_q)_{q\in \Theta\setminus \cQ}$ stay in $[0, 2\log 2]$, so that we do not need to
 gain exponential decay in those variables.  They can therefore be considered to be neutral
 variables.

%%%Local Variables: 
%%% mode: latex
%%% TeX-master: "main"
%%% End: 

\section{The lengths of boundary curves}
\label{s:lengthboundary}
\subsection{Objective and plan of the section}

Recall that, in Proposition \ref{prp:length_g_lambda}, we have obtained expressions for the lengths
of the pair of pants decomposition $(\Gamma_\lambda)_{\lambda\in \Lambda}$ in our new coordinates
$(\vec{L}, \vec{\theta})$ on the Teichm\"uller space $\cT_{\g,\n}^*$. This new expression reads
$\ell_Y(\Gamma_{\lambda})= Q_{m}(\vec{L}^\lambda,\vec{\tau}^\lambda)$, where
$\vec{L}^\lambda = (L^\lambda_1, \ldots, L^\lambda_{m})$ and
$\vec{\tau}^\lambda = (\tau^\lambda_1, \ldots, \tau^\lambda_{m})$ are the parameters
described in \S \ref{s:poly_curve_glambda}, $m=m(\lambda)$, and the functions $F_m$, $Q_m$ are
defined in \eqref{e:Fn} by
\begin{align} \label{e:Ff0}
  F_m(\vec{L}, \vec{\tau})
  = \frac{\cosh (\frac 12 Q_m(\vec{L}, \vec{\tau}))}{\exp(\frac 12 \sum_{j=1}^m\tau_j)}
  = \frac12 \sum_{\alpha  \in \{0, 1\}^{m}}
  (-1)^{|\alpha|}
  \hyp_{-\partial \alpha} \div{\vec{L}} e^{ - \alpha \cdot \vec{\tau}}
  \end{align}
 with $(\partial \alpha)_j =1$ if $\alpha_{j}=\alpha_{j-1}$ and $-1$ if $\alpha_{j}\not=\alpha_{j-1}$.
 % $\=(|u_2-u_1|, |u_3-u_2|, \ldots, |u_{s}- u_{s-1}|, |u_1-u_s|)$.

 The aim of this section is to study the function $F_m$, and in particular its logarithmic
 derivatives, which will appear when we apply our key argument to the
 integral~\eqref{e:int_ell}. The formula~\eqref{e:Ff0} is problematic because of its alternating
 signs: it is not clear which terms dominate the sum, nor is it that the formula gives a positive
 function. In fact, the domain $\domain$ of $(\vec{L}, \vec{\theta})$ is precisely defined by the
 fact that the formula giving $\cosh(\ell_Y(\Gamma_\lambda)/2)$ is greater than $1$ for all
 $\lambda \in \Lambda$, as seen in Proposition \ref{p:super}.

We work around this problem introducing a \emph{favorable region}, i.e. a region of $\domain$ on
 which $F_m$ is well-approximated by its term with $\alpha = (0, \ldots, 0)$, equal to
 $\prod_{j=1}^{m} \sinh(L_j/2)$. This domain is defined in terms of the \emph{heights of cells}
 introduced in \S \ref{sec:height-cell}. We provide estimates on the derivatives of $F_m$ true
 inside the favorable region in Proposition \ref{p:derF}.  Outside the favorable region, the height
 of one (or several) cells must be bounded -- we \emph{neutralize} such cells. The neutralization
 procedure is to low-height cells what the erasing procedure was to negative $\theta$ parameters in
 \S \ref{s:lengthc}. The description of the neutralization procedure while dealing
 simultanously with all of the boundary components $(\Gamma_\lambda)_{\lambda \in \Lambda}$
 requires some technical considerations, addressed in \S \ref{s:simultanous}. In \S
 \ref{s:partition}, we finally obtain a partition of unity of the Teichm\"uller space so that we
 understand the length of $\Gamma_\lambda$ and its derivatives on every piece of the partition.

 \subsection{Polygonal curves, neutralization, derivation}
 In \S \ref{s:bound}, we have used a key observation on the structure of the curves $\Gamma_\lambda$
 to compute their lengths: the fact that they can be viewed as a succession of perpendicular
 geodesic segments of alternating lengths $L_j^\lambda, \tau_j^\lambda$, turning the same direction
 each time (see Figure \ref{fig:Glambda}). We shall introduce a notion of polygonal curve, which
 encapsulates this key property of $\Gamma_\lambda$.  Let us first provide this definition without
 reference to a hyperbolic metric. Subsequently in \S \ref{s:GR}, a hyperbolic metric will be given,
 and we will work with preferred ``geodesic representatives''.

\subsubsection{Definition of polygonal curves} 

In the following, we fix a base surface $\mathbf{S}$ and an oriented multi-curve~$\beta$ on
$\mathbf{S}$.

\begin{defa} \label{d:pc}A polygonal curve $\Gamma$ (on the surface ${\mathbf{S}}$ and based on
  $\beta$) is an oriented closed path without self-intersections, formed by a cyclic sequence
  $(J_j, K_j)_{j \in \Z_m}$ of oriented segments where, for all $j \in \Z_m $,
  \begin{itemize}
  \item $J_j$ meets the multi-curve $\beta$ at its endpoints and not in its interior, and is
    homotopically non-trivial with gliding endpoints along $\beta$;
  \item $K_j$ is a subsegment of $\beta$, possibly with reversed orientation;
  \item $J_j$ arrives at $K_j$ on the right at the point $t(J_j) = o(K_j)$, and $K_j$ arrives at
    $J_{j+1}$ on the right at the point $t(K_j) = o(J_{j+1})$.
  \end{itemize}
\end{defa}
  
Note that the multi-curve $\beta$ is essential to this definition.  The segments $J_j$ and $K_j$
play different roles and will therefore be named differently.
\begin{defa} \label{def:cellbr}
  The segments $K_j$ are called the \emph{cells} of the polygonal curve $\Gamma$, and the $J_j$
  its \emph{bridges}.  We denote 
  $\Cell(\Gamma)=\{K_j, j \in \Z_m\}$ and $\Br(\Gamma)=\{ {J}_{j}, j \in \Z_m\}$.

  For a cell $K = K_j \in \Cell(\Gamma)$, the bridges \emph{adjacent to $K$} refer to the bridges
  $J_j$ and $J_{j+1}$. We respectively call these bridges the \emph{bridge arriving on $K$} and the
  \emph{bridge coming out of $K$}.
\end{defa}

 \begin{nota} \label{nota:clg}
   We call $m$ the \emph{combinatorial length} of $\Gamma$, denoted $|\Gamma|$.
 \end{nota}
 By convention, for $m=0$, a polygonal curve of combinatorial length $0$ is a simple curve that does
 not intersect the family $\beta$ (but may be homotopic to a component of $\beta$). In particular,
 polygonal curves of combinatorial length $0$ do not have any cells.

 \begin{exa}\label{ex:pcgamma}
   We have seen in \S \ref{s:poly_curve_glambda} that each of the curves
   $(\Gamma_{\lambda})_{\lambda\in \Lambda}$ in the pair of pants decomposition defined in \S
   \ref{s:ppdecompo} is a polygonal curve based on $\beta$, where:
   \begin{itemize}
   \item the bridges $J_j$ are among the bars $B_q$;
   \item the cells $K_j$ are among the segments $\cK_{q q'}$ defined in \S \ref{s:tau}.
   \end{itemize}
   The family $(\Gamma_\lambda)_{\lambda \in \Lambda}$ will be the main example to which the whole
   discussion applies.  For $\lambda\in \LambdaBC$, we can even say that the corresponding cells are
   among the segments $(\cK_{q})_{q\in \Theta}$.

   Note that the components $(\beta_\lambda)_{\lambda \in \Lambdabeta}$ are polygonal curves of
   combinatorial length $0$.
 \end{exa}

\subsubsection{Neutralizing a cell} 
\label{s:neutralizing_one}

Let $\Gamma=(J_j, K_j)_{j \in \Z_m}$ be a polygonal curve on $\mathbf{S}$ of combinatorial length
$m\geq 2$. Let $j\in \Z_{m}$. We may construct a new polygonal curve $\Gamma'$ by replacing the
sequence $(J_{j}, K_j, J_{j+1})$ in $\Gamma$ by one single bridge
$\tilde{J}_{j}:=J_{j}\smallbullet K_j \smallbullet J_{j+1}$ (or a smoothened version thereof). This
operation is represented in Figure \ref{fig:neutralization_one}.  We observe that the derived
polygonal curve $\Gamma'$ we obtain is homotopic to $\Gamma$ and that its combinatorial length is
$m-1$.
\begin{figure}[h!]
     \centering
   \begin{subfigure}[b]{0.5\textwidth}
    \centering
     \includegraphics{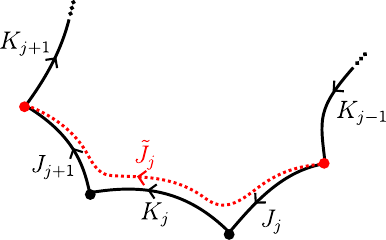}
    \caption{One cell.}
    \label{fig:neutralization_one}
  \end{subfigure}%
   \begin{subfigure}[b]{0.5\textwidth}
    \centering
     \includegraphics{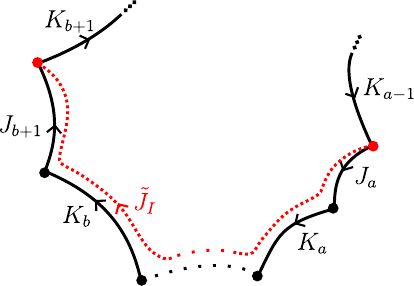}
    \caption{Several cells.}
    \label{fig:neutralization_several}
  \end{subfigure}%
     \caption{Neutralization of one or several cells.}
    \label{fig:neutralization}
  \end{figure}%

  \begin{defa}
    \label{d:der_neutr}
    The polygonal curve $\Gamma'$ is said to be {\em{derived from $\Gamma$ by neutralizing the
        cell~$K_j$}}.  We call $\tilde{J}_{j}$ the {\emph{roof over $K_j$}}.  We denote
    $\tilde{K}_{j-1}:= {K}_{j-1}$, $\tilde{K}_{j+1}:= {K}_{j+1}$ the {\em{new cells of $\Gamma'$}}.
\end{defa}
The motivation for the name ``new cells'' will become apparent when we take geodesic
representatives.  The new polygonal curve $\Gamma'$ satisfies
\begin{itemize}
\item $\Br(\Gamma')= \{\tilde{J}_{j}\} \cup (\Br(\Gamma)\setminus \{ J_{j}, J_{j+1}\} )$;
\item
  $\Cell(\Gamma')= \{\tilde{K}_{j-1}, \tilde{K}_{j+1} \} \cup (\Cell(\Gamma)\setminus \{ K_{j-1},
  K_j, K_{j+1}\})$.
\end{itemize}
  % \begin{figure}[h!]
  %   \includegraphics[height=3cm]{Figures/aJK1}
  %   \caption{The case $m=1$.}
  %   \label{fig:JK1}
  % \end{figure}%

If now $\Gamma$ is a polygonal curve of combinatorial length $1$, we define $\tilde{J}_{1}$ as the
simple closed curve $J_{1}\smallbullet K_1$ (or a smoothened verson thereof). Then
$\Gamma'=\tilde{J}_{1}$ is a polygonal curve of combinatorial length $0$ homotopic to $\Gamma$.

  \subsubsection{Simultaneous neutralization of several cells.}
  \label{s:neutralizing_I}
  
We now consider an interval $I=\{a, \ldots, b\}$ of $\Z_m$ of length $1 \leq \# I < m$.  Let us
define
$\tilde J_I := J_{a} \smallbullet K_{a} \smallbullet \ldots \smallbullet K_b \smallbullet J_{b+1} $
(or a smoothened version thereof) and call $\tilde J_I$ the \emph{roof over $(K_j)_{j\in I}$}.  A
new polygonal curve $\Gamma'$ may be obtained by replacing the sequence
$(J_a, K_a, \ldots, K_b, J_{b+1})$ by $\tilde J_I $, as represented in Figure
\ref{fig:neutralization_several}. We then say that $\Gamma'$ is \emph{derived from $\Gamma$ by
  neutralizing simultaneously all the $(K_j)_{j\in I}$}.

We denote $\tilde{K}_{a-1}:= {K}_{a-1}$, $\tilde{K}_{b+1}:= {K}_{b+1}$ and call them {\em{new
    cells}} of $\Gamma'$.  The new polygonal curve $\Gamma'$ then satisfies:
\begin{itemize}
\item
  $\Br(\Gamma')= \{\tilde{J}_{I}\} \cup (\Br(\Gamma)\setminus \{ J_{a}, J_{a+1}, \ldots, J_{b+1}\}
  )$;
\item $\Cell(\Gamma')= \{\tilde{K}_{a-1}, \tilde{K}_{b+1} \} \cup(\Cell(\Gamma)\setminus \{ K_{a-1},
  K_a, \ldots, K_{b+1}\})$.
\end{itemize}
We extend these definitions to the case $I = \Z_m$, now taking $\Gamma'$ to be a simple curve
freely homotopic to the concatenation
$J_1 \smallbullet K_1 \smallbullet \ldots \smallbullet J_m \smallbullet K_m$, so that it does not
intersect $\beta$.  Then, $\Gamma'$ is polygonal curve of combinatorial length $0$.

In all cases, $\Gamma'$ is homotopic to $\Gamma$, and its combinatorial length now is $m-\# I$. The
two notions of neutralization coincide when $I = \{j\}$, and the polygonal curve derived by
neutralizing simultanously all the $(K_j)_{j \in I}$ is the same as the one derived by successively
neutralizing the cells $K_a, \ldots, K_b$ (in whichever order).

\subsection{Geodesic representative and length inequalities}
\label{s:GR}
We now endow ${\mathbf{S}}$ with a hyperbolic metric $Y \in \cT_{\g,\n}^*$, and consider polygonal
curves as geometric objects.

\subsubsection{Geodesic representative} \label{s:geod_rep_pc}

As suggested in Remark \ref{rem:beta_geod}, we can assume by choosing the right representative in
the Teichm\"uller space of $\mathbf{S}$ that $\beta$ is a multi-geodesic, which then implies that
the cells $K_j$ are geodesic segments.  We replace each bridge $J_j$ by its orthogeodesic
representative~$\bar J_j$.  As the segments $J_j$ glide to their orthogeodesic representatives
$\bar J_j$, the segments $K_j$ glide to new segments $\bar K_j$ whose endpoints match those of the
family $\bar J_j$. We call $(\bar J_j, \bar K_j)_{j \in \Z_m}$ the {\em{piecewise geodesic
    representative}} of $\Gamma$.  This allows us to define the maps
\begin{align*}
  \cT_{\g,\n}^* & \rightarrow \R_{>0} &  \cT_{\g,\n}^* & \rightarrow \R_{>0} \\
  Y & \mapsto \ell_Y(J_j):=\ell(\bar J_j) &  Y & \mapsto \ell_Y(K_j):=\ell(\bar K_j).
\end{align*}
Note that now $\bar J_j$ and $\bar K_j$ form a right angle at their meeting point, and so do
$\bar K_j$ and $\bar J_{j+1}$.

\medskip

\textbf{In the following, unless specified otherwise, we always take geodesic representatives of
  polygonal curves, and we suppress all the bars $\bar{.}$ from the notation.}
  
  \medskip

  We notice that the discussion in \S \ref{s:bound} is valid for any polygonal curve, not only the
  family $(\Gamma_\lambda)_{\lambda \in \Lambda}$ of boundary components of our pair of pants
  decomposition. In particular, if we introduce the short-hand notations
\begin{equation}
L_j:= \ell_Y(J_j)
\quad \text{and} \quad
\tau_j:= \ell_Y(K_j)\label{eq:nota_T_tau_poly}
\end{equation}
then we can apply Proposition
\ref{prp:length_g_lambda} to any polygonal curve $\Gamma$, and obtain that $  \ell_Y(\Gamma) =
Q_m(\vec{L},\vec{\tau})$ for the length-vectors $\vec{L} = (L_j)_{j \in \Z_m}$ and
$\vec{\tau}=(\tau_j)_{j \in \Z_m}$.

% More precisely, we have
% $\tilde{K}_{j-1}= [o(K_{j-1}), o(\tilde{J}_{j}]$, and similarly
% $\tilde{K}_{j+1}= [t(\tilde{J}_{j}), t(K_{j+1})]$.
 
%there is a unique geodesic segment going from $K_{j-1}$ to $K_{j+1}$, orthogonal to $K_{j-1}$ and $K_{j+1}$, homotopic to $(K_{j-1}, J_{j}, K_j, J_{j+1}, K_{j+1})$ with endpoints gliding along $K_{j-1}$ and $K_{j+1}$.
%We denote this segment by $\tilde{J}_{j}$ and 

\subsubsection{Neutralization and lengths}
\label{s:neutralizing_length}

Let us now examine the geometric impact of neutralizing cells in a polygonal curve $\Gamma$.

First, we observe that if $\Gamma'$ is obtained by neutralization of a cell $K_j$ in a polygonal
curve~$\Gamma$, then, due to taking geodesic representatives, the new cell $\tilde{K}_{j-1}$ is
strictly included in ${K}_{j-1}$, and the new cell $\tilde{K}_{j+1}$ in ${K}_{j+1}$ (see Figure
\ref{fig:neutralization_geod}). This justifies the name ``new cells'' introduced in
Definition~\ref{d:der_neutr}. This is also the case when neutralizing several cells simultanously.

We extend the notation \eqref{eq:nota_T_tau_poly} to the definitions introduced in \S
\ref{s:neutralizing_one} and \ref{s:neutralizing_I} related to new cells and bridges, so that:
\begin{itemize}
\item $\tilde{L}_j$, $\tilde \tau_{j-1}$ and $\tilde{\tau}_{j+1}$ respectively denote the lengths of
  the new bridge $\tilde{J}_j$ and the new cells $\tilde{K}_{j-1}$ and $\tilde{K}_{j+1}$ obtained
  when we neutralize the cell $K_j$;
\item $\tilde{L}_I$, $\tilde\tau_{a-1}$ and $\tilde\tau_{b+1}$ respectively denote the lengths of
  the new bridge $\tilde{J}_I$ and the new cells $\tilde{K}_{a-1}$ and $\tilde{K}_{b+1}$ obtained
  when we neutralize simultanously~$(K_j)_{j \in I}$ for $I = \{a, \ldots, b\}$.
\end{itemize}

\begin{figure}[h!]
     \centering
   \begin{subfigure}[b]{0.5\textwidth}
    \centering
     \includegraphics{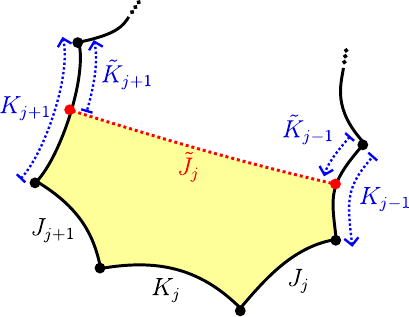}
    \caption{One cell.}
    \label{fig:neutralization_geod_one}
  \end{subfigure}%
   \begin{subfigure}[b]{0.5\textwidth}
    \centering
     \includegraphics{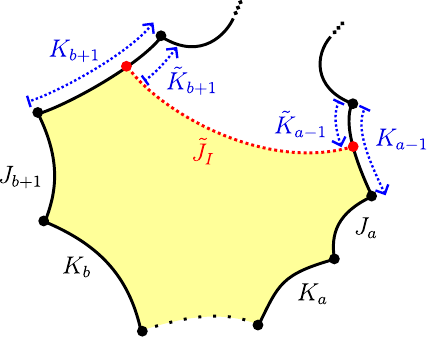}
    \caption{Several cells.}
    \label{fig:neutralization_geod_I}
  \end{subfigure}%
  \caption{Geodesic representative of a polygonal curve and of a neutralization.}
  \label{fig:neutralization_geod}
\end{figure}%

Importantly, if $\Gamma'$ is a polygonal curve derived from $\Gamma$, because $\Gamma'$ is homotopic
to $\Gamma$, we have 
\begin{equation*}
  Q_{|\Gamma|}(\vec{L},\vec{\tau})
  = \ell_Y(\Gamma)
  = \ell_Y(\Gamma')
  = Q_{|\Gamma'|}(\vec{L}', \vec{\tau}')
\end{equation*}
where $\vec{L}', \vec{\tau}'$ are the length parameters of $\Gamma'$.  As a consequence, we obtain
multiple distinct rewritings of the length of the same closed geodesic, depending on which
polygonal curve we use, $\Gamma$, or any derived polygonal curve $\Gamma'$.  Our goal in the
following will be to neutralize enough cells so that the expression of
the length is well-behaved and can be easily understood.

We  make  the following  observation.

\begin{lem}
  \label{lem:obviously}
  Let $\Gamma$ be a polygonal curve of combinatorial length $m$, and $I=\{a, \ldots, b\}$ be an
  interval of $\Z_m$. Then,
\begin{equation}\label{e:obviously}
  \begin{split}
    \tilde L_I
    & =\fn(L_{a}, \ldots, L_{b+1}, \tau_{a}, \ldots, \tau_b) \\
    \tilde \tau_{a-1}
    & =\tau_{a-1}- \fn(L_{a}, \ldots, L_{b+1}, \tau_{a}, \ldots, \tau_b)\\
    \tilde \tau_{b+1}
    & = \tau_{b+1} - \fn(L_{a}, \ldots, L_{b+1}, \tau_{a}, \ldots, \tau_b).
  \end{split}
\end{equation}
\end{lem}

This is a straightforward consequence of applying the following elementary lemma to the highlighted
right-angled polygon in Figure \ref{fig:neutralization_geod}.
  \begin{lem}\label{l:fn_polygon} Let $m \geq 3$.
    In any right-angled convex polygon $(J_j, K_j)_{j \in \Z_m}$, any three consecutive lengths can be
    written as functions of the $2m-3$ others:
    $$(\ell(K_j), \ell(J_j), \ell(K_{j+1}))
    =\fn \big((\ell(J_{k}))_{k \in \Z_m \setminus \{j\}}, (\ell(K_{k}))_{k \in \Z_m \setminus \{j,j+1\}}\big).$$
  \end{lem}
  \begin{proof}[Proof of Lemma \ref{l:fn_polygon}]
    When $m=3$, this is a direct consequence of the basic formulas for convex right-angled hexagons (see
    e.g. \cite[Theorem 2.4.1 (i-ii)]{buser1992}). For $m > 3$, we obtain the result by induction,
    introducing a right-angled hexagon at each step.
  \end{proof}

\subsection{Favorable regions for studying the length of a polygonal curve}

We are now ready to study the variations of the function 
\begin{align} \label{e:Ff}
  F_m(\vec{L}, \vec{\tau})
  =\frac12  \sum_{\alpha  \in \{0, 1\}^{m}}
  (-1)^{|\alpha|}
  \hyp_{-\partial \alpha} \div{\vec{L}} e^{ -\alpha \cdot \vec{\tau}}.
  \end{align}
  We observe that the first term, corresponding to $\alpha=(0, \ldots, 0)$, has the coefficient
  $$\hyp_{-1} \div{\vec{L}} =\prod_{j=1}^{m} \sinh \div{L_j}.$$
  We shall identify a ``favourable'' region of the Teichm\"uller space where this term dominates all
  others, so that $\cosh(\ell_Y(\Gamma)/2)$ is comparable to
  $\exp \paren*{\frac12(\sum_{j=1}^m \tau_j)} \prod_{j=1}^{m} \sinh (L_j/2)$.

\subsubsection{Height of a cell}
\label{sec:height-cell}

First, we introduce a geometric quantity associated to each cell, called its \emph{height}. We will
see in \S \ref{s:favorable} how this is the relevant quantity one needs to introduce to describe the
variations of the function $F_m$.

\begin{defa}\label{d:height}
  Let $\Gamma$ be a polygonal curve of combinatorial length $m$, and $j \in \Z_m$. The \emph{height
    of the cell $K_j$} is defined as
  \begin{equation}\label{e:defZ}
    Z^\Gamma_{K_j} :=
    \tanh(L_j)\tanh(L_{j+1}) \, e^{\tau_j}.
  \end{equation}
  We also use the short-hand notation $Z_j:= Z^\Gamma_{K_j}$, when convenient.
\end{defa}

Note that, if $m=1$, the height is defined as $Z_1=\tanh^2 (L_1)e^{\tau_1}$. If $m=0$, the polygonal
curve has no cell, so there is no notion of height.

We prove the following bounds on the height $Z_j$ of a cell.

\begin{lem}
  \label{lem:height}
  For $m \geq 2$,  $1 \leq  Z_{j} \leq e^{\tau_j}$ and  upon neutralizing the cell $K_j$, we have
\begin{equation}
  \label{e:Tj_vs_Zj}
  \tilde{L}_j \leq \log Z_j + L_j + L_{j+1}.
\end{equation}  
\end{lem}

\begin{proof}
  The upper bound $Z_{j} \leq e^{\tau_j}$ is a direct consequence of the definition. The two other
  bounds can be obtained from the trigonometric formula for right-angled hexagons, see
  e.g. \cite[Theorem 2.4.1 (i)]{buser1992}, which yield
  \begin{align}
    \label{eq:hex_Z}
    \cosh(\tilde{L}_j) = 
    \sinh(L_j) \sinh(L_{j+1}) \cosh(\tau_j) -  \cosh(L_j)\cosh(L_{j+1}).
  \end{align}
  Indeed, $\cosh(\tilde{L}_j) \geq 0$ together with \eqref{eq:hex_Z} imply that
  \begin{align*}
     \tanh(L_j) \tanh(L_{j+1}) \cosh(\tau_j) \geq  1
  \end{align*}
  and hence $Z_j \geq 1$.  Then, we rewrite \eqref{eq:hex_Z} as
    \begin{equation*}
      (\cosh (\tau_j) - 1) \sinh (L_j) \sinh (L_{j+1})
      = \cosh (\tilde{L}_j) + \cosh(L_j-L_{j+1}) \geq \cosh (\tilde{L}_j).
    \end{equation*}
    Since $\cosh (\tau_j) - 1 = 2 \sinh^2 \div{\tau_j} \leq e^{\tau_j}/2$, this implies that
    \begin{equation*}
       2\cosh (\tilde{L}_j) 
      \leq e^{\tau_j} \sinh (L_j) \sinh (L_{j+1})
      = Z_j \cosh (L_j) \cosh (L_{j+1}) 
    \end{equation*}
    which yields the result since $e^x/2 \leq \cosh(x) \leq e^x$.
\end{proof}

We notice that \eqref{e:Tj_vs_Zj} means that the height $Z_j$ of a cell $K_j$ can be used to bound
the length $\tilde{L}_j$ of the roof over it, under the additional assumption that the lengths of
the adjacent bridges are bounded. More generally, the height of the cell $K_j$ can be used to bound
the length of the orthogeodesic $W_j$ constructed below and represented in Figure \ref{fig:Wj}.

\begin{defa}\label{d:Zj}
  For $j \in \Z_m$, we define the orthogeodesic $W_j$ by:
  \begin{itemize}
  \item $W_j=K_j$ if $L_j\geq 1$ and $L_{j+1}\geq 1$;
  \item $W_j$ is the orthogeodesic between $K_{j-1}$ and $J_{j+1}$ if $L_j< 1$ and $L_{j+1}\geq 1$;
  \item $W_j$ is the orthogeodesic between $J_j$ and $K_{j+1}$ if $L_j\geq 1$ and $L_{j+1}< 1$;
\item $W_j=\tilde J_j$ otherwise, i.e. if $L_j < 1$ and $L_{j+1} < 1$.
\end{itemize}
\end{defa}

\begin{figure}[h!]
  \centering
   \begin{subfigure}[b]{0.33\textwidth}
    \centering
     \includegraphics{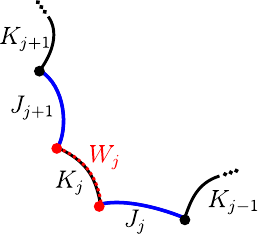}
    \caption{$L_j, L_{j+1} \geq 1$.}
    \label{fig:Wj1}
  \end{subfigure}%
   \begin{subfigure}[b]{0.33\textwidth}
    \centering
     \includegraphics{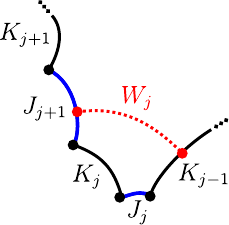}
    \caption{$L_j < 1$ and $L_{j+1} \geq 1$.}
    \label{fig:Wj2}
  \end{subfigure}%
   \begin{subfigure}[b]{0.33\textwidth}
    \centering
     \includegraphics{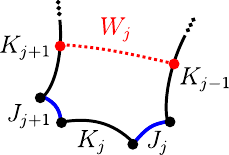}
    \caption{$L_j < 1$ and $L_{j+1} < 1$.}
    \label{fig:Wj3}
  \end{subfigure}%
  \caption{The orthogeodesic $W_j$ depending on the values of $L_j$ and $L_{j+1}$.}
    \label{fig:Wj}
  \end{figure}%

  We then have the following bound.
  \begin{lem}  \label{l:WZ}
    For any $j$, $\ell(W_j)\leq \log(Z_{j})+2$.
  \end{lem} 
  \begin{proof}
    If both $L_j$ and $L_{j+1}$ are smaller than one, then this follows from
    \eqref{e:Tj_vs_Zj}. When both $L_j$ and $L_{j+1}$ are greater or equal to $1$, it is a
    straightforward consequence of the definition and the fact that $\log(1/\tanh(1)) < 1$. In the
    last cases, say $L_j < 1$ and $L_{j+1} \geq 1$, we use the hyperbolic formulas for right-angles
    pentagons (see \cite[Theorem 2.3.4 (i)]{buser1992}) to write
    $\cosh \ell(W_j) = \sinh L_j \sinh \tau_j$, which once again yields the claim since
    $\log(2/\sinh(1)) <1$.
  \end{proof}

\subsubsection{Favourable regions and estimates on $F_m$}
\label{s:favorable}

We are ready to use our notion of height of a cell to identify a condition upon which the
leading term of $F_m$ is the term $\alpha = (0, \ldots, 0)$.

\begin{lem}\label{c:important} 
  For any polygonal curve $\Gamma$ of combinatorial length $m \geq 1$, any $Y \in
  \mathcal{T}_{\g,\n}^*$, any multi-index $\alpha \in \{0, 1\}^m$ distinct from $(0, \ldots, 0)$,
  \begin{align}\label{e:negli}
    \frac{\hyp_{-\partial \alpha} \div{\vec{L}}}{\hyp_{-1}\div{\vec{L}}} e^{ -\alpha \cdot \vec\tau}
    \leq \prod_{j \in \alpha} \frac{4}{Z_{K_j}^\Gamma}. 
  \end{align}
\end{lem}

In particular, if all the $Z_j$ go to $+\infty$, then for all $\alpha\not=(0,\ldots, 0)$,
   $$ \frac{\hyp_{-\partial \alpha} \div{\vec{L}}}{\hyp_{-1}\div{\vec{L}}} e^{ -\alpha \cdot \vec\tau} = o(1).$$
   We can then deduce that $F_m(\vec{L}, \vec\tau) \sim \hyp_{-1}(\vec{L}/2)$ and
   $ \cosh (\ell_Y(\Gamma)/2) \sim \hyp_{-1}\div{\vec{L}} \exp\paren{\frac12(\sum_{j=1}^m \tau_j)}$, as
   quantified in the following statement.

   \begin{cor}
     \label{c:Z_F_m}
     Let $\Gamma$ be a polygonal curve of combinatorial length $m \geq 1$. For a real number
     $a \geq (m+3) \log 2$, if we assume
     \begin{equation*}
       \forall j \in \Z_m, \quad Z_{K_j}^\Gamma \geq e^{a}
     \end{equation*}
     then we have
     \begin{align}\label{e:lowerF}
       F_m(\vec{L}, \vec\tau) \geq \frac14 \hyp_{-1}\div{\vec{L}}
       \quad \text{and} \quad
       \cosh\div{\ell_Y(\Gamma)} \geq \frac{e^{ma/2}}{2^{m+2}}.
     \end{align}
\end{cor}

This motivates the following definitions.

\begin{nota}
  \label{nota:a_Z}
  In the following, we let $a = a(r) = \max((2r+3) \log(2), 2 \log(8 \cosh (3/2)))$, where, recall,
  $r$ is the number of intersections of our generalized eight.  We also fix two smooth functions
  $\fcell{0}, \fcell{\infty}$ on $\R$, with the property that $\fcell{0}$ is supported in
  $[-e^{a+1}, e^{a+1}]$, $\fcell{0}\equiv 1$ on $[-e^a, e^a]$, and
  $\fcell{0}+\fcell{\infty} \equiv 1$.
\end{nota}

\begin{defa} \label{r:favour} Let $\Gamma$ be a polygonal curve of combinatorial length $m \leq
  2r$. We call $\Gamma$-{\em{favourable region}} the region of Teichm\"uller space
  $$\{Y \in \cT_{\g,\n}^* : \quad \forall j \in \Z_m, Z_{K_j}^{\Gamma}>e^a \}.$$
\end{defa}

\begin{rem} \label{rem:triv_bound_ellgamma}
  In particular, using Corollary \ref{c:Z_F_m} and the fact that $e^{a/2} \geq 8 \cosh 3/2$, we
  obtain that, for any polygonal curve $\Gamma$, if $Y$ is in the $\Gamma$-favorable region,
  then $\ell_Y(\Gamma) \geq 3$.
\end{rem}

\begin{proof}[Proof of Lemma \ref{c:important}]
  Let $m\geq 2$. For $\alpha \in \{0, 1\}^{m}$ such that $\alpha \neq (0, \ldots, 0)$, let us write
  $\{j \in \Z_m : \alpha_j=1\}$ as a disjoint union of intervals,
  \begin{align*}
    \{j \in \Z_m : \alpha_j=1 \}= \bigsqcup_{i=1}^l \{k_i, \ldots, l_i\},
  \end{align*}
  so that $\alpha_j=1$ for $k_i \leq j \leq l_i$ but $\alpha_{k_i-1}=0=\alpha_{l_i+1}$.  Then, by
  definition, $(\delta \alpha)_j = -1$ when $j= k_i$ or $l_i+1$ and $+1$ otherwise. 
 It follows that
  \begin{align*}
    \frac{\hyp_{-\partial \alpha} \div{\vec{L}}}{\hyp_{-1}\div{\vec{\vec{L}}}} \, e^{ -\alpha \cdot \tau}
   & = \prod_{i=1}^l \coth \div{L_{k_i}} \coth \div{L_{l_i+1}}
    \prod_{k_i \leq j \leq l_i} e^{-\tau_j}.
  \end{align*}
  We note that, for all $x \geq 0$, we have $\coth(x/2) \leq 2 \coth(x)$. Therefore,
  \begin{equation*}
    \frac{\hyp_{-\partial \alpha} \div{\vec{L}}}{\hyp_{-1}\div{\vec{L}}} \, e^{ -\alpha \cdot \tau}
    \leq \prod_{i=1}^l 4 \coth (L_{k_i}) \coth (L_{l_i+1})
    \prod_{k_i \leq j \leq l_i} e^{-\tau_j}.
  \end{equation*}
  We then  inject $1 \leq \prod_{k_i < j \leq l_i} 4\coth^2(L_j)$ in
  this equation, and obtain
  \begin{equation*}
    \frac{\hyp_{-\partial \alpha} \div{\vec{L}}}{\hyp_{-1}\div{\vec{L}}} \, e^{ -\alpha \cdot \tau}
    \leq \prod_{i=1}^l 
    \prod_{k_i \leq j \leq l_i} 4 \coth (L_{j}) \coth (L_{j+1}) e^{-\tau_j}
    = \prod_{j \in \alpha} \frac{4}{Z_j}
  \end{equation*}
  which is our claim.
  \end{proof}

  \begin{proof}[Proof of Corollary \ref{c:Z_F_m}]
    First we observe that, by hypothesis, for any $j$, $4/Z_j \leq 4e^{-a} \leq 2^{-m-1}$ which is
    smaller than $1$. As a consequence, for any $\alpha \neq (0, \ldots, 0)$, we have
    $\prod_{j \in \alpha} 4/Z_j \leq 2^{-m-1}$.  To then deduce the
    bound~\eqref{e:lowerF}, we write the triangle inequality:
\begin{align*}
  F_m(\vec{L}, \vec{\tau})
  \geq \frac12  \hyp_{-1} \div{\vec{L}} \Big[1-\sum_{\substack{\alpha  \in \{0, 1\}^{m} \\ \alpha\not=(0, \ldots, 0)}}
  \frac{\hyp_{-\partial \alpha} \div{\vec{L}}}{\hyp_{-1} \div{L}} e^{ -\alpha \cdot \tau}\Big]
  \geq \frac{1}{4} \hyp_{-1} \div{\vec{L}}
\end{align*}
 because the sum above is bounded above by
$\sum_{\substack{\alpha\in \{0,1\}^m}} 2^{-m-1} = 1/2$.
 There only remains to prove the last claim, which can be obtained by writing
 \begin{align*}
   \cosh\div{\ell_Y(\Gamma)}
   &=  e^{\sum_{j=1}^m \tau_j/2}  F_m(\vec{L}, \vec\tau) 
     \geq \frac14  e^{ \sum_{j=1}^m \tau_j/2} \prod_{j=1}^m \sinh \div{L_j}
 \end{align*}
 thanks to the bound on $F_m$. Then, since $\sinh(x/2)\geq \tanh(x)/2$ for all $x\geq 0$,
 \begin{equation*}
   \cosh\div{\ell_Y(\Gamma)}
    \geq \frac{1}{2^{m+2}} e^{ \sum_{j=1}^m \tau_j/2} \prod_{j=1}^m \tanh(L_j)
    = \frac{1}{2^{m+2}} \prod_{j=1}^m Z_j^{1/2}
    \geq \frac{e^{ma/2}}{2^{m+2}}.
\end{equation*}
\end{proof}

%\begin{rem}
%  \label{rem:Q1}
 % In the case $m=1$, we can adapt Lemma \ref{c:important} and Corollary \ref{c:Z_F_m}, now noting that
 % \begin{align}\label{e:Q1}
  %  \cosh \div{\ell_Y(\Gamma)}
  %  = \sinh \div{L} \sinh \div{\tau}
  %  =   \frac12 \sinh \div{L} e^{\frac{\tau}2} - \frac 12 \sinh \div{L}e^{\frac{-\tau}2}
%  \end{align}
%  and hence 
 % \begin{align*}
 %   F_1(L, \tau)= \frac12 \sinh \div{L}(1-e^{-\tau}),
%  \end{align*}
%  which satisfies $F_1(L,\tau) \geq \frac 14 \sinh \div{L}$
 % provided $Z = \tanh^2\div{L} e^\tau$ is large enough.
%\end{rem}

\subsubsection{Derivatives of $F_m$}
\label{sec:derivatives-f_m}

Let us now estimate the logarithmic derivatives of the function $F_m$ within the favorable region.

\begin{prp} \label{p:derF} Let $\Gamma=(J_j, K_j)_{j \in \Z_m}$ be a polygonal curve of
  combinatorial length $m \geq 1$.  If the metric $Y$ is in the $\Gamma$-favourable region, then for
  any subset $\pi \subseteq \Z_m$,
 \begin{align*}
   \Big| \prod_{j \in \pi} \frac{\partial}{\partial \tau_j} \big( \log F_m\big) (\vec{L}, \vec\tau) \Big|
   \leq  \fn(m)\, \prod_{j\in \pi}  \frac{1}{Z_{j}}.
 \end{align*}
 \end{prp}
 
 \begin{proof}
   Let $k= \# \pi$.  The partial derivative $\partial^\pi \log F_m$ is a linear combination (with
   universal coefficients depending only on $k$) of fractions of the form
   $\prod_{j=1}^k (\partial^{\pi_j}F_m/F_m)$ where the non-empty subsets $\pi_1, \ldots, \pi_k$ of
   $\Z_m$ form a partition of $\pi$. We then write for each $j \in \{1, \ldots, k\}$
   \begin{align*}
     \frac{\partial^{\pi_j}F_m}{F_m}
      =  \sum_{\alpha \in \{0, 1\}^m} (-1)^{|\alpha|} \frac{  
       \hyp_{-\partial \alpha} \div{\vec{L}}  \partial^{\pi_j}(e^{-\alpha \cdot \vec\tau})}{F_m}.
   \end{align*}
   To bound the denominator, we use \eqref{e:lowerF}, which states that
   $ F_m(\vec{L}, \vec\tau) \geq \frac14 \hyp_{-1}(\vec{L}/2)$. For the numerator, we note that
   $ \partial^{\pi_j}(e^{-\alpha \cdot\vec\tau})=0$ unless $\alpha_i=1$ for all $i\in \pi_j$, and in the
   latter case we have
   $ \partial^{\pi_j}(e^{-\alpha \cdot\vec\tau})=(-1)^{\#\pi_j}e^{-\alpha\cdot\vec\tau}$. In particular, the
   term $\alpha = (0, \ldots, 0)$ does not contribute to the sum. Then, 
   \begin{align*}
     \left|\frac{ \partial^{\pi_j}F_m}{F_m} \right|
     \leq 4 \sum_{\substack{\alpha \in \{0,1\}^m\\ \alpha \neq (0, \ldots, 0)}}
     \frac{ \hyp_{-\partial \alpha} \div{\vec{L}} e^{-\alpha \cdot \vec\tau}}{ \hyp_{-1}\div{\vec{L}}}
     \leq 2^{3m+2} \prod_{i\in \pi_j} \frac{1}{Z_{i}}
   \end{align*}
   by Lemma \ref{c:important}, using the fact that there are at most $2^m$ terms in the sum. 
%If $P_j$ contains intervals of lengths $\geq 2$ this bound is a bit rough, however one cannot do better if $P_j$ contains only intervals of size $1$ (in particular, if $P_j$ has one element).
Taking the product over $j=1, \ldots, k$ yields
$$\prod_{j=1}^k \left| \frac{\partial^{\pi_j}F_m}{F_m} \right|
\leq 2^{(3m+2)k} \prod_{i\in \pi} \frac{1}{Z_{i}}.$$ Summing over all partitions
$\pi_1, \ldots, \pi_k$ of $\pi$, and noticing that $k\leq m$, yields the announced bound. 
 \end{proof}

\subsubsection{Decoration for one polygonal curve}
\label{sec:decor-one-polyg}

We recall that our aim in this section is to provide a well-described formula for the length of a
polygonal curve $\Gamma$.  Proposition~\ref{p:derF} tells us that this is achieved as soon as our
metric $Y$ lies in the $\Gamma$-favorable region. If this is not the case, then we shall replace
$\Gamma$ by a derived polygonal curve $\Gamma'$ obtained by neutralizing a cell $K$ such that
$Z_K^{\Gamma} \leq e^{a}$, and hope that $Y$ now falls in the $\Gamma'$-favourable region.  Since
the combinatorial length strictly decreases under derivation, we can iterate this process finitely
many times until we reach a favourable region for a polygonal curve $\Gamma''$ derived from
$\Gamma$. Since $\Gamma''$ is homotopic to $\Gamma$, this yields a well-behaved formula for the
length $\ell_Y(\Gamma)$ in terms of the lengths of the cells and bridges of $\Gamma''$.

The \emph{decorations}, defined below, will serve to enumerate all possible cases in this recursive
process. Later on (\S \ref{s:partition}), decorations will be used to construct a partition of unity
on the Teichm\"uller space: a decoration $\infty$ on a cell $K$ will mean that we consider the
region of Teichm\"uller space where the height $Z^{\Gamma}_K$ is high, and on the opposite the
decoration $0$ will mean that we consider the region where the height $Z^{\Gamma}_K$ is low.

Let us provide the general definition of a decoration for a polygonal curve $\Gamma$.
\begin{defa}\label{d:deco}
  A decorated polygonal curve $(\Gamma, d)$ is a polygonal curve $\Gamma$ of combinatorial length
  $|\Gamma|\geq 1$, together with a map $d: \Cell(\Gamma)\To \{0, \infty\}$.  A cell
  $K\in \Cell(\Gamma)$ is called \emph{favourable} if $d(K)=\infty$ and \emph{unfavourable} if
  $d(K)=0$. A decorated polygonal curve $(\Gamma, d)$ is called {\em{favourable}} if all cells are
  favourable.

  A decorated polygonal curve $(\Gamma', d')$ is said to be {\em{derived from}} $(\Gamma, d)$ if
  $\Gamma'$ is derived from $\Gamma$ by neutralizing a cell $K\in \Cell(\Gamma)$ such that
  $d(K)=0$. In addition, we ask that the map $d'$ coincides with $d$ on
  $\Cell(\Gamma')\cap \Cell(\Gamma)$.
\end{defa}

The following definition allows us to consider the iterative derivation process described above.

\begin{defa} \label{d:exhaust}Let $\Gamma$ be a polygonal curve.
  An \emph{exhaustion} for $\Gamma$ is a sequence of decorated polygonal curves $(\mathbf{\Gamma},
  \mathbf{d})=(\Gamma^{k}, d^{k})_{1 \leq k \leq t}$, such that
\begin{itemize}
\item $\Gamma^{1}=\Gamma$;
\item for every $k<t$, $(\Gamma^{k}, d^{k})$ is not favourable, and $(\Gamma^{k+1}, d^{k+1})$ is
  derived from $(\Gamma^{k}, d^{k})$;
\item for $k=t$, $(\Gamma^{t}, d^{t})$ is either favourable or of combinatorial length $=0$.
\end{itemize}
\end{defa}

\begin{rem}
  In particular, any polygonal curve of combinatorial length $0$ is left untouched by the exhaustion
  process. It follows that, if we apply this process to a polygonal curve $\beta_\lambda$ with
  $\lambda \in \Lambdabeta$ (i.e. a component of $\beta$), then nothing happens.
\end{rem}

\begin{nota}
  \label{nota:term_ex}
  For an exhaustion $(\mathbf{\Gamma}, \mathbf{d}) = (\Gamma^k,d^k)_{1 \leq k \leq t}$ of a
  polygonal curve $\Gamma$, we denote as $(\Gamma^\infty,d^\infty) = (\Gamma^t,d^t)$ the decorated
  polygonal curve at the terminal step of the exhaustion.
\end{nota}

\subsection{Families of polygonal curves and simultanous decorations}
\label{s:simultanous}

The integral \eqref{e:int_ell}, our main focus, simultanously involves all the lengths
$(\ell_Y(\Gamma_\lambda))_{\lambda\in\Lambda}$ through the function $\bbbone_{\domain}$. We will
therefore need to consider simultaneously all the polygonal curves
$(\Gamma_\lambda)_{\lambda\in\Lambda}$. This motivates the need to study families of polygonal
curves, which we do in this section.

In the following, our families are indexed by our set $\Lambda$ introduced in \S
\ref{s:ppdecomponumb} to index our pair of pants decomposition, but the precise structure of
$\Lambda$ does not matter. The one thing that is important is that we have ordered the set $\Lambda$
with the order relation $\lambda_1<\lambda_2$ if $\step(\lambda_1)<\step(\lambda_2)$, where
$\step(\lambda)$ is the step of the construction of the pair of pants decomposition at which
$\Gamma_\lambda$ appears. The considerations below can be applied to any partially ordered set
$\Lambda$.

\subsubsection{Compatibility of polygonal curves}

When considering families of polygonal curves, we will need to understand how the different
polygonal curves in our family intersect one another. Without any hypothesis, two polygonal curves
can have quite wild intersections and combinatorics. The following notion of compatibility will be
key to restrict the possible behaviours and describe our families of polygonal curves.

\begin{defa}[Compatibility of two polygonal curves]
  \label{r:compa}
  Let $\Gamma_1$, $\Gamma_2$ be two polygonal curves based on $\beta$. We say that $\Gamma_1$ and
  $\Gamma_2$ are \emph{compatible} if:
\begin{itemize}
\item if two cells $K^1 \in \Cell(\Gamma_1)$, $K^2 \in \Cell(\Gamma_2)$ lie in the same component of
  $\beta$ {\em{and have the same orientation}}, then they are either disjoint or one is included in
  the other;
\item two bridges $J^1 \in \Br(\Gamma_1)$, $J^2 \in \Br(\Gamma_2)$ are either equal or
  disjoint (in particular, they cannot have an endpoint in common unless they are equal).
\end{itemize}
\end{defa}

As a consequence of our definition of polygonal curves, for any cells $K^1, K^2$ within one
component $\beta$ and oriented the same way, the bridges adjacent to $K^1$ and $K^2$ all lie on
the same side of that component of~$\beta$.
The compatibility condition then implies that any two compatible polygonal curves $\Gamma_1$ and
$\Gamma_2$ are homotopic to simple curves \emph{not intersecting one another}.

\begin{defa}[Inclusion of cells]
  \label{d:inclusion}
  Let $\Gamma_1$, $\Gamma_2$ be two compatible polygonal curves.  Let $K^1 \in \Cell(\Gamma_1)$ and
  $K^2 \in \Cell(\Gamma_2)$ be two cells of $\Gamma_1$, $\Gamma_2$ respectively. We say that the
  cell $K^1$ is \emph{included} in the cell $K^2$, and write $K^1 \subseteq K^2$, if $K^1$ is a subset
  of $K^2$ and their orientations are the same. We will say that the two cells \emph{coincide}, and
  write $K^1=K^2$, if $K^1\subseteq K^2$ and $K^2\subseteq K^1$. In this case, we say that $\Gamma_1$
  and $\Gamma_2$ have \emph{a cell in common}. We say that $K^1$ is strictly included in $K^2$, and
  write $K^1 \subsetneq K^2$, if $K^1 \subseteq K^2$ and $K_1 \neq K_2$.
\end{defa}

Because of the second item in Definition \ref{r:compa}, if $K^1=K^2$, then the bridges adjacent 
to this common cell  are equal.

\begin{rem} The reason we insist on only comparing cells with the same orientation is that we want
  the notion of inclusion to be purely topological, that is, independent on the choice of a
  hyperbolic metric.  If $K^1\subseteq K^2$ for topological polygonal curves $\Gamma_1, \Gamma_2$,
  then, by the Useful Remark \ref{r:useful}, we have that the geodesic representatives $\bar{K}^1$
  and $\bar{K}^2$ satisfy $\bar{K}^1 \subseteq \bar{K}^2$ in the sense of the inclusion of sets.
\end{rem}

\begin{defa}\label{r:compat} We say a family of polygonal
  curves $(\Gamma_\lambda)_{\lambda\in\Lambda}$ form a \emph{compatible family}~if
  \begin{itemize}
  \item these polygonal curves are pairwise compatible;
  \item for $K^1\in \Cell(\Gamma_{\lambda_1})$ and $K^2\in \Cell(\Gamma_{\lambda_2})$, if
    $K^1 \subsetneq K^2$, then $\lambda_2<\lambda_1$.
 \end{itemize} 
\end{defa}

\begin{exa}
  By construction, the family $(\Gamma_\lambda)_{\lambda\in\Lambda}$ of boundary components of our
  pair of pants decomposition is a compatible family of polygonal curves.
\end{exa}

\begin{nota}
  \label{nota:height_fam}
  For $\Gamma = (\Gamma_\lambda)_{\lambda \in \Lambda}$ a compatible family of polygonal curves, we
  let $$\Cell(\Gamma)=\bigcup_{\lambda\in\Lambda} \Cell(\Gamma_\lambda),$$ where the union runs over
  the $\lambda\in \Lambda$ such that $|\Gamma_\lambda|\geq 1$. For $K\in \Cell(\Gamma)$, if
  $\lambda\in \Lambda$ is such that $K\in \Cell(\Gamma_\lambda)$, then we denote by $\tilde J(K)$
  the roof over~$K$ in $\Gamma_\lambda$ and $Z_K^{\Gamma}:= Z_K^{\Gamma_\lambda}$ the height of $K$.
  We call \emph{complexity} of $\Gamma$, and denote $|\Gamma|$, its number of cells.
\end{nota}
Thanks to the compatibility condition, $\tilde{J}(K)$ and $Z_K^\Gamma$ do not depend on the choice of a
$\lambda \in \Lambda$ for which $K \in \Cell(\Gamma_\lambda)$.  The following property will allow us to
compare heights within a compatible family. It says that the height is (almost) increasing for the
inclusion relation.

\begin{prp}\label{p:6} Consider a compatible family $\Gamma=(\Gamma_\lambda)_{\lambda\in\Lambda}$. Let
  $K, K'\in \Cell(\Gamma)$, such that $K'\subseteq K$. Then
  \begin{align}\label{e:6}
    Z_{K}^{\Gamma}\geq e^{-2} Z_{K'}^{\Gamma}.
  \end{align} 
\end{prp}
 
\begin{proof}
  Let us lift the polygonal curve to the hyperbolic plane. Let $K$ and $K'$ denote two lifts of our
  cells, such that $K' \subseteq K$. Let $J_-$, $J_+$ denote lifts of the bridges arriving on $K$ and
  coming out of $K$, and similarly $J_-'$, $J_+'$ for $K'$. By definition of the inclusion of cells,
  all of those bridges lie on the same side of~$K$. We denote as $L_\pm$ and $L_\pm'$ the lengths of
  these bridges, and $\tau$, $\tau'$ the lengths of the cells $K$, $K'$.

  Let $t_\pm$ denote the distance between the origins of $J_\pm$ and $J_\pm'$ along $K$.  We shall
  only prove the property when $t_+, t_- > 0$, i.e. the bridges $J_\pm$ and $J_\pm'$ are disjoint --
  the proof is simpler when one or both of these lengths is equal to $0$.

  Let $\tilde{\beta}_+$, $\tilde{\beta}_+'$ denote the lifts of the multicurve $\beta$ on which the
  bridges $J_+$, $J_+'$ terminate. Since there is no right-angled quadrilateral in $\IH$, these two
  lifts must differ. Then, completing the picture with the orthogeodesic between $\beta_+$ and
  $\beta_+'$, we obtain a convex right-angled hexagon, see Figure \ref{fig:A6j}.

    \begin{figure}[h!]
     \includegraphics{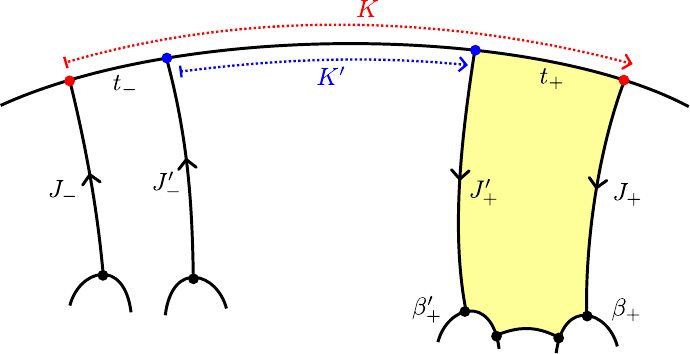}
     \caption{Illustration of the proof of Proposition \ref{p:6}.}
     \label{fig:A6j}
  \end{figure}%

  In a convex right-angled hexagon, the lengths $\ell_1, \ell_2$ of any two consecutive edges must
  satisfy $\sinh(\ell_1) \sinh(\ell_2) \geq 1$. Indeed, one can cut the hexagon along one of its
  orthogeodesics to obtain a right-angled pentagon with two consecutive edges of lengths
  $\ell_1, \ell_2$, which then must satisfy the aforementioned inequality by \cite[Therorem
  2.3.4(i)]{buser1992}.

  It then follows that $\sinh (t_+) \sinh(L_+) \geq 1$.  If $L_+ \leq 1$, then
  $\sinh(L_+) \geq \tanh(L_+)/e$ and we hence obtain
  \begin{equation}
    \label{eq:bound_comp_bridge_+}
    e^{t_+} \tanh (L_+) \geq 1/e.
  \end{equation}
  Otherwise, $e^{t_+} \geq 1$ and $\tanh(L_+) \geq \tanh(1) \geq 1/e$, so the inequality
  \eqref{eq:bound_comp_bridge_+} also holds. By symmetry, \eqref{eq:bound_comp_bridge_+} remains
  true replacing all $+$ by $-$.
  
  Then, we write $\tau = \tau' + t_+ + t_-$, and hence
  \begin{align*}
    Z_K^\Gamma
    = \tanh(L_+) \tanh(L_-) \, e^{\tau'+t_++t_-} 
    \geq \frac{e^{\tau'}}{e^2}
    \geq e^{-2} \tanh(L_+') \tanh(L_-') e^{\tau'} = e^{-2} Z_{K'}^\Gamma
  \end{align*}
  which was our claim.
  \end{proof}

\subsubsection{Neutralization in a compatible family}
\label{sec:neutr-comp-family}

In \S \ref{s:neutralizing_one}, we defined a notion of neutralization of a cell $K$ for a single
polygonal curve.  Now, we define and discuss the notion of neutralization of a cell $K$ in a
compatible family of polygonal curves $\Gamma= (\Gamma_\lambda)_{\lambda\in\Lambda}$.

  \begin{defa}\label{d:deriv2} 
    Let $\Gamma= (\Gamma_\lambda)_{\lambda\in\Lambda}$ be a family of polygonal curves. Let
    $K\in \Cell(\Gamma)$.

    For every $\lambda \in \Lambda$, we let $\Gamma'_\lambda$ be the polygonal curve derived from
    $\Gamma_\lambda$ by simultaneously neutralizing all cells $K'\in \Cell(\Gamma_\lambda)$ such that
    $K'\subseteq K$. We say that the family $\Gamma' = (\Gamma'_\lambda)_{\lambda \in \Lambda}$ is \emph{derived
      from $\Gamma$ by neutralizing $K$}.
  \end{defa}

  We show, using the Useful Remark, that the notion of compatibility is stable by derivation.
 
  \begin{prp} \label{p:compatible} If $\Gamma$ is a compatible family and $\Gamma'$ is derived from
    $\Gamma$ by neutralizing a cell, then $\Gamma'$ is a compatible family.
\end{prp}

\begin{proof}
  Let $\tilde{K}^1, \tilde{K}^2 \in \Cell(\Gamma')$ be two cells with non-empty intersection and oriented
  the same way. We wish to prove that either $\tilde{K}^1 \subseteq \tilde{K}^2$ or
  $\tilde{K}^2\subseteq \tilde{K}^1$, and that the bridges adjacent to these cells are equal or disjoint.

  For $i = 1, 2$, let $\lambda_i \in \Lambda$ denote an index for which
  $\tilde{K}^i \in \Gamma'_{\lambda_i}$. By definition, $\Gamma'_{\lambda_i}$ is derived from $\Gamma_{\lambda_i}$
  by neutralizing some cells, and therefore the cell $\tilde{K}^i$ is contained in a cell $K^i$ of
  $\Gamma_{\lambda_i}$. 

  By the compatibility condition, two lifts of $\Gamma_{\lambda_1}$ and $\Gamma_{\lambda_2}$ to $\IH^2$ cannot
  intersect transversally. Consider two such lifts, say $(J_j^1, K_j^1)_{j\in \Z}$ and
  $(J^2_j, K^2_j)_{j\in \Z}$, chosen so that $K_0^1$ and $K_0^2$ are respective lifts of $K^1$ and
  $K^2$ and $K_0^1 \cap K_0^2 \neq \emptyset$.

  Let us consider the index
  \begin{equation*}
    m =\inf\{j\geq 0 :   t(K^1_j) \not=  t(K^2_j)\},
  \end{equation*}
  i.e. the first place where the two polygonal curves $\Gamma_{\lambda_1}$ and $\Gamma_{\lambda_2}$ start
  disagreeing after meeting along $K_0^1$ and $K_0^2$ (see Figure \ref{fig:compatible}).
  First, we observe that if $m = +\infty$, then $\Gamma_{\lambda_1}= \Gamma_{\lambda_2}$ and thus
  $\Gamma'_{\lambda_1}= \Gamma'_{\lambda_2}$, which is enough to conclude.

  If now $m < \infty$, then $J^1_{j}=J^2_{j}$ for all $j \in \{1, \ldots, m\}$. Thus, if $m>0$, then
  $K^1_{m}$ and $K^2_{m}$ have the same origin and, in particular, $K^1_m$ and $K_m^2$ are two
  distinct cells, with an intersection and oriented the same way. This is obvisouly also true by
  hypothesis if $m=0$. Now, because $\Gamma_{\lambda_1}$ and $\Gamma_{\lambda_2}$ are compatible, we
  either have $K_m^2 \subsetneq K_m^1$ and $\lambda_1 < \lambda_2$, or $K_m^1 \subsetneq K_m^2$ and
  $\lambda_2 < \lambda_1$ -- let us assume w.l.o.g. that we are in the first situation.

  If we now consider the index
  \begin{equation*}
      n=\sup\{j\leq 0,  o(K^1_j)\not=  o(K_j^2)\},
  \end{equation*}
  since we assumed that $m < + \infty$, $\Gamma_{\lambda_1} \neq \Gamma_{\lambda_2}$, and hence
  $n > - \infty$. By the same reasoning as above, the cells $K_n^1$ and $K_n^2$ are distinct, with
  an intersection, and oriented the same way. Since we have now specified that
  $\lambda_1 < \lambda_2$, we must have $K_n^2 \subsetneq K_n^1$.
 
  \begin{figure}[h!]
    \includegraphics{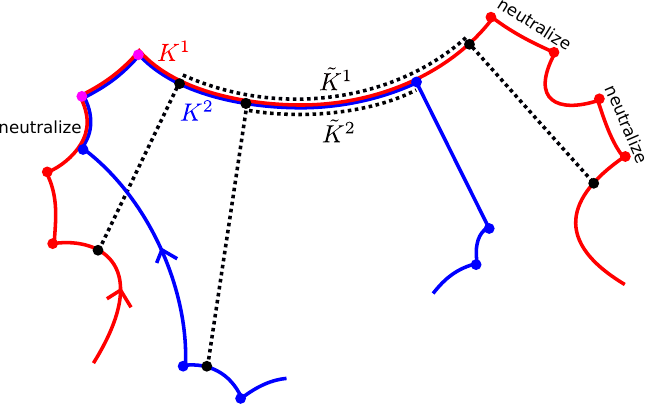}
    \caption{Example of two compatible polygonal curves under a neutralization procedure. The full
      lines are the initial polygonal curves. The dotted lines are the new cells $\tilde{K}^1$ and
      $\tilde{K}^2$ that we need to compare, and the dashed lines the new bridges. On this example,
      we have $m=0$ and $n=-1$.}
    \label{fig:compatible}
  \end{figure}%

  Because two lifts of $\Gamma_{\lambda_1}$ and $\Gamma_{\lambda_2}$ cannot cross transversally, our lift of
  $\Gamma_{\lambda_2}$ then lies entirely on the right of our lift of $\Gamma_{\lambda_1}$, as represented in
  Figure \ref{fig:compatible}.

  Now, for $i = 1, 2$, let $\tilde{J}^i$ denote the bridge coming out of $\tilde{K}^i$ in
  $\Gamma_{\lambda_i}'$. By definition of the neutralization procedure, the bridge $\tilde{J}^i$ is the
  orthogeodesic between the cell $K^i_0$ and a cell $K_{j_i}^i$ for an index $j_i \geq 0$, and the
  cells $K_j^i$ for $0 < j <j_i$ are neutralized. There are now two cases to consider.
  \begin{itemize}
  \item If $j_2 \geq m$, we can use the Useful Remark \ref{r:useful} to deduce that the origin of
    $\tilde{J}^2$ lies on the right of the origin of $\tilde{J}^1$ on $K_0^1$, and the bridges
    $\tilde{J}^1$ and $\tilde{J}^2$ are equal or do not cross, which is exactly our claim.
  \item Otherwise, $j_2 < m$, or, in other words, the two polygonal curves $\Gamma_{\lambda_1}$ and
    $\Gamma_{\lambda_2}$ coincide up until the index $j_2$. By definition of the neutralization for a
    family of curve, if a cell $K^2_{j}$ of $\Gamma_{\lambda_2}$ is contained in a cell $K_j^1$ of
    $\Gamma_{\lambda_1}$ and if $K_j^1$ is neutralized, then $K^2_{j}$ must also be
    neutralized. Therefore, $j_1 \leq j_2$, and we can once again use the Useful Remark to conclude.
  \end{itemize}
  The same reasoning also allows to prove that the origin $o(\tilde K^2)$ lies on the left of
  $o(\tilde K^1)$, and the bridges arriving on $\tilde{K}^1$ and $\tilde{K}^2$ are equal or disjoint.
\end{proof}

 \subsubsection{Simultaneous decoration}
\label{sec:simult-decor}
We are now ready to extend the notion of decoration introduced in \S \ref{sec:decor-one-polyg} to
families of polygonal curves.
 
\begin{defa}\label{d:deco2} Let $\Gamma=(\Gamma_\lambda)_{\lambda\in\Lambda}$ be a compatible family of
  polygonal curves.  A \emph{(simultaneous) decoration} of $\Gamma$ is a map
  $d: \Cell(\Gamma)\To \{0, \infty\}$ such that any cell common to several $\Gamma_\lambda$ bears only one
  decoration.
\end{defa}

\begin{defa}
  Let $(\Gamma,d)$ be a decorated family. A cell $K\in \Cell(\Gamma)$ is called \emph{unfavourable}
  (for the decoration $d$) if $d(K)=0$ or if $K\subseteq K'$ with $d(K')=0$. It is called
  \emph{favourable} otherwise.  A \emph{maximal unfavourable cell} is a cell which is unfavourable
  and maximal for inclusion.

  A simultaneous decoration $d$ of the family $\Gamma$ is called {\em{favourable}} if all its cells are
  favourable. 
\end{defa}

By convention, recall that polygonal curves of combinatorial length $0$ have no cells. As a
consequence, if $\Gamma = (\Gamma_\lambda)_{\lambda \in \Lambda}$ contains only polygonal curves
such that $|\Gamma_\lambda|=0$, then $\Cell(\Gamma) = \emptyset$ and in particular $\Gamma$ is
favorable.

\begin{defa}
  A decorated family $(\Gamma', d')$ is said to be {\em{derived from}} $(\Gamma, d)$ if $\Gamma'$ is derived from
  $\Gamma$ by neutralizing a maximal unfavourable cell and the map $d'$ coincides with $d$ on all cells
  in common between $\Gamma$ and $\Gamma'$.
\end{defa}

\begin{defa} \label{d:exhaust2}Let $\Gamma=(\Gamma_\lambda)_{\lambda\in\Lambda}$ be a compatible family of
  polygonal curves.  An \emph{exhaustion} of $\Gamma$ is a sequence of decorated families
  $(\mathbf{\Gamma}, \mathbf{d})=(\Gamma^{k}, d^{k})_{1 \leq k \leq t}$, where:
  \begin{itemize}
  \item $\Gamma^{1} =\Gamma$;
  \item for every $1 \leq k \leq t$, $\Gamma^{k}=(\Gamma^{k}_\lambda)_{\lambda\in\Lambda}$ is a
    compatible family of polygonal curves;
  \item for every $1 \leq k<t$, $(\Gamma^{k}, d^{k})$ is not favourable, and
    $(\Gamma^{k+1}, d^{k+1})$ is derived from $(\Gamma^{k}, d^{k})$;
\item for $k=t$, $(\Gamma^{t} , d^{t})$ is favourable.
\end{itemize}
We let $\Cell(\mathbf{\Gamma})=\bigcup_{k=1}^t \Cell(\Gamma^k)$.  We denote as
$|\mathbf{\Gamma}| = \# \Cell(\mathbf{\Gamma})$ the \emph{complexity} of the
exhaustion~$\mathbf{\Gamma}$.  If $K\in \Cell(\mathbf{\Gamma})$, we denote $\mathbf{d}(K)= d^k (K)$
for any $k$ such that $K\in\Cell(\Gamma^k)$.
\end{defa} 
Note that, by construction, the value of the decoration $\mathbf{d}(K) = d^k(K)$ is independent from
the choice of an index $k$ such that $K \in \Cell(\Gamma^k)$. This is also true of the roof over a
cell $\tilde{J}(K)$ and height of a cell $Z_K^{\mathbf{\Gamma}}$ introduced in Notation \ref{nota:height_fam}.

\begin{nota}[Terminal cells] \label{d:terminal} Let
  $(\mathbf{\Gamma}, \mathbf{d})=( \Gamma^{k}, d^{k})_{1 \leq k \leq t}$ be a sequence of decorated
  families of polygonal curves. We denote as $(\Gamma^\infty,d^\infty)$ the decorated polygonal
  curve $(\Gamma^t,d^t)$. The cells of the last family $\Gamma^{\infty}$ are called the {\em{terminal
      cells}} of $(\mathbf{\Gamma}, \mathbf{d})$. The set of terminal cells associated with
  $(\mathbf{\Gamma}, \mathbf{d})$ is denoted $\Term({\mathbf{\Gamma}}, {\mathbf{d}})$, in other
  words
  $$\Term({\mathbf{\Gamma}}, {\mathbf{d}}):= \Cell(\Gamma^{\infty}).$$
\end{nota}

\subsection{Another partition of unity on Teichm\"uller space}\label{s:partition}
\label{s:supercutoff}

Let us now focus on our compatible family of polygonal curves
$\Gamma=(\Gamma_\lambda)_{\lambda \in \Lambda}$ defined as the boundary components of the pair of
pants decomposition introduced in \S \ref{s:ppdecompo}.

We recall that we have fixed a cutoff value $a = a(r)$ and associated cutoff functions
$\fcell{0}, \fcell{\infty}$ such that $\fcell{0} + \fcell{\infty}=1$ in Notation \ref{nota:a_Z}.

\begin{nota}
  Let $({\mathbf{\Gamma}}, {\mathbf{d}}) = (\Gamma^k,d^k)_{1 \leq k \leq t}$ be an exhaustion of the
  family $(\Gamma_\lambda)_{\lambda\in\Lambda}$. For a hyperbolic metric
  $Y \in \mathcal{T}_{\g,\n}^*$, we let
  \begin{align}\label{e:THEtestfn}
    \Psi_{{\mathbf{\Gamma}}, {\mathbf{d}}}(Y)
    :=  \prod_{K\in \Cell(\mathbf{\Gamma})}\fcell{\mathbf{d}(K)}({Z^\mathbf{\Gamma}_K}).
  \end{align}
\end{nota}
The function $\Psi_{{\mathbf{\Gamma}}, {\mathbf{d}}}$ is a cut-off function on the Teichm\"uller space
$\cT^*_{\g, \n}$. It achieves our purpose described throught this section. Indeed, if $Y
\in \cT_{\g,\n}^*$ lies in the support of $\Psi_{\mathbf{\Gamma},\mathbf{d}}$, then for any cell $K \in
\Cell(\mathbf{\Gamma})$,
\begin{itemize}
\item if $\mathbf{d}(K) = 0$, then $Z_K^{\mathbf{\Gamma}} \leq e^{a+1}$;
\item if $\mathbf{d}(K) = \infty$, then $Z_K^{\mathbf{\Gamma}} \geq e^a$.
\end{itemize}

We can extract from the family of functions
$\Psi_{{\mathbf{\Gamma}}, {\mathbf{d}}}$ with $(\mathbf{\Gamma}, \mathbf{d})$ an exhaustion a
partition of unity on Teichm\"uller space~:

\begin{prp}\label{p:partition}
There exists a family of exhaustions $\mathfrak{E}$ of the family $(\Gamma_\lambda)_{\lambda\in\Lambda}$ such that
\begin{align}\label{e:partition}
  \forall Y \in \mathcal{T}_{\g,\n}^*, \quad 
  \sum_{({\mathbf{\Gamma}},  {\mathbf{d}})\in \mathfrak{E}}\Psi_{{\mathbf{\Gamma}},  {\mathbf{d}}}(Y) = 1. 
\end{align}
\end{prp}

We omit the proof, a tedious but surpriseless recursive procedure on the number of cells of the
family of polygonal curves $(\Gamma_\lambda)_{\lambda \in \Lambda}$.

 \subsection{Neutral variables associated with $({\mathbf{\Gamma}}, {\mathbf{d}})$.} \label{e:neutr-active}

 Let us fix an exhaustion $({\mathbf{\Gamma}}, {\mathbf{d}}) = (\Gamma^k, d^k)_{1 \leq k \leq t}$ of the
 family $(\Gamma_\lambda)_{\lambda \in \Lambda}$ of boundary components of our pair of pants
 decomposition of $\mathbf{S}$. On the support of $\Psi_{{\mathbf{\Gamma}}, {\mathbf{d}}}$, we shall
 enlarge the set of neutral variables originally defined in \S \ref{s:act_ntr}, and hence diminish
 the number of variables with respect to which we need to apply the operator $\cL$ under the
 integral \eqref{e:int_ell}.

 Amongst the terminal cells of $(\mathbf{\Gamma},\mathbf{d})$, some are particularly important.
\begin{nota}\label{nota:termbc}
  We denote as $\TermBC \subseteq \Term(\mathbf{\Gamma},\mathbf{d})$ the set of cells of
  $(\Gamma_\lambda^\infty)_{\lambda \in \LambdaBCFN}$. We call these cells \emph{boundary terminal cells}.
\end{nota}

We observe that, by definition of the neutralization procedure, any $K \in \TermBC$ is included in a
cell of the initial polygonal curve $(\Gamma_\lambda^1)_{\lambda \in \LambdaBCFN}$. The cells of
this family are amongst the segments $(\cK_q)_{q \in \Theta}$ defined in \S \ref{s:tau}. We also
note that each segment $(\cK_{q})_{q \in \Theta}$ can contain at most one boundary terminal cell.

 We now recall that, in \S \ref{s:tau}, we have defined for each $q \in \Theta$ a set
 $\Theta_\tau(q) \subseteq \Theta$ such that we can write
 the length $\tau_q$ of $\cK_q$ as $\tau_q = \sum_{q' \in \Theta_\tau(q)} \theta_{q'}$.

 \begin{defa}
   \label{defa:surv}
   Let $q \in \Theta$.  We call $\cK_q$ a \emph{surviving boundary segment} if $\cK_q$ contains a
   boundary terminal cell. In that case, we let
   \begin{align}\label{e:bigZ}Z_{{\mathbf{\Gamma}}, {\mathbf{d}}}(\cK_q):=Z_K^{{\mathbf{\Gamma}}}
   \end{align}
   where $K$ is the unique boundary terminal cell contained in $\cK_q$.  We denote as $\Thetasurv$
   the set of indices $q \in \Theta$ such that $\cK_q$ is a surviving boundary segment and
   $\tausurv = (\tau_q)_{q \in \Thetasurv}$.
 \end{defa}

 \begin{nota}
   \label{nota:neutral_cell}
   We define additional {\em{neutral variables}} associated to the exhaustion
   $(\mathbf{\Gamma}, \mathbf{d})$, 
   \begin{align}\label{e:NGd}
     \ThetaNe({\mathbf{\Gamma}}, {\mathbf{d}})
     := \bigcup_{q \in \Theta \setminus \Thetasurv} \Theta_\tau(q).
   \end{align}
   For a vector $\vec{\theta} \in \R^\Theta$, we denote as
   $\thetaNe({\mathbf{\Gamma}}, {\mathbf{d}}) = (\theta_{q})_{q \in \ThetaNe({\mathbf{\Gamma}}, {\mathbf{d}})}
   \in \R^{\ThetaNe(\mathbf{\Gamma},\mathbf{d})}$ the restriction of the vector~$\vec\theta$ to the
   indices in $\ThetaNe({\mathbf{\Gamma}}, {\mathbf{d}})$. 
 \end{nota}

 \begin{rem}
   \label{rem:ntr_comb_0}
   If a boundary component $\lambda\in\LambdaBC$ is reduced in the exhaustion procedure to a
   component of combinatorial length $0$, i.e. if $|\Gamma^\infty_\lambda|=0$, then all variables
   $\theta_q$ with $\lambda(q) = \lambda$ are neutral variables.
 \end{rem}

 Lemma \ref{lem:obviously} allows us to draw the following observations as to the variables
 appearing in a terminal polygonal curve.
 \begin{lem}
   \label{lem:expre_neutr_exh}
   Let $(\mathbf{\Gamma},\mathbf{d})$ be an exhaustion of $(\Gamma_\lambda)_{\lambda \in
     \Lambda}$. Then,
   \begin{itemize}
   \item for any $q \in \Thetasurv$, if $K$ is the unique boundary terminal cell contained in $\cK_q$,
     \begin{align}\label{e:relationK}
       \tau_{q}=\ell(K)
       + \fn(\vec{L},  \thetaNe({\mathbf{\Gamma}}, {\mathbf{d}}));
     \end{align}
   \item any bridge $J \in \Br(\Gamma^\infty)$ of the terminal family of polygonal curves
     $\Gamma^\infty$ satisfies
     \begin{align}\label{e:relationJ}\ell(J)
       =\fn(\vec{L},  \thetaNe({\mathbf{\Gamma}}, {\mathbf{d}})).
     \end{align}   
\end{itemize}
\end{lem}

\subsection{Structure and derivatives of the function
  $ \Psi_{{\mathbf{\Gamma}}, {\mathbf{d}}}(Y)$}

Later on, we will need to be able to estimate the derivatives of the function $
\Psi_{{\mathbf{\Gamma}}, {\mathbf{d}}}(Y)$ with respect to certain variables. 
We prove the following.

\begin{prp}\label{prp:der_psi}
  We can rewrite $\Psi_{\mathbf{\Gamma},\mathbf{d}}(Y)$ under the form
  \begin{equation*}
    \Psi_{\mathbf{\Gamma},\mathbf{d}}(Y)
    = \tilde{\Psi}_{\mathbf{\Gamma},\mathbf{d}}(\vec\tau, \vec{L}, \thetaNe(\mathbf{\Gamma}, \mathbf{d}))
  \end{equation*}
  for a smooth function
  $\tilde{\Psi}_{\mathbf{\Gamma},\mathbf{d}} : \R_{>0}^{\Theta}\times \R_{>0}^r \times
  \R^{\ThetaNe(\mathbf{\Gamma},\mathbf{d})} \rightarrow \R$ satisfying, for any set
  $\tilde{\pi} \subseteq \Theta$,
  \begin{equation*}
    \Big|
    \frac{\partial^{|\tilde{\pi}|} \tilde{\Psi}_{\mathbf{\Gamma},\mathbf{d}}}
    {\prod_{q \in \tilde{\pi}}\partial \tau_q}\Big|
    \leq \fn(|\mathbf{\Gamma}|)
    \Big(\max_{0 \leq k \leq |\tilde{\pi}|}
    \sup_{x \geq 0}
    \big( |\bfcell{0}^{(k)}(x)|
   + |\bfcell{\infty}^{(k)}(x)| \big)  \Big)^{|\mathbf{\Gamma}|}
 \end{equation*}
  where $\bfcell{d}(x):= \fcell{d}(e^x)$ for $d \in \{0, \infty\}$.
\end{prp}

Recall that the quantity $|\mathbf{\Gamma}|$ is the complexity of $\mathbf{\Gamma}$, i.e. its number
of cells, which is bounded above by $2 r \# \Lambda \leq 6r^2$.

\begin{proof}
  Let us first make a useful observation.  For any cell $K \in \Cell(\mathbf{\Gamma})$, there exists
  $K^1 \in \Cell(\Gamma^1)$ containing $K$. As a consequence, $K = K^1 \setminus I_K$ for a
  union of intervals $I_K$, corresponding to some intervals that were removed in the neutralization
  procedure.  Recalling that the cells of $\Gamma^1$ are amongst the collection of segments
  $(\cK_{qq'})_{q,q'\in \Theta}$ introduced in \S \ref{s:tau}, we obtain that there exists a set of
  indices $\Theta_\Gamma(K) \subseteq \Theta$ such that
 \begin{equation}
   \label{eq:thetaK}
   \ell_Y(K)=\sum_{q\in \Theta_\Gamma(K)}\tau_q- \ell_Y(I_K).
 \end{equation}

 We now write 
 \begin{align}\label{e:uKTK}
   \Psi_{{\mathbf{\Gamma}}, {\mathbf{d}}}(Y)
   = \prod_{K\in \Cell(\mathbf{\Gamma})} \bfcell{\mathbf{d}(K)}(\ell_Y(K) + T_K)
 \end{align}
 where for $K \in \Cell(\mathbf{\Gamma})$,
 \begin{itemize}
 \item $\ell_Y(K)$ is the length of the cell $K$;
 \item if $J_\pm(K)$ denote the two bridges adjacent to $K$, then
   $$T_K=\log \tanh (\ell_Y(J_-(K)))+ \log \tanh (\ell_Y(J_+(K))).$$
 \end{itemize}
 By \eqref{eq:thetaK} this is equal to
 \begin{align}\label{e:uKTK}
   \Psi_{{\mathbf{\Gamma}}, {\mathbf{d}}}(Y)
   = \prod_{K\in \Cell(\mathbf{\Gamma})} \bfcell{\mathbf{d}(K)}
   \Big(\sum_{q\in \Theta_\Gamma(K)}\tau_q- \ell_Y(I_K) + T_K\Big).
 \end{align}
 We prove, by the same method as Lemma \ref{lem:expre_neutr_exh}, that for all
 $K \in \Cell(\mathbf{\Gamma})$,
 $$T_K-\ell_Y(I_K) = \fn(\vec{L},\thetaNe(\mathbf{\Gamma},\mathbf{d})).$$
 Calling $f_K$ this function, define
 \begin{equation*}
   \tilde{\Psi}_{\mathbf{\Gamma},\mathbf{d}}(\vec{\tau}, \vec{L},
   \thetaNe(\mathbf{\Gamma},\mathbf{d}))
   := \prod_{K \in \Cell(\mathbf{\Gamma})}
   \bfcell{\mathbf{d}} \Big( \sum_{q \in \Theta_\Gamma(K)} \tau_q + f_K(\vec{L},\thetaNe(\mathbf{\Gamma},\mathbf{d})) \Big).
 \end{equation*}
 The upper bound on the derivatives is a simple application of the chain rule.
\end{proof} 

\section{Proof of the main theorem for generalized eights}
\label{s:proof}
We are now ready to prove the main result, Theorem \ref{t:main}, for generalized eights.

\subsection{Writing as a pseudo-convolution}\label{s:theintegral}

We recall that, in \S \ref{s:noncross}, we have reduced the problem to the study of the integral
\eqref{e:int_ell} 
\begin{align*} 
  \mathbf{Int}(\ell)=
  \int_{h(\thetaAc)=\ell}
  \varphi_0(\thetaAc)
  \prod_{q \in \ThetaAc} \theta_q^{\rK_q} e^{\theta_q}
  \,  \frac{  \d \thetaAc}{\d \ell}
\end{align*}
with the function $h$ equal to the length of $\mathbf{c}$ in the coordinates $(\vec{L}, \vec{\theta})$,
see \eqref{e:spec_h}, and the density~$\varphi_0$ given in \eqref{e:spec_phi}. Under this form, the
integral $\mathbf{Int}(\ell)$ is already written as a pseudo-convolution of Friedman--Ramanujan
functions. However, as seen in \S \ref{s:problems}, we cannot directly apply Theorem
\ref{t:intermediate} (or its variants) to it to conclude to Theorem \ref{t:main}.

We shall therefore input the considerations of \S \ref{s:lengthc} and \ref{s:lengthboundary}.  More
precisely, let us use the partitions of unity $(\Psi_{\xi, \cQ})_{\xi \in \Xi, \cQ \subseteq I(\xi)}$
and $(\Psi_{\mathbf{\Gamma}, \mathbf{d}})_{(\mathbf{\Gamma},\mathbf{d}) \in \mathfrak{E}}$ defined in \S
\ref{s:cutoffQ} and \ref{s:partition} to subdivide the integral into
\begin{align*}
  {\mathbf{Int}}(\ell)=\sum_{(\xi, \cQ), (\mathbf{\Gamma}, \mathbf{d})}
  {\mathbf{Int}}_{(\xi, \cQ), (\mathbf{\Gamma}, \mathbf{d})}(\ell),
\end{align*}
where $\xi$ ranges over $\Xi$, $\cQ$ over the subsets of $I(\xi)$, $(\mathbf{\Gamma}, \mathbf{d})$
over the set of exhaustions $\mathfrak{E}$ constructed in Proposition \ref{p:partition}, and
\begin{equation*}
  {\mathbf{Int}}_{(\xi, \cQ), (\mathbf{\Gamma}, \mathbf{d})}(\ell)
  =   \int_{h(\thetaAc)=\ell}
  \varphi_0(\thetaAc) \Psi_{ \xi, \cQ}(\vec{L}, \vec{\theta})
  \Psi_{\mathbf{\Gamma}, \mathbf{d}}(\vec{L}, \vec{\theta})
  \prod_{q \in \ThetaAc} \theta_q^{\rK_q} e^{\theta_q}
  \,  \frac{  \d \thetaAc}{\d \ell}.
\end{equation*}

We now further enlarge the set of neutral parameters.

\begin{nota}
  We define the full set of neutral parameters as 
  $$\ThetaNe'=\ThetaNe\cup \ThetaNe(\xi, \cQ)\cup \ThetaNe(\mathbf{\Gamma}, \mathbf{d}),$$
  where $\ThetaNe(\xi, \cQ)$ and $\ThetaNe(\mathbf{\Gamma}, \mathbf{d})$ are defined in
  \eqref{e:neutralxiQ} and \eqref{e:NGd} respectively. The \emph{active variables} are
  $\thetaAc' := (\theta_q)_{q\in\ThetaAc'}$ with $\ThetaAc' = \Theta \setminus \ThetaNe'$. We denote
  $\thetaAc^{\sim} = (\tilde{\theta}_q)_{q \in \ThetaAc'}$ the set of modified $\theta$ parameters
  introduced in Notation \ref{nota:erase_lengths}.
\end{nota}
From now on, all the \emph{neutral variables} $\thetaNe' := (\theta_q)_{q\in \ThetaNe'}$ will be
frozen, and any function which depends only on the neutral variables may be treated as a
constant. Our variables will be the $\thetaAc^\sim$, which are always $\geq \log(2)$, and hence
circumvent the problems raised by crossings.  We saw in Lemma \ref{lem:new_var_erase} that
$\tilde \theta_q$ and $\theta_q$ just differ by a constant (that is, a function of neutral
variables).

We shall therefore consider the restriction of the integral
${\mathbf{Int}}_{(\xi, \cQ), (\mathbf{\Gamma}, \mathbf{d})}$ defined as
\begin{equation}
  \label{e:mainfocusbis} 
  {\mathbf{Int}}'_{(\xi, \cQ), (\mathbf{\Gamma}, \mathbf{d})}(\ell)
  =   \int_{h(\thetaAc^\sim)=\ell}
  \varphi_0(\thetaAc) \Psi_{ \xi, \cQ}(\vec{L}, \vec{\theta})
  \Psi_{\mathbf{\Gamma}, \mathbf{d}}(\vec{L}, \vec{\theta})
  \prod_{q \in \ThetaAc'} \theta_q^{\rK_q} e^{\theta_q}
  \,  \frac{  \d \thetaAc^\sim}{\d \ell}.
\end{equation}
Note that the only modification is that we have removed the functions
$\theta_q^{\rK_q} e^{\theta_q}$ for $q\in \ThetaNe' \setminus \ThetaNe$ (now constant), and
restricted the integral to the variables $\thetaAc^\sim$, which are identical to $\thetaAc'$, up to
an additive function of the neutral parameters.

We can now see the integral ${\mathbf{Int}}'_{(\xi, \cQ), (\mathbf{\Gamma}, \mathbf{d})}(\ell)$ as the
$(h, \varphi)$-convolutions of the functions $\theta_q^{\rK_q} e^{\theta_q}$, indexed by
$q\in \ThetaAc'$ of cardinal $\nac = \# \ThetaAc'$, where
\begin{itemize}
\item the function $h$ is now given by \eqref{e:newc}, which reads
 \begin{align*}
   h(\thetaAc^\sim)
   =\ell_Y({\mathbf{c}})
   = \sum_{i=1}^{\cc} M_{n'_i}\Big( (\eps \tilde L_q)_{q\in I(\xi^i)}, (\tilde \theta_q)_{q\in I(\xi^i)}\Big)
 \end{align*}
 where $n_i' := \# I(\xi^i)$ is the number of non-erased $\theta$ parameters in the $i$-th component
 $\mathbf{c}_i$, and $\tilde{L}_q, \tilde{\theta}_q$ are the modified parameters introduced in \S
 \ref{s:newexp} after erasing;
\item the function $\varphi$  given by
   \begin{align}\label{e:varphi}
     \varphi(\thetaAc^\sim)
     = \Psi_{ \xi, \cQ}(\vec{L}, \vec{\theta}) \,
     \Psi_{\mathbf{\Gamma}, \mathbf{d}}(\vec{L}, \vec{\theta}) \,
     \bbbone_{\domain}(\vec{L}, \vec{\theta})
     \prod_{\substack{\lambda \in V_+^\beta}} \fz(y_\lambda)^2
     \prod_{\lambda\in \VFN\setminus \WBCFN}
     \Big(\log F_{m_\lambda}(\vec{L}^\lambda, \vec\tau^\lambda) \Big)^{d_\lambda}.
  \end{align}
\end{itemize}

% Thanks to the relation \eqref{e:chasles2}, $\tilde \theta_q$ is just a translate of $\theta_q$ by a ``constant'' function, i.e. a function of the neutral variables. 
% In the purely non-crossing case, $I(\xi^i)=\Theta^i$, $\tilde \theta_q=  \theta_q$ and $\tilde L_q=L_q$.
%   On the support of $\Psi_{ \xi, \cQ}$, all variables $(\tilde \theta_q)_{q\ThetaAc'}$ are
%   greater than $\log 2$.

In the next two sections, we show that the functions $h$ and $\varphi$ have the structure described
in the second variant of Theorem \ref{t:intermediate}, presented in~\S \ref{s:variant}.

\subsection{Properties of the function $h$} \label{s:checkingh}

By definition, on the support of the function $\Psi_{\xi,\cQ}$, we have
$\tilde{\theta}_q \geq \log(2)$ for every $q \in \ThetaAc'$. We can therefore apply Proposition
\ref{p:mainclass}, Corollary \ref{r:argcosh} and Remark~\ref{r:partition} and prove that the
function $h$, seen as a function of the variables $\thetaAc^\sim$, belongs in
the class $\cE_{\nac}^{(\log 2)}$ introduced in Definition \ref{def:cE}.
We can then deduce that the function $h$ is of class $\cC^\infty$ on the support of
$\Psi_{\xi,\cQ}$, and that on this set,
\begin{itemize}
\item for every $q\in \ThetaAc'$,  
\begin{align*} 
  |\partial_q h(\thetaAc^\sim)-1| \leq C e^{- \tilde \theta_q};
\end{align*}
\item for every $\pi\subseteq \ThetaAc'$ such that $|\pi|\geq 2$, 
\begin{align*} 
  | \partial^\pi h(\thetaAc^\sim)| \leq C e^{-\sum_{q\in \pi} \tilde \theta_q},
\end{align*}
\end{itemize}
where $C=C(\log 2,n)$ is the constant from Proposition \ref{p:mainclass}. As a consequence, the
function $h$ satisfies the hypothesis \textbf{($h$)} defined in Notation \ref{nota:assum_1} with
respect to the variables $\thetaAc^\sim$, with the constant $C$.

\subsection{Derivatives of the function $\varphi$}
\label{s:checkingphi}

We shall rewrite the function $\varphi$ under the special form required in \eqref{e:special}, i.e.
split it into a product $\varphi=\phi\times \psi$ where
\begin{align*} 
  \psi(\thetaAc^\sim)
  =  \Psi_{ \xi, \cQ}(\vec{L}, \vec{\theta}) \,  \bbbone_{\domain} (\vec{L}, \vec{\theta})  
  \prod_{\substack{\lambda \in V_+^\beta}} \fz(y_\lambda)^2
\end{align*} 
and 
\begin{align*}   
  \phi(\tausurv)
  =   \Psi_{\mathbf{\Gamma}, \mathbf{d}}(\vec{L}, \vec{\theta})
  \prod_{\lambda\in \VFN \setminus \WBCFN}
  \Big(\log F_{m_\lambda}(\vec{L}^\lambda, \vec\tau^\lambda)  \Big)^{d_\lambda}
\end{align*}
where we recall that $\tausurv$ is defined in Definition \ref{defa:surv}.
To see that $\phi$ may be expressed as a function of the parameters $\tausurv$ (once we have set the
neutral parameters as constants), we refer to Lemma~\ref{lem:expre_neutr_exh}.  We prove the
following.
\begin{prp}
  The functions $\psi$ and $\phi$ are $\cC^\infty$ on the support of
  $\Psi_{\mathbf{\Gamma},\mathbf{d}}$. Furthermore,
  \begin{itemize}
  \item {for any $\pi \subseteq \ThetaAc'$,
   $|\partial^\pi \psi|
      \leq \fn(r)  e^{- \sum_{q \in \pi} \tilde\theta_q}$;
    }
  \item for any $\tilde \pi \subseteq \Thetasurv$, if
    $d := \sum_{\lambda \in V^\Gamma \setminus \WBCFN} d_\lambda$,
    \begin{align*} \Big| \frac{\partial^{|\tilde \pi|} \phi}{\prod_{q \in \tilde \pi}\partial
        \tau_q} \Big| \leq \fn(r, d) \, \ell_Y(\mathbf{c})^{d} \prod_{\cK \in \tilde \pi}
      \frac1{Z^{1/2}_{\mathbf{\Gamma}, \mathbf{d}}(\cK)}.
    \end{align*} 
  \end{itemize}  
\end{prp}

These inequalities play the role of the derivative controls in Assumption
\textbf{($\varphi$-CE-form)}, see Notation \ref{nota:CEform}.

\begin{proof}
  Let us first establish the smoothness of $\psi$ and $\phi$.  

  We study the discontinuity of the indicator function $ \bbbone_{\domain}$. By Remark
  \ref{rem:ntr_comb_0}, if $\lambda\in\Lambda$ satisfies $|\Gamma^\infty_\lambda|=0$, then
  $\Theta_\ell(\lambda) \subseteq \ThetaNe'$. In other words, if $\lambda \in \Lambda$ satisfies
  $q\in \Theta_\ell(\lambda)$ for an active parameter $q\in \ThetaAc'$, then we must have
  $|\Gamma^\infty_\lambda|\geq 1$. In that scenario, by definition of the exhaustion, all the cells
  $K\in \Cell(\Gamma^\infty_\lambda)$ must be favorable. Then, on the support of
  $\Psi_{\mathbf{\Gamma}, \mathbf{d}}$, by definition, we must have $Z^{\Gamma}_K\geq e^a$ for all
  cells in $\Gamma_\lambda^\infty$; in other words, $Y$ is in the $\Gamma_\lambda^\infty$-favorable
  region.  By Remark \ref{rem:triv_bound_ellgamma}, we obtain that $\ell_Y(\Gamma_\lambda) \geq
  3$. Combined with Proposition \ref{p:super}, this yields that $\Psi_{\mathbf{\Gamma}, \mathbf{d}}$
  vanishes near the discontinuities of the indicator function $\bbbone_{\domain}$. In other words,
  the derivatives of $ \bbbone_{\domain}$ with respect to $\thetaAc^\sim$ are identically equal to
  zero on the support of $\Psi_{\mathbf{\Gamma}, \mathbf{d}}$.

  Let us now examine the regularity of $\fz(y_\lambda)$ for $\lambda \in \Vbeta_+$. Let
  $q \in \ThetaAc'$ with $q\in \Theta_\ell(\lambda)$ (otherwise $\fz(y_\lambda)$ does not depend on
  $\theta_q$). Then there also exists $q'$ such that $q\in \Theta_\tau(q')$: otherwise,
  $\beta_\lambda$ would not be reached by any of the bars $B_1, \ldots, B_r$, contradicting the fact
  that the multi-loop~$\mathbf{c}$ was assumed to be connected, and not a simple curve. Necessarily,
  the corresponding $\cK_{q'}$ must be a surviving boundary segment (otherwise, we would have
  $q\in \ThetaNe(\mathbf{\Gamma}, \mathbf{d})$ by definition of the set of neutral variables). This
  implies that $e^{\tau_{q'}} \geq Z_{{\mathbf{\Gamma}}, {\mathbf{d}}}(\cK_{q'})\geq e^a$.  Because
  $\beta_\lambda$ contains the segment $\cK_{q'}$, we then have $y_j\geq \tau_{q'}\geq a >1$. But
  the function $\fz$ is assumed to be constant on $[1, +\infty)$. Thus, we obtain that the
  derivatives of $\fz(y_\lambda)$ with respect to $\thetaAc^\sim$ are identically equal to zero on the
  support of $\Psi_{\mathbf{\Gamma}, \mathbf{d}}$.

  The smoothness of the functions
  $\Psi_{\mathbf{\Gamma},\mathbf{d}}$, $\Psi_{\xi,\cQ}$ and
  $\log(F_{m_\lambda}(\vec{L}^\lambda,\vec\tau^\lambda))$ are immediate consequences of their
  definitions.
  To estimate the derivatives $\partial^{\tilde \pi}\phi$ we use the Leibniz rules. The derivatives
  of $\phi$ will contain derivatives of $ \log F_{m_\lambda}(\vec{L}^\lambda, \tau^\lambda)$ and
  derivatives of $\Psi_{{\mathbf{\Gamma}}, {\mathbf{d}}}$. We have all tools to estimate both.

  By Lemma \ref{lem:expre_neutr_exh}, the lengths of terminal cells of
  $({\mathbf{\Gamma}}, {\mathbf{d}})$ coincide (up to additive ``constants'') with the components of
  $\tausurv = (\tau_q)_{q \in \Thetasurv}$. We may therefore reformulate Proposition~\ref{p:derF} as
  follows: on the support of the test function $\Psi_{{\mathbf{\Gamma}}, {\mathbf{d}}}$, for every
  $\lambda\in \LambdaBCFN$, for every set $\tilde \pi$ of surviving boundary segments with
  $|\tilde \pi|\geq 1$, we have
\begin{align*}
  \Big| \frac{\partial^{\tilde \pi}
  \log F_{m_\lambda}(\vec{L}^\lambda, \vec{\tau}^\lambda)}
  {\prod_{q \in \tilde \pi}\partial \tau_q} \Big| \leq  
  \fn(m_\lambda)\, \prod_{q \in \tilde \pi} \frac1{Z_{\mathbf{\Gamma}, \mathbf{d}}(\cK_q)}.
\end{align*}
 As a consequence,
 \begin{align*}
   \Bigg| \frac{\partial^{\tilde \pi}
   \Big(\big( \log F_{m_\lambda}(\vec{L}^\lambda, \vec{\tau}^\lambda)\big)^{d_\lambda}\Big)}
   {\prod_{q \in \tilde \pi}\partial \tau_q} \Bigg| 
   \leq   \fn(m_\lambda)\,
   d_\lambda^{|\tilde \pi|} \big| \log F_{m_\lambda}(\vec{L}^\lambda, \vec{\tau}^\lambda) \big|^{d_\lambda-|\tilde \pi|}
   \prod_{q \in \tilde \pi} \frac1{Z_{\mathbf{\Gamma}, \mathbf{d}}(\cK_q)}.
%\\  \nonumber &\leq   K(m_\lambda)
 %d_\lambda^{|\tilde \pi|}  \ell^{d_\lambda-|\tilde \pi|}
%\prod_{\cK \in \tilde \pi} \frac1{Z_{(\mathbf{\Gamma}, \mathbf{d})}(\cK)}.
\end{align*}
Since $m_\lambda$ is always less than $2r$, we can write $\max_{m_\lambda \leq
  2r}\fn(m_\lambda)=\fn(r)$.
By definition of $F_{m_\lambda}$,
$$  \big|\log    F_{m_\lambda}(\vec{L}^\lambda, \vec{\tau}^\lambda)\big|
\leq Q_{m_\lambda}(\vec{L}^\lambda,\vec{\tau}^\lambda)
= \ell_Y(\Gamma_\lambda)
\leq  \ell_Y(\mathbf{c})$$
by Proposition \ref{p:apriori1}.
We then deduce that
\begin{align}
  \label{e:bound_log_F_in_proof}
  \Bigg| \frac{\partial^{\tilde \pi}
  \Big(\big( \log F_{m_\lambda}(\vec{L}^\lambda, \vec{\tau}^\lambda)\big)^{d_\lambda}\Big)}
  {\prod_{q \in \tilde \pi}\partial \tau_q} \Bigg| 
  \leq   \fn(r, d_\lambda) \, \ell_Y(\mathbf{c})^{d_\lambda}
  \prod_{q \in \tilde \pi} \frac1{Z_{\mathbf{\Gamma}, \mathbf{d}}^{1/2}(\cK_q)}
  % \\  \nonumber &\leq   K(m_\lambda)
  % d_\lambda^{|\tilde \pi|}  \ell^{d_\lambda-|\tilde \pi|}
  % \prod_{\cK \in \tilde \pi} \frac1{Z_{(\mathbf{G}, \mathbf{d})}(\cK)}.
\end{align}
using the trivial but later useful observation that
$ 1/ Z_{\mathbf{\Gamma}, \mathbf{d}}(\cK_q) \leq 1/Z^{1/2}_{\mathbf{\Gamma}, \mathbf{d}}(\cK_q)$.

Let us turn to the derivatives of $\Psi_{{\mathbf{\Gamma}}, {\mathbf{d}}}$.  It follows from
Proposition \ref{prp:der_psi}, the bound $|\mathbf{\Gamma}| \leq 6r^2$ and the fact that our
partition functions $\fcell{0}$, $\fcell{\infty}$ are fixed by Notation \ref{nota:a_Z}, that
\begin{equation*}
   \Big|\frac{\partial^{\tilde{\pi}} \Psi_{{\mathbf{\Gamma}}, {\mathbf{d}}} }
   {\prod_{q\in \tilde{\pi}}\partial \tau_q} \Big| \leq \fn(r).
\end{equation*}
We then harmonize this upper bound with \eqref{e:bound_log_F_in_proof} by artificially introducing a
product on the right-hand side. In order to do so, we note that, if $q \in \tilde{\pi}$ is a
surviving boundary segment, then it contains a terminal cell $K$. On the support of the derivative
above, $Z_K^{\Gamma^\infty} \leq e^{a+1}$ which, by Proposition \ref{p:6}, implies that
$Z_{\mathbf{\Gamma},\mathbf{d}}(\cK_q) \leq e^{a+3}$. We obtain
 \begin{align}\label{e:derPsi}
   \Big|\frac{\partial^{\tilde{\pi}} \Psi_{{\mathbf{\Gamma}}, {\mathbf{d}}} }
   {\prod_{q\in \tilde{\pi}}\partial \tau_q}\Big|
   \leq \fn(r)  \prod_{\cK \in \tilde{\pi}} \frac{1}{Z^{1/2}_{\mathbf{\Gamma}, \mathbf{d}}(\cK)}.
 \end{align}

 { The bounds on the derivatives of $\psi$ are trivial consequences of the Leibniz
   rules together with the previous observations on the derivatives of $\bbbone_{\domain}$ and
   $\fz(y_\lambda)$. The only derivatives appearing are the derivatives
   $\partial^\pi\Psi_{\xi,\cQ}$, which are bounded by $\fn(\# \cQ, |\pi|) \leq \fn(r)$, and vanish
   for $\tilde{\theta}_q \geq 2 \log(2)$, and hence satisfy
 \begin{equation*}
   |\partial^\pi\Psi_{\xi,\cQ}|
   \leq  \fn(r) 4^{|\pi|} \prod_{q \in \pi}e^{-\tilde{\theta}_q}
   \leq   \fn(r) \prod_{q \in \pi}e^{-\tilde{\theta}_q}
 \end{equation*}
 which leads to the claim.}
\end{proof}    

\subsection{Comparison estimate}
\label{sec:comparison-estimate-apply}

Now, the last thing we need to check is the hypothesis \eqref{e:comparison} in Assumption
\textbf{($\varphi$-CE-form)}, i.e., the comparison estimate. It will take the form of the following
inequality. 
       
\begin{prp} \label{p:comparison}
  On the support of $\Psi_{\mathbf{\Gamma},\mathbf{d}}$, we have
  \begin{align*}
    \sum_{\lambda\in \LambdaBC\setminus W} \ell_Y(\Gamma_\lambda)
    \leq
    \ell_Y(\mathbf{c})
        + \sum_{q\not\in \Theta_\tau(\tilde{\pi})}\tilde \theta_{q}
    + \sum_{q\in \tilde \pi}( \log Z_{\mathbf{\Gamma}, \mathbf{d}}(\cK_q) +2)
    +  \fn(r).
  \end{align*} 
\end{prp}
We can rewrite this inequality under the form \eqref{e:comparison} in terms of the variables
$\thetaAc^\sim$ with the comparison functions
\begin{align*}
  \log Z_q& := \log Z_{\mathbf{\Gamma}, \mathbf{d}}(\cK_q) +2\\
  \cH(\thetaAc^\sim)
          &:= \frac 12 \Big(\sum_{\lambda\in \LambdaBC\setminus W} \ell_Y(\Gamma_\lambda) -  \fn(r) \Big).
\end{align*}      
The proof of the proposition can be found in \S \ref{s:horrible}, see Proposition
\ref{p:boundGlambda}. We first show how it allows to end the proof of our main theorem.

\subsection{Applying the operator $\cL$} 
We are now ready to apply the operator $\cL$ to the integral~\eqref{e:mainfocusbis}, in order to
prove that the result is a Friedman-Ramanujan remainder (an element of $\FRrem_w$). By Proposition
\ref{p:charFR}, this implies that the desired function  is a
Friedman-Ramanujan function, which is our claim. In order to write down the terms we obtain when we
apply the operator $\cL$, we will use Proposition \ref{l:yet2}, the second variant of
our stability result.

The integration variable $x_1, \ldots, x_n$ are the active parameters $\tilde\theta_q$ with
$q \in \ThetaAc'$ (and in particular $n = \nac = \# \ThetaAc'$), and the linear forms
$\tau_1, \ldots, \tau_m$ are the components of the vector $\tausurv = (\tau_q)_{q \in \Thetasurv}$,
given by \eqref{e:tautheta2}. The functions $h, \phi$ and $\psi$ are those described in \S
\ref{s:checkingh} and \S\ref{s:checkingphi}, and the functions $f_1, \ldots, f_n$ are the functions
$\tilde\theta_q \mapsto \theta_q^{\rK_q} e^{\theta_q}$ for $q \in \ThetaAc'$.

The previous discussion proves that, as a function of the variables $\thetaAc^\sim$, the functions
$h$, $f_1, \ldots, f_n$ and $\varphi$ satisfy Assumptions \textbf{$(h)$}, \textbf{$(f)$} and
\textbf{($\varphi$-CE-form)} from Notation \ref{nota:assum_1} and \ref{nota:CEform} (provided we
prove the comparison estimate). The constant $\rK$ from Assumption \textbf{($f$)} is
\begin{align}\label{e:choice_K}
  \rK=\sum_{q\in \ThetaAc'}(\rK_q+1)
%  \leq \sum_j(\rK_j -1) +2r-\# \ThetaNe'
\end{align}
which,  by \eqref{eq:sum_K}, is equal to
\begin{equation}
  \label{eq:sum_K_proof}
  \rK   = \sum_{q\in \ThetaAc'}\rK_q + \nac
  \leq \sum_{j \in V}(\rK_j-1) + \nac
  \leq \sum_{j \in V}(\rK_j-1) + 2r.
\end{equation}
We note that the functions $\theta_q^{\rK_q}e^{\theta_q}$ only have a principal term here (so we
could take $\rN_q=0$), but because we assume that $\rK_q \leq \rN_q$ in Assumption \textbf{$(f)$},
we take $\rN_q=\rK_q$ for all $q \in \ThetaAc'$.

We can then apply Proposition \ref{l:yet2}. We obtain that applying $\cL^\rK$ to
\eqref{e:mainfocusbis} yields a sum of terms bounded by
\begin{align}
  \label{e:partialint}
  \fn(r,\rK)
  \prod_{j \in V}\|\tilde{g}_j\|_{\cF^{\rK_j,\rK_j}}
   (1+ \ell)^{\rK} e^{\frac{\ell}2}
  & \int
    \bbbone_{\domain} (\vec{L}, \vec{\theta})
    \prod_{\lambda \in \Vbeta_+}  \fz(y_\lambda)^2
        \prod_{q \in \ThetaAc'} e^{\theta_q}
    \prod_{\lambda\in \LambdaBC\setminus W} e^{-y_\lambda/2}
 \frac{\d \thetaAc^\sim}{\d\ell}.
\end{align}
% \begin{align}
% \label{e:partialint} 
%   \begin{split}
%     \int_{h(\thetaAc^\sim)=\ell}
%     \bbbone_{\domain} (\vec{L}, \vec{\theta})
%     & \prod_{\lambda \in \Vbeta_+}  \fz(y_\lambda)^2
%           \prod_{t\leq \tilde n}  (1+ \theta_{v_t})^{\rK_{v_t}} e^{  \theta_{v_t}}
%       \prod_{q\not\in \{v_1, \ldots, v_{\tilde n}\}}    \theta_q^{\rK_q}e^{\theta_q}
%      \\
%     & \exp \Big(\frac12\sum_{j\in W_1} y_j
%       -\frac12  \sum_{\lambda\in \oLambdaBC\setminus W_2} y_\lambda \Big)
%       \frac{\d \thetaAc^\sim}{\d\ell}
%   \end{split}
% \end{align}
% with a factor $\fn(r, (\rK_j)) (1+\ell )^{k_1} e^{\ell/2}$
As a consequence, our goal will be achieved if we prove that the integral in \eqref{e:partialint} is
polynomially bounded, after integration with respect to all the frozen variables $\thetaNe'$.

All the variables that were frozen are recorded in \eqref{e:takenout} and just after
\eqref{e:mainfocusbis}. Re-integrating them in the calculation, we are lead to consider the product
of $\fn(r,\rK) \prod_{j \in V}\|\tilde{g}_j\|_{\cF^{\rK_j,\rK_j}}$ with
\begin{align*}
  (1+ \ell)^{\rK} e^{\frac{\ell}2}
  \int_{h(\vec{L},\vec{\theta})=\ell}
  \bbbone_{\domain} (\vec{L}, \vec{\theta})
  &
    \prod_{\lambda \in \Vbeta_+}  \fz(y_\lambda)^2
    P(\vec{L},\vec{\theta}) 
    T_{V,\mathrm{m}}(\vec{L},\vec{\theta})
    \prod_{i=1}^r \sinh (L_i)\frac{\d^r \vec{L} \d^{2r} \vec\theta}{\d\ell}.
\end{align*}
where
\begin{equation*}
  P(\vec{L},\vec{\theta}) 
  = \prod_{q \in \Theta_\ell(V_+) \cap \ThetaNe} \theta_q^{\rK_q}
  \prod_{q \in \ThetaNe' \setminus \ThetaNe} \theta_q^{\rK_q}
  \prod_{q \in \ThetaAc} e^{\theta_q}
  \prod_{\lambda \in W} r_\lambda(\vec{L},\vec{\theta})
  \prod_{\lambda\in \LambdaBC\setminus W} e^{-y_\lambda/2}.
\end{equation*}
% \begin{equation}
%   \begin{split}
%     \int_{\ell({\mathbf{c}})=\ell}
%     \bbbone_{\domain} (\vec{L}, \vec{\theta})  
%       \prod_{\substack{1 \leq j \leq N \\ j \in V}}  \fz(y_j)^2
%     \prod_{j\in W_1}  r_j(y_j)
%     \prod_{\lambda\in W_2} U_\lambda(\vec{L}, (\theta_q)_{q\in \ind(\lambda)})
%     \prod_{t\leq \tilde n}  (1+ \theta_{v_t})^{\rK_{v_t}} e^{  \theta_{v_t}}\\
%     \prod_{\substack{q\not\in \{v_1, \ldots, v_{\tilde n}\} \\ \text{or } q \in \ThetaNe'\setminus \ThetaNe}}
%     \theta_q^{\rK_q}e^{\theta_q}
%     \exp \Big(\frac12\sum_{j\in W_1} y_j-\frac12  \sum_{\lambda\in \oLambdaBC\setminus W_2} y_\lambda\Big)
%       T_V(\vec{L}, \vec{\theta})
%       \hyp_-(\vec{L})\frac{ \d^r \vec{L} \d^{2r} \vec{\theta}}{\d \ell}.
%   \end{split}
% \end{equation}
% \todo{missing the neutral variables?}
Recall that we have assumed that we have the decay
$$|r_\lambda(\vec{L},\vec{\theta})| \leq
\begin{cases}
  (1+y_\lambda)^{\rN_\lambda-1}e^{y_\lambda/2} & (\lambda \in W^\beta) \\
  (1+y_\lambda)^{\rN_\lambda-1}e^{-y_\lambda/2} & (\lambda \in \WBCFN).
\end{cases}
$$ 
We also remember that $\sum_{q \in \ThetaAc}{\theta_q} \leq \sum_{\lambda \in \Vbeta_+ \setminus \Wbeta}
y_\lambda = \sum_{\lambda \in \Vbeta \setminus W^\beta} y_\lambda + \sum_{\lambda \in \Lambdainbeta} y_\lambda$. It follows that
\begin{align*}
  P(\vec{L},\vec\theta)
  &  \leq (1+\ell)^{\rN'}
    \prod_{\lambda \in \Vbeta \setminus \Wbeta} e^{y_\lambda}    
    \prod_{\lambda \in \Lambdainbeta} e^{y_\lambda}    
    \prod_{\lambda \in \Wbeta} e^{y_\lambda/2}
    \prod_{\lambda \in \WBCFN} e^{-y_\lambda/2}
    \prod_{\lambda\in \LambdaBC\setminus W} e^{-y_\lambda/2} 
\end{align*}
for $\rN' \leq \sum_{\lambda \in \ThetaNe'} \rN_\lambda$.
Therefore, since $(\Vbeta \setminus \Wbeta) \sqcup (\VFN \setminus \WBCFN) \subseteq \LambdaBC \setminus W$,
\begin{align*}
  P(\vec{L},\vec\theta)
  &  \leq (1+\ell)^{\rN'} \prod_{\lambda \in \Vbeta} e^{y_\lambda/2}
    \prod_{\lambda \in \Lambdainbeta} e^{y_\lambda}
    \prod_{\lambda\in \VFN} e^{-y_\lambda/2}.
\end{align*}

Putting everything together, we have proved that the integral \eqref{e:mainfocusbis} is bounded by a
sum of terms of the form
\begin{equation*}
  \int_{h(\vec{L},\vec{\theta})=\ell}
  \bbbone_{\domain} (\vec{L}, \vec{\theta})
  \frac{\prod_{\lambda \in \Vbeta} \fz(y_\lambda)e^{y_\lambda/2}
  \prod_{\lambda \in \Lambdainbeta} \fz^2(y_\lambda)e^{y_\lambda}}
    {\prod_{\lambda\in \VFN} e^{y_\lambda/2}}
    T_{V,\mathrm{m}}(\vec{L},\vec{\theta})
    \prod_{i=1}^r \sinh (L_i)\frac{\d^r \vec{L} \d^{2r} \vec\theta}{\d\ell},
  \end{equation*}
  multiplied by $  \fn(r,\rK)
  \prod_{j \in V}\|\tilde{g}_j\|_{\cF^{\rK_j,\rK_j}}   (1+\ell)^{\rK+\rN'} e^{\frac \ell 2}$. 
We shall now bound this by a quantity for which we can use our change of variables, Corollary
\ref{c:maindet-dirac}, backwards. Indeed, we note that $\fz(y)e^{y/2} \leq c \sinh(y/2)$ for a
universal constant $c$, and hence the quantity above is smaller than
\begin{align*}
  & \fn(r)
    \int_{h(\vec{L},\vec{\theta})=\ell}
    \bbbone_{\domain} (\vec{L}, \vec{\theta})
    \frac{\prod_{\lambda \in \Vbeta} \sinh \div{y_\lambda}
    \prod_{\lambda \in \Lambdainbeta} \sinh^2 \div{y_\lambda} }
    {\prod_{\lambda\in \VFN}\sinh \div{y_\lambda}}
    T_{V,\mathrm{m}}(\vec{L},\vec{\theta})
    \prod_{i=1}^r \sinh (L_i)\frac{\d^r \vec{L} \d^{2r} \vec\theta}{\d\ell} \\
  & = \fn(r)
    \int_{h(\vec{L},\vec{\theta})=\ell}
    \prod_{j=1}^k x_{\vi}\delta(x_{\vi}-x_{\vi'})
    \frac{\d \Volwp[\g,\n](Y)}{\d\ell}.   \end{align*}
  We have shown that this quantity is polynomially bounded in the weak sense in Corollary
  \ref{c:pvolume}, giving a bound
\begin{align*} \int_{\ell_Y({\mathbf{c}})\leq \ell}  
  \prod_{j=1}^k x_{\vi} \delta(x_{\vi} - x_{\vi'})
   \d \mathrm{Vol}^{\mathrm{WP}}_{g_{\mathbf{S}}, n_{\mathbf{S}}}(\x,Y)  \leq
  (3 \ell)^{3 \chi(\mathbf{S})}.
 \end{align*}

 As a conclusion, for $\rK = \sum_{j \in V} (\rK_j -1)+2r$ we have shown that
$$\cL^{\rK} \cJ \in \cR^{\rN}$$
where $\rN=\rK+\rN'+3\chi(\Sf)+1 \leq\sum_{j \in V}\rN_j+3\chi(\Sf)+1$, implying that $\cJ$ belongs to
$\cF_w^{\rK, \rN}$. Furthermore we have shown that
\begin{align}\label{e:multi_bound}
\norm{\cJ}_{\cF_w^{\rK, \rN}} \leq \fn(\g, \n, \rK, \rN)\, \prod_{j\in V} \norm{\tilde g_j}_{\cF^{\rK_j, \rN_j}}.
\end{align}
This proves Theorem \ref{t:main}.

\subsection{Theorem \ref{t:main-kappa}}

As a step towards the version of \S \ref{s:variantLC}, we add a small modification to Theorem \ref{t:main}. The only difference is that we impose the lengths of some components $\mathbf{c}_j$ to stay bounded. 
  
 \begin{thm}\label{t:main-kappa}
   With the notations of Theorem \ref{t:main}, let $W$ be a subset of $\{1, \ldots, \cc\}$ and
   $\kappa=(\kappa_j)_{j\in W}$ be positive real numbers.  Then, the function
   $\cJ=\cJ_{\mathbf{T},V, \mathrm{m}, f, W, \kappa}$ defined by
\begin{align}\label{e:mainint-kappa}
  \cJ( \ell )
  =   \int_{\ell_Y({\mathbf{c}})=\ell}
  \prod_{j\in W}\bbbone_{[0, \kappa_j]}(\ell_Y(\mathbf{c}_j))
  \prod_{j\in V}  f_j(x_j)
                \prod_{j=1}^k x_{\vi} \delta(x_{\vi} - x_{\vi'})
                \frac{ \d \mathrm{Vol}^{\mathrm{WP}}_{g_{\mathbf{S}}, n_{\mathbf{S}}}(\x, Y) }{\d\ell}
\end{align}
is a Friedman--Ramanujan function in the weak sense.  More precisely, let $\rho[{\mathbf{T}}, W]$ be
the number of unshielded simple portions of the multi-loop $\mathbf{c}$ that are not in
$(\mathbf{c}_j)_{j\in W}$, 
\begin{equation*}
  \rK= \sum_{j\in V\setminus W} (\rK_j -1) +\rho[{\mathbf{T}}, W]
  \quad \textrm{and} \quad
  \rN=\sum_{j \in V}\rN_j + 3 \chi(\mathbf{S}) +1.
\end{equation*}
Then $\cJ\in \cF_w^{\rK, \rN}$, and
\begin{align}\label{e:multi_bound}
  \norm{\cJ}_{\cF_w^{\rK, \rN}}
  \leq \fn(\g, \n, \rK, \rN)\,
  e^{\frac12\sum_{j\in W} \kappa_j} \prod_{j\in V} \norm{\tilde g_j}_{\cF^{\rK_j, \rN_j}}.
\end{align}
  \end{thm}

  \begin{proof}
    The proof of Theorem \ref{t:main-kappa} follows exactly the same lines as the proof of Theorem
    \ref{t:main}, except that the set of neutral variables is enlarged to
    $\ThetaNe'(W):= \ThetaNe'\cup \{q\in \Theta, \comp(q)\in W\}$ (with the notation $\comp(q)$ from
    \S \ref{s:relabel_Theta}, indicating to which component $c_i$ the index $q$ belongs). The set of
    active variables is therefore reduced to $\Theta \setminus \ThetaNe'(W)$.

Proposition \ref{p:comparison} is modified as follows:
  \begin{align*}
    \sum_{\lambda\in \LambdaBC\setminus W} \ell_Y(\Gamma_\lambda)
    \leq \ell_Y(\mathbf{c})
    + \sum_{q\notin \Theta_\tau(\tilde{\pi})}\tilde \theta_{q}
    + \sum_{q\in \tilde \pi}( \log Z_{\mathbf{\Gamma}, \mathbf{d}}(\cK_q) +2)
    + \sum_{j\in W} \kappa_j
    +   \fn(r)
      \end{align*} 
and this last term is responsible for the extra factor $e^{\frac12 \sum_{j\in W}\kappa_j}$ in the quantitative statement of Theorem \ref{t:main-kappa}.
  \end{proof}

\begin{rem}Note that $g$, the genus of the large surface in which $\Sf$ is embedded at the beginning
  of the paper, does not appear in Theorems \ref{t:main} and \ref{t:main-kappa}.
However, a notable point is that the numbers $(\kappa_j)_{j \in W}$
 may be taken to be functions of $g$. This will happen in the application to Theorem~\ref{t:dream}, where we have constraints such as $\ell_Y(\mathbf{c}_j)\leq \kappa_j$ with $\kappa_j=\kappa \log g$, with $\kappa>0$ arbitrarily small.
  \end{rem}

  \begin{rem}
    Another point is that the new exponents $\rK$, $\rN$ are bounded above by the ones obtained in
    the initial result, Theorem \ref{t:main}. In particular, the property also holds with the
    exponents from Theorem \ref{t:main}.
  \end{rem}

%%%Local Variables: 
%%% mode: latex
%%% TeX-master: "main"
%%% End: 

\section{Proof of the crucial comparison estimate (Proposition \ref{p:comparison})}
\label{s:horrible}

The proof of the comparison estimate, Proposition \ref{p:comparison}, is divided in two parts: a
lower bound on $\ell_Y(\mathbf{c})$ (presented in \S \ref{s:lblt}) and an upper bound on
$\sum_{\lambda\in \LambdaBC \setminus W^\Gamma} \ell_Y(\Gamma_\lambda)$ (in \S \ref{s:ubsl}).

\subsection{Lower bound on $\ell_Y(\mathbf{c})$}
\label{s:lblt}

Let us first prove a lower bound on the length of $\mathbf{c}$. We first prove a general lemma, and
then proceed to the proof in \S \ref{s:lbl}.

\subsubsection{A lemma}
\label{s:iterated}

Let us first prove a comparison inequality in a simple case.
Let $\gamma_1, \ldots, \gamma_M$ be $M$ complete oriented geodesics in the hyperbolic plane, assumed
to be pairwise aligned, with $\gamma_{k+1}$ on the right of $\gamma_k$ for all $k$. Let us define
the following notations.
\begin{itemize}
\item For $1 \leq k \leq M-1$, $H_{k, k+1}$ is the orthogeodesic segment from $\gamma_k$ to
  $\gamma_{k+1}$, of length denoted as $\theta_k$, with endpoints $G_k \in \gamma_k$ and
  $D_k \in \gamma_{k+1}$.
  % (and call $H^\inftD_{k, k+1}$ the full orthogeodesic).
\item $D_0$ and $G_{M}$ are two arbitrary points on $\gamma_1$ and $\gamma_M$ respectively.
\item For $1 \leq k \leq M$, we denote as $L_k$ the distance between $D_{k-1}$ and $G_k$.
\end{itemize}
We then prove the following.
 
\begin{lem}
  We have
  \begin{align} \label{ineq:staircase}
    e^{\dist(D_0, G_M)}
    \geq \frac{1}{4^{M-1}}
    \exp \Big(\sum_{k=1}^M L_k\Big) \prod_{k=1}^{M-1} E_k(\theta_k)
  \end{align}
  where the functions $E_k$ are defined by
  \begin{equation*}
    E_k(\theta) =
    \begin{cases}
      \frac{1}{2} \exp(\theta) & \text{if the geodesic from } D_0 \text{ to } G_M \text{ crosses } H_{k,
        k+1} \\
      \cosh \theta- 1= 2 \sinh^2\div{\theta} & \text{otherwise.}
    \end{cases}
  \end{equation*}
\end{lem}

\begin{proof}
  The proof relies on the following classic formula for hyperbolic birectangles, see \cite[Equation
  (2.3.2)]{buser1992}. For a geodesic quadrilateral with consecutives sides $\rho_1, t, \rho_2, d$
  with right angles between $\rho_1, t$ and between $t, \rho_2$, we have :
  \begin{align}
    \cosh d =\cosh \rho_1 \cosh \rho_2 \cosh t- \sinh \rho_1 \sinh \rho_2.
    \label{e:birect}
  \end{align}
  We then obtain that:
  \begin{itemize}
  \item if $\rho_1 \rho_2\geq 0$, 
    \begin{align}
      e^{d} \geq \cosh d
      \geq \cosh \rho_1 \cosh \rho_2 (\cosh t- 1)
      \geq \frac{1}{4} e^{\rho_1+\rho_2} (\cosh t- 1);
      % =2 \cosh\rho_1 \cosh \rho_2 \sinh^2\div{t}
      \label{ineq:birect}
    \end{align}
    % If $t\geq \log 2$ then $\cosh t- 1\geq \frac{e^t}8$ and thus
    % \begin{align}
    %   e^{|d|} \geq \frac1{2^5} e^{|\rho_1|+|\rho_2|+|t|}
    %   % =2 \cosh\rho_1 \cosh \rho_2 \sinh^2\div{t}
    %   \label{ineq:birect1}.
    % \end{align}
  \item and if $\rho_1 \rho_2\leq 0$, 
    \begin{align}
      e^{|d|} \geq \cosh d \geq \cosh \rho_1 \cosh \rho_2 \cosh t \geq
      \frac1{2^3} e^{|\rho_1|+|\rho_2|+|t|}.
      \label{ineq:birect-cross2}
    \end{align}
  \end{itemize}
  
  Let us now prove the result by induction on $M \geq 2$. For $M=2$, this is a direct consequence of
  \eqref{ineq:birect} and \eqref{ineq:birect-cross2} applied to the quadrilateral of vertices $D_0$,
  $G_1$, $D_1$, $G_2$.

  For $M>2$, assume the result known for $M-1$ geodesics, and let us pass to rank $M$. We observe
  that the geodesic from $D_0$ to $G_M$ has to cross successively
  $\gamma_2, \ldots, \gamma_{M-1}$. Call $\tilde G_{M-1}$ its intersection point with
  $\gamma_{M-1}$. The induction assumption at rank $M-1$ tells us that
\begin{align*}  
  e^{\dist(D_0, \tilde G_{M-1})}
  \geq \frac{1}{4^{M-2}} \exp \Big(\sum_{k=1}^{M-2} L_k + \dist(D_{M-2},\tilde G_{M-1})\Big)
  \prod_{k=1}^{M-2} E_k(\theta_k).
\end{align*}
By the case $M=2$,
\begin{align*} 
e^{\dist( \tilde G_{M-1}, G_M)} \geq  \frac{1}{4} \exp(\dist(\tilde G_{M-1}, G_{M-1})+L_M) \, E_{M-1}(\theta_{M-1}).
\end{align*}
The result at rank $M$ follows by putting together the two previous inequalities, given that we have the triangular inequality
$$ L_{M-1}=\dist(D_{M-2}, G_{M-1}) \leq  \dist(D_{M-2},\tilde G_{M-1}) + \dist(\tilde G_{M-1}, G_{M-1}),$$
and the equality
$$ \dist(D_0, G_M) = \dist(D_0, \tilde G_{M-1})+ \dist( \tilde G_{M-1}, G_M).$$
\end{proof}

We will actually use the following variant.

\begin{lem}
  Let $\cQ =\{\psi_1<\psi_2<\ldots <\psi_j\} \subseteq \{1, \ldots, M-1\}$. If we assume that
  $\theta_k\geq \log 2$ for all $k\in \cQ$, then if we define $\psi_0 := 0$ and $\psi_{j+1} := M$,
  \begin{align} \label{ineq:staircase'}
    \dist(D_0, G_M)
    \geq 
    \sum_{k \in \cQ} \theta_k
    + \sum_{k=0}^{j} \dist(D_{\psi_k}, G_{\psi_{k+1}}) 
     - 5 (M-1)\log 2.
\end{align}
\end{lem}

\begin{proof}
  The proof is the same as above, now noting that the hypothesis on $\cQ$ implies that
  $E_k(\theta_k)\geq \frac{1}8e^{\theta_k}$ for all $k \in \cQ$.
\end{proof}

\subsubsection{The lower bound}

\label{s:lbl} Recall the partition of unity
$\Psi_{\xi, \cQ}$ defined in Definition \ref{defa:part_unity_cross}, indexed by elements
$\xi \in \Xi$ and subsets $\cQ$ of the set of parameters $I (\xi) \subseteq \Theta$ (which
intuitively corresponds to the set of ``nice'' variables $\theta_q$, i.e. the ones we do not erase).
We have the decomposition $I (\xi)=\bigsqcup_{i=1}^{\cc} I(\xi^i)$ depending on the index of the
component $\mathbf{c}$ we consider.  Consider $\tilde I(\xi^i) \subseteq \Z$ the periodic extension of
$I(\xi^i)$, and let $\varphi^i :\Z \rightarrow \Z$ be an increasing numbering of it:
\begin{equation*}
  \tilde I(\xi^i)=\{\varphi_k^i, k\in \Z\}, \qquad
  \varphi_k^i<\varphi_{k+1}^i.
\end{equation*}
Denoting $n_{\xi^i}=\# I(\xi^i)$, we have the periodicity property that
$\varphi_{k+n_{\xi^i}}^i=\varphi_{k}^i+n_i$.

We also recall some notation related to the lifts of the geometric situation to the hyperbolic plane
introduced in \S \ref{s:lift}. For a component $i \in \{1, \ldots, \cc\}$,
$(\tilde \beta(q))_{q\in \mathbf{\Theta}_i}$ is a sequence of lifts of geodesics in the family
$\beta$, which satisfies the following periodicity property: $\tilde \beta(q+n_i)$ is the image of
$\tilde \beta(q)$ by the hyperbolic translation of axis $\tilde {\mathbf{c}}_i$ and of length
$\ell_Y({\mathbf{c}}_i)$.

Last, recall the intervals $\cI'_q =[G_q, D_q]$ introduced in Notation \ref{nota:GD}.

\begin{nota}
  \label{nota:fix_Q}
  Throughout this section, we fix $\xi \in \Xi$ and $\cQ \subseteq I(\xi)$. For
  $i \in \{1, \ldots, \cc\}$, let $\cQ_i:=\cQ \cap I(\xi^i) \subseteq \Theta^i \simeq \Z_{n_i}$. We
  pick a numbering of $\cQ_i$
  \begin{equation*}
    \cQ_i = \{ \psi_1^i < \ldots < \psi_{j_i}^i \} \subseteq \{1, \ldots, n_i\}.
  \end{equation*}
  Let $\tilde{\cQ}^i = \{ \psi_{k}^i \}_{k \in \Z} \subseteq \mathbf{\Theta}^i \simeq \Z$ denote its
  cyclic extension.
\end{nota}

Note that, since $\cQ^i \subseteq I(\xi^i)$, the sequence $(\psi_k^i)_{k \in \Z}$ is an extraction
of $(\varphi_k^i)_{k \in \Z}$.

 \begin{defa}
   { We call \emph{junction} any geodesic segment $[D_{\psi_k^i}, G_{\psi_{k+1}^i}]$, with
     $i \in \{1, \ldots, \cc\}$ and $k \in \{1, \ldots, j_i\}$.  We denote $\Junc(\cQ)$ the set of
     all junctions. Two segments of $\Junc(\cQ)$ will be said to be \emph{disjoint} if they
     intersect only transversally.  }
 \end{defa}
 Remark that certain junctions are bars (if $\psi_{k+1}^i= \psi_k^i+1$).
\begin{prp}
  For any $Y$ in the support of $\Psi_{\xi,\cQ}$,
  \begin{align} \label{e:uppppper}
    \ell_Y({\mathbf{c}})
    \geq \sum_{\lambda \in \Lambdabeta} \ell_Y(\beta_\lambda)
    + \sum_{p \in \Junc(\cQ)} \ell(p) -12 \, r \log 2 .
  \end{align}  
\end{prp}

\begin{rem} \label{r:junction} The proof consists, for each $i \in \{1, \ldots, \cc\}$, in
  constructing a representative $\mathbf{c}_{i}^{\mathrm{f}}$ of the homotopy class of
  $\mathbf{c}_i$, made of the concatenation of the intervals $\cI'_q$ and of junctions. Inequality
  \eqref{e:uppppper} means that
  $\mathbf{c}^{\mathrm{f}}=(\mathbf{c}^{\mathrm{f}}_1, \ldots, \mathbf{c}^{\mathrm{f}}_{\cc})$
  almost minimizes the length in its homotopy class.
 \end{rem}

 \begin{proof}
   Let $Y$ be a metric on the support of $\Psi_{\xi,\cQ}$.
  We prove this property component by component and then sum the different bounds. Let us therefore
  fix $i \in \{1, \ldots, \cc\}$. 

  Denote by $\gamma_k$ the complete orthogeodesic between $\tilde \beta_{\varphi_{k-1}^i}$ and
  $\tilde \beta_{\varphi_k^i}$: it contains the geodesic segment
  {$\tilde{B}_{\varphi_{k-1}^i \varphi_k^i}$.} By construction of the set $I(\xi^i)$, all these
  geodesics are parallel, and can be oriented so that $\gamma_{k+1}$ is on the right of
  $\gamma_{k}$.  We are in the geometric situation described in \S \ref{s:iterated}, with $M=n_{\xi^i}+1$.
  The distance between $\gamma_k$ and $\gamma_{k+1}$ is $\tilde \theta_{\varphi_k^i}$. The
  orthogeodesic segment denoted by $H_{k, k+1}$ in \S \ref{s:iterated} is
  $[G_{\varphi_k^i}, D_{\varphi_k^i}]$.

  The infinite geodesic $\tilde {\mathbf{c}}_i$ intersects successively all geodesics
  $\gamma_k$. Calling $D_0$ its intersection point with $\gamma_1$ and $G_{n_{\xi^i}+1}$ its
  intersection with $\gamma_{n_{\xi^i}+1}$, we have $\ell_Y({\mathbf{c}}_i)=\dist(D_0, G_{n_{\xi^i}+1})$.
  We can apply inequality \eqref{ineq:staircase'} to the set of indices $\psi_k^i$, which applies
  here because all the $\tilde \theta_{\psi_k^i}$ are greater than $\log 2$ on the support of
  $\Psi_{\xi, \cQ}$.  We obtain that
  \begin{align} \label{e:uppperone}
    \ell_Y({\mathbf{c}}_i) \geq  
    \sum_{q\in \cQ^i} \tilde \theta_q
    + \sum_{k=1}^{j_i} \dist (D_{\psi_k^i}, G_{\psi_{k+1}^i})
    -5 n_i\log 2.
  \end{align}  

  By summing over $i \in \{1, \ldots, \cc\}$ and using
  $\ell_Y({\mathbf{c}}) =\sum_{i=1}^{\cc} \ell_Y({\mathbf{c}}_i)$, we obtain:
   \begin{align*}
     \ell_Y({\mathbf{c}}) 
     \geq \sum_{q\in \cQ}\tilde \theta_q
     +\sum_{i=1}^{\cc} \sum_{k=1}^{j_i} \dist (D_{\psi_k^i}, G_{\psi_{k+1}^i})
      -10 r \log 2.
  \end{align*}  
  Because the $\tilde \theta_q$ are all smaller than $2\log 2$ for $q\not\in \cQ$, we may write
  \begin{align*}  \sum_{q\in \cQ}\tilde \theta_q &\geq \sum_{q\in I(\xi)}\tilde \theta_q -2r \log 2
  \\ &\geq  \sum_{q\in \Theta} \theta_q -2r \log 2= \sum_{\lambda \in \Lambdabeta} \ell_Y(\beta_\lambda) -2r \log 2
    \end{align*}
    where we used \eqref {e:funnier} in the second line. This is the claim. Note that the right-hand
    side represents the total length of the representative
    $\mathbf{c}^{\mathrm{f}}=(\mathbf{c}_1^{\mathrm{f}}, \ldots, \mathbf{c}_{\cc}^{\mathrm{f}})$
    mentioned in Remark \ref{r:junction}.
\end{proof}
%   We fix a component
% $i\in \{1, \ldots, \cc\}$. We shall prove a lower bound on $\ell_Y({\mathbf{c}}_i)$, Lemma
% \ref{l:ci}.  We introduce the notation  and
% $\tilde \cQ_i \subseteq {\mathbf{\Theta}}^i \simeq\Z$ the corresponding $n_i\Z$-invariant set.

  % Similarly, let $\psi:\Z\To \Z $ be an increasing numbering of the elements of $\tilde\cQ^1$:
  %  \begin{equation*}
  % \tilde\cQ^1=\{\psi(j), j\in \Z\}
  % \end{equation*}

\subsubsection{The polygons $\Pol(q)$}
\label{sec:polygons-polq}

We observe that, for any $1 \leq i \leq \cc$ and $1 \leq k \leq j_i$, the junction
$J_k^i = [D_{\psi_k^{i}}, G_{\psi_{k+1}^i}]$ leaves from $\tilde \beta_{\psi_k^{i}}$ and arrives at
$\tilde \beta_{\psi_{k+1}^{i}}$. It is homotopic with gliding endpoints to the concatenation
\begin{align*}
  \left(\smallbullet_{q=\psi_k^{i}+1}^{\psi_{k+1}^{i}-1} \tilde B_{q-1, q}\smallbullet \tilde\cI_{q} \right)\smallbullet \tilde B_{\psi_{k+1}^{i}-1, \psi_{k+1}^{i}}.
\end{align*}
Examining the points at infinity of the infinite geodesic containing a lift of $J_k^i$, and using
the fact that it has to cross both $\tilde \beta_{\psi_k^{i}}$ and $\tilde \beta_{\psi_{k+1}^{i}}$,
we make the following two observations.
\begin{itemize}
\item The junction $J_k^i$ crosses all the geodesics $\tilde \beta_{q}$ with $q \in \{\psi_k^{i}+1,
  \ldots, \psi_{k+1}^{i}\}$, assuming $q$ is a crossing parameter.
\item The junction $J_k^i$ crosses all the bars $\tilde B_{q}$ with
  $q \in \{\psi_k^{i}+1,..., \psi_{k+1}^{i}\}$, assuming both $q-1$ and $q$ are U-turn parameters.
\end{itemize}
We deduce the following.

\begin{lem}
  Let $q \in \Theta$, and let $1 \leq i \leq \cc$ and $1 \leq k \leq j_i$ for which
  $q \in \{\psi_k^{i}+1,..., \psi_{k+1}^{i}\}$. There exists a contractible geodesic polygon
  $\Pol(q)$, which is bordered, on one side by the full bar $\tilde B_{q}$, on other sides, by
  subsegments of $J_k^i$, $\tilde \beta_{q-1}$, $\tilde \beta_{q}$ and possibly
  $\tilde \beta_{q-2}$, $\tilde \beta_{q+1}$, $\tilde B_{q-1}$ and $\tilde B_{q+1}$, and so that
  none of the geodesics $\tilde \beta_{q'}$ or the bars $\tilde B_{q'}$ meet the interior of
  $\Pol(q)$.
\end{lem}

\begin{proof}
  The construction is made by exhaustion of all cases, depending on whether $q-1$ and $q$ are U-turn
  parameters or crossing parameters.
  \begin{itemize}
  \item If $q-1$ and $q$ are crossing parameters, we simply consider the polygon with sides
    delimited by $\tilde B_q$, and subsegments of $\tilde{\beta}_q$, $\tilde{\beta}_{q-1}$ and
    $J_k^i$, as represented by the shaded area in Figure \ref{fig:polq_cc}.
    \begin{figure}[h]
      \centering
      \includegraphics[scale=0.9]{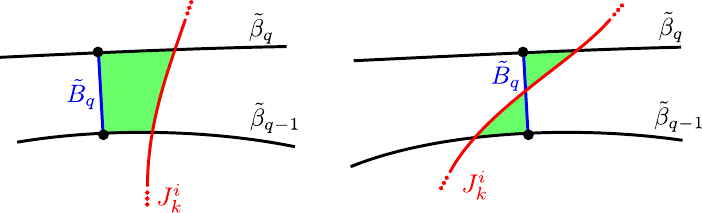}
      \caption{The polygon $\Pol(q)$ when $q-1$ and $q$ are crossing parameters.}
      \label{fig:polq_cc}
    \end{figure}
  \item If $q-1$ is a crossing parameter and $q$ is a U-turn, then there are four cases to consider,
    depending on whether $J_k^i$ intersects $\tilde{B}_q$ or not, and whether it first intersects
    $\tilde{B}_{q+1}$ or~$\tilde{\beta}_{q+1}$. See Figure \ref{fig:polq_uc}.
        \begin{figure}[h]
      \centering
      \includegraphics[scale=0.9]{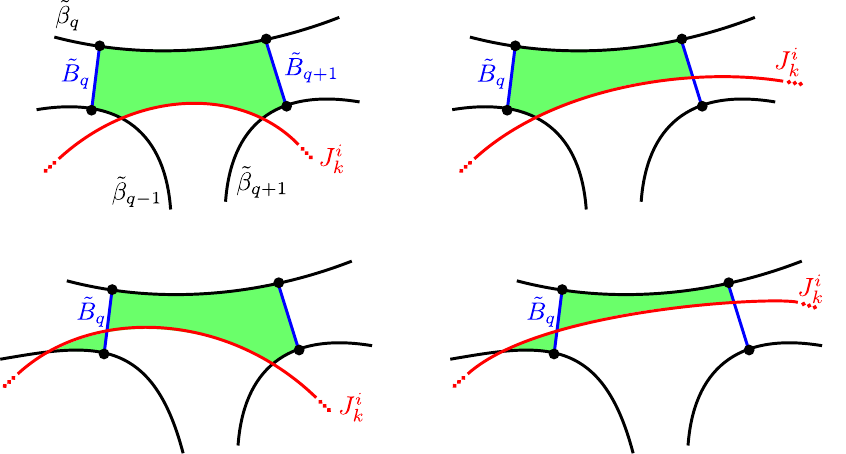}
      \caption{The polygon $\Pol(q)$ in the mixed case.}
      \label{fig:polq_uc}
    \end{figure}
  \item The symmetric case, when $q-1$ is a U-turn parameter and $q$ a crossing parameter.
  \item In the case when both $q-1$ and $1$ are U-turn parameters, there are also four cases to
    consider, depending on whether $J_k^i$ first intersects $\tilde{\beta}_{q-2}$ or
    $\tilde{B}_{q-1}$, and $\tilde{\beta}_{q+1}$ or $\tilde{B}_{q+1}$. See Figure \ref{fig:polq_uu}.
    \begin{figure}[h]
      \centering
      \includegraphics[scale=0.9]{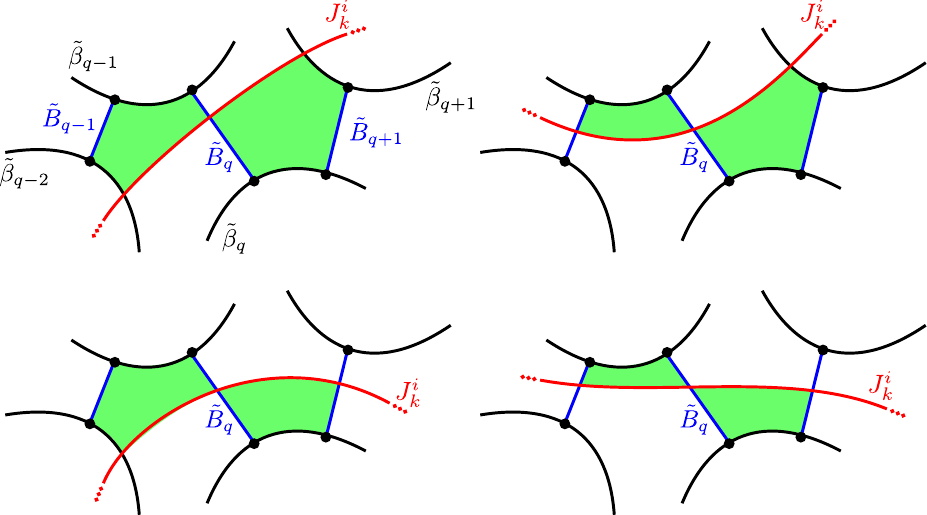}
      \caption{The polygon $\Pol(q)$ when $q-1$ and $q$ are U-turn parameters.}
      \label{fig:polq_uu}
    \end{figure}  \end{itemize}
\end{proof}

\begin{rem}
  In the case the junction $J_k^i$ intersects the bar $\tilde B_{q}$, the polygon $\Pol(q)$ is not
  convex but rather made of two disjoint true polygons $\Pol(q)=\Pol^1(q)\sqcup \Pol^2(q)$, one at
  the left of the bar $\tilde B_{q}$, the other one at its right. It can also happen that the
  polygon $\Pol(q)$ is degenerate, which happens if and only if the junction $J_k^i$ coincides with
  the bar $\tilde B_{q}$.
\end{rem}

Because of the presence of these contractible polygons $\Pol(q)$, all the bars $\tilde B_{q}$ (or
their projections $\overline{B}_q$ to $\Sf$), for $q \in \{\psi_k^{i}+1,..., \psi_{k+1}^{i}\}$, are
said to be \emph{superseded} by the polygon $\Pol(q)$. Each bar in the family $\overline{B}$ is
superseded twice: if $\iota : \Theta \rightarrow \Theta$ is the involution switching signs, then
$\overline{B}_q$ is superseded by $\Pol(q)$ and $\overline{B}_{\iota q}$ by $\Pol(\iota q)$.
  %Examples of respective situations of the polygons
  %$\Pol(q)$ and $\Pol(\iota q)$ are shown in Figure \ref{fig:polygons}.

\begin{rem}
  \label{rem:seg_shield}
  In the special case when $q\in \{ \psi_k^{i}+1, \ldots, \psi_{k+1}^{i}-1\}$ is a U-turn parameter,
  then the junction $J_k^i$ does not intersect $\tilde \beta_q$, and the entirety of the simple
  portion $\tilde{\cI}_q$ is included in the boundary of $\Pol(q)$. We say that $\overline{\cI}_q$
  is \emph{shielded} by the junction $J_k^i$: there is a contractible component of
  $\Sf\setminus \mathbf{c}^{\mathrm{f}}$ containing $\overline{\cI}_q$ in its boundary.
\end{rem}

\subsection{Upper bound on $\ell_Y(\Gamma_\lambda)$}  \label{s:ubsl}

The upper bound will be related to neutralizations within the family of polygonal curves
$(\Gamma_\lambda)_{\lambda \in \Lambda}$. We invite the reader to remind themselves of the notations
introduced in \S \ref{s:lengthboundary}, and in particular the notion of height of a cell
$Z^{\mathbf{\Gamma}}_K$ from Notation \ref{nota:height_fam}, the conventions related to simultanous
decorations from \S \ref{sec:simult-decor}, and the notion of surviving boundary segment
$\cK_q$, $q \in \Thetasurv$, and associated height $Z_{{\mathbf{\Gamma}}, {\mathbf{d}}}(\cK_q)$,
introduced in Definition \ref{defa:surv}.

\subsubsection{A useful lemma}

The following lemma, illustrated in Figure \ref{fig:variational}, is an obvious consequence of the
variational characterization of geodesics and orthogeodesics. Here we use Lemma~\ref{l:WZ}, which
allows us to compare certain orthogedesics with the height function introduced in Definition
\ref{e:defZ}.

\begin{prp} \label{p:simulbridge2}Let $\Gamma=(J_j, K_j)_{j \in \Z_m}$ be a polygonal curve. Let
  $I=\{a, \ldots, b\}$ be an interval of $\Z_m$ of length $\geq 2$.
  % Define $\bar I=[a, b+1]$.
  For $x, y$ be two points in the roof $\tilde J_{I}$ over $(K_i)_{i \in I}$, denote as
  $\gamma_{xy}$ the subsegment of $\tilde J_{I}$ between $x$ and $y$. Let
  $\gamma =\gamma_1 \smallbullet \kappa_1 \smallbullet \gamma_2 \smallbullet \cdots \smallbullet
  \kappa_n\smallbullet \gamma_{n+1}$ (with $n\geq 1$) be a path joining $x$ and $y$, homotopic to
  $\gamma_{xy}$, such that:
  \begin{itemize}
  \item for $1 \leq i \leq n+1$, $\gamma_i$ is a piecewise geodesic path;
  \item for $1 \leq i \leq n$, $\kappa_i$ is a geodesic segment contained in a segment $K_{m_i}$
    for a $m_i\in I$.
  \end{itemize}
  Then, for any  arbitrary subset $\pi$ of $\Z_m$,
  \begin{align}\label{e:xy_bis}\ell(\gamma_{xy})\leq \sum_{\substack{1\leq i \leq n \\ \kappa_i\notin \pi}}
    \ell(\kappa_i) +\sum_{i=1}^{n+1} \ell(\gamma_i) + \sum_{j\in I \cap \pi} (\log Z^\Gamma_{K_j}+2)
    % +  \sum_{i\in I\cap \cO}  \ell(\tilde J_i)
    ,\end{align}
  where $\kappa_i\notin \pi$ means that the segment $K_{m_i}$ containing $\kappa_i$ has index $m_i\not\in \pi$.
\end{prp}

\begin{figure}[h!]
  \includegraphics{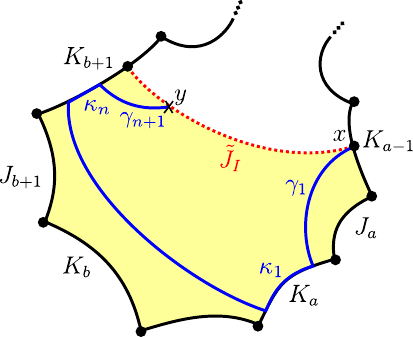}
  \caption{Illustration of the length inequality in a neutralized polygonal curve.}
  \label{fig:variational}
\end{figure}%

In practice, $\gamma$ will be simple (which implies that the indices $m_i$ are monotonically
ordered), and the paths $\gamma_i$ will not intersect the cells $(K_j)_{j \in\Z_m}$ except at their
endpoints.

\begin{rem}
  If $x$ is the origin of $\tilde J_{I}$, then we can replace in this upper bound $\gamma_1$ by any
  piecewise geodesic path~$\gamma'_1$, starting anywhere on $K_{a-1}$ and ending at the origin of
  $\kappa_1$. A similar remark holds if $y$ is the endpoint of $\tilde J_{I}$ lying in $K_{b+1}$.

  The same result also holds with $I =\Z_m$, $\tilde J_{I}$ replaced by $\Gamma$, and
  $\gamma=\gamma_1 \smallbullet \kappa_1 \smallbullet \gamma_2 \smallbullet \cdots \smallbullet \kappa_n$ a
  piecewise geodesic curve (hence closed), freely homotopic to $\Gamma$ (the only difference being
  that there is the same number of $\kappa_i$ and $\gamma_i$, i.e. \eqref{e:xy_bis} now runs for
  $1 \leq i \leq n$ for both sums).
\end{rem}

% Comparing \eqref{e:xy_bis} to \eqref{e:xy}, we see that for the segments $k_i$ included in $K_j$
% with $j\in \pi$, we have preferred to use $(\log Z^\Gamma_{K_j}+2)$ instead of $\ell(k_i)$ in the
% upper bound. 

\begin{proof}
  First we observe that, under the hypotheses of the statement, by the variational definition of the
  geodesic,
  \begin{align}\label{e:xy}\ell(\gamma_{xy})\leq \sum_{i=1}^n \ell(\kappa_i) +\sum_{i=1}^{n+1}
    \ell(\gamma_i).\end{align}
  We then observe that the path $\gamma$ has to intersect each of the segments $W_{j_i}$ defined
  in Definition~\ref{d:Zj}. More precisely, there is a portion of $\gamma$ containing $\kappa_i$ and homotopic to a portion of $W_{j_i}$.
  Thus, we replace $\gamma$ by a piecewise geodesic path in the same homotopy class, where each segment
  $\kappa_i$ such that $m_i\in \pi$ is suppressed and replaced by a portion of $W_{m_i}$. The result then follows from the minimizing property of $\gamma_{xy}$
  together with Lemma \ref{l:WZ}.
\end{proof}

\subsubsection{New representants of $\Gamma_\lambda$} In the spirit of the previous length
inequality, we construct a new representant to each free homotopy class $\Gamma_\lambda$, which
approximates its length well.

\begin{nota}
  Throughout the rest of this section, we fix an exhaustion
  $(\mathbf{\Gamma}, \mathbf{d})=(\Gamma^{k}, d^{k})_{1 \leq k \leq  t}$ of the family of
  polygonal curves $\Gamma=(\Gamma_\lambda)_{\lambda\in\Lambda}$.   For
  $\lambda\in \Lambda$, write
  $$\Gamma^{k}_\lambda=(J_j^{\lambda, k}, K_j^{\lambda, k})_{j \in \Z_{n_\lambda^k}}.$$
  For $k=t$ the terminal polygonal curve, we simply denote
  $\Gamma^{\infty}_\lambda=(J_j^\lambda, K_j^\lambda)_{j \in \Z_{n_\lambda}}$.
\end{nota}

Remember that we also fixed a $\xi \in \Xi$ and $\cQ \subseteq I(\xi)$, see Notation \ref{nota:fix_Q}.

In the following we focus on the elements $\lambda \in \LambdaBCFN$, that is to say, boundary
curves.  Let us call $\gamma_\lambda$ the closed geodesic homotopic to $\Gamma^1_\lambda$.  For
$k \in \{1,\ldots, t\}$ there exists a closed annular region $\rA_\lambda^k$ bordered, on one side
by $\gamma_{\lambda}$, on the other side by the polygonal curve $\Gamma^{k}_\lambda$.
%on peut eventuellement le dire bien avant.
Note that we have the inclusion $\rA_\lambda^k\subseteq \rA_\lambda^{k-1}$ for $k \geq 2$.
% Denote  $\rA^*(k, \lambda)= \rA(k, \lambda)\setminus \rA(k+1, \lambda)$: $\rA^*(k, \lambda)$ consists of a finite number
%of contractible right-angled polygons, made of the regions that disappear when certain cells of $\Gamma^{k}_\lambda$ are neutralized to get $\Gamma^{k+1}_\lambda$. By definition, the sets $\rA^*(k, \lambda)$ are disjoint.

The following lemma is an immediate consequence of our construction.

\begin{lem}
  For $k \in \{1, \ldots, t\}$, we can find a simple, piecewise geodesic representative
  $\tilde\Gamma_\lambda^k$ of the free homotopy class $\Gamma_\lambda$, contained in the annulus
  $\rA_\lambda^k$, written as a cyclic concatenation
  \begin{align}
    \tilde\Gamma_\lambda^k
    =\gamma_1^{\lambda,k} \smallbullet \kappa_1^{\lambda,k}  \smallbullet
    \ldots \smallbullet \gamma^{\lambda,k}_{s_\lambda^k} \smallbullet \kappa^{\lambda,k}_{s_\lambda^k}
    % \smallbullet \gamma_{s(k)+1}^k
  \end{align}
  where:
  \begin{itemize}
  \item $\gamma_j^{\lambda,k}$ is piecewise geodesic, made of portions of junctions and bridges of
    $\Gamma^k_\lambda$;
      \item $\kappa_j^{\lambda,k}$ is a geodesic segment contained in one of the cells
    of $\Gamma_{\lambda}^k$.
  \end{itemize}
  The representative $\tilde\Gamma_\lambda^k$ is unique if we ask that the annulus
  $\tilde\rA_\lambda^k$ bordered by $\gamma_{\lambda}$ and by $\tilde\Gamma_\lambda^k$ does not
  contain any junction in its interior.
\end{lem}

In particular, $\gamma_j^{\lambda,k}$ does not intersect the cells of $\Gamma_\lambda^{k}$ except at
its endpoints.

For $k=t$, we write $\gamma_j^\lambda = \gamma_j^{\lambda,t}$ and
$\kappa_j^\lambda=\kappa_j^{\lambda,t}$ for $1 \leq j \leq s_\lambda = s_\lambda^t$.
% \begin{nota}
%   The segments $(\kappa_j^{\lambda,k})_{1 \leq j \leq s_\lambda^k}$ are called \emph{residual segments} (of
%   generation $k$). We denote by $j_1^{\lambda,k}, \ldots, j_{m_\lambda^k}^{\lambda,k}$ the portions of bridges
%   $ J_1^{\lambda, k}, ..., J_{n_\lambda^k}^{\lambda, k}$ contained in
%   $\bigcup_{j=1}^{s_\lambda^k}\gamma_j^{\lambda,k}$.
% \end{nota}
% \begin{rem}
%   For $k=1$, the bridges $J_j^{\lambda, 1}$ coincide with the bars $\overline{B}_q$, and therefore
%   $j_1^{\lambda,1}, \ldots, j_{m_\lambda^1}^{\lambda,1}$ are subsegments of these bars.
% \end{rem}
% Denote $\kappa_j^{*, k}:= \kappa_j^k \cap \rA^*(k, \lambda)$, $j_m^{*, k}:= j_m^k\cap \rA^*(k, \lambda)$, $\gamma_m^{*, k}:= \gamma_m^k \cap \rA^*(k, \lambda)$.

\subsubsection{Length inequalities: junctions}

We prove the following comparison between the lengths of the paths
$(\gamma_j^{\lambda,1})_{1 \leq j \leq s_\lambda^1}$ and the total length of all junctions.

\begin{lem}
  \label{lem:junc}
  We have that
  $$ \sum_{\lambda\in \LambdaBCFN} \sum_{j=1}^{s_\lambda^1} \ell(\gamma_j^{\lambda,1})
\leq \sum_{p \in \Junc(\cQ)} \ell(p). $$
\end{lem}

\begin{proof}
  We observe that, for $\lambda \in \LambdaBCFN$, the union of paths
  $\bigcup_{1 \leq j \leq s_\lambda^1} \gamma_j^{\lambda,1}$ is composed of disjoint subsegments of
  junctions, that we call $(j_l^\lambda)_{1 \leq l \leq M_\lambda}$, and subsegments of the bridges of
  $\Gamma^1$, which we denote as $(b_{l}^{\lambda})_{1 \leq l \leq m_\lambda}$. Since the bridges of
  $\Gamma^1$ are bars $\overline{B}_q$, the paths $(b_{l}^{\lambda})_{1 \leq l \leq m_\lambda}$ are
  subsegments of bars.

  We claim that to each segment $b_l^\lambda$ we can associate a subset $P(b_l^\lambda)$ satisfying
  the following properties:
  \begin{itemize}
  \item $P(b_l^\lambda)$ is included in $\bigcup_{p \in \Junc(\cQ)} p$ and disjoint from all
    $(j^{\lambda'}_{l'})_{\lambda', 1 \leq l' \leq M_{\lambda'}}$;
  \item $\ell(b_l^\lambda)\leq \ell(  P(b_l^\lambda) )$;
  \item the $P(b_l^\lambda)$ are mutually disjoint for distinct $(l, \lambda)$.
  \end{itemize}
  Indeed, $b_l^\lambda$ is included in the bar $\overline{B}_q$ for an index $q$. This means that
  $b_l^\lambda$ is one of the boundary segments of the polygon $\Pol(q)\cup \Pol(\iota q)$. The two
  polygons $\Pol(q)$ and $\Pol(\iota q)$ share a side, namely $\overline{B}_q$. Because their union
  still contains the subset $b_l^\lambda$ of $ \overline{B}_q$ in its boundary, their intersection
  $\Pol(q)\cap \Pol(\iota q)$ must be a polygon, the boundary of which consists of the union of:
\begin{itemize}
\item the entire segment $b_l^\lambda$ itself;
\item a subset of $\bigcup_{p \in \Junc(\cQ)} p$ made of pieces of boundary of $\Pol(q)$ in the
  interior of $\Pol(\iota q)$, or pieces of boundary of $\Pol(\iota q)$ in the interior of
  $\Pol(q)$;
\item and, perharps, some pieces of $\beta_{\lambda(q)}$ or $\beta_{\lambda(q-1)}$, orthogonal to
  $b_l^\lambda$.
\end{itemize}
We define $P(b_l^\lambda)$ to be the set of points in the boundary of $\Pol(q)\cap \Pol(\iota q)$
which belong in $\bigcup_{p \in \Junc(\cQ)} p$.  By an enumuration of all cases, a few of which are
illustrated in Figure \ref{fig:junc}, we check that this set satisfies the length inequality, using
either the variational characterization of geodesics or orthogeodesics.

\begin{figure}[h]
  \centering
  \includegraphics{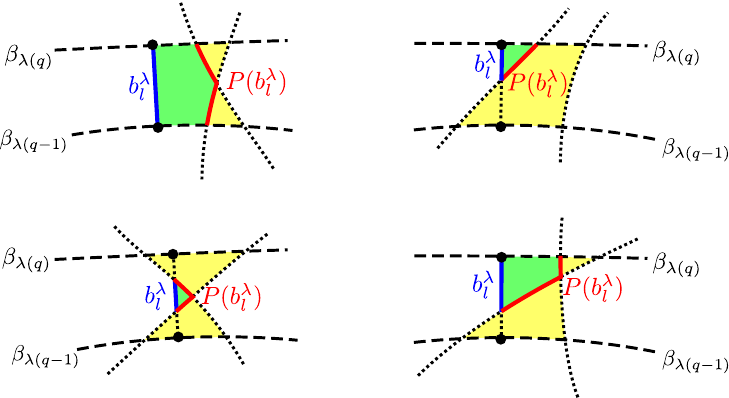}
  \caption{The construction of $P(b_l^\lambda)$ in all cases where $q$ and $\iota(q)$ are crossing
    parameters. The overall shaded area is the union of $\Pol(q)$ and $\Pol(\iota(q))$, with the
    intersection represented in green. The dotted lines are the irrelevant portions of the junctions
  and bars, and the dashed lines the multi-curve $\beta$.}
  \label{fig:junc}
\end{figure}

   We can then write
   \begin{align*} \sum_{\lambda\in \LambdaBCFN}
     \sum_{j=1}^{s_\lambda^1} \ell(\gamma_j^{\lambda,1})
     &= \sum_{\substack{\lambda \in \LambdaBCFN \\ 1 \leq l \leq M_\lambda}} \ell(j^\lambda_l )
       + \sum_{\substack{\lambda \in \LambdaBCFN \\ 1 \leq l \leq m_\lambda}} \ell(b_{l}^{\lambda}) \\
     &\leq \ell \Big( \bigsqcup_{\substack{\lambda \in \LambdaBCFN \\ 1 \leq l \leq M_\lambda}} j_l^\lambda
     \sqcup
     \bigsqcup_{\substack{\lambda \in \LambdaBCFN \\ 1 \leq l \leq m_\lambda}} P(b_l^\lambda)\Big)
     \leq \sum_{p \in \Junc(\cQ)} \ell(p)
   \end{align*}
   because the  $j^\lambda_l$ and $P(b_l^\lambda)$ are disjoint subsegments of the set of junctions.
 \end{proof}

 \subsubsection{Length inequalities and double points}

 Let us now prove some length inequalities related to the portions of cells
 $(\kappa_j^\lambda)_{\lambda, j}$.  We observe that the sets
 $(\beta_\lambda)_{\lambda \in \LambdaBCbeta}$ and
 $(\kappa_j^{\lambda})_{\lambda\in \LambdaBCFN, 1 \leq j \leq s_\lambda}$ are included in $\beta$,
 and that each point of $\beta$ belongs to at most two of them. In particular,
 \begin{equation}
   \label{eq:ineq_bf_double}
   \sum_{\lambda \in \LambdaBCbeta}\ell_Y(\beta_\lambda)
   +\sum_{\lambda\in \LambdaBCFN} \sum_{1\leq j \leq s_\lambda} \ell(\kappa_j^\lambda)  \leq 2 \ell_Y(\beta).
 \end{equation}

 This is also true if we restrict the left-hand-side to smaller sets.
 \begin{nota}
   In the rest of this section, we shall consider arbitrary subsets $V^\Gamma\subseteq \LambdaBCFN$,
   $V^\beta\subseteq \LambdaBCbeta$, $\pi \subseteq \Thetasurv$. We let
   $V = V^\beta \sqcup V^\Gamma$. We write $\kappa_j^\lambda \notin \pi$ iff the cell containing
   $\kappa_j^\lambda$ is not in $\pi$. We denote
   $\Theta_\tau(\pi):=\bigcup_{q\in\pi} \Theta_\tau(q)$ the set of indices corresponding to the
   segments $\cK_{q}$ that are in $\pi$.
 \end{nota}
 Plugging in these subsets, \eqref{eq:ineq_bf_double} directly implies
 \begin{equation}
   \label{eq:ineq_bf_double_rest}
   \sum_{\lambda \in V^\beta}\ell_Y(\beta_\lambda)
   +\sum_{\lambda\in V^\Gamma}
   \sum_{\substack{1\leq j \leq s_\lambda \\ \kappa_j^\lambda \notin \pi}} \ell(\kappa_j^\lambda)
   \leq 2 \ell_Y(\beta).
 \end{equation}
 The aim of this subsection is to improve this inequality by determining the points which are
 counted twice.

 \begin{defa}
   We call a point in $\beta$ \emph{simple} if it satisfies either of these conditions:
   \begin{itemize}
   \item it belongs to $\bigcup_{\lambda \in \Lambdainbeta}\beta_\lambda$ and to at most one
     interval $\kappa_j^\lambda$ with $\lambda\in V^\Gamma$ and
     $\kappa_j^\lambda \notin \pi$;
   \item it belongs to $\bigcup_{\lambda\in V^\beta}\beta_\lambda$ and to none of the intervals
     $\kappa_j^\lambda$ with $\lambda\in V^\Gamma$ and $\kappa_j^\lambda \notin \pi$;
   \item or it belongs to $\bigcup_{\lambda \in \LambdaBCbeta\setminus V^\beta}\beta_\lambda$.
   \end{itemize}
   Other points are called \emph{double}. We call $D$ the set of double points.
 \end{defa}

 Obviously, if $|D|$ stands for the 1-dimensional Lebesgue measure of $D$ as a subset of $\beta$,
  \begin{equation}
   \label{eq:ineq_bf_double_rest_D}
   \sum_{\lambda \in V^\beta}\ell_Y(\beta_\lambda)
   +\sum_{\lambda\in V^\Gamma}
   \sum_{\substack{1\leq j \leq s_\lambda \\ \kappa_j^\lambda \notin \pi}} \ell(\kappa_j^\lambda)
   \leq \ell_Y(\beta) + |D|.
 \end{equation}
 We now prove a bound on $|D|$.

 \begin{prp}
   \label{prp:bound_D}
   We have that
   \begin{equation*}
     |D|\leq
     \sum_{\substack{q \in \cQ \cap \Theta_\ell(V) \\ q \notin \Theta_\tau(\pi)}} \tilde \theta_q.
   \end{equation*}
 \end{prp}

 The proof essentially revolves around the following lemma.

\begin{lem} \label{l:shield} All the following points are simple:
\begin{itemize}
\item[(a)] points in $\cK_q$ where $\cK_q$ is a cell of $\Gamma_\lambda$ with $\lambda\notin V^\Gamma$
  or $ \kappa_l^\lambda\in \pi$;
\item[(a')] points in $\bigcup_{\lambda\notin V^\beta}\beta_\lambda$;
\item[(b)] for $q\in I(\xi)$, points in $\overline{\cI}_q \setminus [G_q, D_q]$;
\item[(c)] points in $\overline{\cI}_q$ where $q\not \in I(\xi)$;
\item[(d)] for $q\in I(\xi)\setminus \cQ$, all points in $\overline{\cI}_q$;
\end{itemize}
\end{lem}

\begin{proof}[Proof of Lemma \ref{l:shield}]
  Items (a), (a') are obvious. 
  
  Item (b) is also straightforward, because there are two junctions arriving at $G_q$ and $D_q$,
  which are responsible for the points in $\overline{\cI}_q \setminus [G_q, D_q]$ being simple.

  For item (c), if $q\not \in I(\xi)$ then there is a junction that crosses $\beta_{\lambda(q)}$
  transversally, say at a point $x_q$, which is a vertex both of $\Pol(q)$ and of $\Pol(q+1)$. The
  two intervals $[x_q, o(\overline{\cI}_{q+1})]$ and $[x_q, t(\overline{\cI}_q)]$ consist of simple
  points (see Figure \ref{fig:cross_shield}), and their union contains $\overline{\cI}_q$.

  \begin{figure}[h]
    \centering
    \includegraphics{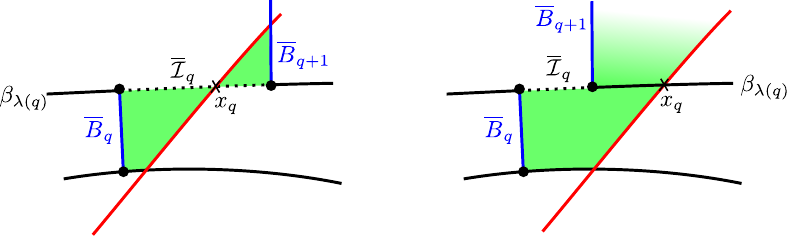}
    \caption{Illustration of case (c) in the proof of Lemma \ref{l:shield}. The two highlighted
      polygonals are $\Pol(q)$ and $\Pol(q+1)$. The dashed path is $\overline{\cI}_q$.}
    \label{fig:cross_shield}
  \end{figure}

  Regarding item (d), if $q\in (I(\xi)\setminus \cQ)\cap \cC$, there is again a junction that
  crosses $\beta_{\lambda(q)}$ transversally, and we can argue as for Item (c). If
  $q\in (I(\xi)\setminus \cQ)\cap \cU$ we use Remark \ref{rem:seg_shield}, showing that the
  points in $\overline{\cI}_q $ are simple.
\end{proof}

\begin{proof}[Proof of Proposition \ref{prp:bound_D}]
  We write $\beta=\bigcup_{q\in \Theta} \overline{\cI}_q$, where the geodesic segments are just
  defined as subsets of $\beta$, that is to say, we forget about their orientations. These intervals
  may overlap: this is the case if some of the $\theta_q$ were negative, or were positive but very
  large. By Lemma \ref{l:shield}, the only possibly double points are contained in segments
  $[G_q, D_q]$ where $q\in \cQ \cap \Theta_\ell(V)$ with $q \notin \Theta_\tau(\pi)$. This
  implies the claim if we recall that $\tilde \theta_q$ is the length of $[G_q, D_q]$.
\end{proof}

\subsubsection{The upper bound} We are now ready to state and prove our upper bound.

\begin{prp}\label{p:boundGlambda}
  Let $\xi \in \Xi$, $\cQ \subseteq I(\xi)$, and $(\mathbf{\Gamma},\mathbf{d})$ be an exhaustion of
  $(\Gamma_\lambda)_{\lambda \in \Lambda}$. For any arbitrary sets $V \subseteq \LambdaBC$ and
  $\pi \subseteq \Thetasurv$, on the support of $\Psi_{\mathbf{\Gamma}, \mathbf{d}}$, we have
  \begin{align*}
    \sum_{\lambda\in V}\ell_Y(\Gamma_\lambda) 
    \leq \ell_Y({\mathbf{c}}) +
    \sum_{\substack{q \in \cQ \cap \Theta_\ell(V) \\ q \notin \Theta_\tau(\pi)}}\tilde \theta_q
    +\sum_{q\in \pi} (\log Z_{\mathbf{\Gamma}, \mathbf{d}}(\cK_q) +2) 
    + \fn(r)
  \end{align*}
  where $\Theta_\tau(\pi) = \bigcup_{q \in \pi} \Theta_\tau(q)$.
\end{prp}

\begin{proof}
  Denote $\Vbeta = V \cap \Lambdabeta$ and $\VFN=V \cap \LambdaFN$.
  For $\lambda \in \LambdaBCFN$ and $k \in \{1, \ldots, t-1\}$, we define
  $$\cS_\lambda^k=\{K\in \Cell(\Gamma_\lambda^k) : \exists K'\in \Cell(\Gamma^{k+1}), K'\subseteq K\},$$
  the set of cells of $\Gamma_\lambda^k$ that ``survive'' at step $k+1$ of the exhaustion. Denote
  $\cZ_\lambda^k=\Cell(\Gamma_\lambda^k)\setminus \cS_\lambda^k $ the set of cells of
  $\Gamma_\lambda^k$ that are neutralized at step $k+1$.
  By repeated application of Proposition~\ref{p:simulbridge2}, we obtain that, for any $Y$ in the
  support of $\Psi_{\mathbf{\Gamma}, \mathbf{d}}$,
  \begin{align*}
    \ell_Y(\Gamma_\lambda)=\ell(\gamma_\lambda)
    \leq \sum_{j=1}^{s_\lambda} \ell(\gamma_j^\lambda)
    +\sum_{\substack{1 \leq j \leq s_\lambda \\ \kappa_j^\lambda \notin \pi}}
     \ell(\kappa_j^\lambda)
     +\sum_{q \in \pi \cap \Theta_\ell(\lambda)} (\log Z_{\mathbf{\Gamma}, \mathbf{d}}(\cK_q) +2)
  \end{align*}
  and, for $k \in \{2, \ldots, t\}$,
  \begin{align*} 
    \sum_{j=1}^{s_\lambda^k} \ell(\gamma_j^{\lambda,k})
    \leq  \sum_{j=1}^{s_\lambda^{k-1}} \ell(\gamma_j^{\lambda,k-1})
    +\sum_{K\in \cZ_\lambda^{k-1}} (\log Z^{\mathbf{\Gamma}}_K   +2).
 \end{align*}
 Put together, this yields
 \begin{align*}
   \ell_Y(\Gamma_\lambda)\leq & \sum_{j=1}^{s_\lambda^1} \ell(\gamma_j^{\lambda,1})
   + \sum_{\substack{1 \leq j \leq s_\lambda \\ \kappa_j^{\lambda} \notin
   \pi}} \ell(\kappa_j^\lambda) \\
  & +\sum_{q\in \pi \cap \Theta_\ell(\lambda)} (\log Z_{{\mathbf{\Gamma}}, {\mathbf{d}}}(\cK_q)  +2) 
   +\sum_{k=1}^{t-1}\sum_{K\in \cZ_\lambda^{k}} (\log Z^{\mathbf{\Gamma}}_K +2).
   \end{align*}

   Let us now use the decoration to further estimate this length.  If $k\leq t-1$, then for any cell
   $K\in \cZ_\lambda^{k}$, there is a cell $K'$ of $\Gamma^k$, containing $K$, and bearing the
   decoration $0$ for the exhaustion $(\mathbf{\Gamma}, \mathbf{d})$. As a result, by definition of
   the test function $\Psi_{\mathbf{\Gamma}, \mathbf{d}}$, on its support, we have
   $\log Z^\mathbf{\Gamma}_{K'}\leq a+1$. Proposition \ref{p:6} then implies
   $\log Z^\mathbf{\Gamma}_K\leq a+3$.
   We can then deduce from the previous bound, after summation over $\lambda \in V^\Gamma$,
   that 
   \begin{align*}
     \sum_{\lambda\in V^\Gamma}\ell_Y(\Gamma_\lambda)
     \leq & \sum_{\substack{\lambda\in V^\Gamma \\ 1 \leq j \leq s_\lambda^1}} \ell(\gamma_j^{\lambda,1})
       + \sum_{\lambda\in V^\Gamma} \sum_{\substack{1 \leq j \leq s_\lambda \\ \kappa_j^{\lambda} \notin
     \pi}} \ell(\kappa_j^\lambda) 
    \\ & +\sum_{q\in \pi} (\log Z_{\mathbf{\Gamma}, \mathbf{d}}(\cK_q) +2) 
       + \sum_{\substack{1 \leq k \leq t-1 \\ \lambda\in V^\Gamma \\ K\in \cZ_\lambda^{k}}} (a+9).
   \end{align*}
   In the last sum, the number of terms in the triple sum
   $\sum_{k=1}^{t-1} \sum_{\lambda\in V^\Gamma}\sum_{K\in \cZ_\lambda^{k}}$ is less than the total
   number of cells of all $\Gamma_\lambda$ with $\lambda\in V^\Gamma$; this is less than $4r$.
   
   We now use Lemma \ref{lem:junc}, equation \eqref{eq:ineq_bf_double_rest_D} and Proposition
   \ref{prp:bound_D} to bound the sum over $\gamma_j^\lambda$ and $\kappa_j^\lambda$. We obtain that
   \begin{align*}
     & \sum_{\lambda \in V} \ell_Y(\Gamma_\lambda)
       =  
       \sum_{\lambda\in V^\Gamma}\ell_Y(\Gamma_\lambda)
     + \sum_{\lambda \in V^\beta} \ell_Y(\beta_\lambda) \\
    & \leq \sum_{p \in \Junc(\cQ)} \ell(p)
     + \sum_{\lambda \in \Lambda^\beta}\ell_Y(\beta_\lambda)
     +     \sum_{\substack{q \in \cQ \cap \Theta_\ell(V) \\ q \notin \Theta_\tau(\pi)}}\tilde \theta_q
     +\sum_{q\in \pi} (\log Z_{\mathbf{\Gamma}, \mathbf{d}}(\cK_q) +2) 
       + \fn(r).
   \end{align*}
   To conclude, we use the lower bound proven at the beginning of this section, \eqref{e:uppppper}.    
 \end{proof}

%%%Local Variables: 
%%% mode: latex
%%% TeX-master: "main"
%%% End: 

\section{Arbitrary loop topologies}
\label{s:othercases}
We now indicate how to prove Theorem \ref{t:main} when the multi-loop ${\mathbf{c}}$ is not a generalized eight.
In this case, the representatives of the homotopy class of ${\mathbf{c}}$ in minimal position may not be isotopic to one another. We choose arbitrarily one such representative -- for instance, by putting an auxiliary hyperbolic metric on ${\mathbf{S}}$ and choosing the geodesic representative of ${\mathbf{c}}$ (possibly perturbing it to get rid of multiple intersections).

Let us denote by $\cN \subseteq \Sf$ a regular neighbourhood of ${\mathbf{c}}$. The pair
$(\cN, {\mathbf{c}})$ is a generalized eight, and we may apply the construction of the pairs of
pants decomposition of \S \ref{s:ppdecompo}. We keep the same notation
$(\Gamma_{\lambda})_{\lambda\in \Lambda}$ for the pair of pants decomposition of $\cN$, with the
indices $\LambdaBC = \LambdaBCbeta \sqcup \LambdaBCFN$ corresponding to the boundary components of
$\cN$, split into the ones which are, or not, components of the simple multi-loop $\beta$.

In $\mathbf{S}$, some of the curves $(\Gamma_{\lambda})_{\lambda \in \LambdaBC}$ may be
contractible. % We denote by $\LambdaBC = \LambdaC \sqcup \LambdaNC$ the indices corresponding to
% contractible v.c. non-contractible components $(\Gamma_\lambda)_{\lambda \in \LambdaBC}$.  We define
% the intersections
% \begin{align*}
%   & \LambdaCbeta := \LambdaC \cap \LambdaBCbeta
%   & \LambdaCFN := \LambdaC \cap \LambdaBCFN \\
%   & \LambdaNCbeta := \LambdaNC \cap \LambdaBCbeta
%   & \LambdaNCFN := \LambdaNC \cap \LambdaBCFN,
% \end{align*}
% i.e., for instance, $\LambdaCbeta$ is the set of components of $\beta$ which are boundary components
% of $\cN$ contractible in $\mathbf{S}$.
 Proposition \ref{p:comparison1} tells us that the existence of contractible $\beta_\lambda$ and
$\Gamma_{\lambda}$ should help to prove Theorem \ref{t:main}, by making comparison estimates easier.
On the other hand, contractible boundary components cause two new difficulties.
\begin{itemize}
\item While it is relatively easy to enumerate topological types of generalized eights, it seems
  daunting to make a list of all topological types of multi-loops. We sidestep this difficulty by
  only ever working with a fixed topological type, in this article. In the applications to spectral
  gaps of random hyperbolic surfaces \cite{Expo}, the discussion involves all possible topological
  types appearing in the Selberg trace formula, but we condition the Weil--Petersson measure on the
  set of tangle-free surfaces, in order to drastically reduce the number of local topological types
  to consider.
\item When we vary the metric $Y$ on $\mathbf{S}$ and consider the piecewise geodesic representative
  of~$\mathbf{c}$, as done in \S \ref{s:straightening2}, the orthogeodesics
  $(\overline{B}_q)_{q \in \Theta}$ may be ill-defined. This happens if $B_q$ is a \emph{problematic
    bar} in the sense of Definition \ref{d:problem}.
\end{itemize}

 \subsection{Diagrams and the reconstitution procedure}
 \label{s:reconstitution}

 We recall the notion of diagram $\mathbf{D}$ representing the multi-loop $\mathbf{c}$ introduced in
 Definition \ref{defa:diagram}. A diagram is the data of the simple multi-loop $\beta$ and the bars
 $B = (B_1, \ldots, B_r)$.

 \subsubsection{Simple portions of diagrams}
 
 We extend the notions of simple portions, introduced in Definition~\ref{d:double_fill} for
 multi-loops, to diagrams as follows: a \emph{simple portion} of $\mathbf{D}$ is a maximal
 subsegment of~$\beta$ that does not meet the endpoints of the bars $B$. These simple portions may
 be \emph{shielded} or \emph{unshielded}, following once again Definition \ref{d:double_fill}.  The
 (un)shielded simple portion of~${\mathbf{c}}$ are in natural bijection with (un)shielded simple
 portion of $\bD$.

 We do not apply this terminology to the bars $B_k$, because they do not play the same role as the
 curves $\beta_i$: indeed, when we describe the isotopy class of $\mathbf{c}$ by following its
 diagram, all the bars are traversed twice, whilst each simple portion of $\beta$ only once.

 \subsubsection{Reconstitution procedure}
\label{sec:reconst-proc}

Given a diagram $\mathbf{D}=(\beta, B)$, we may reconstitute a multi-loop~$\mathbf{c}$ stemming from
$\mathbf{D}$ by the following procedure, guided by the \emph{orientation rules} in~\S
\ref{s:orientation_bars}.

The endpoints of the bars cut $\beta$ into $2r$ non-oriented segments $\cI_j$.
Pick a starting segment $\cI_{j_0}$ and pick arbitrarily one of the two possible orientations of
this segment. There a unique closed path starting with $\cI_{j_0}$, obtained by concatenating
alternatively bars $B_k$ and segments $\cI_j$, and respecting the orientation
rules. If this path goes through all segments $\cI_j$, then we obtain a loop (with one component)
stemming from $\mathbf{D}$. If not, select another segment $\cI_j$ that has not been traversed yet,
and repeat the procedure until exhaustion of all segments $\cI_j$.

Starting from a diagram $\mathbf{D}$ with $r$ bars, and letting
$\Theta_{\mathbf{D}}:=\{1, \ldots, r\}\times \{\pm\}$, applying the reconstitution procedure gives
us an orientation of each segment $\cI_j$ and their labelling by elements of $\Theta_{\mathbf{D}}$,
as well as two oriented copies $(B_1^\pm, \ldots, B_r^\pm)$ of each bar, labelled by
$\Theta_{\mathbf{D}}$.

If we start with a multi-loop $\mathbf{c}$ and obtain a diagram $\mathbf{D}$ by opening the
intersections (either way), the reconstitution procedure will give us back a multi-loop isotopic to
$\mathbf{c}$ (modulo orientation of each component).

The reader may check that a different choice of the original orientation of $\cI_{j_0}$ will just
reverse the orientation of the corresponding component of $\mathbf{c}$.  This means that if a
multi-loop stemming from $\mathbf{D}$ has $\cc$ components, then the reconstitution procedure can
lead to $2^{\cc}$ multi-loops stemming from $\mathbf{D}$, all having the same geometric image, and
differing only by orientation of each component.

%There are some caveats if we want to go in the other direction. If we start from $\mathbf{D}$ and apply the reconstitution procedure, a multi-loop $\mathbf{c}$ stemming from $\mathbf{D}$ may not be in minimal position. Calling $\mathbf{c}^{min}$ the representative of the homotopy class of $\mathbf{c}$ with minimal intersection number, $\mathbf{D}$ does \emph{not} necessarily represent $\mathbf{c}^{min}$. More generally,
%two multi-loops that are homotopic without being isotopic are associated with distinct diagrams.
% Notice also that even when $\mathbf{D}$ represents $\mathbf{c}$, the opening intersections, \emph{but not necessarily the first way}.

It seems difficult to determine when two diagrams lead to homotopic multi-loops; we do not address
this problem in this paper. In the case of generalized eights, this problem does not occur and we can
state the following:
\begin{rem}\label{p:D8}
  Let $\bD$ be a diagram. Assume that ${\mathbf{ S}}\setminus \bD$ has no connected component which
  is a disk. Then a multi-loop stemming from $\bD$ is a generalized eight and,
%and is in minimal position.
  as a consequence,~$\bD$ represents a generalized eight; all generalized eights leading to the
  diagram $\bD$ are isotopic (modulo orientation of each component).
\end{rem} 

Note that each component of a multi-loop $\mathbf{c}$ obtained by the reconstitution procedure may
be written by a concatenation such as the one we have obtained in \S \ref{sec:cycle-decomp-rewr},
see \eqref{e:decompo-i}.

\subsubsection{Extended diagrams}

We can extend the notation of diagram, discarding the condition that the bars do not self-intersect:
\begin{defa}
  An extended diagram $\mathbf{D}$ in $\mathbf{S}$ is the data (modulo isotopy) of a collection of
  simple disjoint loops $\beta$, and of a collection of bars $B=(B_1, \ldots, B_r)$, defined as
  closed segments with distinct endpoints lying in $\beta$, and disjoint from $\beta$ otherwise.
\end{defa} 

The reconstitution procedure of \S \ref{sec:reconst-proc} also applies to an extended diagram
$\mathbf{D}$. \emph{Simple portions} of an extended diagram are defined as for usual diagrams. A
simple portion of $\mathbf{D}$ is said to be shielded (in $\Sf$) if it lies in the boundary of a
contractible component of $\Sf\setminus \mathbf{D}$.

\begin{defa}\label{d:minimality}
  Let $\mathbf{D}=(\beta, B)$ be an extended diagram in $\mathbf{S}$, and let
  $\mathbf{c}=(\mathbf{c}_1, \ldots, \mathbf{c}_\cc)$ be a multi-loop obtained from $\mathbf{D}$ by
  the reconstitution procedure. We say that $\mathbf{D}$ \emph{satisfies the minimality condition}
  if:
\begin{itemize}
\item the components $\mathbf{c}_i$ are non-contractible;
\item if $\tilde{\mathbf{c}}_i$ is a lift of a component of $\mathbf{c}$ in the universal cover of
  $\Sf$, then $\tilde{\mathbf{c}}_i$ has no self-intersection;
\item any two lifts $\tilde{\mathbf{c}}_i$, $\tilde{\mathbf{c}}_j$ of components of $\mathbf{c}$ in
  the universal cover of $\Sf$ can intersect at most once (with the convention that a shared bar
  counts for 1 intersection point).
\end{itemize}
\end{defa}

Remark that if we start with a multi-loop $\mathbf{c}$ in minimal position, as we always do, then
any diagram~$\mathbf{D}$ representing it automatically satisfies the minimality condition; so this
is a natural condition to ask for.
 
\begin{defa}\label{d:problem} A bar of an extended diagram is \emph{problematic} if it satisfies
  either of the following two conditions:
\begin{itemize}
\item one of its endpoints lies on a contractible loop $\beta_i$;
\item its two endpoints lie on one same loop $\beta_i$, and the bar is homotopically trivial with
  endpoints gliding along $\beta_i$.
\end{itemize}
\end{defa} 

\begin{defa}\label{d:cylind} A bar of an extended diagram is \emph{cylindrical} if it runs between $\beta_i$ and $\beta_j$ with $i\not=j$ and $\beta_i,\beta_j$ are non-contractible, in the same homotopy class.
\end{defa} 
Cylindrical bars count among the non-problematic bars.

\begin{defa}\label{d:good_ext}
An extended diagram $\mathbf{D}$ is said to be \emph{good} if it has no problematic bar and if it satisfies the minimality condition. 
\end{defa}

\subsection{An auxiliary generalized eight: the skeleton} 
\label{s:skeleton}

The forthcoming Proposition \ref{p:remove} says that, by opening intersections from a multi-loop
${\mathbf{ c}}$, we can create a generalized eight that has length \emph{smaller} than
${\mathbf{ c}}$ for any hyperbolic metric, \emph{and still fills ${\mathbf{ S}}$} (or a surface
obtained by removing cylinders from ${\mathbf{ S}}$).  Such an auxilliary loop is called a
\emph{skeleton}; simple examples are represented in Figure \ref{fig:non-isotopic}.

\subsubsection{Removing one disk}
\label{sec:removing-one-disk}

We first prove the following lemma. 
   
\begin{lem}\label{l:G8}
  Let $\bD$ be a diagram filling ${\mathbf{ S}}$. Assume that at least one of the connected
  components of ${\mathbf{ S}}\setminus \bD$ is homeomorphic to a disk.  Then we can remove one of
  the bars of $\mathbf{D}$ so that the resulting diagram $\bD'$
%\item is still admissible,
  fills a surface ${\mathbf{ S}}'$ obtained by removing cylinders from $\mathbf{S}$.
%Since $\bD'$ has exactly one bar less than $\bD$, this is equivalent to the fact that 
\end{lem}

We recall that we have defined the ways to remove a bar in a diagram in \S \ref{s:removing}.  By
construction, ${\mathbf{ S}}\setminus \bD'$ has exactly one less embedded disk than
${\mathbf{ S}}\setminus \bD$ among its connected components.

\begin{proof}
  If one of the loops $\beta_i$ are contractible, it borders a disk $D_i$.  We remove a bar $B_k$
  touching $\beta_i$ (and another components $\beta_j$) in the second way, to obtain $\bD'$. The
  filled surface ${\mathbf{ S}}'$ has same Euler characteristic as ${\mathbf{ S}}$.  If $B_k$ lies
  inside $D_i$, then $\beta_j$ is also contractible and obviously ${\mathbf{ S}}' ={\mathbf{ S}}$.
  Otherwise, the bar $B_k$ is outside the disk $D_i$. We have ${\mathbf{ S}}' ={\mathbf{ S}}$ except
  in one particular case, which occurs when $\beta_i=\beta_j$. In this case, ${\mathbf{ S}}' $ is
  obtained from ${\mathbf{ S}}$ by removing a cylinder.

Assume now that none of the $\beta_i$ are contractible. Consider a contractible connected component $D$ of ${\mathbf{ S}}\setminus \bD$: pick a bar $B_k$ in the boundary of this contractible component and remove it the first way. Either ${\mathbf{ S}}' ={\mathbf{ S}}$, or ${\mathbf{ S}}' $ is obtained from ${\mathbf{ S}}$ by removing a cylinder (this can only happen if the two sides of $B_k$ touch the open disk $D$).
\end{proof}

Examples of both cases appearing in this lemma can be found in Figure \ref{fig:non-isotopic}.

  \begin{figure}[h!]
     \includegraphics[height=9cm]{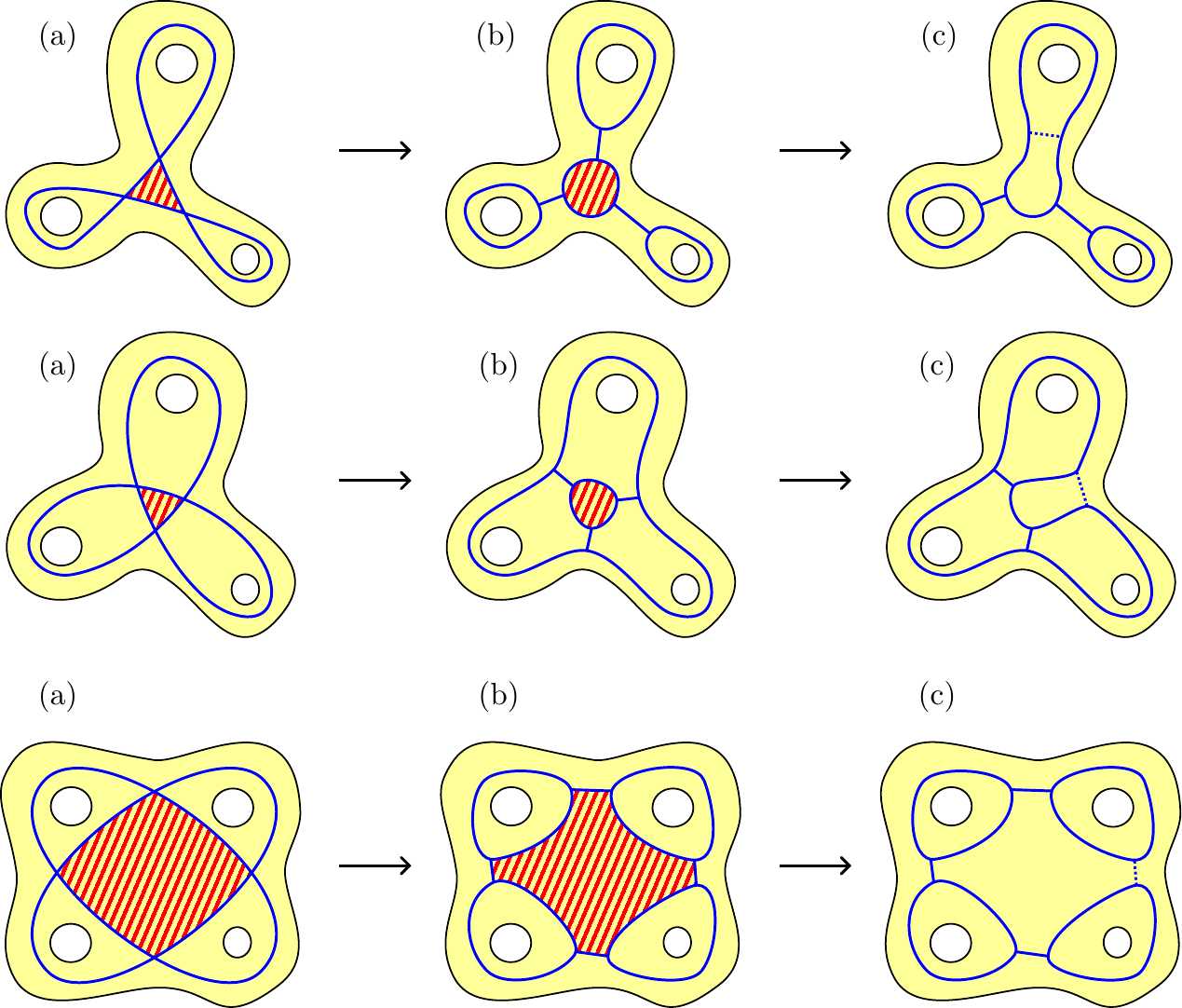}
  \caption{Three examples of (a) loops, (b) diagrams $D^{(1)}$ and (c) skeletons (shown in solid lines). The diagram $D^{(2)}$ is obtained by reincorporating the dashed bar to the skeleton. The two first examples are homotopic, non-isotopic loops.}    
    \label{fig:non-isotopic}
  \end{figure}

\subsubsection{Skeleton $\bD'$}
\label{sec:skeleton-bd}

By iteration of the process, and using Remark \ref{p:D8}, we obtain:
\begin{cor}\label{c:G8}
  For any diagram $\bD$ filling ${\mathbf{S}}$, we can remove bars so that the resulting
  diagram~$\bD'$ fills a surface ${\mathbf{S}}'$ obtained by removing cylinders from $\mathbf{S}$,
  and so that no connected components of ${\mathbf{ S}}'\setminus \bD'$ is contractible.  The
  diagram $\bD'$ represents a generalized eight ${\mathbf{ c}}'$.
\end{cor}

\begin{defa}
  \label{def:skeleton}
  Let ${\mathbf{ c}}$ be a multi-loop with $r$ self-intersections, and $\bD^{(1)}=(\beta, B)$ be the
  diagram obtained by opening all the intersections of ${\mathbf{ c}}$ the first way.  We call
  \emph{skeleton} of ${\mathbf{ c}}$ any diagram $\bD' = (\beta', B')$ obtained from $\bD^{(1)}$ by
  application of Corollary \ref{c:G8}.  We denote as $r' \leq r$ its the number of bars and
  $\Theta'=\{1, \ldots,r'\}\times\{\pm\}$.  We also refer to the generalized eight ${\mathbf{ c}}'$
  stemming from $\bD'$, obtained by removing intersections from ${\mathbf{ c}}$, as a skeleton of
  $\mathbf{c}$.
\end{defa}

\begin{rem}
  We have $\chi(\Sf)=\chi(\Sf')=r'$.
\end{rem}

\begin{rem} \label{e:simple} 
A skeleton is not uniquely defined, because it depends on the order in which we choose the intersection points to remove when iterating Lemma \ref{l:G8}. It suffices for us to know that a skeleton exists.
The forthcoming discussion is completely dependent on the choice of a skeleton.
\end{rem}

\begin{rem}
  \label{rem:disco_skeleton}
  Even when $\mathbf{S}$ is connected, the surface ${\mathbf{ S}}'$ may be disconnected, and some
  connected components may be cylinders containing a component of ${\mathbf{ c}}'$ which is a simple
  curve. This happens, in particular, when $\beta_i$ and $\beta_j^{\pm 1}$ are non-contractible
  curves bordering a cylinder containing a third curve $\beta_k$, in which case $\beta_k$ becomes a
  separate connected component of $\mathbf{S}'$ once all its adjacent bars are removed the first
  way. 
\end{rem}

Note however that if a bar of $\bD$ joins two non-contractible curves $\beta_i$, and is not in the
boundary of a contractible component of ${\mathbf{ S}}\setminus \bD$, then this bar stays in each
skeleton $\bD'$.  A skeleton $\bD'$ never contains the cylindrical bars and $\Sf'$ may have more
connected components than $\Sf$, including cylinders (filled by a simple curve).

The multi-curve $\beta'$ corresponding to a skeleton ${\mathbf{ c}}'$ is homotopic to the collection
of non-contractible components of $\beta$ in $\bD^{(1)}$. In particular, its number of components is
smaller.

\subsubsection{Length inequality and skeleton}
\label{sec:length-inequality-skeleton}

We are now ready to prove our length inequality.

\begin{prp}\label{p:remove}
  Let ${\mathbf{c}}$ be a multi-loop in minimal position, filling ${\mathbf{ S}}$. Let
  ${\mathbf{ c}}'$ be a skeleton of~${\mathbf{c}}$. Then, for any hyperbolic Riemannian metric $Y$
  on ${\mathbf{ S}}$,
\begin{align}
\ell_Y({\mathbf{ c}}')\leq \ell_Y({\mathbf{ c}}).
\end{align}
\end{prp}

The proof relies on the following lemma.

\begin{lem}\label{lem:reid_remove}
  Let $c_1$ and $c_2$ be two multi-loops differing by one third Reidemeister move. Let $c_1'$ be
  obtained from $c_1$ by removing some intersections (in the first or the second way). Then there
  exists $c'_2$, homotopic to $c_1'$, obtained from $c_2$ by removing some intersections.
\end{lem}
\begin{proof}[Proof of Lemma \ref{lem:reid_remove}]
  This can be checked by simple inspection of a finite number of cases. Two cases are shown in
  Figure \ref{fig:Reidemeister}.
  \begin{figure}[h!]
     \includegraphics[height=7cm]{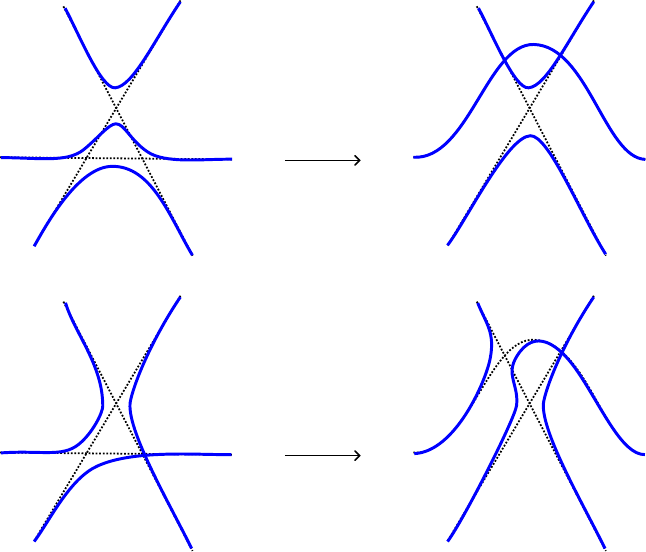}
     \caption{Illustration of two cases arising in the proof of Lemma \ref{lem:reid_remove}.}
    \label{fig:Reidemeister}
  \end{figure}
\end{proof}

\begin{proof}[Proof of Proposition \ref{p:remove}]
  A skeleton ${\mathbf{ c}}'$ is obtained from ${\mathbf{ c}}$ by removing some intersections.  Let
  $Y$ be a hyperbolic Riemannian metric $Y$ on ${\mathbf{ S}}$ and $\gamma$ be the geodesic
  representative of ${\mathbf{ c}}$. By definition, $\ell_Y({\mathbf{ c}})=\ell(\gamma)$. Since
  $\gamma$ is a geodesic, $\gamma$ is in minimal position: by the results of \cite{graaf1997} it is
  isotopic to ${\mathbf{ c}}$ up to a finite number of third Reidemeister moves.

  For each Reidemeister move, we shall use Lemma \ref{lem:reid_remove}. We obtain a multi-loop
  $\gamma'$ homotopic to the skeleton ${\mathbf{c}}'$, obtained from $\gamma$ by opening some
  intersections.  In the homotopy class of $\gamma'$, there is a representative $\gamma_\varepsilon$
  which just differs from $\gamma$ in an $\varepsilon$-neighbourhood of the intersections (with
  $\varepsilon>0$ arbitrarily small) and such that $\ell(\gamma_\varepsilon)<\ell(\gamma)$. More
  precisely, $\gamma_\varepsilon$ may be obtained from local ``switchings'', described as follows:
  take embedded $\varepsilon$-disks $D_1, \ldots, D_r$ around each of the intersection points
  $a_1, \ldots, a_r$, such that $D_i\cap \gamma$ is the union of two simple geodesic segments
  intersecting once, say $[x_1, x_3]\cup [x_2, x_4]$. The boundary $\partial D_i \cap \gamma$
  consists of the four cyclically ordered points $x_1, x_2, x_3, x_4$, and the multi-loop
  $\gamma_\eps$ with removed intersections is such that
  $D_i\cap\gamma_\eps= [x_1, x_2]\cup [x_2, x_3]$. Outside the disks $D_i$, we have
  $\gamma_\eps =\gamma$. By the triangle inequality,
  $\ell([x_1, x_2]\cup [x_2, x_3]) < \ell ([x_1, x_3]\cup [x_2, x_4])$. Equivalently,
  $\ell(\gamma_\varepsilon)<\ell(\gamma)$.  To prove the proposition, we combine the two
  inequalities $\ell_Y({\mathbf{ c}}')\leq \ell(\gamma_\varepsilon)$, and
  $\ell(\gamma_\varepsilon)<\ell(\gamma)=\ell_Y({\mathbf{ c}})$.
 \end{proof}

\subsection{Good representative of $\mathbf{c}$}

The goal of this section is to construct a new representative $\dopt$ of ${\mathbf{c}}$ which is a
good extended diagram: recall that this means a diagram without problematic bars (Definition
\ref{d:problem}) and satisfying the minimality condition (Definition \ref{d:minimality}).

%If two paths with the same endpoints are homotopic, we say that \emph{they do not intersect in the universal cover} if they have lifts to the universal cover that share the same endpoints and do not intersect, apart from
%the endpoints.

\subsubsection{Diagrams $\bD^{(2)}$ and $\bD''$}
\label{sec:diagram-bd2}

From the construction of ${\mathbf{ c}}'$, we see that there exists a diagram $\bD^{(2)}$ on~$\Sf$,
representing ${\mathbf{ c}}$, obtained by:
\begin{itemize}
\item gluing to some cylinders to $\Sf'$ to obtain $\Sf$; 
\item adding a collection of bars to the diagram $\bD'$ in $\Sf$.
\end{itemize}

\begin{rem}\label{r:D2} The diagram $\bD^{(2)}$ is in general different from $\bD^{(1)}$, because
  $\bD^{(1)}$ was obtained by opening all intersections the first way, whereas $\bD^{(2)}$ is
  obtained from ${\mathbf{ c}}$ by opening the second way certain intersections (running between 
 contractible components of $\beta$).  
\end{rem}

The added bars may be problematic or non-problematic (cylindrical or not).  Since the components of
$\beta'$ in $\bD^{(2)}$ are non-contractible, the problematic bars of $\bD^{(2)}$ are only of one
kind: they are homotopically trivial bars, running from a component of $\beta'$ to itself.

Recall that the (un)shielded simple portions of ${\mathbf{ c}}$ are in natural bijection with
(un)shielded simple portions of $\bD^{(1)}$, and of $\bD^{(2)}$.  A simple portion of $\bD'$ can be
written as a disjoint union of simple portion of $\bD^{(2)}$, only one of which may be
unshielded. Since $\bD'$ is a generalized eight, it has exactly $2\chi(\Sf)$ simple portions. We
deduce the following.

\begin{lem}
  $\bD^{(1)}$ and $\bD^{(2)}$ have at most $2\chi(\Sf)$ unshielded simple portions.
\end{lem}

\begin{defa}
  A simple portion of $\bD'$ which contains only shielded simple portions of $\bD^{(2)}$ will be
  called \emph{a priori neutral}.
\end{defa}

We introduce a last auxilliary diagram.
\begin{nota}\label{d:D''}
 Let $\bD''$ be the diagram obtained by removing from $\bD^{(2)}$ its problematic bars.
\end{nota}

\begin{lem} $\bD''$ is a good diagram filling $\Sf$.
\end{lem}

\subsubsection{Diagram $\dopt$}

We now prove the following.
 
 \begin{thm}\label{t:opt}
   Let $\mathbf{c}$ be a multi-loop in minimal position. Assume that $\mathbf{c}$ is connected and
   not double-filling.  Then there exists good extended diagram $\dopt$, containing all
   the non-problematic bars of $\mathbf{D}^{(2)}$, such that $\mathbf{c}$ is homotopic to a
   multi-loop stemming from $\dopt$.  \end{thm}

 \begin{rem}
 We may assume w.l.o.g. that $\mathbf{c}$ is connected. If the surface $\mathbf{S}$ filled by $\mathbf{c}$ is not connected, the previous theorem applies in each connected component of $\mathbf{S}$.
 \end{rem}

 An example of a loop which is not a generalized eight, and a good extended diagram for this loop,
 is represented in Figure \ref{fig:1example}.
 
  \begin{figure}[h!]
    \centering
      \begin{subfigure}[b]{0.5\textwidth}
    \centering
    \includegraphics{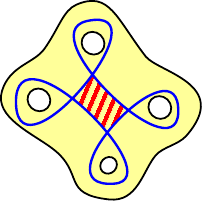}
    \caption{The loop $\mathbf{c}$.}
    \label{fig:ex_opt_1}
  \end{subfigure}%
      \begin{subfigure}[b]{0.5\textwidth}
    \centering
    \includegraphics{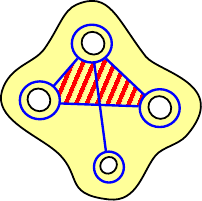}
    \caption{The diagram $\dopt$.}
    \label{fig:ex_opt_2}
  \end{subfigure}%
     \caption{Example of Theorem \ref{t:opt}.}
     \label{fig:1example}
   \end{figure}
 
 \begin{proof}
   We prove the result by induction on the total number $n$ of problematic bars of
   $\mathbf{D}^{(2)}$.  The case $n=0$ is obvious as we can take $\dopt= \mathbf{D}^{(2)}$.

   % Our proof will also show that
 %\begin{itemize}
% \item the bars of $\mathbf{D}^{opt}$ do not intersect $\beta'_1, \ldots, \beta'_{N'}$;
% \item if two bars of $\mathbf{D}^{opt}$ arriving at cannot intersect;
%\item $;
% \item the bars of $\mathbf{D}^{opt}$ that are not bars of $\mathbf{D}^{''}$ have the following property: if such a bar $B$ goes from $\beta_i'$ to $\beta_j'$, there are two m.s.p. of $\mathbf{D}^{''}$, say $\cI_i=[x_i, y_i]\subset \beta'_i$ and $\cI_j=[x_j, y_j]\subset \beta'_j$ such that $o(B)\in \cI_i$ and $t(B)\in \cI_j$, all the bars of $\mathbf{D}^{opt}$ arriving in $[x_i, o(B)]$ and in $[x_j, t(B)]$ arrive there on the same side, and finally $B$ is homotopic (without intersecting it in the universal cover) to a simple surve
%  $$ [o(B), x_i]\smallbullet B'_{a_i} \smallbullet \gamma \smallbullet B'_{a_j} \smallbullet [x_j, t(B)]$$
% where $B'_{a_i} $ is the bar of $\mathbf{D}^{''}$ leaving from $x_i$, $B'_{a_j} $ is the bar of $\mathbf{D}^{''}$ arriving at $x_j$, and $\gamma$ itself is made of a concatenation of at least one m.s.p. of $\mathbf{D}^{''}$ and some bars of $\mathbf{D}^{''}$
%(remark that the m.s.p. of $\mathbf{D}^{''}$ involved in the last part of the sentence are necessarily shielded in $\mathbf{D}^{opt}$).
 % \item when applying the reconstitution procedure to $\mathbf{D}^{opt}$, we never find a homotopically trivial path from $\beta'_i$ to itself .
%\end{itemize}
%These three facts will be helpful to go from $n$ to $n+1$ in the induction.  
 
   Assume that $n\geq 1$, and that the expected property is known until rank
   $n-1$.  By hypothesis, the diagram $\mathbf{D}^{(2)}$ contains a problematic bar
   $B_q$. Because the bar is problematic regardless of its orientation, we assume w.l.o.g. that
   $\sign(q)= +$.  By construction of $\bD^{(2)}$, the bar
   $B_q$ goes from one component
   $\beta'_\lambda$ to itself and is homotopically trivial with gliding endpoints. Denote by
   $D$ the topological disc bordered by $B_q$ and
   $\beta'_\lambda$.  One of the simple portions adjacent to
   $B_q$ is shielded. To fix ideas, we assume it is $\cI_q$.

   A representative $\cop$ of the homotopy class of $\mathbf{c}$ may be obtained by applying the
   reconstitution procedure, and allows to write the homotopy class of $\textbf{c}$ as an
   alternating concatenation of bars and simple portions.  Denote $i_1=\comp(q)$ and
   $i_2=\comp(\iota q)$ so that the bars $B_q$ and $B_{\iota q}$ belong to the components
   $\cop_{i_1}$ and $\cop_{i_2}$ of $\cop$ respectively.

   Recall that $B_{\iota q}$ is the same bar as $B_q$ with either the same or the reverse
   orientation. We treat the case where it has reverse orientation, the other case being similar
   modulo changes of notation. The local situation near the bar $B_q$ is shown on Figure
   \ref{fig:loc_sit}. % Another example where a multi-loop can lead to
   % such a situation is pictured in Figure \ref{fig:opt_proc}.
   
 % In \eqref{e:decompo-ibis}, there exists a component $i\in \{1, \ldots, \cc\}$ and $\qq\in \Theta^i$
 % such that 
 %(we may also assume that
% $\cI_{\tau^{-1}{\qq}}$ is unshielded, if this was not possible, $\mathbf{c}$ would be double-filling).  

 % Assume (without loss of generality) that $j=1, i=1$.
 %  This means that the bar $B({\qq})$ belongs to the component $\cop_1$
 % and goes from $\beta'_1$ to itself. 
 % There is another index, say $j_o$, such that the bar $B(\iota {\qq})$ belongs to the component $\cop_{j_o}$. 
 
  \begin{figure}[h!]
    \centering
    \begin{subfigure}[b]{0.5\textwidth}
      \centering
      \includegraphics{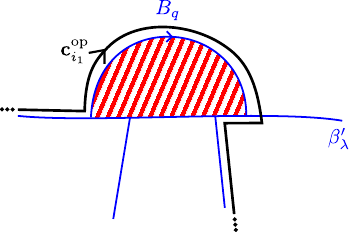}
      \caption{$B_q$ and $\cop_{i_1}$.}
    \label{fig:loc_sit_1}
  \end{subfigure}%
    \begin{subfigure}[b]{0.5\textwidth}
      \centering
      \includegraphics{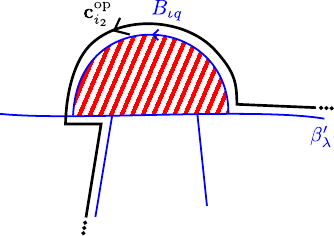}
      \caption{$B_{\iota q}$ and $\cop_{i_2}$.}
    \label{fig:loc_sit_2}
  \end{subfigure}%
  \caption{The local situation near the bar $B_q$.}
     \label{fig:loc_sit}
   \end{figure}
   
   %   \begin{figure}[h!]
   %   \centering
   %   \includegraphics[height=6cm]{Figures_opt_procedure}
   %   \caption{An example with three equivalent pictures. In the surface $\Sf_{\circ}$, the dashed curve is non-contractible, whereas it is contractible in $\Sf$.}
   %   \label{fig:opt_proc}
   % \end{figure}
 
   We introduce a new surface $ \Sf_{\circ}\subseteq \Sf$ by taking a regular neighbourhood $\cN$ of
   $\mathbf{D}^{(2)}$ in $\Sf$ and gluing to $\cN$ all connected components of
   $\Sf\setminus \mathbf{D}^{(2)}$ that are contractible, except the connected component contained
   in the disc $D$ and touching the bar $B_q$. %See Figure \ref{fig:opt_proc}.
      
   In the diagram $\mathbf{D}^{(2)}$, now viewed as a diagram on $\Sf_{\circ}$, the bar $B_q$ is
   non-problematic. We can therefore apply the induction hypothesis at rank $n-1$ to the multi-loop
   $\mathbf{c}$ on $\Sf_\circ$.  We obtain a good extended diagram $\dopt_\circ$ on $\Sf_\circ$,
   containing all non-problematic bars of $(\Sf_\circ,\bD^{(2)})$ (and in particular $B_q$), and
   such that $\mathbf{c}$ is homotopic on $\Sf_\circ$ to a multi-loop stemming from $\dopt_\circ$.

   In $\Sf_{\circ}$ the diagram $\dopt_{\circ}$ is good, and in $\Sf$ it has exactly one problematic
   bar, namely $B_q$. We shall define the desired diagram $\dopt$ by replacing $B_{q}$ by a new
   non-problematic bar $B_{\mathrm{new}}$.

   % We denote by $\Theta_{\circ}:=\{1, \ldots, r_\circ\} \times \{\pm\}$ the parameter set which
   % serves to index the simple portions of $ \dopt_{\circ}$.
   % We may assume, without loss of
   % generality, that the bars are numbered so that $B_q$ is associated to $q_\circ=(1,+)$.
   Applying the reconstitution procedure to $\dopt_{\circ}$, for each $1 \leq i \leq \cc$, we obtain
   a representative $\mathbf{c}_i^\circ$ of the homotopy class of ${\mathbf{c}}_i$ as a
   concatenation of bars and simple portions of $\dopt_\circ$.  The construction will consist in
   replacing the two components ${\mathbf{c}}^{\circ}_{i_1}$ and ${\mathbf{c}}^{\circ}_{i_2}$,
   containing $B_q$ and $B_{\iota q}$ respectively, by new representatives
   ${\mathbf{c}}^{\mathrm{opt}}_{i_1}$ and ${\mathbf{c}}^{\mathrm{opt}}_{i_2}$. The other components
   $({\mathbf{c}}^{\circ}_i)_{i \neq i_1, i_2}$ will stay untouched.

   Lifting the situation to the hyperbolic plane, for $l \in \{1,2\}$, we can write the lift
   $\tilde{\mathbf{c}}_{i_l}^\circ$ of $\mathbf{c}_{i_l}^\circ$ as an infinite concatenation
   $\smallbullet_{k \in \Z} \tilde{p}_k^l$ where, for each $k$, $\tilde{p}_k^l$ is the concatenation
   of lifts $\tilde{B}_{k}^l$ and $\tilde{\cI}_k^l$ of a bar and a simple portion of
   $\dopt_\circ$. We pick the indices so that $\tilde{p}^1_0$ and and $\tilde{p}_0^2$ contain a lift
   of $B_q$ and $B_{\iota q}$ respectively.

   Denote by $\tilde \beta_k^l$ the complete geodesic on $\IH^2$ containing $\tilde \cI^l_k$. We have
   $\tilde \beta_0^1=\tilde \beta_0^{2}$ and we call $m>0$ the smallest integer such that
   $\tilde \beta_m^1\not=\tilde \beta_m^{2}$. By the minimality condition for
   $ {\mathbf{D}}^{\mathrm{opt}}_{\circ}$, all $k \in \{0, \ldots, m-1\}$ are crossing indices.  We define
   the bar $\tilde B_{\mathrm{new}}$ on $\IH^2$ as the segment going from $t(\tilde B^1_m)$ to
   $t(\tilde B^{2}_m)$, and let $B_{\mathrm{new}}$ be its projection to $\Sf$. See Figure
   \ref{fig:Bnew}. The desired diagram $\dopt$ is defined from $\dopt_{\circ}$ by removing $B_q$ and
   adding the non-problematic bar $B_{\mathrm{new}}$.

\begin{figure}[h!]
     \centering
     \includegraphics[height=7cm]{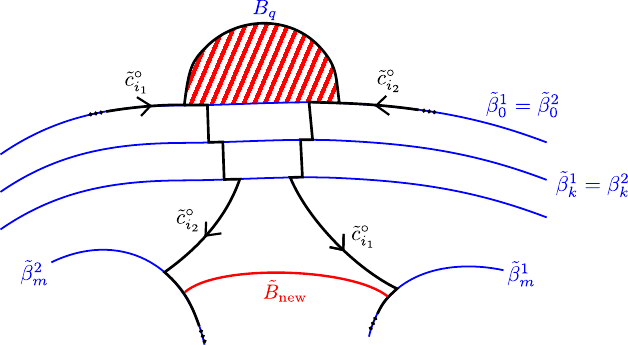}
     \caption{The new bar $B_{\mathrm{new}}$ replacing $B_q$.}
     \label{fig:Bnew}
   \end{figure}
 \end{proof}
 
 % In Figure \ref{fig:example_new} we show the new bars in our example of Figure \ref{fig:opt_proc}.

 % \begin{figure}[h!]
 %     \centering
 %     \includegraphics[height=5cm]{Figures_example_new}
 %     \caption{Continuation of example of Figure \ref{fig:opt_proc}.}
 %     \label{fig:example_new}
 %   \end{figure}
 
 \begin{rem} At each induction step the number of unshielded simple portions of $\dopt$ in $\Sf$
   stays the same as for the original $ \mathbf{D}^{(2)}$, hence it is smaller that $2\chi(\Sf)$.
 \end{rem}

 \subsection{Sketch of proof for general topologies}

 We now explain how to prove Theorem \ref{t:main} for general multi-loops $\mathbf{c}$.
  
 \subsubsection{Coordinates on Teichm\"uller spaces} 
\label{s:apriorintr}
  
 Since the diagram $\mathbf{D}'$ represents a generalized eight~$\mathbf{c}'$ in $\Sf'$, we can use
 it to define coordinates $(\vec{L}, \vec{\theta}) \in \R_{\geq 0}^{r'} \times \R^{\Theta'}$ on
 $\cT^*_{\Gg, \nn}$.

 The surface $\Sf$ is obtained by gluing cylinders to $\Sf'$. This gives rise to coordinates
 $(\vec{L}, \vec{\theta}, \vec{\alpha})$ on $\cT^*_{\g, \n}$, obtained by adding some twist
 coordinates $\vec{\alpha}$ along the boundary of these cylinders, and getting rid of an equivalent
 number of parameters $\theta_q$ that have become redundant (corresponding to components of $\beta'$
 homotopic in $\Sf$). Note that $\cT^*_{\Gg, \nn}$ and $\cT^*_{\g, \n}$ have the same dimension,
 that is, $3 \chi(\Sf)$.

 Recall that certain simple portions of $\mathbf{c}'$ are declared to be \emph{a priori
   neutral}. The corresponding parameters $\theta_q$ are \emph{a priori neutral} coordinates. A
 priori neutral parameters, as well as all the twist parameters $\vec{\alpha}$, will be neutral
 parameters throughout this discussion.

 \begin{rem}
   If $\mathbf{c}$ is double-filling, all parameters $(\theta_q)_{q \in \Theta}$ are a priori neutral.
 \end{rem}
 
\subsubsection{Expression of $\ell_Y(\mathbf{c})$} 
Denote by $\dopt=(\beta^{\mathrm{opt}},B^{\mathrm{opt}})$ the good extended diagram constructed in
Theorem~\ref{t:opt}.  The reconstitution procedure, applied to the diagram $\dopt$, leads to a
representative of the homotopy class of ${\mathbf{c}}$ in the form of an alternance of bars and
segments similar to \eqref{e:decompo-ibis}.
 
Since none of the bars of $\dopt$ is homotopically trivial with gliding endpoints, we can replace
the writing \eqref{e:decompo-ibis} by a geodesic representative, as described in \S
\ref{s:straightening2}.  Indeed, we pick the representative of the metric $Y \in \cT_{\g,\n}^*$ such
that $\beta'$ is a multi-geodesic, and glide the bars $B^{\mathrm{opt}}_q$ to their orthogeodesic
representatives $\overline B_q$, and then take a representative of $\mathbf{c}$ of the
form~\eqref{e:decompo-ig}. In the special case of multiple components $\beta'_i$ in the same
homotopy class, their geodesic representatives are collapsed to the same curve, and the cylindrical
bars $\overline B_q$ running between them are reduced to a point.  Theorem
\ref{prop:form_l_gamma_try} then provides an expression for the length of each component
${\mathbf{c}}_i$ with $1 \leq i \leq \cc$, with the bars and segments corresponding to those of
$\dopt$.

\subsubsection{Rest of the proof}

Once the expression for $\ell_Y(\mathbf{c})$ is obtained, the proof is similar to the one for
generalized eights:
\begin{enumerate}
\item we get rid of negative crossing parameters using the erasing procedure from \S \ref{s:lengthc};
\item we write the integral in the new coordinates, with an additional set of a priori neutral
  coordinates defined in \S \ref{s:apriorintr} (note that each non-neutral coordinate $\theta$
  coincides with the length of an unshielded simple portion of $\dopt$, up to addition of a function
  of neutral variables);
\item we apply our results about generalized convolutions in the active variables
  $\ThetaAc'' \subseteq \Theta$;
\item we conclude with a comparison estimate, which is easier to prove now that some lengths
  $\ell(\Gamma_\lambda)$ (corresponding to contractible components) vanish.
\end{enumerate}

\begin{rem}\label{r:opti_K} 
  The choice of $\rK$ in Theorem \ref{t:main} is still given by equation \eqref{eq:sum_K_proof}, but
  now the the a-priori-neutral parameters have been removed from the new set of active parameter
  $\ThetaAc''$. Hence,
 \begin{align}
   \rK=\sum_{q\in \ThetaAc''}(\rK_q+1)
   \leq \sum_{j \in V'}(\rK_j -1) + \rho[\mathbf{T}]
\end{align} 
where the sum $V' \subseteq V$ is the set of $\lambda \in \LambdaBCbeta \cap V$ such that
$\beta_\lambda$ contains at least an unshielded simple portion of $\mathbf{c}$, as mentioned in
Remark \ref{rem:improve}. In particular, we can take $\rK=0$ if $ \rho[\mathbf{T}]=0$.
\end{rem}

%%%Local Variables: 
%%% mode: latex
%%% TeX-master: "main"
%%% End: 

\section{Variant of the main theorem for lc-surfaces}
\label{s:variantLC}
Our core motivation for proving Theorem \ref{t:main} is the spectral gap problem, i.e. proving
that $\lambda_1$ is close to $1/4$ for typical hyperbolic surfaces of high genus. We have explained
the link between these questions, and presented the arguments leading to our spectral gap result
(Theorem \ref{t:dream}), in \cite{Expo}. In the proof, we need to restrict the Weil--Petersson
measure to the set of tangle-free surfaces.  The presence of an indicator function in the restricted
integral causes problems, that are overcome by the introduction of a so-called \emph{Moebius
  function}, constructed in~\cite{Moebius}. The $\mathbf{T}$-averages $\av[\mathbf{T}]{F}$ of a test
function $F$, defined in \eqref{e:avT}, are then replaced by expressions of the form
  \begin{align} \label{e:nu-ave-intro}  
   \avN[\ftype]{\mu \otimes F}
   :=   \Ewpo \Bigg[ \sum_{\substack{ \curv \in \mathcal{G}(X) \\\tau\in \cS(X)\\ (\curv, \tau)\sim {\ftype}}} 
 \mu(\tau) F( \ell(\curv) ) \Bigg]
 \end{align}
 where:
 \begin{itemize}
 \item $\mu$ is the Moebius function constructed in \cite{Moebius}, defined on the ``extended moduli
   space'' of metric spaces made of disjoint unions of circles and hyperbolic surfaces with
   boundary;
 \item $\curv$ is a geodesic loop of length $\ell(\curv)$;
 \item $\tau$ is a subset of $X$ with various connected components, which may be either surfaces with geodesic boundary, or non-contractible simple closed geodesics; 
 \item $\ftype$ is now a local topological type of pairs $(\curv, \tau)$.
  \end{itemize}

  In the following, a \emph{test function} $F$ is a measurable function $F : \R_{\geq 0} \rightarrow
  \R$ which is bounded and compactly supported.
  
  The precise definition of $\mu$ and of the ``extended moduli space'' is given in \cite[\S
  2]{Moebius} and partially recalled below.  We also recall the new notion of local topological type
  adapted to pairs $(\curv, \tau)$, and finally explain how to extend Theorem \ref{t:main} to the
  average \eqref{e:nu-ave-intro}.  This provides the required technical generalization to complete
  the proof of Theorem \ref{t:dream}.

\subsection{c-surfaces}

Let us recall the notion of c-surface, developed in more detail in \cite[\S 2]{Moebius}. In a nutshell, a
c-surface is the disjoint union of a surface and of simple curves.
   
Let $\N_0 = \Z_{\geq 0}$ be the set of non-negative integers.  We order the pairs $(g, n)\in \N_0^2$
by the order relation: $(g, n)\leq (g', n')$ if $2g+n-2 < 2g'+n'-2$, or if $2g+n-2 = 2g'+n'-2$ and
$g\leq g'$.  For $j\geq 1$, denote by ${\mathfrak{S}}^{(j)}$ the set of $2j$-tuples
$(\vg, \vn)=((g_1, n_1), \ldots, (g_j, n_j))\in (\N_0^2)^j$, satisfying $(g_i,n_i)=(0,2)$ or
$2g_i-2+n_i> 0$ for $1 \leq i \leq j$, and $(g_i, n_i)\leq (g_{i+1}, n_{i+1})$ for $1 \leq i < j$.
 
   \begin{defa} A \emph{c-surface of signature
     $(\vg, \vn)\in {\mathfrak{S}}^{(j)}$} is a non-empty topological space $\Sf$ with $j$ connected components
     $\Sf= \bigsqcup_{i=1}^j \tau_i$ labelled from $1$ to $j$, where:
     \begin{itemize}
     \item if $(g_i, n_i)=(0, 2)$, then $\tau_i$ is a $1$-dimensional oriented manifold
       diffeomorphic to a circle;
     \item if $2g_i+n_i-2>0$, then $\tau_i$ is a compact $2$-dimensional orientable manifold with
       (or without) boundary, of signature $(g_i, n_i)$, with boundary components numbered by
       $\{1, \ldots, n_i\}$.
     \end{itemize}
   \end{defa}

   \begin{nota} The c-surface $\Sf$ can be decomposed as $\Sf=(c, \sigma)$,
     where $c =(c_1, \ldots, c_{j_1}) $ and $\sigma =(\sigma_1, \ldots, \sigma_{j_2})$ are,
     respectively, collections of $1$-dimensional oriented manifolds and $2$-dimensional oriented
     compact manifolds with boundary ($j=j_1+j_2$ is the total number of connected components).  We
     shall denote by $\mathfrak{c}(\Sf)=j_1$ the number of 1d components, and $\chi(\Sf)$ the total
     absolute Euler characteristic, that is
     $$\chi(\Sf)=\sum_{i=1}^{j_2} \chi(\sigma_i)
     \quad \text{where} \quad \chi(\sigma_i)=2g_{j_1+i}+n_{j_1+i}-2.$$ We say that $\Sf$ is purely
     2-dimensional if $\mathfrak{c}(\Sf)=0$, and purely 1-dimensional if $\chi(\Sf)=0$.
   \end{nota}

   \begin{nota}[Boundary of a c-surface] \label{r:order}
     For $\Sf=(\tau_1, \ldots, \tau_j)=(c_1, \ldots, c_{j_1}, \sigma_1, \ldots, \sigma_{j_2})$, if
     $\partial \sigma_i$ is the boundary of the surface $\sigma_i$,
     we let $\partial \Sf$ be the (usual) boundary of $\Sf$
     \begin{equation*}
       \partial \Sf = \Big(\bigsqcup_{i=1}^{j_1}  c_i \Big)
       \sqcup \Big( \bigsqcup_{i=1}^{j_2} \partial \sigma_i \Big).
     \end{equation*}
     The total number of boundary components of $\sigma$ is defined as
     $$n_{\sigma} := \sum_{i=1}^{j_2} n_{j_1+i}.$$
   \end{nota}
   
   \begin{defa}[Sub-c-surface] \label{d:subsurface} If $\Sf$ is a c-surface, we call
     {\em{sub-c-surface}} of $\Sf$ a subset $\Sf'\subseteq \Sf$ which (for the induced topology and
     differential structure) is a c-surface, and such that $\partial \Sf'$ is a multi-curve
     in~$\Sf$, i.e. its components are non-contractible and homotopically distinct.
 \end{defa}

 \begin{nota}[Geodesic representative] \label{d:geo_rep}Let $\Sf$ be a c-surface of signature
   $(\vg, \vn)$ and $\Sf' = (c',\sigma')$ be a sub-c-surface of~$\Sf$.  Assume $\Sf$ is endowed with
   a metric $ Z\in \cT^*_{\vg, \vn}$, the Teichm\"uller space of Riemannian metrics on $\Sf$ that
   are hyperbolic on the 2-dimensional components. Then $\Sf'$ is isotopic to a geodesic
   sub-c-surface $\Sf'_Z=(c'_Z, \sigma'_Z)$: we mean by this a sub-c-surface with geodesic boundary,
   that we call the {\em{geodesic representative}} of $\Sf'$ in $Z$.

   We denote by $\cS(Z)$ the set of geodesic sub-c-surfaces of $\Sf$ endowed with the metric~$Z$.
\end{nota}

In \cite[\S 2]{Moebius}, we also defined the extended moduli space $\mathbf{M}$, which is -- roughly speaking -- the moduli space of all possible Riemannian structures on all
possible c-surfaces, with hyperbolic metrics on the 2d-components. In the context of Definition \ref{d:geo_rep}, $\Sf'_Z$ may be seen as a representative of an element of $\mathbf{M}$.

The following notation will also be useful.

   \begin{nota}
     For $\Sf$ as in Notation \ref{r:order}, if now $\tilde{\partial} c_i:=c_i\times \{\pm\}$
     (interpreted as the two boundary curves of a cylinder containing $c_i$), we let
     $$\tilde{\partial} \Sf=\Big(\bigsqcup_{i=1}^{j_1} \tilde\partial c_i \Big)
     \sqcup \Big( \bigsqcup_{i=1}^{j_2} \partial \sigma_i \Big).$$
     The components of $\tilde{\partial} \Sf$ are labelled by the set
     $$(\{1, \ldots, j_1\} \times \{\pm\}) \sqcup \{(i, l), 1 \leq i \leq j_2, 1 \leq l \leq  n_i\}.$$
     When needed, we can put a total lexicographic order on this set.  We let
     $n_{\Sf}:= 2j + n_{\sigma}$. Notice that, listing all the boundary components in the natural
     order, we can also identify
     \begin{equation*}
       \tilde{\partial} \Sf = \tilde{\partial}c \sqcup \{1, \ldots, n_\sigma\} = (\{1, \ldots, j_1) \times \{\pm\}) \sqcup \{1, \ldots, n_\sigma\}.
     \end{equation*}
 \end{nota}

\subsection{lc-surfaces and their local topological types}
\label{s:lcsurface}

We enrich the notion of local topological type, to a notion designed to sort out expressions such as
\eqref{e:nu-ave-intro}. The topological object to consider is now {$(\Sf, \curv, \tau)$,
  where $\curv$ is a loop and $\tau$ is a c-surface, which together fill $\Sf$}. Such a pair will
be called a lc-surface. We define these objects in \cite[Appendix A]{Moebius}.
  
   \begin{nota}\label{n:FTlc} To any  $(\vg, \vn)\in {\mathfrak{S}}^{(j)}$, $j\geq 0$,
     we associate a \emph{fixed} smooth oriented c-surface~$\Sf$ of signature $(\vg, \vn)$. The data
     of $(\vg, \vn)$, or equivalently of $\Sf$, is called a \emph{c-filling type}.
\end{nota}

\begin{defa}
  \label{def:local_type}
  A \emph{local lc-surface} is a triplet $(\Sf,\curv, \tau)$, where $\Sf$ is a c-filling type,
  $\curv$ is a loop in $\Sf$, $\tau$ is a sub-c-surface of $\Sf$, and the pair $(\curv, \tau)$ fills
  $\Sf$.

  Two local lc-surfaces $(\Sf,\curv, \tau)$ and $(\Sf',\curv', \tau')$ are said to be \emph{locally
    equivalent} if $\Sf=\Sf'$ (i.e.
  $(\mathbf{g}_\Sf,\mathbf{n}_\Sf) = (\mathbf{g}_{\Sf'},\mathbf{n}_{\Sf'})$), and there exists a
  positive homeomorphism $\psi : \Sf \rightarrow \Sf$, possibly permuting the boundary components
  and connected components of $\Sf$, such that $\psi \circ \curv$ is freely homotopic to $\curv'$
  and $\psi(\tau)=\tau'$. We ask that the homotopy between $\psi \circ \curv$ and $\curv'$ preserves
  the orientation, and that $\psi$ respects the numbering of boundary components of $\tau$ and
  $\tau'$. 
\end{defa}

This defines an equivalence relation $\eq$ on local lc-surfaces. Equivalence classes for this
relation are denoted as $\eqc{\Sf, \curv, \tau}$ and called \emph{local topological types} of
lc-surfaces.

\begin{nota} \label{n:SS} If $(\Sf,\curv, \tau)$ is a local lc-surface, we will denote by
  $\Sb \subseteq \Sf$ the c-surface filled by the multi-loop $(\curv, \partial \tau)$, and $\Sinn$
  its complement in $\Sf$.
\end{nota}
Necessarily, the 1d-components of $\tau$ which lie in the 2d-components of $\Sf$ must intersect
$\curv$, otherwise $\tau\cup \curv$ would not fill $\Sf$. As a consequence, $\Sb$ contains the 1d
part of $\tau$, and $\Sinn$ is contained in the 2d-part of $\tau$.
 
  \begin{defa}
    Let $\ftype = \eqc{\Sf, \curv, \tau}$ be a local topological type.  For a pair
    $(\gamma, \tau_0)$ formed by a loop~$\gamma$ and a sub-c-surface $\tau_0$ on a hyperbolic
    surface~$X$, we denote as $S(\gamma, \tau_0)$ the surface filled by $(\gamma, \tau_0)$ in $X$
    (defined in details in \cite[Definition 4.1]{Ours1}). Then, $(\gamma, \tau_0)$ is said to
    \emph{belong to the local topological type}~$\eqc{\Sf, \curv, \tau}$ if there exists a positive
    homeomorphism $\phi : S(\gamma, \tau_0) \rightarrow \Sf$ such that the loops $\phi ( \gamma)$
    and $\curv$ are freely homotopic in~$\Sf$, and such that $\phi(\tau_0)=\tau$. In that case, we
    write $(\gamma, \tau_0) \sim \ftype$.
\end{defa}

\begin{nota}
  For any $j \geq 1$, there is an natural bijection from the set of c-surfaces with
  $\mathfrak{c}(\Sf)=0$ to the set of c-surfaces with $\mathfrak{c}(\Sf)=j$, which consists in
  adding $j$ 1-dimensional connected components: we denote it by $\rho_j$.

  Similarly, if ${\ftype}=\eqc{{{\Sf, \curv, \tau}}}$ is a local topological type of lc-surface such
  that $\mathfrak{c}(\Sf)=0$, we denote by $\rho_j{\ftype}$ the local topological type obtained by
  adding $j$ 1-dimensional connected components to the local topological type $\ftype$. If $\ftype$
  fills $\Sf$, then $\rho_j\ftype$ fills $\rho_j\Sf$.
 \end{nota}

\subsection{Volume functions for local topological types of lc-surfaces} 

We now study expressions such as \eqref{e:nu-ave-intro} and state analogues of Theorem
\ref{con:FR_type}. As explained in \cite{Expo}, to prove our spectral gap result, Theorem
\ref{t:dream}, we can restrict attention to subsurfaces of $X$ that disconnect the hyperbolic
surface $X$ into a bounded number of components. This is formalised as follows.

For an integer $Q$, we denote as ${\MC}_X(Q)$ the set of multi-curves which separate $X$ into at
most $Q$ connected components, following \cite[Notation A.2]{Ours1}.  Let us define
$\cS_Q(X)\subset \cS(X)$ as the set of geodesic sub-c-surfaces of~$X$, with a 1d-part belonging to
${\MC}_X(Q)$. For any test function $F$, any bounded {measurable} function $\nu$ on the extended
moduli space $\mathbf{M}$, we let
 \begin{align}  \label{e:total}
   \avN[\ftype]{\nu \otimes F}
   :=   \Ewpo \Bigg[ \sum_{\substack{ \gamma \in \mathcal{G}(X) \\\tau\in \cS_Q(X)\\ (\gamma, \tau)\sim {\ftype}}} 
 \nu(\tau) F( \ell(\gamma) ) \Bigg].
 \end{align}

 \begin{nota}\label{n:nu-requirements}
   In the following, the function $\nu$ on $\mathbf{M}$ will be assumed to satisfy the following
   properties for some positive real numbers $\kappa, \tfL$ with $\kappa < \tfL$.
  \begin{itemize}
  \item If $\tau =(c_1, \ldots, c_{j_1},\sigma_1, \ldots, \sigma_{j_2})$ is such that
    $\nu(\tau)\not=0$, then
    \begin{itemize}
    \item for $1 \leq i \leq j_1$, $\ell(c_i)\leq \kappa$, where $\ell(c_i)$ is the length of $c_i$
      for the metric $\tau$;
    \item for $1 \leq i \leq j_2$, $\ell(\partial \sigma_i)\leq 9\tfL \chi(\sigma_i)$, where
      $\ell(\partial \sigma_i)$ is the total boundary length of $\sigma_i$.
    \end{itemize}
  \item There exists a bound of the form
    \begin{equation}
      \label{eq:bound_nu}
      |\nu(\tau)| \leq \mathfrak{f}(\mathfrak{c}(\tau), \chi(\tau)).
    \end{equation}
  \end{itemize}
  We assume that $ \mathfrak{f}(j,\chi) \leq \mathfrak{f}(j',\chi')$ if $j\geq j'$ and
  $\chi\leq \chi'$. 
\end{nota}

\begin{rem}
  \label{rem:nu_bound}
  Notice that the above conditions imply that, for any $\tau$ in the support of $\nu$,
  \begin{equation}
    \label{eq:bound_nu_boundary}
    \ell(\partial \tau)\leq \kappa \mathfrak{c}(\tau) + 9\tfL \chi(\tau).
  \end{equation}
  If we further assume that $\tau$ is a sub-c-surface of an ambiant c-surface $\Sf$ with
  $\mathfrak{c}(\Sf)=0$, then  $    \ell(\partial \tau)\leq 12\tfL \chi(\Sf)$ since the maximal
  number of components for a multi-curve in $\Sf$ is $3\chi(\Sf)$.
\end{rem}

\begin{exa} The hypotheses above are satisfied for $\nu= \mu_{\kappa, \tfL}$, the Moebius function
  constructed in \cite{Moebius}, with the bound
  \begin{equation}
    \label{eq:bound_nu_mu}
    \mathfrak{f}(j, \chi)  := \frac{U_1(\chi)}{j!}  e^{\tfL U_2(\chi)}
  \end{equation}
  where $U_1$ and $U_2$ are increasing sequences (for which we do not need an explicit expression).
\end{exa}

 The following generalizes Proposition \ref{prp:existence_density}, with exactly the same proof:
 
\begin{prp} 
  \label{thm:volumen_lc}
  Let ${\ftype}=\eqc{{{\Sf, \curv, \tau}}}$ be a local topological type of lc-surface with
  $\mathfrak{c}(\Sf)=0$.  For any $j\geq 0$, there exists a continuous functions
  $V_{g, Q}^{\nu, {\ftype}, j}$ such that, for any test function $F$, \begin{align*}
    \avN[\rho_j{\ftype}]{\nu \otimes F} = \int_0^{+ \infty} F(\ell)V_{g, Q}^{\nu, {\ftype}, j}(\ell)
    \d \ell. \end{align*}
 \end{prp}

 We are now ready to state the Friedman--Ramanujan property that we will prove.
 
 \begin{thm} 
  \label{thm:leviathan_lc}
  Let ${\ftype}=\eqc{{{\Sf, \curv, \tau}}}$ be a local topological type of lc-surface such that
  \mbox{$\mathfrak{c}(\Sf)=0$}. %Let $\chi(\tau)= \chi(\tau)$ be the absolute Euler characteristic of $\tau$.
  For any integers $Q, \ord\geq 0$, any large enough $g$, any $j\leq \log g$, there exists a
  Friedman--Ramanujan function $h_{g, Q, \ord}^{\nu, {\ftype}, j}$ such that, for any test function
  $F$, any $\eta>0$,
  \begin{align*}
   % \label{eq:exist_asympt_type_intro}
    & \Big| \avN[\rho_j{\ftype}]{\nu \otimes F} -
      \frac{1}{g^{\chi(\Sf)}}
      \int_0^{+ \infty} F(\ell)h_{g, Q, \ord}^{\nu, {\ftype}, j}(\ell) \d \ell \Big|
    \\ &
         \leq \fn(Q,\ord,\chi(\Sf),\eta) \, M^{j} \mathfrak{f}(j, \chi(\tau))
         e^{13 \tfL\chi(\Sf)}
         \frac{\norm{ F(\ell) \, e^{(1+\eta)\ell}}_\infty}{g^{\ord+1}}
  \end{align*}
  where $M = \fn(Q,\ord,\n,\kappa)=2^{2(Q+\n)}\tilde a_\ord^{2}(e^{\kappa}-1)$ with the constant
  $\tilde a_\ord$ from \cite[Theorem B.1]{Ours1}. More precisely, there exists
  $\rK = \rK(\ord, \chi(\Sf))$ and $\rN = \rN(\ord, \chi(\Sf))$ such that
  $h_{g, Q, \ord}^{\nu, {\ftype}, j}\in \cF_w^{\rK,\rN}$, and
    \begin{equation}
    \label{eq:bound_FR_norm_lc}\norm{h_{g,Q,\ord}^{\nu, {\ftype}, j}}^w_{\cF^{\rK,\rN}}
    \leq  \fn(Q, \ord, \chi(\Sf))\, 
     \tilde a_\ord^{2j} \mathfrak{f}(j, \chi(\tau)) \, e^{6 \tfL  \chi(\Sf)}.
    \end{equation}
 \end{thm} 

 It is essential in this statement that the exponents and constants are independent from the number
 of 1d components $j$ and the specific topology of $(\curv,\tau)$ amongst pairs filling $\Sf$.

 % In applications, we will take $\nu = \mu_{\kappa, \tfL}$ the Moebius function, and sum this result
 % for all values $j \geq 0$, using the following statement.
 % \begin{cor}
 %     \label{cor:leviathan_lc_sum}
 %     Let ${\ftype}=\eqc{{{\Sf, \curv, \tau}}}$ be a local topological type of lc-surface such that
 %     \mbox{$\mathfrak{c}(\Sf)=0$}. %Let $\chi(\tau)= \chi(\tau)$ be the absolute Euler characteristic of $\tau$.
 %     For any integers $Q, \ord\geq 0$, any large enough $g$, there exists a Friedman--Ramanujan
 %     function $h_{g, Q, \ord}^{\nu, {\ftype}}$ such that, for any test function $F$, any $\eta>0$,
 %  \begin{align*}
 %   % \label{eq:exist_asympt_type_intro}
 %    & \Big| \sum_{j=0}^\infty\avN[\rho_j{\ftype}]{\nu \otimes F} -
 %      \frac{1}{g^{\chi(\Sf)}}
 %      \int_0^{+ \infty} F(\ell)h_{g, Q, \ord}^{\nu, {\ftype}}(\ell) \d \ell \Big|
 %    \\ &
 %         \leq \fn(Q,\ord,\chi(\Sf),\eta,\kappa)  e^{(13 \chi(\Sf) + U_2(\chi(\Sf))) \tfL}
 %         \frac{\norm{ F(\ell) \, e^{(1+\eta)\ell}}_\infty}{g^{\ord+1}}
 %  \end{align*}
 %  and there exists $\rK = \rK(\ord, \chi(\Sf))$ and $\rN = \rN(\ord, \chi(\Sf))$ such that
 %  $h_{g, Q, \ord}^{\nu, {\ftype}}\in \cF_w^{\rK,\rN}$, and
 %    \begin{equation}
 %    \label{eq:bound_FR_norm_lc}\norm{h_{g,Q,\ord}^{\nu, {\ftype}}}^w_{\cF^{\rK,\rN}}
 %    \leq  \fn(Q, \ord, \chi(\Sf))\, 
 %      e^{6 \tfL  \chi(\Sf)}.
 %    \end{equation}
 % \end{cor}

\subsection{Writing down the volume functions for lc-local types}
\label{sec:real-fill-type}

Let us now sketch the proof of Theorem \ref{thm:leviathan_lc}, which is very similar to the proof of
Theorem \ref{t:main}.

As before, we start by writing down an explicit expression for the volume function
$V_{g, Q}^{\nu, {\ftype}, j}$. We then use the asymptotic expansion from \cite[Theorem
1.1]{anantharaman2022} to obtain an approximate expression for this volume function, up to errors
decaying like $1/g^{\ord+1}$. The Friedman--Ramanujan property follows from a variant of Theorem
\ref{t:main}. The main new difficulty here is to control the dependency of the estimates on the
integer $j$, the number of 1d-components of our lc-local-type (as we will need to understand this
dependency well for the spectral gap result).

\subsubsection{Realizations of a lc-local type}

In order to write down the volume functions $ V_{g, Q}^{\nu, {\ftype}, j}$ with a formula similar to
\eqref{e:disint_phi}, we need an enumeration of the possible embeddings of the c-filling type $\Sf$
into a surface $S_g$ of genus $g$. We therefore extend from \cite[\S 4.5]{Ours1} the notion of
realization of a filling type into $S_g$.

For the sake of convenience, we treat separately purely 2d and 1d c-filling types, before merging
the two definitions.  Recall from Notation \ref{r:order} that we can put an order on the labels of
the connected components of $\partial \Sf$.

\begin{defa}
  \label{def:realizations}
  Let $\mathbf{S}=(\sigma_1, \ldots, \sigma_j)$ be a purely 2-dimensional c-filling type, and let
  $g \geq 2$. We call \emph{realization of $\mathbf{S}$ in a connected surface $S_g$ of genus $g$}
  any pair $\mathfrak{R} = (\vec{I}, \vec{g})$, where:
  \begin{itemize}
  \item $\vec{I}$ is a partition of $\partial \mathbf{S} $ into $\mathfrak{q} \geq 1$ non-empty sets
    $I_1, \ldots, I_{\mathfrak{q}}$, numbered such that $i \mapsto \min I_i$ is an increasing
    function;
  \item we ask furthermore the partition to be \emph{transitive} in the sense given below;
  \item $\vec{g} = (g_1, \ldots, g_{\mathfrak{q}})$ is a vector of non-negative integers;
  \item if for any $1 \leq i \leq \mathfrak{q}$, we denote $n_i := \# I_i$ and
    $\chi_i := 2g_i-2+n_i$, then $\chi_i > 0$ or $(g_i,n_i)=(0,2)$, and we have:
       \begin{equation}
     \label{eq:add_euler_realization}
     \chi(\Sf) + \sum_{i=1}^{\mathfrak{q}} \chi_i = 2g-2.
   \end{equation}  
 \end{itemize}
 The set of realizations of $\mathbf{S}$ in $S_g$ is denoted as $R_g(\mathbf{S})$.
\end{defa}
Recall that $\partial \mathbf{S} $ comes naturally partitioned into
$\bigsqcup_{i=1}^j \partial \sigma_i$. The partition $\vec{I}$ being \emph{transitive} means that
there does not exist a strict non-empty subset of $\partial \mathbf{S} $ that can be realized simultanously
as a union of sets $(\partial \sigma_i)_{1 \leq i \leq j}$ and as a union of sets
$(I_i)_{1 \leq i \leq \mathfrak{q}}$. 

Realizations are meant to enumerate all possible embeddings of a local type in a closed surface of
genus $g$ (i.e. absolute Euler characteristic $2g-2$), modulo the action of the mapping class group
(\cite[\S 4.5]{Ours1}). The transitivity assumption guarantees that the resulting surface is
\emph{connected}.

For a purely 1d c-filling type, i.e. a multi-curve, we have a similar definition for realizations of
a multicurve in a connected surface $S_{g, n}$ of signature $(g,n)$.

\begin{defa}
  \label{def:realizations-c}
  Let $c=(c_1, \ldots, c_j)$ be a purely 1-dimensional c-filling type. Let $(g,n) \in \N_0^2$ be
  such that $2g-2+n > 0$ or $(g,n)=(0,2)$. We call \emph{realization of $c$ in a connected surface
    $S_{g, n}$ of signature $(g,n)$} any pair $\mathfrak{R} = (\vec{J}, \vec{g})$, where:
  \begin{itemize}
  \item $\vec{J}$ is a partition of $ \{1, \ldots, n\}\sqcup  \tilde{\partial}c $ into
    $\mathfrak{q} \geq 1$ non-empty sets $J_1, \ldots, J_{\mathfrak{q}}$, numbered so that
    $i \mapsto \min J_i$ is an {increasing} function;
  \item we ask furthermore the partition to be \emph{transitive} in the sense given below;
  \item $\vec{g} = (g_1, \ldots, g_{\mathfrak{q}})$ is a vector of non-negative integers;
  \item if for any $1 \leq i \leq \mathfrak{q}$, we denote $n_i := \# J_i$ and
    $\chi_i := 2g_i-2+n_i$, then $\chi_i>0$ or $(g_i,n_i)=(0,2)$, and in the latter case, $J_i$
    must contain exactly one element of $ \{1, \ldots, n\}$ and one element of $\tilde{\partial} c$;
  \item we have:   \begin{equation}
     \label{eq:add_euler_realization}
     \sum_{i=1}^{\mathfrak{q}} \chi_i = 2g-2+n.
   \end{equation}   
 \end{itemize}
 The set of realizations of $(c_1, \ldots, c_j)$ in $S_{g, n}$ is denoted as
 $R^1_{g, n}(j)$.
\end{defa}

We recall that here $\tilde{\partial}c = \{1, \ldots, j\} \times \{\pm\}$ is the set of boundary
curves of the cylinders containing the multi-curve $c$.
Note that the superscript $1$ refers to the fact that we consider realizations of 1d curves. If the
curves, instead of being numbered by $\{1, \ldots, j\}$, are numbered by an abstract set $B$ of
cardinality $j$, we write $R^1_{g, n}(B)$.  Similarly, if the boundary components of $S_{g, n}$ are
indexed by a set $I$ of cardinality $n$, we denote $S_{g, I}$ and $R^1_{g, I}(B)$.

Realizations are meant to enumerate (modulo the action of the mapping class group) all possible
embeddings of a multi-curve with $j$ components into a surface of signature $(g,n)$.  In the fourth
point of the definition, we impose that two curves cannot be in the same homotopy class, but we
authorize the case where one of the curves is homotopic to a boundary component.  In particular, if
$(g, n)=(0, 2)$ we must have $j\leq 1$.

Because we are only interested in \emph{connected} surfaces $S_{g, n}$, we again impose a
transitivity condition, defined as follows.  Remark that
$ \{1, \ldots, n\}\sqcup \tilde{\partial}c$ comes naturally partitioned into
$$ \Big(\bigsqcup_{k=1}^{n}\{k\} \Big) \sqcup \Big(\bigsqcup_{i=1}^j(\{i\}\times \{\pm\})\Big).$$
The partition $\vec{J}$ being \emph{transitive} means that there does not exist a strict non-empty
subset of $ \{1, \ldots, n\}\sqcup \tilde{\partial}c$ that can be realized simultanously as a union
of sets $ \{k\}$ and $\{i\}\times \{\pm\}$ as above, and as a union of sets $J_i$ with
$1 \leq i \leq \mathfrak{q}$.

Remark that the definition makes sense for $j=0$, in this case we must have $\mathfrak{q}=1$.

The definitions above can be extended for realizations of an arbitrary c-filling types
$\Sf=(c, \sigma)$ (having 1d and 2d components $c$ and $\sigma$) into a connected surface $S_g$. For
the purpose of controlling the dependency of volume functions on $j$, we opt for the following
formulation in two steps: we first realize the 2d part $\sigma$ into $S_g$, and in a second step
we realize $c$ in $S_g\setminus \sigma$.

\begin{defa}
  \label{def:realizations-lc}
  Let $\Sf=(c, \sigma)$ be a general c-filling type, separated into its 1d and 2d components
  $c=(c_1, \ldots, c_j)$ and $\sigma$. Let $g \geq 2$. We call \emph{realization of $\Sf$ in a
    connected surface $S_g$ of genus $g$} the data of
  $\mathfrak{R}=(\mathfrak{R}_\sigma, \pi, (\mathfrak{R}_l)_{1 \leq l \leq \mathfrak{q}})$, where
  \begin{itemize}
  \item
    $\mathfrak{R}_\sigma = (\vec{I}, \vec{g})=(I_1, \ldots, I_{\mathfrak{q}}, g_1, \ldots,
    g_{\mathfrak{q}})$ is a realization of $\sigma$ in $S_g$ with $\mathfrak{q}$ components;
  \item $\pi$ is a partition of $\{1, \ldots, j\}$ into (possibly empty) sets
    $\pi_1, \ldots, \pi_{\mathfrak{q}}$;
  \item for each $1 \leq l \leq \mathfrak{q}$,
    $\mathfrak{R}_l= (\vec{J}_l, \vec{g}_l) \in R^1_{g_l, I_l}(\pi_l)$ is a realization of $(c_i)_{i\in \pi_l}$
    in a connected surface $S_l=S_{g_l, n_l}$ of signature $(g_l, n_l)$, where $n_l= \# I_l$.
   \end{itemize}
 The set of realizations of $\mathbf{S}$ in $S_g$ is denoted as $R_g(\mathbf{S})$.
 \end{defa}

For such a realization, the integer $\mathfrak{q}$ corresponds to the number of connected components
of $S_g\setminus \sigma$. If those components are called $S_1, \ldots, S_{\mathfrak{q}}$, then $S_l$
has $n_l$ boundary components indexed by $I_l$. The set $\pi_l$ corresponds to the curves in
$(c_1, \ldots, c_j)$ that lie inside the component $S_l$.

For $1 \leq l \leq \mathfrak{q}$, we denote $\vec{J}_l=(J_l^1, \ldots, J_l^{{\mathfrak{q}}_l})$,
i.e. we let $\mathfrak{q}_l$ be the number of connected components of the surface
$S_l\setminus (\bigcup_{i\in \pi_l} c_i)$ (some of those components may be peripheral annuli). Then
$(J_{l}^i)_{1 \leq i \leq {\mathfrak{q}}_l}$ forms a partition of
$I_l \sqcup (\pi_l\times \{\pm\})$.

The collection $(J_{l}^i)_{\substack{1 \leq l \leq {\mathfrak{q}}, 1 \leq i \leq {\mathfrak{q}}_l}}$
forms a partition of $\tilde{\partial} \Sf = \{1, \ldots, n_\sigma\} \sqcup
\tilde{\partial}c$. Letting $n_l^i = \# J_l ^i$, we have
$$\sum_{i=1}^{{\mathfrak{q}}_l}n_l^i = n_l+2|\pi_l|
\quad \text{and}
\quad \sum_{l=1}^{\mathfrak{q}}\sum_{i=1}^{{\mathfrak{q}}_l}n_l^i=n_\sigma+2j=\n.$$

\begin{nota}
  With the notations of Definition \ref{def:realizations-lc}, for $Q\geq 0$ an integer, we denote by
  $R_{g, Q}(\mathbf{S})$ the set of realizations in $ R_g(\mathbf{S})$ for which
  $c\in \MC_{S_g}(Q)$, that is to say, the 1d part $c$ disconnects $S_g$ into at most $Q$ connected
  components.
\end{nota}

Remark that the number of connected components of $S_g\setminus \Sf$ coincides with
$\sum_{l=1}^{\mathfrak{q}} {\mathfrak{q}}_l$. For a realization in $R_{g, Q}(\mathbf{S})$, by
definition, the surface $S_g\setminus c$ has at most $Q$ connected components. {Each
  connected component of $S_g \setminus \Sf$ is either a connected component of $S_g \setminus c$,
  or shares a boundary with one boundary component of $\sigma$.  Hence,
$$\sum_{l=1}^{\mathfrak{q}} {\mathfrak{q}}_l \leq Q+ n_\sigma.$$}

\subsubsection{Volume function associated to a realization}
\label{sec:volume-funct-rea}

For a realization $\mathfrak{R} \in R_g(\Sf)$, we define the volume function $V_{\mathfrak{R}}(\z)$,
which is a function of $\n=n_\sigma+2j$ variables
$\z=(z_t)_{t\in \tilde{\partial} \Sf}$ defined by
\begin{align*}
  V_{\mathfrak{R}}(\z):= \prod_{l=1}^{{\mathfrak{q}}} \prod_{i=1}^{{\mathfrak{q}}_l} V_{g_l^i,
  n_l^i}(\z_{J_l^i})
  \quad \text{where} \quad
  \z_{J_l^i} = (z_t, t\in J_{l}^i),
  % \prod_{k=1}^j \frac{1}{z_{k, +}}\delta(z_{k, +}-z_{k, -}),
\end{align*}
where we recall that for any $(g,n)$ with $2g-2+n>0$, $V_{g,n}(x_1, \ldots, x_n)$ is the total
volume of the moduli space of surfaces of signature $(g,n)$ with boundary lengths
$x_1, \ldots, x_n$, and
$$V_{0, 2}(x_1, x_2) =\frac1{x_1} \delta(x_1-x_2)$$  where $\delta$ denotes the Dirac mass at $0$.
 
 For $1 \leq l \leq {\mathfrak{q}}$, let
 $$ J_{l}^{\mathrm{cyl}} := \bigcup_{\substack{1 \leq i \leq {\mathfrak{q}}_l\\ \chi_l^i=0}} J_l^i,
 $$
 where $\chi_l^i=2g_n^i-2+n_l^i \geq 0$.  We then define
 $J^{\mathrm{cyl}} := \bigcup_{l=1}^{{\mathfrak{q}}} J_{l}^{\mathrm{cyl}}$. This corresponds to the set of
 boundary components of $\Sf$ that are glued to one another by cylinders in the
 realization. 
  % We let
 %$$\pi'_l=\pi_l\setminus\{i, (i, +) \mbox{ or } (i, -) \in J_{l}^{\# 2} \},$$
 %corresponding to the curves $c_i, i\in \pi_l$ that are not peripheral in $S_l$.
  %We have 
%$$\sum_{\substack{i=1 \\ n_l^i\not=2}}^{q_l}n_l^i= n_l+2\# \pi'_l$$.

\subsubsection{Upper bound on volume functions}
\label{sec:upper-bound-volume}

Let us prove a rough upper bound on the volume functions.

\begin{lem}
  \label{lem:u_vol_lc}
  For a fixed partition $(\vec{J}_l)_{1 \leq l \leq \mathfrak{q}}$ of $\tilde{\partial}{\Sf}$ as
  above, the sum over all corresponding realizations possessing this datum satisfies, for a
  universal constant $C$,
  \begin{align*}\frac{1}{V_g}\sum_{\mathfrak{R}}
    \prod_{t\in \tilde{\partial}\Sf} z_t \,\,V_{\mathfrak{R}}(\z)
    \leq \frac{C}{g^{\chi(\Sf)}}
    \prod_{  t\in  \tilde{\partial}\Sf \setminus J^{\mathrm{cyl}}} 2  \sinh \div{ z_t} 
  \prod_{\substack{\chi_l^i = 0 \\ J_l^i =\{t, t'\}}}  z_t \delta(z_t- z_{t'})
  \end{align*}
  to be interpreted as an inequality between two positive measures.  The number of possible data
  sets $(\vec{J}_l)_{1 \leq l \leq \mathfrak{q}}$ is at most $2^{\n (Q+n_\sigma)}$.
\end{lem}

\begin{proof}
  The choice of $(\vec{J}_l)_{1 \leq l \leq \mathfrak{q}}$ automatically fixes ${\mathfrak{q}}$,
  $\vec{I}$, $\pi$ and $J^{\mathrm{cyl}}$. This reduces to a sum over all possible values of
  $g_1, \ldots, g_{\mathfrak{q}}$ and $\vec{g}_l$. This sum is controlled by \cite[Lemma
  24]{nie2023}:
  \begin{align*}\sum_{\mathfrak{R}}
    \prod_{t\in \tilde{\partial}\Sf} z_t \,\,V_{\mathfrak{R}}(\z)\leq
    \prod_{l=1}^{{\mathfrak{q}}}   W_{\chi_l}
    \prod_{  t\in  \tilde{\partial}\Sf \setminus J^{\mathrm{cyl}}} 2  \sinh \div{ z_t}
      \prod_{\substack{\chi_l^i=0 \\ J_l^i =\{t, t'\}}}  z_t \delta(z_t- z_{t'})
  \end{align*}
  where, for $\chi > 0$,
  \begin{equation*}
    W_\chi =
    \begin{cases}
      V_{\frac{\chi}{2}+1,0} & \text{if } \chi \text{ is even};\\
      V_{\frac{\chi+1}2,1} & \text{otherwise.}
    \end{cases}
  \end{equation*}
  The product $\prod_{l=1}^{{\mathfrak{q}}} W_{\chi_l}$ itself is bounded by
  $\cO(V_g/g^{\chi(\Sf)})$.  Then, $2^{(n_\sigma+2j)(Q+n_\sigma)}$ is an upper bound on the number
  of partitions of a set with $\n = n_\sigma+2j$ elements into at most $Q+n_\sigma$ subsets.
\end{proof}

 \begin{rem}\label{r:constraints} If we want to further impose that $\Sf=(c, \sigma)$ is realised as a sub-c-surface, we have to impose that $\chi_j >0$ and $\chi_l^i>0$ in Definitions \ref{def:realizations-c} and \ref{def:realizations-lc}.
   It is also possible to impose this restriction only on some elements $J_j$ or $J_l^i$ of the
   partitions; that is to say, forbid certain components to be cylinders. Obviously, imposing
   restrictions on the set of allowed realizations does not modify the following discussion.
\end{rem}

\subsection{Proof of our main results}

We now turn to the proof of the main results of this section, Proposition \ref{thm:volumen_lc} and
Theorem \ref{thm:leviathan_lc}.  We omit technical details that are similar to earlier
considerations: we focus on the main new difficulty, which is to control the dependency on the
integer $j$. This requires to used refined estimates, with explicit dependency in the
  constants, as developed in \cite[Appendix B]{Ours1}.
\begin{rem}
  At some point it seems necessary to assume some upper bound, say $j\ll g^{1/2}$. It is not clear
  if this is strictly necessary, or only a consequence of our imperfect optimisation of the
  estimates. In any case, the application to spectral gap results only requires $j\leq \log g$.
\end{rem}

\subsubsection{Expression of the volume function}

Let ${\ftype}=\eqc{{{\Sf, \curv, \tau}}}$ be a local topological type of lc-surface such that
$\mathfrak{c}(\Sf)=0$ (thus, $n_\sigma = \n$). Recall that $\rho_j\ftype$ stands for the local
topological type obtained by adding $j$ 1-dimensional connected components to~$\ftype$. If $\ftype$
fills the c-surface $\Sf$, then $\rho_j\ftype$ fills the c-surface~$\rho_j\Sf$ and the decomposition
of $\rho_j\Sf$ into 1d and 2d-components can be written as $\rho_j\Sf=(c_1, \ldots, c_j, \Sf)$.

We now straightforwardly apply the method developed in \cite[Thm. 5.7 and Prop. 5.22]{Ours1} to
establish Proposition~\ref{thm:volumen_lc}, and see that the volume function
$V_{g, Q}^{\nu, {\ftype}, j}(\ell)$ has the form
\begin{align}\label{e:disint_phi_nu}
  \int_{\substack{\vec{\alpha} \in \R_{>0}^j}}
  \int_{\substack{\x \in \R_{>0}^{\n}}}
  \int_{\substack{Y\in \cM_\x( \Sf |  \Sf_{\partial}) \\ \ell_Y({\curv})=\ell }}
  \phi(\x, \vec{\alpha}) \,
  \nu(\tau_{(\x,Y)}, \vec{\alpha})
  \frac{\d\mathrm{Vol}^{\mathrm{WP}}_{\g, \n}(\x, Y) }{\d\ell} \d \vec\alpha ,\end{align}
where, for $\vec{\alpha}=(\alpha_1,\ldots,\alpha_j) \in \R_{>0}^j$ and $\x \in \R_{>0}^{\n}$,
$$\phi(\x, \alpha_1, \ldots, \alpha_j) =  \frac{1}{\alpha_1 \ldots \alpha_j}
\Phi_{g,Q}^{\ftype,j}(\x, \alpha_1, \alpha_1, \ldots, \alpha_j, \alpha_j)$$
and $\Phi_{g,Q}^{\ftype,j}$ is defined by
\begin{align}\label{e:enumeration}
  \Phi_{g,Q}^{\ftype,j}(\z):=\frac{1}{n(\ftype)}\frac{1}{V_g}
  \prod_{t\in \tilde{\partial}\Sf } z_t \,\sum_{\mathfrak{R} \in R_{g, Q}(\rho_j \Sf)} V_{\mathfrak{R}}(\z).
\end{align}
Recall from Notation \ref{d:geo_rep} that for $Z=(\x,Y)$, $\tau_{Z}\in \cS(Z)$ is the geodesic
representative of $\tau$ for the metric $Z$. We should actually restrict the sum to realizations
$\mathfrak{R}$ such that $(c_1, \ldots, c_j, \tau)$ stands as a sub-c-surface in $S_g$ (see Remark
\ref{r:constraints}).

Careful attention needs to be paid to the domain of integration $ \cM_\x( \Sf | \Sf_{\partial})$.
Recall from Notation \ref{n:SS} that $\Sf=\Sb\sqcup \Sinn$ where $\Sb$ is filled by $\curv$ and
$\partial \tau$.  By hypothesis, the function $Y\mapsto \nu(\tau_{(\x,Y)}, \vec\alpha)$ is invariant
by the mapping class group of $\tau$. Thus the pair
$(\ell_Y(\curv), \nu(\tau_{(\x,Y)}, \vec\alpha))$, initially defined on the Teichm\"uller space
$\cT_\x(\Sf)$, is invariant by $\mathrm{MCG}(\Sinn)$. It follows that the integral takes place on
$ \cM_\x( \Sf | \Sb)$, defined as  the quotient of $\cT_\x(\Sf)$ by $\mathrm{MCG}(\Sinn)$ and the
Dehn twists along the shared boundary components of $\Sb$ and $\Sinn$. The
combinatorial factor $n(\ftype)$ is the number of diffeomorphisms of $\Sb$ stabilising
$(\curv, \partial\tau)$, modulo homotopy.
   
 \subsubsection{Rough upper bounds on the volume function associated to a type}

 Let us prove the following rough upper bound.

 \begin{lem}\label{lem:rough_u_lc}
   Let $\ftype = \eqc{\Sf,\curv,\tau}$ be a local topological type of lc-surface such that
   $\mathfrak{c}(\Sf)=0$.  For any large enough $g$, any $j, Q \geq 0$, any test function $F$, any
   $0 < \eta \leq 1/9$,
   \begin{align*}
     &\Big|\int_{\R_{>0}} F(\ell) V_{g,Q}^{\nu,\ftype,j}(\ell) \d \ell\Big|\\
     &\leq \fn(\chi(\Sf), \eta)
     2^{(\n+2j)(Q+\n)} (e^{\kappa}-1)^j \mathfrak{f}(j, \chi(\tau))
     e^{13 \tfL  \chi(\Sf)} 
     \frac{\norm{ F(\ell) \, e^{(1+\eta)\ell}}_\infty}{g^{\chi(\Sf)}}.
   \end{align*}
 \end{lem}
 
 \begin{proof}
   We use the formula for the volume function and the upper bound from
   Lemma~\ref{lem:u_vol_lc}. In order to control the exponential terms, we note that, in each term
   appearing with associated lengths $(z_t)_{t \in \tilde{\partial} \Sf}$, if
   $\x=(x_1, \ldots, x_{\n})$ is the vector of boundary lengths of the surface $\Sf$ filled by
   $(\curv, \tau)$ and $\vec{\alpha}=(\alpha_1, \ldots, \alpha_j)$ the vector of the 1d-component,
   then
   \begin{equation}
     \label{e:fill1} \frac12
      \sum_{t \in \tilde{\partial} \Sf} z_t = \sum_{j=1}^j \alpha_j + \frac 12 \sum_{i=1}^{\n} x_{i}.
    \end{equation}
   By a straightfoward adaptation from Proposition
 \ref{p:comparison1},  for every hyperbolic metric $Y$ on~$\Sf$,
  \begin{align} \label{e:fill}
    \frac 12 \sum_{i=1}^{\n} x_i \leq \ell_Y(\curv)+\ell_Y(\partial \tau)
    \leq \ell_Y(\curv) + 12 \tfL \chi(\Sf)
  \end{align}
  on the support of $\nu$, by Remark \ref{rem:nu_bound}.  As a result, the integral
  $|\int_{\R_{>0}} F(\ell) V_{g,Q}^{\nu,\ftype,j}(\ell) \d \ell|$ is bounded above by
  $2^{(\n+2j)(Q+\n)} e^{12 \tfL \chi(\Sf)} \norm{ F(\ell) \,
      e^{(1+\eta)\ell}}_\infty g^{-\chi(\Sf)}$ times
 \begin{align*}
     \int_{\substack{\vec\alpha, \x >0 \\ Y\in \cM_\x( \Sf |  \Sf_{\partial})}}
      e^{-\eta \ell_Y({\curv})} 
   \prod_{i=1}^j e^{\alpha_i}\bbbone_{[0, \kappa]}(\alpha_i) 
  | \nu(\tau_{(\x,Y)}, \vec\alpha)|
 \d\mathrm{Vol}^{\mathrm{WP}}_{\g,\n}(\x, Y)      \d \vec\alpha.
  \end{align*}
  We now integrate $\vec\alpha$ on $[0,\kappa]^j$, using the upper bound \eqref{n:nu-requirements},
  to bound the previous integral by
 \begin{align*}
    (e^{\kappa}-1)^j \mathfrak{f}(j, \chi(\tau))
   \int_{\x,Y \in \cM_\x( \Sf |  \Sf_{\partial})} e^{-\eta \ell_Y({\curv})} 
                    \d\mathrm{Vol}^{\mathrm{WP}}_{\g,\n}(\x,Y)
 \end{align*}
 where the integral runs over the set of $\x >0$ and $Y\in \cM_\x( \Sf | \Sf_{\partial})$.  We
 conclude by noticing that the integral above is bounded by
 $\fn(\chi(\Sf), \eta) e^{9\eta \tfL \chi(\tau)}$.  Indeed, we showed in \cite{Moebius} that, because
 $\tau$ is a derived tangle, we can construct a multi-loop $\gamma_\tau$, filling $\tau$, of total
 length $\leq 9 \tfL \chi(\tau)$. We then conclude using an adaptation of Corollary \ref{c:pvolume},
 proving that the Weil--Petersson volume of the set
 $$\{(\x, Y), \x > 0, Y\in  \cM_\x( \Sf |  \Sf_{\partial}) , \ell_Y({\curv})
 + \ell_Y(\gamma_\tau)\leq \ell\}$$ grows at most polynomially in $\ell$, because
 $(\curv, \gamma_\tau)$ fills $\Sf$.
\end{proof}

\subsubsection{Preliminary step}

We are now ready to detail the proof of Theorem \ref{thm:leviathan_lc}, which we do in the next few
paragraphs. First, \cite[Lemma B.2]{Ours1} shows that we can restrict attention to realizations of
rank $\leq \ord$ in the sum \eqref{e:enumeration}, modulo an error
$$ \leq \fn(Q,\ord,\chi(\Sf)) \frac{2^{2j(Q+\n)}}{g^{\ord+1}} \exp \Big( \frac 12 \sum_{t \in \tilde{\partial}\Sf}
z_t \Big).$$ Given $Q$, $\ord$ and $\Sf$, the number of such realizations is bounded: we are reduced
to dealing with a finite number of terms in \eqref{e:enumeration}. We therefore pick a realisation
$\mathfrak{R} \in R_{g,Q}(\rho_j\Sf)$.

\subsubsection{First step}

We now use \cite[Theorem B.1]{Ours1} to write an approximate expression (up to errors decaying like
$1/g^{\ord+1}$) of each term
$\frac{1}{V_g} \prod_{t\in \tilde{\partial} \Sf } z_t V_{\mathfrak{R}}(\z)$ in the definition of
$\Phi_{g,Q}^{\ftype,j}(\z)$, see equation \eqref{e:enumeration}. More precisely, the quantity above
is well-approximated by a function $F_{g,\mathfrak{R}}^{(\ord)}(\z) $ of the form
\begin{align}F_{g,\mathfrak{R}}^{(\ord)}(\z)  \label{e:FgR}
  := \sum_{V_+,  V_-, V_0}
  P_{g,\mathfrak{R}}^{(\ord,V_{\pm})}(\z) \prod_{t \in V_+} \cosh
  \div{z_t}
  \prod_{t \in V_-} \sinh \div{z_t}
  \prod_{\substack{\chi_i^l = 0 \\ J_l^i =\{t, t'\}}}  z_t \delta(z_t- z_{t'})
\end{align}
where:
\begin{itemize}
\item $V_+ , V_- ,V_0 $ runs over partitions of the set
  $V := \tilde{\partial} \Sf \setminus J^{\mathrm{cyl}}$;
\item the function $ P_{g,\mathfrak{R}}^{(\ord,V_{\pm})}$ is a polynomial function in the variables
  $(z_t)_{t\in V}$, of total degree in $(z_t)_{t\in V_+\sqcup V_-}$ is $\leq 2\ord$ and partial
  degree with respect to each $(z_t)_{t\in V_0}$ less than a constant~$a_{\ord+1}$ depending only on
  $\ord$.
\end{itemize}
The error made in this approximation can be bounded above by
$$ \fn(Q,\ord,\chi(\Sf)) \, \tilde a_\ord^{2j}\frac{(\|\z\|+1)^{3\ord+1}}{g^{\ord+1}}
\exp \Big( \frac 12 \sum_{t \in \tilde{\partial} \Sf} z_t \Big).$$ By \cite[Theorem B.1]{Ours1} and
\cite[eq. (B.4)]{Ours1}, the coefficients of the polynomials $P_{g,\mathfrak{R}}^{(\ord,V_{\pm})}$ are
bounded by
$$\fn(Q, \ord, \chi(\Sf)) \tilde a_\ord^{\n+2j} \frac{V_{\g,\n}}{V_g} V_{\mathfrak{R}}(\mathbf{0})
\leq \fn(Q, \ord, \chi(\Sf)) \frac{\tilde a_\ord^{2j}}{g^{\chi(\Sf)}}.$$

  \subsubsection{Second step: the Friedman--Ramanujan property}
  
  The leading term in Theorem~\ref{thm:leviathan_lc} is obtained by replacing the function
  $\Phi_{g,Q}^{\ftype,j}(\z)$ by its approximate expression up to errors $1/g^{\ord+1}$. It is thus of
  the form \eqref{e:disint_phi_nu}, with
  $$\phi(\x, \vec\alpha)
  =  \frac{1}{\alpha_1 \ldots \alpha_j}
  \tilde\Phi_{g,Q}^{\ftype,j}(\x, \alpha_1, \alpha_1, \ldots, \alpha_j, \alpha_j) ,$$
   where $\tilde\Phi_{g,Q}^{\ftype,j}(\z)$ is a sum of finitely many terms of the form
 \eqref{e:FgR}. We are left with the question of proving that each function of $\ell$ defined by
   \begin{align}\label{e:dominant-lc}
  \int_{\substack{\alpha, \x >0}} \int_{\substack{Y\in \cM_\x( \Sf |  \Sf_{\partial}) \\ \ell_Y({\curv})=\ell }} \frac{1}{\alpha_1 \ldots \alpha_j} F_{g,\mathfrak{R}}^{(\ord)}(\x, \alpha_1, \alpha_1, \ldots, \alpha_j, \alpha_j) 
      \, \nu(\tau_{(\x,Y)}, \vec\alpha)
     \d \vec\alpha \frac{\d\mathrm{Vol}^{\mathrm{WP}}_{\g,\n}(\x, Y) }{\d\ell}
   \end{align}
   is a Friedman--Ramanujan function, with estimates on its norm.

   The proof of this fact is the same as the proof of Theorem \ref{t:main-kappa}, applied to the
   multi-loop $(\curv, \partial \tau, c_1, \ldots, c_j)$, with the set $W$ corresponding to loops in
   $\partial \tau \cup (c_1, \ldots, c_j)$, $\kappa_i:=\kappa$ for loops in $ (c_1, \ldots, c_j)$,
   and $\sum_{i\in \partial \tau} \kappa_i \leq 12\tfL \chi(\Sf)$.
   
   The set of \emph{a priori neutral} variables is enlarged to include the parameters $q\in \Theta$ such that the simple portion $\cI_q$ is contained in the interior of the 2d-part of $\tau$ (indeed, those simple portions cannot be included 
   in a boundary component of $\Sf$).
   
   Theorem \ref{t:main-kappa} then shows that \eqref{e:dominant-lc} is a Friedman--Ramanujan
   function, with the same parameters as in Remarks \ref{r:firstKN}. Hence, the parameters
   $\rK(\ord, \Sf)= 2\ord + 2\chi(\Sf)$ and $\rN(\ord, \Sf)= \n (a_{\ord+1}+1) + 3 \chi(\Sf)+1$
   satisfy the claim of Theorem \ref{thm:leviathan_lc}.  We obtain an estimate of the
   Friedman--Ramanujan norm of the function \eqref{e:dominant-lc} by:
   \begin{align*}
 \fn(Q, \ord, \chi(\Sf))\, \tilde a_\ord^{2j} \mathfrak{f}(j, \chi(\tau)) e^{6 \tfL\chi(\Sf)},
   \end{align*}
   where the term $ \mathfrak{f}(j, \chi(\tau))$ comes from the requirements on the function $\nu$
   given in Notation \ref{n:nu-requirements}.

  \subsubsection{Third step: estimate of the remainder term}
 
As in the proof of Lemma \ref{lem:rough_u_lc}, 
  using \eqref{e:fill1}, \eqref{e:fill} and the requirements on the function $\nu$ coming from
  Notation \ref{n:nu-requirements}, we see that the remainder term in Theorem \ref{thm:leviathan_lc}
  is bounded by
 \begin{align*}
   &\fn(Q,\ord,\chi(\Sf))
   \frac{\norm{ F(\ell) \, e^{(1+\eta)\ell}}_\infty}{g^{\ord+1}}
   \left(2^{2j(Q+\n)}  +  \tilde a_\ord^{2j}\right) e^{12 \tfL \chi(\Sf)}
     \\
   & \int_{\substack{\alpha, \x >0}}
     \int_{Y\in \cM_\x( \Sf |  \Sf_{\partial}) }
     e^{-\eta\ell_Y({\curv})} 
     \prod_{i=1}^j e^{\alpha_i} \bbbone_{[0, \kappa]}(\alpha_i) 
     |\nu(\tau_Y, \vec\alpha)|
  \d \vec\alpha \d\mathrm{Vol}^{\mathrm{WP}}_{\g,\n}(\x, Y).
 \end{align*}
 As in the proof of Lemma \ref{lem:rough_u_lc}, we conclude by integrating $\vec\alpha$ on
 $[0,\kappa]^j$ and bounding the final integral, introducing a multi-loop filling $\tau$.

%%%Local Variables: 
%%% mode: latex
%%% TeX-master: "main"
%%% End: 

\section{Proof of the spectral gap result}
\label{s:wrapup}

Theorem \ref{thm:leviathan_lc} was the last brick in the proof of Theorem \ref{t:dream}, namely
that,  given $\alpha < 1/2$ and arbitrary small $\epsilon >0$, 
\begin{align} \label{e:dream1} \lim_{g \rightarrow \infty}
  \Pwp{\lambda_1(X) \leq \frac{1}{4}  -\alpha^2- \epsilon} =0.
\end{align}

The full argument leading to Theorem \ref{t:dream} from Theorem \ref{thm:leviathan_lc} is given in
our exposition paper~\cite{Expo}. Here, we outline the main steps.

\subsection{Trace method}

The trace method consists in controlling the expectation of a sum over the set $\geod(X)$ of all
primitive closed geodesics of $X$:
\begin{equation}
  \label{eq:trace_after_TF}
 \pavbtf{F} =  \Ewp{\sum_{\gamma \in \geod(X) } F(\ell(\gamma))
       \, \1{\Atf}(X)}
\end{equation}
where we restricted the Weil--Petersson measure to the set $\Atf$ of all
$(\kappa, \tfL)$-tangle-free surfaces, following the definition of \cite{monk2021a,Moebius}. The
test function $F=F_{m,L}$ will be of the form
\begin{align}\label{e:formF}
F_{m,L}(\ell)= \ell e^{-\ell /2}
  \D^m H_L(\ell) 
\end{align}
where $H_L$ is itself a smooth function whose support is $[-L, L]$, {constructed by dilation of
  the function for $L=1$, i.e. $H_L(\ell) := H_1(\ell/L)$} and $\cD$ is the differential operator
$\frac14- \partial^2$.

The introduction of the operator $\cD$ is one of the keys of our approach, as it will give rise to
cancellations in \eqref{eq:trace_after_TF}, leading to subexponential growth. The parameters
$\kappa$, $\tfL$ and $L$ will be determined later; $\kappa$ is a fixed, arbitrary small number, and
$\tfL=\kappa \log g$; similarly, $L$ will grow logarithmically with the genus $g$. The important
point is that the tangle-free set $\Atf$ is a set of large probability (see \cite[Lemma 4.4]{Expo}),
namely that, for $\kappa < \kappa_0 = \min(2/3, 2 \argsh 1)$,
\begin{equation}
  \label{eq:proba_tf}
    1-\Pwp{\Atf} \leq C_{\mathrm{tf}}(\kappa^2 + g^{\frac 3 2 \kappa-1}).
\end{equation}

From \cite{Expo}, it suffices to show that the expectation \eqref{eq:trace_after_TF} is
$o(e^{(\alpha+\eps)L})$ as $g \rightarrow + \infty$.  More precisely, the Selberg trace formula
together with the Markov inequality allows to write
 \begin{align}
  \Pwp{\delta \leq \lambda_1 \leq \frac 1 4 - \alpha^2 - \epsilon} 
   \leq \frac{C}{e^{(\alpha + \epsilon) L}}
   \Big( \pavbtf{ F_{m,L}} + g \log(g)^2\Big)
   + 
   C_{\mathrm{tf}}(\kappa^2 + g^{\frac 3 2 \kappa-1})\label{e:trace_meth}
 \end{align}
 for all $\delta>0$ (see \cite[Lemma 2.7]{Expo}).  We therefore need to have
 $e^{(\alpha+\eps)L}\geq g^{1+\eta}$ for a $\eta>0$, in order to balance the terms growing linearly
 with $g$ (including, in particular, the topological term from the Selberg trace formula).  From now
 on, we take $L=A\log g$ for an integer $A\geq \alpha+\eps$.

 \subsection{Restriction to a finite number of types}

 We now restrict the number of topological configurations in the average $\pavbtf{F_{m,L}}$ by using
 asymptotic expansions in powers of $1/g$ and the tangle-free hypothesis.  First, letting
 $\chic = 2A +1$, we observe that we can replace the average $ \pavbtf{F_{m,L}} $ by a
 truncated version $\avbtf{F_{m,L}}$, where for a general test function $F$,
   \begin{align*}  
  \avbtf{F}  
   % \label{e:avtbtf}
     :=\Ewp{ \sum_{\type : \chi(\type) \leq \chic}
     \sum_{\substack{\gamma \in \cG(X) \\\gamma\sim \type}} F(\ell(\gamma))
     \, \bbbone_{\Atf}(X)}.
 \end{align*}
  Indeed, the possibility to restrict the sum to local topological types $\type$ of absolute
   Euler characteristic $\chi(\type)\leq \chic$ comes from standard estimates presented in
   \cite[Lemma 2.2 and Proposition A.1]{Ours1}, which allow to write
   \begin{equation}
     \label{eq:prescribe_chic}
   \av{F}
   = \sum_{\type : \chi(\type) \leq \chic} \av[\type]{F}
   + \O[\chic]{\frac{L^{c(\chic)+1}e^{2L}}{g^{\chic}} \|F\|_\infty}
 \end{equation}
 if $F$ is supported on $[-L,L]$, the remainder being $\O[A]{ \|F\|_\infty}$ with $L=A \log g$ and
 $\chic = 2A+1$.

We now decompose this average according to the local types it contains, i.e. we write
\begin{align}  
  \avbtf{F} 
  \label{e:reduce} 
= \sum_{\type \in \mathrm{Loc}_{\chic}^{\kappa,\tfL, L}}\avTbtf{F}
\end{align}
where $\mathrm{Loc}_{\chic}^{\kappa,\tfL, L}$ is the set of local topological types of loops,
filling a surface of Euler characteristic bounded by $\chic$, that can be realized by a closed
geodesic of length $\leq L$ in a $(\kappa, \tfL)$-tangle-free surface, and for
$\type \in \mathrm{Loc}_{\chic}^{\kappa,\tfL, L}$
\begin{align} \label{e:avtbtf}  \avTbtf{F}
  :=\Ewp{\sum_{\substack{\gamma \in \cG(X) \\\gamma\sim \type}} F(\ell(\gamma)) \, \bbbone_{ \Atf}(X)}.
 \end{align}
 We know from \cite{Moebius} that for $\tfL=\kappa \log g$ and $L = A \log g$, 
    \begin{equation}
    \label{eq:bound_number_type_TF}
    \# \mathrm{Loc}_{\chic}^{\kappa,\tfL, L} \leq \fn(\kappa, A,\chic) (\log g)^{n_{\mathrm{typ}}}
\end{equation}
for a constant $n_{\mathrm{typ}} = \fn(\kappa, A, \chic) > 0$ depending only on $\kappa$, $A$ and
$\chic$.

\subsection{Expression of the indicator function using the Moebius function}

Thanks to the result of \cite{Moebius}, we can express the indicator function $\1{\Atf}$ in terms of the Moebius function $\mu_{\kappa, \tfL} $~: for every hyperbolic surface $X$,
\begin{align} \label{e:inversionformula}1- \1{\Atf}(X)= \sum_{\tau\in \cS(X)} \mu_{\kappa, \tfL}
  (\tau)
\end{align}
where the sum runs over all the geodesic sub-c-surfaces of $X$.

This formula will allow to rewrite \eqref{e:reduce} as
\begin{align}\label{e:tensor}
  \sum_{\type \in \mathrm{Loc}_{\chic}^{\kappa,\tfL, L}} \av[{\mathbf{T}}]{F}
  -   \sum_{\substack{\Sf \\ \chi(\Sf)< \chictau}} \sum_{{\ftype}\in \LocTFlc }
  \avN[{\ftype}]{\mu_{\kappa, \tfL} \otimes F}
\end{align}
up to a small error term. The set $\LocTFlc$ now is the set of local topological types $\ftype$ of
lc-surfaces $(\Sf, \curv, \tau)$ such that:
    \begin{itemize}
    \item the loop $\curv$ fills a surface of absolute Euler characteristic $\leq \chic$, and has a
      topology that can be realised in a tangle-free surface with a geodesic having length less than
      $L$;
    \item the absolute Euler characteristic of $\tau$ is less than $\chitau$;
    \item the topology of the pair $(\zeta, \tau)$ is constrained by the fact that it can be
      realised in a hyperbolic surface $Z$, in which $\ell_Z(\zeta)\leq L$ and $\tau$ is a
      $(\kappa, \tfL)$-derived tangle.
 \end{itemize}
 We recall from \eqref{e:total} that
 \begin{align}  
 \label{e:recall_total_mu}
   \avN[\ftype]{\mu_{\kappa, \tfL}  \otimes F}
   :=   \Ewpo \Bigg[ \sum_{\substack{ \gamma \in \geod(X) \\\tau\in \cS_Q(X)\\ (\gamma, \tau)\sim {\ftype}}} 
\mu_{\kappa, \tfL} (\tau) F( \ell(\gamma) ) \Bigg].
 \end{align}

 Let us bound the error made going from \eqref{e:reduce} to \eqref{e:tensor}. We prove the
 following, a generalization of \cite[Lemma 10.11]{Ours1} to the general inclusion-exclusion using
 the Moebius function.

\begin{prp} \label{prp:proba_small_Moebius} Let $\type \in \mathrm{Loc}_{\chic}^{\kappa,\tfL, L}$,
  for $L = A \log g$, $\tfL = \kappa \log g$, $\chitau = 2A+4$, $Q = 2A + 5+6\chitau$. Provided that
  $ \kappa (U_2(\chitau) + 37 \chitau) < 1$, with $U_2$ the sequence in \eqref{eq:bound_nu_mu}, we
  have:
  \begin{align} \label{e:muav} \avTbtf{F} = \avb[\type]{F} -\Ewpo \Bigg[ \sum_{\substack{(c,
        \sigma)\in \cS(X) \\ \chi(\sigma) <\chitau \\ c\in \, {\MC}_X(Q)}} \sum_{\substack{\gamma
        \in \geod(X)\\ \gamma\sim \type}} \mu_{\kappa,\tfL}(c,\sigma) F(\ell(\gamma)) \Bigg] +
    \O[A]{\|F\|_\infty}.
 \end{align}
\end{prp}

\begin{proof}
  By definition, we have that $\Atf \subset\mathcal{B}_g^{\kappa, Q}$ for any $Q$. Furthermore, if
  $\tau_{\kappa,\tfL}(X)$ denotes the maximal $(\kappa,\tfL)$-derived tangle of $X$, defined in
  \cite{Moebius}, we have $\Atf \subset \{\chi(\tau_{\kappa,\tfL}(X)) < \chitau\}$ because, on
  $\Atf$, the property ``all derived tangles have absolute Euler characteristic less than $\chitau$'' is
  trivially true. We can rewrite these inclusions in the form
  $$\1{\Atf} = \1{\Atf} \1{\mathcal{B}_g^{\kappa, Q}} \1{[0,\chitau)}(\chi(\tau_{\kappa,\tfL}(X)).$$
  The Moebius inversion formula from \cite{Moebius} allows to rewrite the error term in
  \eqref{e:muav} as $E_1+E_2$, where
  \begin{equation}
    \label{eq:E1}
    E_1
    = \Ewp{\1{[\chitau, + \infty)}(\chi(\tau_{\kappa,\tfL}(X)))
      \Bigg(1 -  \sum_{\substack{(c, \sigma)\in \cS(X) \\ \chi(\sigma) <\chitau \\ c\in \,
          {\MC}_X(Q)}} \mu_{\kappa,\tfL}(c,\sigma)\Bigg) \sum_{\substack{\gamma \in \geod(X)\\ \gamma\sim
          \type}}  F(\ell(\gamma))}
  \end{equation}
  and
  \begin{equation}
    \label{eq:E2}
    E_2
    = \Ewp{(\1{\mathcal{B}_g^{\kappa,Q}}(X)-1)
    \1{[0,\chitau)}(\chi(\tau_{\kappa,\tfL}(X)))
    \Bigg(1 -  \sum_{\substack{(c, \sigma)\in \cS(X) \\ \chi(\sigma) <\chitau \\ c\in \,
    {\MC}_X(Q)}} \mu_{\kappa,\tfL}(c,\sigma)\Bigg) \sum_{\substack{\gamma \in \geod(X)\\ \gamma\sim
    \type}}  F(\ell(\gamma))}.
\end{equation}
We shall bound these two terms separately, using the bound \eqref{eq:bound_nu_mu} on the Moebius
function, \cite[Lemma 2.2]{Ours1}, as well as our probabilistic bounds on the events
$\{\chi(\tau_{\kappa,\tfL}(X)) \geq \chitau\}$ and $\mathcal{B}_g^{\kappa, Q}$ proven in
\cite[Appendix A]{Ours1}. This is done in the two following lemmas.

\begin{lem}
  We have
  \begin{equation*}
    |E_1| \leq \fn(Q,\chitau,\kappa)   \|F\|_{\infty}
    \frac{Le^{L+R(36\chitau+U_2(\chitau))}}{g^{\chitau/2-1}}
  \end{equation*}
  and, in particular, if we specify $L=A\log g$, $\tfL=\kappa\log g$, $\chitau=2A+4$ and fix a value
  of $\kappa$ such that $ \kappa < ( U_2(\chitau) + 36 \chitau)^{-1}$, then
  $|E_1| \leq \fn(Q,A) \|F\|_\infty$.
\end{lem}

\begin{proof}
  We use the triangle inequality and the deterministic bound from \cite[Lemma 2.2]{Ours1} to obtain
  \begin{equation}
    \label{eq:E1}
    |E_1|
    \leq 10^4 L g e^L \|F\|_\infty
    \Ewp{\1{[\chitau, + \infty)}(\chi(\tau_{\kappa,\tfL}(X)))
      \Bigg(1 +  \sum_{\substack{(c, \sigma)\in \cS(X) \\ \chi(\sigma) <\chitau \\ c\in \,
          {\MC}_X(Q)}} |\mu_{\kappa,\tfL}(c, \sigma)|\Bigg) } .
  \end{equation}

  Let us now bound the contribution of the Moebius function. We first use the fact that, if
  $(c,\sigma)$ with $c=(c_1, \ldots, c_j)$, by \cite{Moebius},
  $$ | \mu_{\kappa,\tfL}(c, \sigma)|= |\mu_{\kappa,\tfL}(c)| |\mu_{\kappa,\tfL}(\sigma)|
  = \frac1{2^j j!}  \bbbone_{[0, \kappa]}(\ell_X^{\mathrm{max}}(c))|\mu_{\kappa,\tfL}(\sigma)|.$$
  Introducing the function $$\cY_{\kappa,Q}(X) = \sum_{j=1}^\infty \frac{1}{j!}
  \sum_{\substack{c \in \MC_X(Q) \\
      \mathfrak{c}(c)=j}} \1{[0,\kappa]}(\ell_X^{\mathrm{max}}(c))$$ as in \cite[\S A.2]{Ours1}, we
  obtain
  \begin{align*}
    1 +
    \sum_{\substack{(c, \sigma)\in \cS(X) \\ \chi(\sigma) < \chitau \\ c\in \, {\MC}_X(Q)}} | \mu(c, \sigma)|
    \leq (1 + \cY_{\kappa,Q})
    \Big( 1+ \sum_{\substack{ \sigma \in \cS(X) \\ \chi(\sigma) < \chitau \\ \mathfrak{c}(\sigma)=0}} |\mu(\sigma)|\Big).
\end{align*}
This allows us to reduce ourselves to bounding the sum of the Moebius function for purely 2d
c-surfaces. This is done in \cite{Moebius}, and we obtain a bound of the form
\begin{equation}
  \label{eq:bound_mu_chitau}
  \sum_{\substack{\sigma\in \cS(X)\\ \chi(\sigma) <\chitau \\ \mathfrak{c}(\sigma)=0}} |\mu(\sigma)|
  \leq \fn(\chitau) g^{3\chitau} e^{\tfL U_2(\chitau)}.
\end{equation}

The factor $g^{3 \chitau}$ in \eqref{eq:bound_mu_chitau} is inconvenient. In order to compensate it,
we split the event $\{\chi(\tau_{\kappa, \tfL}(X))\geq \chitau\}$ depending on the relative position
of $\chi(\tau_{\kappa,\tfL})$ and $8 \chitau$:
  \begin{equation*}
    \1{[\chitau, + \infty)}(\chi(\tau_{\kappa,\tfL}(X)))
    =
    \1{[\chitau, 8 \chitau)}(\chi(\tau_{\kappa,\tfL}(X)))
    + \1{[8\chitau, + \infty)}(\chi(\tau_{\kappa,\tfL}(X))).
  \end{equation*}
  This is helpful because, now, we can improve \eqref{eq:bound_mu_chitau} on the set where
  $\chi(\tau_{\kappa,\tfL}(X)) < 8 \chitau$:
  $$\bbbone_{[\chitau,8\chitau)}(\chi(\tau_{\kappa,\tfL}(X)))
  \sum_{\substack{\sigma\in \cS(X)\\ \chi(\sigma) <\chitau \\ \mathfrak{c}(\sigma)=0}}
  |\mu(\sigma)|  \leq \fn(\chitau) e^{\tfL U_2(\chitau)}.$$
  Altogether, we obtain
  \begin{align*}
    |E_1|
    \leq & \, \fn(\chitau)  \|F\|_\infty  L \, g \, e^{L+\tfL U_2(\chitau)}\\
         & \Ewp{(1+\cY_{\kappa,Q})
           \Big(
           \1{[\chitau,+\infty)}(\chi\big(\tau_{\kappa,\tfL}(X))\big)
           +g^{3\chitau} \1{[8\chitau,+\infty)}(\chi\big(\tau_{\kappa,\tfL}(X))\big)\Big)}.
  \end{align*}
  We then use the Cauchy--Schwarz inequality to bound
  the expectation above by
  \begin{equation*}
    (1+\Ewpo\brac*{\cY^2_{\kappa,Q}}^{1/2}) \,
    \Big(\Pwp{\chi(\tau_{\kappa,\tfL}(X)) \geq \chitau}^{1/2}+
    g^{3 \chitau}\Pwp{\chi(\tau_{\kappa,\tfL}(X)) \geq 8\chitau }^{1/2}\Big)
  \end{equation*}
  which, by \cite[Propositions A.1 and A.5]{Ours1}, is bounded by
  \begin{equation*}
    \fn(Q,\chitau,\kappa) \tfL^{c(\chitau)}
    \Bigg(
    \frac{e^{9\tfL\chitau /2}}{g^{\chitau/2}}
    + g^{3 \chitau}\frac{e^{36\tfL\chitau}}{g^{4\chitau}}
    \Bigg)
    \leq \fn(Q,\chitau,\kappa) \frac{e^{36 \tfL \chitau}}{g^{\chitau/2}}
  \end{equation*}
  which is our claim.
\end{proof}

Now, with exactly the same proof, now using \cite[Proposition A.4]{Ours1} to bound the probability
of the complement of $\cB_g^{\kappa,Q}$, we obtain:

\begin{lem} \label{l:small_proba2}
  We have:
\begin{align*}
  |E_2|  \leq \fn(Q, \chitau) 
  \norm{F}_\infty  \frac{L e^{L+\tfL U_2(\chitau)}}{g^{(Q-1) /2 - 1 - 3 \chitau}}.
  \end{align*}
  In particular, if we specify $L=A\log g$, $\tfL=\kappa\log g$, $Q=2A+5+6\chitau$ and fix $\kappa$ such
  that $\kappa < U_2(\chitau)^{-1}$, then $|E_2| \leq \fn(A, \chitau) \|F\|_\infty$.
\end{lem}

Combining the two lemmas allows us to conclude.
\end{proof}

The following proposition shows that we can further restrict the Euler characteristic of $\Sf$, the
surface filled by $(\curv, \tau)$:
 \begin{prp}\label{p:chictau}
  If $L=A\log g$ and $\tfL=\kappa\log g$, we may find
 $\chictau=\fn(A, \chitau)$ such that, for all $\kappa< 1/10$,
 \begin{align*} 
   \sum_\Sf \sum_{{\ftype}\in \LocTFlc }   \avN[{\ftype}]{\mu_{\kappa, \tfL} \otimes F}
   =\sum_{\substack{\Sf \\ \chi(\Sf)< \chictau}} \sum_{{\type}\in \LocTFlc }
   \avN[{\ftype}]{\mu_{\kappa, \tfL} \otimes F} +\cO_{A , Q}(1).
 \end{align*}
 \end{prp}
 
 We omit the proof, a medley of all the previous arguments.

 Note that the sum $\sum_\Sf$ in \cref{e:tensor} now runs over all c-surfaces of absolute Euler
 characteristic less than some $\chictau$ to be chosen later (we assume without loss of generality
 that $\chic, \chitau \leq \chictau$). The set of such c-surfaces is infinite, because $\Sf$ could
 contain an arbitrary number of 1d-components (which do not contribute to the Euler characteristic).
 Thus we prefer to rewrite the second sum in \eqref{e:tensor} as
 \begin{align} \label{e:totalj}
   \sum_{\substack{\Sf \\ \chi(\Sf)< \chictau \\ \mathfrak{c}(\Sf)=0}}
   \sum_{{\ftype}\in \LocTFlc}
   \sum_{j=0}^{+\infty}  \avN[{\rho_j\ftype}]{\mu_{\kappa, \tfL} \otimes F},
\end{align} 
where the first sum, running over purely 2d c-surfaces of bounded Euler characteristic, is now
finite.  Recall that the operation $\rho_j$ consists in adding $j$ 1d components to a c-surface.

Now, we showed in \cite{Moebius} that, if $L=A\log g$, $\tfL=\kappa \log g$, $\mathfrak{c}(\Sf)=0$
and $\chi(\Sf)< \chictau$, then
\begin{equation} \label{e:control_loc}
  \# \LocTFlc \leq \fn(\kappa, A, \chictau) g^{10 \chitau\chictau \kappa}.
\end{equation}
Therefore, we shall take $\kappa$ very small so that this polynomial contribution in $g$ does not
interfere with our method.

% \begin{rem}
%   To be more precise, we will show in \S \ref{s:A3} that if we take
%   \begin{align*}
%     & \chic= 2A  +1
%     && Q=2A+5+6\chitau \\
%     & \chitau=2A+4 
%     && \kappa < ( U_2(\chitau) + 37 \chitau)^{-1} \\
%    & \chictau= \fn(A, \chitau) && 
%   \end{align*}
%   where $U_2$ is the sequence introduced to bound the Moebius function (see \eqref{eq:bound_nu_mu})
%   and the function for $\chictau$ is explicit, the difference between \eqref{eq:trace_after_TF} and
%   \eqref{e:tensor} is bounded above by $\fn(A) \# \mathrm{Loc}_{\chic}^{\kappa,\tfL, L} $.
% %  We can even restrict the sum in \eqref{e:totalj} to $j\leq \log g$.
%   \end{rem}

  The important thing here is that we have removed the indicator function $\1{\Atf}$ from the average
 we consider. This allows to use Theorems \ref{thm:FR_type} and \ref{thm:leviathan_lc} that describe
 the special structure of the volume functions associated to $\av[{\mathbf{T}}]{F}$ and
 $\avN[\ftype]{\mu_{\kappa, \tfL} \otimes F}$.

\subsection{Cancellation and the Friedman--Ramanujan hypothesis}

Theorems \ref{thm:FR_type} and \ref{thm:leviathan_lc} can be combined with the integration by parts
phenomenon presented in {\cite[Proposition 3.19]{Ours1}} to obtain the following propositions. The
underlying principle is that $\cD^{m}$ vanishes on $\C_{m-1}[\ell]e^{\ell/2}$, hence the choice of
our test function \eqref{e:formF}.
     
First, Theorem \ref{thm:FR_type} implies the following, related to the average
$\av[{\mathbf{T}}]{F_{m,L}}$.

\begin{prp}[{\cite{Ours1}}]
  \label{cor:FR_implies_small}
  Let $\type=[\Sf, \curve]_{loc}$ be a local topological type with $\chi(\Sf)\leq\chic$. For any
  integer $\ord \geq 0$, there exists constants $c, \rK = \fn(\ord, \chic)$ such that, for any large
  enough $g$, any $m \geq \rK$, any $\eta >0$ and any $L \geq 1$,
  \begin{equation*}
    % \label{e:contribTIBP}
    \av[\type]{F_{m,L}}
    \leq \fn(\ord,\chic,\eta,m)
    \Big(L^{c} + \frac{e^{L(\frac{1}2 +\eta)}}{g^{\ord+1}}\Big).
  \end{equation*}
\end{prp}
 
 Summing over $\type\in\mathrm{Loc}_{\chic}^{\kappa,\tfL, L}$, and using
 \eqref{eq:bound_number_type_TF} to bound the number of terms in the sum, we obtain that for $\tfL =
 \kappa \log g$ and $L=A \log g$, 
 \begin{align}
  \label{c:almost-} 
   \sum_{\type \in \mathrm{Loc}_{\chic}^{\kappa,\tfL, L}}
   \av[{\mathbf{T}}]{F_{m,L}}
   \leq
   \fn(\ord, \kappa, A,\chic, m, \eta) L^{n_{\mathrm{typ}}}
   \Big(L^{c} + \frac{e^{L(\frac{1}2 +\eta)}}{g^{\ord+1}}\Big)
  \end{align} 
  where we recall that $c= \fn(\ord,\chic)$ and $n_{\mathrm{typ}}=\fn(\kappa,A,\chic)$.
  
  We now state a similar property for the average \eqref{e:totalj}, now using Theorem
  \ref{thm:leviathan_lc}.
     
   \begin{prp}
     \label{cor:FR_implies_small_final} 
     Let ${\ftype}\in \LocTFlc$ for a filling type $\Sf$ such that $\mathfrak{c}(\Sf)=0$ and
     $\chi(\Sf) \leq \chictau$. For any integer $\ord \geq 0$, there exists constants
     $c, \rK = \fn(\ord, \chictau)$ such that, for any large enough~$g$, any $m \geq \rK$, any
     $\eta >0$ and any $L \geq 1$, $\tfL=\kappa \log(g)$, $U_3(\chi):=U_2(\chi)+13\chi$,
     \begin{align*} 
       \sum_{j=0}^{+\infty}
       \avN[\rho_j{\ftype}]{\mu_{\kappa, \tfL} \otimes F_{m,L}} 
        \leq \fn(Q,\ord,\chictau,\kappa,m,\eta)
        \, g^{\kappa U_3(\chitau)}
         \left(L^{c} + \frac{e^{L(\frac 1 2  +\eta)}}{g^{\ord+1}}\right).
       \end{align*}
     \end{prp}
{
  \begin{proof}
    First, we observe that we can truncate the sum to values $j \leq \log(g)$ due to the decay in
    $j$, as done in \cite[Lemma 10.12]{Ours1}.
       We then use our cancellation argument with the Friedman--Ramanujan function
       $h_{g, Q, \ord}^{\mu_{\kappa,\tfL},\ftype,j}$, and sum the contribution associated to each
       $j$, using the fact that, for the Moebius function $\mu_{\kappa, \tfL}$, $\mathfrak{f}(j,\chi
       ) = U_1(\chi) e^{R U_2(\chi)}/j!$ and hence the constant appearing is smaller than
       \begin{equation}
         \label{eq:bound_final_tg}
         \fn(Q,\ord,\chictau)
         e^{R (U_2(\chitau)+13 \chictau)}
         \sum_{j=0}^{+ \infty} \frac{M^j}{j!}
         = \fn(Q,\ord,\chictau) g^{\kappa U_3(\chitau)}e^M
       \end{equation}
       where $M = \fn(Q,\ord,\chictau,\kappa)$.
     \end{proof}}
   
   Summing over all filling types $\Sf$ such that $\mathfrak{c}(\Sf)=0$ and
   $\chi(\Sf) \leq \chictau$ and all local topological types ${\ftype}\in \LocTFlc$ (the cardinality
   of which is bounded by \eqref{e:control_loc}), we obtain an integer $\rK$ such that, for any $m
   \geq \rK$,
   \begin{align}
     \sum_{\substack{\Sf\\\chi(\Sf) < \chictau \\ \mathfrak{c}(\Sf)=0}}
     \sum_{{\ftype}\in \LocTFlc}
     \sum_{j=0}^{+\infty}   
     \avN[\rho_j{\ftype}]{\mu_{\kappa, \tfL} \otimes  F_{m,L}} 
     \leq \fn(Q,\ord,\chictau,\kappa,m,\eta)
      g^{\kappa U_4(\chictau)}
      \left(L^{c} + \frac{e^{L(\frac 1 2  +\eta)}}{g^{\ord+1}}\right)
   \end{align}
   where $U_4(\chi) := U_2(\chi) + 13 \chi + 10 \chi^2$.
   The important thing here is that the bound is subexponential in $L$: there is a polynomial term
   $L^{c}$, and an exponential term $e^{L(\frac 1 2 +\eta)}$ which is compensated by the factor
   $g^{-(\ord+1)}$.  Balancing these terms, we can conclude the proof of Theorem \ref{t:dream}.

\subsection{Final choice of parameters}

Let $\alpha \in (0,1/2)$ and $0<\eps< \frac14 -\alpha^2$.
We pick our parameters the following way.
\begin{itemize}
\item The order $\ord$ of the expansion is chosen to be an integer so that
  $2 \alpha (\ord+2) > 1$.
\item We take $L=A\log g$ with $A=2(\ord+2)$; this ensures that $e^{L/2}/g^{\ord+1}=g$, to
  match the topological term of the trace formula.
\item $\chic = 2A +1$ to be able to limit the local types of loops filling a surface of Euler
  characteristic $< \chic$ (as prescribed after \eqref{eq:prescribe_chic}).
\item $\chitau=2A+4$ and $Q=2A+5+6\chitau$, as recommended in Proposition
  \ref{prp:proba_small_Moebius}, to reduce the possibilities to derived tangles $\tau$ with
  $\chi(\tau) < \chitau$ with a 1d-component separating at most $Q$ connected components.
\item $\chictau=\fn(A, \chic,\chitau)$ is prescribed in Proposition \ref{p:chictau}, implying that
  we can restrict attention to c-filling types $\Sf$ with $\chi(\Sf)< \chictau$.
\item We can now apply Propositions \ref{cor:FR_implies_small} and \ref{cor:FR_implies_small_final}
  with a parameter $\eta<\eps $.
\item This allows to fix the integer $m$ in the definition of $F_{m,L}$: it is chosen as the maximum
  of the integers given by Propositions \ref{cor:FR_implies_small} and
  \ref{cor:FR_implies_small_final}, for the finite number of c-filling types $\Sf$ with
  $\chi(\Sf) < \chictau$ and $\mathfrak{c}(\Sf)=0$.
\item As announced, we take $\tfL=\kappa \log g$, where we now pick the parameter
  $\kappa$ so that
  \begin{equation*}
    \begin{cases}
      0<\kappa< \min(1/10, 2\argsh 1) \\
      (U_2(\chitau) + 36 \chitau)\kappa < 1 \\
      U_4(\chictau) \kappa < \alpha A -1
    \end{cases}
  \end{equation*}
  as recommended in Proposition \ref{prp:proba_small_Moebius} and to ensure that the upper bound
  obtained in equation \eqref{eq:bound_final_tg} is $\cO(e^{(\alpha+\eta)L})$.
\end{itemize}

Then, for $F_{m,L}(\ell)= \ell \cD^m H_L(\ell)e^{-\ell/2}$, the expected trace
\eqref{eq:trace_after_TF} is $o(e^{(\alpha + \eps)L})$ as $g \rightarrow + \infty$. By the
inequality \eqref{e:trace_meth}, this allows to conclude that
 \begin{align*} 
 \limsup_{g \rightarrow + \infty}  \Pwp{\delta \leq \lambda_1 \leq \frac 1 4 - \alpha^2 - \epsilon} \leq  
   C_{\mathrm{tf}} \kappa^2
 \end{align*}
 {for the universal constant $C_{\mathrm{tf}}$ from \eqref{eq:proba_tf}}.
 Since $\kappa$ can be taken to be arbitrarily small, this concludes the proof of Theorem
 \ref{t:dream}.

%%%Local Variables: 
%%% mode: latex
%%% TeX-master: "main"
%%% End: 

\appendix  
% \section{Discarding sets of extremely small probability}
% \label{s:A3}
% \input{app_cauchy_sch}

\section{List of notations}
\label{s:notations}

\renewcommand{\arraystretch}{2}
\begin{longtable}{lll}
  \toprule 
  $\fn((u_i)_{i \in \cV})$ & Function of the parameters $(u_i)_{i \in \cV}$.
  & Nota \ref{n:fn} \\
  \midrule $S_g$, $g \geq 2$ & Fixed compact surface of genus $g$.
  & \\
  $X \in \mathcal{M}_g$ & A metric $X$ on $S_g$, up to isometry.
  & \\
  $\cG(X)$ & Set of primitive closed geodesics of $X$.
  & Def \ref{def:av_type} \\
  $\ell_X(\gamma)$ & Length of the geodesic representative of $\gamma$ in $X$.
  & Def \ref{def:av_type} \\
  $\Pwpo$, $\Ewpo$ & Weil--Petersson probability measure and expectation.
  & \\
  $\lambda_1(X)$ & First non-trivial eigenvalue of Laplacian.
  & Thm \ref{t:dream} \\
  \midrule $\mathbf{T} = \eqc{\mathbf{S},\mathbf{c}}$ & \makecell[l]{Local topological type
    $\mathbf{T}$, \\ i.e. a surface $\mathbf{S}$ filled by a multi-loop $\mathbf{c}$.}
  & Def \ref{def:av_type} \\
  $\av[\mathbf{T}]{F}$ & Average
  $\Ewpo \Big(\sum_{\gamma \text{ of type } \mathbf{T}} F(\ell_X(\gamma)) \Big)$.
  & Def \ref{def:av_type} \\
  $V_g^{\mathbf{T}}(\ell)$ & Volume function corresponding to the average
  $\av[\mathbf{T}]{\, \cdot \,}$.
  & Prop \ref{prp:existence_density} \\
  $f_k^{\mathbf{T}}(\ell)$ & $k$-th coefficient of the asymptotic expansion of
  $V_g^{\mathbf{T}}(\ell)$.
  & Thm \ref{thm:exist_asympt_type_intro} \\
  \midrule $\FR$ and $\FRrem$ & Friedman--Ramanujan functions and remainders.
  & Def \ref{def:FR}  \\
  $\FR^{\rK}$, $\FR^{\rK,\rN}$, $\FRrem^{\rN}$ & FR functions and remainders of bounded growth.
  & Def \ref{def:FR} \\
  $\FR_w$ and $\FRrem_w$ & Weak forms of $\FR$ and $\FRrem$.
  & Def  \ref{def:weakFR} \\
  $\norm{\cdot}_{\FRrem^\rN}$, $\norm{\cdot}_{\FR^{\rK,\rN}}$ & Norms on those spaces (see also weak
  forms version).
  & Def \ref{d:normFR} \\
  \midrule $(\g,\n)$, $\chi({\mathbf{S}})$, $\partial \mathbf{S}$ & \makecell[l]{Signature of
    $\mathbf{S}$, absolute value of \\ Euler characteristic and boundary of $\mathbf{S}$.}
  & Nota \ref{nota:S} \\
  $\mathcal{T}_{\g,\n}^*$
  & \makecell[l]{Teichm\"uller space of $\Sf$ (with free $>0$ boundary length).}
  & \S \ref{s:teich} \\
  $\mathcal{T}_{\g,\n}(\x)$
  & \makecell[l]{Teichm\"uller space of $\Sf$ (with fixed boundary lengths $\x$).}
  & \S \ref{s:teich} \\
  $\d\mathrm{Vol}^{\mathrm{WP}}_{\g,\n}$ & Weil--Petersson measure on $\cT_{\g,\n}^*$.
  & \S \ref{s:teich} \\
  $\rho[\mathbf{T}]$
  & Number of unshielded simple portions of $\mathbf{T}$.
  & Def \ref{d:double_fill}\\
  \midrule
  $V$ & A subset of $\partial \mathbf{S}$.
  &  Nota \ref{nota:mt} \\
  $\mathrm{m}$ & A matching of the complement of $V$ as $\bigsqcup_j \{\vi, \vi'\}$.
  &  Nota \ref{nota:mt} \\
  $\fz$ & A smooth function vanishing at $0$.
  & Nota \ref{nota:mt} \\
  $\cJ=\cJ_{\mathbf{T}, V, \mathrm{m}, f}$ & The main integral to understand.
  & Eq \eqref{e:mainint} \\
  $f_j$, $\tilde{g}_j$
  & Functions appearing in $\cJ$.
  & Nota \ref{nota:mt} \\
  $\rK_j$, $\rN_j$, $\rK$, $\rN$
  & Friedman--Ramanujan exponents.
  & Nota \ref{nota:mt} \\
  % $T_V(\vec{L},\vec{\theta})$ & Expression of $\prod_j x_{\vi} \delta(x_{\vi}-x_{\vi'})$ in terms of
  % $(\vec{L},\vec{\theta})$.
  % & Coro \ref{c:maindet-dirac} \\
  % \midrule
  \midrule $\cP = \cP_x$, $\cL = \cL_x$ & Operators $\int_0^x$ and $\id - \int_0^x$.
  & \S \ref{s:cL} \\
  \midrule $\mathbf{c} = (\mathbf{c}_1, \ldots, \mathbf{c}_{\cc})$ & A multi-loop filling
  $\mathbf{S}$.
  & Def \ref{def:mult_loop}. \\
  $J$, $o(J)$, $t(J)$ & A segment $J$, its origin and terminus.
  &  Nota \ref{n:origin} \\
  $p \smallbullet q$ & Concatenation of the paths $p$ and $q$.
  & Nota \ref{n:concat} \\
  $\bar{J}$ & Orthogeodesic homotopic to $J$ with gliding endpoints.
  & Rem \ref{rem:beta_geod} \\
  $\IH^2$ & The hyperbolic plane.
  & \\
  $\Dist$ & The algebraic distance along a geodesic.
  & Def \ref{def:alg_d} \\
  $\ell(\gamma)$, $\ell_Y(\gamma)$ & Length of a path v.s. the geodesic in a free homotopy class.
  & \S \ref{nota:length} \\
  \midrule $r$ & Number of intersections of $\mathbf{c}$.
  & Nota \ref{n:r} \\
  $\beta, (B_1, \ldots, B_r)$ & \makecell[ll]{Simple multi-loop and bars obtained \\ when opening
    the intersections of $\mathbf{c}$.}
  & \S \ref{sec:open-inters} \\
  $\mathbf{c}^{\mathrm{op}}$ & The representative of $\mathbf{c}$ obtained after opening.
  & \S \ref{sec:open-inters} \\
  $(\beta_\lambda)_{\lambda \in \Lambdabeta}$ & The components of the multi-loop $\beta$.
  & Rem \ref{rem:lambda_beta} \\
  $B_j^\pm$ & Choice of orientation of the bars.
  & \S \ref{s:orientation_bars} \\
  $\Theta$ & The set of $q = (j, \epsilon) \in \{1, \ldots, r\} \times \{+, -\}$.
  & \S \ref{s:labelling_msp} \\
  $\mathcal{I}_q$, $q \in \Theta$ & Labelling of the simple portions.
  & \S \ref{s:labelling_msp} \\
  $\sign(q)$, $\iota(q)$ & Sign function and sign switching involution.
  & Def \ref{def:sign} \\
  $\lambda(q)$ & The component of $\beta$ in which $B_q$ terminates.
  & Nota \ref{nota:ind} \\
  $\Theta_t(\lambda)$, $\Theta_t(W)$ & The set of $q \in \Theta$ s.t. $\lambda(q) = \lambda$ or
  $\lambda(q) \in W$.
  & Nota \ref{nota:ind} \\
  $\beta_{\lambda,{\mathrm{init}}}$, $\mathcal{I}_{q,\mathrm{init}}$ & Paths on $\mathbf{c}$
  corresponding to $\beta_\lambda$ and $\mathcal{I}_q$.  & Rem \ref{r:init} \\ \midrule $\mathbf{D}$
  & A diagram $(\beta, B)$.
  & Def \ref{defa:diagram} \\
  $\mathbf{D}^{(1)}$ & Diagram obtained by opening all intersections the first way.
  & \S \ref{sec:diagrams} \\
  \midrule $p(q)$ & The concatenation of $B_q$ and $\mathcal{I}_q$.
  & Eq \eqref{e:psegment} \\
  $\sigma$ & The permutation of $\Theta$ allowing to go along $\mathbf{c}$.
  & Eq \eqref{e:decompo-i} \\
  $n_i$ & Length of the cycle of $\sigma$ describing $\mathbf{c}_i$.
  & Eq \eqref{e:decompo-i} \\
  $\Z_n$ & Cyclic set $\Z \diagup n\Z$.
  & \S \ref{sec:cycle-decomp-rewr} \\
  $\comp(q)$ & Label of the component $\mathbf{c}_i$ of $\mathbf{c}$ in which $q$ appears.
  & \S \ref{s:relabel_Theta} \\
  $\Theta^i = (q_k^i)_{k \in \Z_{n_i}}$ & Set of $q \in \Theta$ of component $\comp(q)=i$.
  & \S \ref{s:relabel_Theta} \\
  $\mathbf{\Theta}$, $\mathbf{\Theta}^i$ & Periodic extensions of $\Theta$ and $\Theta^i$.
  & \S \ref{s:periodic_Theta} \\
  \midrule $\mathcal{C}$, $\mathcal{U}$ & Set of crossing v.s. non-crossing indices.
  & Def \ref{def:crossing} \\
  \midrule $\vec{L}, \vec{\theta}$ & Coordinate system on $\mathcal{T}_{\g,\n}^*$.
  & \S \ref{s:coord}\\
  $L(q)$, $\eps L(q)$ & Length of $B_q$ and its signed version.
  & Nota \ref{nota:signL} \\
  \midrule $\mathbf{P}_1, \ldots, \mathbf{P}_r$ & Decomposition of $\mathbf{S}$ in pairs of pants.
  & \S \ref{s:ppdecompo} \\
  \makecell[l]{$\Sigma_0, \ldots, \Sigma_r$ \\ $\mathbf{S}_0, \ldots, \mathbf{S}_r$} & Auxilliary
  families of multi-curves and surfaces.
  & \S \ref{s:ppdecompo} \\
  \midrule $(\Gamma_\lambda)_{\lambda \in \Lambda}$ & Set of boundary components of
  $(\mathbf{P}_1, \ldots, \mathbf{P}_r)$.
  & \S \ref{sec:Lambda_def} \\
  $\LambdaBC$, $\Lambdain$ & Indices $\lambda$ such that $\Gamma_\lambda$ is a boundary / inner
  curve.
  & \S \ref{sec:part_lambda} \\
  $\Lambdabeta$, $\LambdaFN$ & Indices $\lambda$ such that $\Gamma_\lambda$ is / isn't a component
  of $\beta$.
  & \S \ref{sec:part_lambda} \\
  $\LambdaBCbeta$, $\Lambdainbeta$, $\LambdaBCFN$, $\LambdainFN$ & Intersection of the above sets.
  & \S \ref{sec:part_lambda} \\
  $\mathrm{step}(\lambda)$, $\lambda < \lambda'$ & Step when $\Gamma_\lambda$ is constructed, and
  induced order on $\Lambda$.
  & Def \ref{def:order_Lambda} \\
  \midrule $y_\lambda, \alpha_\lambda$ & Fenchel--Nielsen parameters associated to
  $\mathbf{P}_1, \ldots, \mathbf{P}_r$.
  & \S \ref{s:FNc} \\
  $\domain$ & Image of the map $Y \mapsto (\vec{L},\vec{\theta})$.
  & Thm \ref{t:maindet}\\
  $V^\beta$, $V^\Gamma$
  & Intersection of $V$ with $\Lambdabeta$ and $\LambdaFN$.
  & Cor \ref{c:maindet-dirac} \\
  $T_{V,\mathrm{m}}$
  & Expression of a Dirac-distribution in $(\vec{L},\vec{\theta})$.
  & Cor \ref{c:maindet-dirac} \\
  \midrule
  $a^t$, $w^t$, $k^\theta$
 & Matrices in $\mathrm{PSL}_2(\IR)$ representing different moves in $T^1\IH^2$.
  & \S \ref{s:sl2} \\
  $\bar{p}(q) = \overline{B}_q \smallbullet \overline{\cI}_q$
                           & Path associated to $q \in \Theta$.
  & \S \ref{s:lift}\\
  $\tilde{\beta}_q$, $\tilde{B}_q$, $\tilde{\cI}_q$, $\tilde{p}(q)$
                           & Lifts of $\beta_{\lambda(q)}$, $\overline{B}_q$, $\overline{\cI}_q$ and
                             $\bar{p}(q)$ to $\IH^2$, for $q \in \mathbf{\Theta}^i$.
  & \S \ref{s:lift}\\
  $\tilde{B}_{mn}$, $L_{mn}$, $z_{mn}$
                           & Orthogeodesic from $\tilde{\beta}_m$ to $\tilde\beta_n$, length and terminus.
                       & \S \ref{s:lift}\\        
  \midrule
    $\hyp_+$ and $\hyp_-$ & The hyperbolic functions $\cosh$ and $\sinh$.
  & Nota \ref{nota:hypdelta} \\
  $\hyp_\delta(t)$ & The product of the hyperbolic functions $\hyp_{\delta_i}(t_i)$.
  & Nota \ref{nota:hypdelta} \\
  $E_\pm(s)$
                           & Useful matrices.
  & Eq \eqref{e:epm} \\
  $M_n(x,t)$
  & Function used to compute $\ell_Y(\mathbf{c}_i)$ as a function of $\vec{L},\vec{\theta}$.
  & Nota \ref{nota:Mn}\\
  \midrule
  $\cK_{qq'}$, $\overline{\cK}_{qq'}$, $\tau_{qq'}$
                           & Segment between two bars of the same sign, and its length.
  & \S \ref{s:tau} \\
  $\cK_q$, $\tau_q$
                           & Special case where $q'$ is the next bar of the same side.
  & \S \ref{s:tau} \\
  $\Theta_\tau(qq')$
                           & Set of indices needed to write $\tau_{qq'}$ in terms of $\vec\theta$.
  & \S \ref{s:tau} \\
  $\Theta_\tau(W)$
                           & Set of indices needed to write $(\tau_q)_{q \in W}$ in terms of $\vec\theta$.
  & \S \ref{s:tau} \\
  \midrule
  $\bar{p}_k^\lambda$,  $q_k^\lambda$, ${q'}_k^\lambda$
                           & Segments appearing when we describe $\Gamma_\lambda$ and their indices.
  & \S \ref{s:poly_curve_glambda} \\
  $m(\lambda)$
                           & Number of segments in the writing of $\Gamma_\lambda$.
  & \S \ref{s:poly_curve_glambda} \\
  $\vec{L}^\lambda$, $\vec{\tau}^\lambda$
                           & Length parameters appearing in $\Gamma_\lambda$.
  & \S \ref{s:poly_curve_glambda} \\
  $\nabla \alpha$, $\delta \alpha$
                           & Sequences of indices associated to $\alpha$.
  & \makecell[l]{Thm \ref{prp:length_g_lambda}\\ Nota \eqref{nota:QFm}} \\
  $Q_m(x,t)$, $F_m(x,t)$
                           &  Functions used to compute $\ell_Y(\Gamma_\lambda)$ as a function of $\vec{L},\vec{\theta}$.
  & Nota \ref{nota:QFm}\\
  $\Theta_\ell(\lambda)$, $\Theta_\ell(W)$
  & Sets of indices needed to write $\ell_Y(\Gamma_\lambda)$, or $(\ell_Y(\Gamma_\lambda))_{\lambda \in
    W}$.
  & Nota \ref{nota:indlambda} \\
  \midrule
  $V^\beta$, $V^\Gamma$, $V_+^\beta$, $V_+$
                           & Sets of variables of interest in $\Lambda$. 
  & \S \ref{s:ch_var_applied} \\
  $W, W^\beta, W^\Gamma$, $r_\lambda$
                           & Set of variables for which we took a remainder term $r_\lambda$.
  & \S \ref{s:conc_subtit_all} \\
  $d_\lambda$, $\rK_q$
                           & Exponents appearing in the integral $\cJ$.
  & \S \ref{s:conc_subtit_all} \\
  $\ThetaAc$, $\ThetaNe$
                           & Set of active and neutral parameters.
  & Nota \ref{nota:ac_ne} \\
  $\thetaAc$, $\thetaNe$
                           & Corresponding subvectors of $\vec{\theta}$.
  & Nota \ref{nota:ac_ne} \\
  \midrule
  $f_1\star \ldots \star f_n|^h_\varphi$
                           & Pseudo-convolution of $(f_i)_{1 \leq i \leq n}$ with height $h$ and weight $\varphi$.
  & Nota \ref{nota:pc} \\
  $\frac{\d x}{\d \ell}=\frac{\d x_1 \ldots \d x_n}{\d \ell}$
                           & Desintegration of Lebesgue measure on the level-sets of $h$.
  & Eq \eqref{e:density}\\
  $F_j^\rK$, $G_j$, $H_j^t$, $I_j^t$
                           & Auxilliary functions associated to $f_j$.
  & Nota \ref{nota:FGHI} \\
  $\mathbf{Int}_{\pi}$, $B$, $\Phi_j$
                           & \makecell[l]{Integrals we obtain when we apply our operator $\cL$, \\
  and associated quantities.}
                           & Thm \ref{t:thebigone} \\
  $E_n^{(a)}$, $\cE_n^{(a)}$
                           & Functional spaces in which we hope to find $\varphi$ and $h$.
  & Def \ref{def:Ena}, \ref{def:cE} \\
  $\| \cdot \|_{\alpha, a}$
                           & Semi-norms on $E_n^{(a)}$.
  & Def \ref{def:Ena}\\
  \textbf{($h$)}, \textbf{($\varphi$)}, \textbf{($f$)} 
                           &
                             Assumptions on the functions $h$, $\varphi$, $(f_j)_{1 \leq j \leq n}$.
  & Nota \ref{nota:assum_1}\\
  \textbf{($\varphi$-CE)}, $Z_j$
                           & Assumption as a comparison estimate and height function.
  & Nota \ref{nota:phi_CE} \\
  \midrule
  $\overline{B}_{ij}$, $L_{ij}$, $z_{ij}$ &
                                            Orthogeodesics, their lengths and endpoints.
  & \S\ref{s:fam_aligned} \\
  $I_j^\delta$
                           & The crossing intervals.
  & Nota \ref{nota:cross_int} \\
  $\Xi(I_j^\delta)$, $P_{\xi_j^\delta}$, $I_j^\delta(\xi)$
                           & The partitioning associated to the crossing interval $I_j^\delta$.
  & \S \ref{sec:appl-eras-proc} \\
  $\Xi^i$, $P_{\xi_i}$, $I(\xi^i)$, $n_{\xi^i}$
                           & The partitioning associated to $\mathbf{c}_i$.
  & Nota \ref{nota:erasing_i} \\
  $\cI_q'=[G_q,D_q]$, $\tilde{\theta_q}$, $\tilde{L}_q$
                           & The new segments and lengths after erasing.
  & Nota \ref{nota:erase_lengths} \\
  $\Xi$, $P_{\xi}$, $I(\xi)$
                           & Product over all partitions for $1 \leq i \leq \cc$.
  & Nota \ref{sec:product-over-all} \\
  $\fcr{0}$, $\fcr{\infty}$
                           & Functions used to construct partition of unity.
  & Nota \ref{nota:fcr} \\
  $(\Psi_{\xi, \cQ})_{\xi \in \Xi, \cQ \subseteq I(\xi)}$
                           & Partition of unity to deal with crossing parameters.
  & Def \ref{defa:part_unity_cross} \\
  $\ThetaNe(\xi, \cQ)$
                           & Neutral parameters associated to $(\xi,\cQ)$.
  & Nota \ref{nota:neutr_cr} \\
  \midrule
  $\Gamma = (J_j,K_j)_{j \in \Z_m}$
                           & Polygonal curve, its bridges and cells.
  & Def \ref{d:pc} \\
  $\Cell(\Gamma)$, $\Br(\Gamma)$
                           & Cells and bridges of $\Gamma$.
  & Def \ref{def:cellbr} \\
  $|\Gamma|$
                           & Combinatorial length of $\Gamma$.
  & Nota \ref{nota:clg} \\
  $L_j$, $\tau_j$
                           & Length of geodesic representatives of the bridges and cells.
  & \S \ref{s:geod_rep_pc} \\
  $Z_K^\Gamma$
                           & Height of a cell.
  & Def \ref{d:height} \\
  $\fcell{0}$, $\fcell{\infty}$, $a$
                           & Functions used to construct partition of unity, cutoff value.
  & Nota \ref{nota:a_Z} \\
  $(\Gamma,d)$
                           & Decorated polygonal curve.
  & Def \ref{d:deco} \\
  $(\mathbf{\Gamma},\mathbf{d})$
                           & Exhaustion.
  & Def \ref{d:exhaust} \\
  $(\Gamma^\infty,d^\infty)$
                           & Terminal step of the exhaustion.
  & Nota \ref{nota:term_ex} \\
  $\Term({\mathbf{\Gamma}}, {\mathbf{d}})$
                           & Set of terminal cells of the exhaustion.
  & Nota \ref{d:terminal} \\
  $(\Psi_{{\mathbf{\Gamma}}, {\mathbf{d}}})_{(\mathbf{\Gamma}, \mathbf{d}) \in \mathfrak{E}}$
                           & Partition of unity associated to $(\Gamma_\lambda)_{\lambda \in \Lambda}$.
  & \S \ref{s:supercutoff} \\
  $\TermBC$
                           & Set of boundary terminal cells of an exhaustion in $\mathfrak{E}$.
  & Nota \ref{nota:termbc} \\
  $\Thetasurv$, $\tausurv$
                           & Set of surviving boundary segments and their
                             lengths. 
  & Def \ref{defa:surv} \\
  $Z_{{\mathbf{\Gamma}}, {\mathbf{d}}}(\cK_q)$
                           & Height of a surviving boundary segments.
  & Def \ref{defa:surv} \\
  $\ThetaNe({\mathbf{\Gamma}}, {\mathbf{d}})$,
  $\thetaNe({\mathbf{\Gamma}}, {\mathbf{d}})$
                           & Neutral variables associated to $({\mathbf{\Gamma}}, {\mathbf{d}})$ and their vector.
  & Nota \ref{nota:neutral_cell} \\
  \bottomrule
\end{longtable}

\pagebreak

%%%Local Variables: 
%%% mode: latex
%%% TeX-master: "main"
%%% End: 

\bibliographystyle{plain}
\bibliography{bibliography}
 \end{document}